\documentclass[11pt]{amsart}

\usepackage[english]{babel}
\usepackage{amssymb}
\usepackage{amsmath}
\usepackage{amsthm}
\usepackage{color}
\numberwithin{equation}{section}

\usepackage{amscd}
\usepackage{a4wide}

\usepackage[utf8]{inputenc}
\usepackage{amsmath,amsthm,amsfonts,amssymb}
\usepackage{braket}
\usepackage{hyperref}
\usepackage{tikz}

\newtheorem{theorem}{Theorem}[section]
\newtheorem{definition}[theorem]{Definition}
\newtheorem{lemma}[theorem]{Lemma}
\newtheorem{proposition}[theorem]{Proposition}
\newtheorem{corollary}[theorem]{Corollary}
\newtheorem{remark}[theorem]{Remark}
\newtheorem{examples}[theorem]{Examples}
\newtheorem{example}[theorem]{Example}
\usepackage{amscd}

\newcommand{\hh}{{\mathbb{H}}}

\newcommand{\cc}{{\mathbb{C}}}
\newcommand{\rr}{{\mathbb{R}}}

\title[New expansions of harmonic, regular
and poly functions]
{New Taylor and Laurent series of axially harmonic, Fueter regular and polyanalytic functions}
\date{}

\author[F. Colombo]{Fabrizio Colombo}
\address{(FC) Politecnico di Milano, Dipartimento di Matematica, Via E. Bonardi 9, 20133 Milano, Italy}
\email{fabrizio.colombo@polimi.it}

\author[A. De Martino]{Antonino De Martino}
\address{(ADM) Politecnico di Milano, Dipartimento di Matematica, Via E. Bonardi 9, 20133 Milano, Italy}
\email{antonino.demartino@polimi.it}

\author[I. Sabadini]{Irene Sabadini}
\address{(IS) Politecnico di Milano, Dipartimento di Matematica, Via E. Bonardi 9, 20133 Milano, Italy}
\email{irene.sabadini@polimi.it}

\thanks{The first and second author are members of GNAMPA of INdAM; the third author is a member of GNSAGA of INdAM.}

\begin{document}
	
	\maketitle

	\begin{abstract}
		The Fueter-Sce mapping theorem stands as one of the most profound outcomes in complex and hypercomplex analysis, producing hypercomplex generalizations of holomorphic functions.
		In recent years, delving into the factorization of the second operator appearing in the Fueter-Sce mapping theorem has uncovered its potential to generate novel classes of functions and their respective functional calculi.
		The sets of functions obtained from this factorization and the associated functional calculi define the so-called {\em fine structures on the $S$-spectrum}. This paper aims to comprehensively investigate the function theories for the fine structures of Dirac type in the quaternionic framework, presenting new series expansions for axially harmonic, Fueter regular, and axially polyanalytic functions.
		These series expansions are highly nontrivial.
		In fact, when considering  the hypercomplex realm, specifically the quaternionic or the Clifford setting, extending the concept of complex power series expansion is not immediate, and different Taylor and Laurent expansions appear with different sets of convergence.
		Additionally, our objectives include establishing the representation formulas for these function spaces; such formulas encode the fundamental properties of the functions and have numerous consequences.
		Finally, in the last section of this paper,
		we explain the applications of the fine structures in operator theory.

	\end{abstract}
	
	
	
	\tableofcontents
	
	\section{Introduction}

	In hypercomplex analysis the concept of series expansion diverges significantly from the one of complex analysis.
	Specifically, in the quaternionic or Clifford algebra setting, this concept deeply varies
	depending on the type of hyperholomorphic functions being considered.
	In fact, when considering Cauchy-Fueter regular functions in the quaternionic setting, commonly known as Fueter regular functions, the powers of a quaternion are not regular in this sense.
	Similar considerations hold for monogenic functions, i.e., for functions in the kernel of the Dirac operator for Clifford algebra-valued functions, see \cite{BLUEBOOK,green,GILBERT}.
	The analogues of complex power series in this context require, for example, the Fueter polynomials, or may exploit the Clifford-Appell polynomials or the Gelfand-Tsetlin bases, see \cite{BOCK} and the references therein.
	We note that this function theory is important for its links with harmonic analysis and also since it has an associated spectral theory based on the monogenic spectrum which holds for special classes of quaternionic and Clifford operators, see \cite{JBOOK, TAOBOOK}.
	
	\medskip
	Another class of functions that is more similar, in some sense, to holomorphic functions are the so-called slice hyperholomorphic functions. They admit a series expansion in terms of the quaternionic variable, but the similarity ends with the expansion around real points. For points that are not on the real line, there are two distinct series expansions with different topologies associated with the convergence. Moreover,
	for slice hyperholomorphic functions there exists the so called {\em structure formula} or {\em representation formula} that is of crucial importance in proving the most important properties of these functions, such as the Cauchy formula,
	and it is the heart of the whole function theory, see
	\cite{CSS, CSS2,GSS}. This class of functions also has an associated spectral theory based on the notion of $S$-spectrum, see \cite{BOOKSCHUR,bookCG,CGK,CSS,GANTMEM}.
	
	\medskip
	The connection between the two classes of hyperholomorphic functions mentioned earlier is established through the second map in the Fueter-Sce-Qian (mapping) theorem. In the quaternionic context this result is simply called Fueter mapping theorem, in honor of Fueter who first proved it in \cite{Fueter}. In the context of Clifford algebras, this extension theorem was proven by Sce in the case of odd dimensions, and by Qian in the even dimensional case. For more in-depth information, interested readers are encouraged to refer to the translation of Sce's work with commentaries in \cite{CSS3} (and the references therein) and the survey paper \cite{Q1}.
	
	\medskip
	In the quaternionic algebra $\mathbb{H}$, the Fueter mapping theorem establishes a connection between slice hyperholomorphic functions and axially Fueter regular functions using the four-dimensional Laplace operator $\Delta$. The different factorizations of the Laplace operator in this theorem, give rise to various classes of functions. These functions contribute to a new branch in spectral theory, specific for quaternionic and for Clifford operators, referred to as {\em the fine structures on the $S$-spectrum}. Further details on applications to operator theory will be provided in the concluding section of this paper.
	
	\medskip
	It is worthwhile to highlight that there is another connection between slice hyperholomorphic functions and Fueter regular or monogenic functions, established through the Radon and dual Radon transform, which is studied in the paper \cite{CLSSmathAn}.

	\medskip
	The primary goal of this work is to systematically explore function spaces derived from fine structures in the quaternionic case. Our specific objectives include demonstrating representation formulas for these functions and determining various series expansions for axially Fueter regular, axially harmonic, and axially polyanalytic functions of order 2 in the vicinity of a generic quaternion, i.e., we determine the $*$-Taylor and the spherical expansions for each class of functions. Moreover, the respective Laurent series expansions are studied in detail. Due to the noncommutativity of the quaternions, the left and right Taylor and Laurent series differ. Throughout this paper, we restrict our attention to the left case, as the right case can be handled in a completely analogous way with suitable modifications.
	
	To provide precise definitions for the function spaces under investigation in this paper, we need additional definitions.
	The algebra of quaternions $\mathbb H$ is defined as
	\begin{equation}\label{mathH}
 \mathbb{H}:= \{q=q_0+q_1e_1+q_2e_2+q_3e_3 \ \  | \ \   q_0, q_1, q_2, q_3 \in \mathbb{R}\},
 \end{equation}
	where the imaginary units satisfy the relations
	$$ e_1^2=e_2^2=e_3^2=-1,$$
	and
	$$ e_1e_2=-e_2e_1=e_3, \, \, e_2e_3=-e_3e_2=e_1, \, \, e_3e_1=-e_1e_3=e_2.$$

	We use the notation $\mathcal{SH}_L(U)$, see Definition \ref{sh}, to represent the set of left slice hyperholomorphic functions  defined on an axially symmetric domain $U$, and similarly $\mathcal{AM}_L(U)$ denotes the class of left axially Fueter regular functions on $U$ (see Definition \ref{AXREG}). The Fueter mapping theorem is a two-step process that extends the set of holomorphic functions  $\mathcal{O}(\Omega)$ on an open set $\Omega\subseteq \mathbb{C}$,  first to slice hyperholomorphic functions and then to axially Fueter regular functions. Specifically, the Fueter mapping theorem can be illustrated by the following diagram:
	
	\begin{equation*}
		\begin{CD}
			\textcolor{black}{\mathcal{O}(\Omega)} @>T_{F1}>> \textcolor{black}{\mathcal{SH}_L(U_\Omega)} @>\ \ T_{F2}=\Delta >>\textcolor{black}{\mathcal{AM}_L(U_\Omega)}.
		\end{CD}
	\end{equation*}
	
	The first operator, $T_{F1}$, is called the slice operator and extends holomorphic functions from the domain $\Omega\subseteq \mathbb{C}$ to slice hyperholomorphic functions defined on the open set $U_\Omega\subseteq \mathbb{H}$ induced by $\Omega$.
	Meanwhile, the subsequent operator, referred to as $T_{F2}=\Delta$, takes slice hyperholomorphic functions to axially Fueter regular functions.
	It is noteworthy that in contemporary terms, the second map (generally referred to as the Fueter map) takes slice hyperholomorphic functions to Fueter regular functions (or monogenic functions in the case of Clifford algebras). This mapping is then defined on all slice hyperholomorphic functions, and is not limited to be defined only on those arising from holomorphic functions, namely the so-called
	intrinsic slice hyperholomorphic functions. This generality is the reason why, in the sequel, the open set on which slice hyperholomorphic functions are defined  will be denoted by $U$ and not by $U_\Omega$.
	
	\medskip
	Having established the framework, we are now able to introduce the function spaces associated with the fine structures of Dirac type, in the quaternionic case.
	These fine structures are based on the two different factorizations of the second map $T_{F2}$, that is the $4$-dimensional Laplacian
	\begin{equation}\label{Eq_Laplace_operator}
		\Delta= {D}\bar{{D}}=\bar{{D}} {D},
	\end{equation}
	using the Fueter operator (also called Dirac-operator or generalized Cauchy-Riemann operator in the Clifford algebra setting) and its conjugate
	$$
	{D}:=\partial_{q_0}+ e_1\partial_{q_1}+e_2\partial_{q_2}+e_3\partial_{q_3}, \ \ \ \
	\bar{D}:=\partial_{q_0}- e_1\partial_{q_1}-e_2\partial_{q_2}-e_3\partial_{q_3},
	$$
respectively.
	Depending on whether $D$ or $\bar{D}$ acts first on $f\in\mathcal{SH}(U)$, we get different function spaces.

	The operators $D$ and $\bar{D}$ may act  on left or on right slice hyperholomorphic functions; the two function spaces are somewhat equivalent but different and for this reason in this work we
	concentrate on the case left slice hyperholomorphic functions $\mathcal{SH}_L(U)$. In short, the Fueter mapping theorem translates into the statement:
	$$
	D\Delta f(q)=0, \ \ \ \ \ \ {\rm for \ all} \ \ f\in  \mathcal{SH}_L(U).
	$$
	
	\noindent
Next we define the set of {\em  $D$-axially harmonic functions} or
	{\em  axially harmonic functions}, for short, which is defined as
	$$
	\mathcal{AH}_L(U):=D(\mathcal{SH}_L(U))=\Set{ Df | f\in\mathcal{SH}_L(U)}.
	$$
These functions are indeed harmonic and of axial type, by the Fueter construction.
	They are motivated by the factorization $\Delta= \bar{D} D$ which leads to the following re-writing of the Fueter mapping theorem:
	$$
	\Delta (Df(q))=0 \ \ \ \ {\rm for \ all} \ \ f\in  \mathcal{SH}_L(U).
	$$
In summary, the function spaces of the harmonic fine structure of Dirac type are given by
	\begin{equation*}
		\begin{CD}
			\textcolor{black}{\mathcal{SH}_L(U)} @>D >> \textcolor{black}{\mathcal{AH}_L(U)}  @>\ \   \bar{D} >>\textcolor{black}{\mathcal{AM}_L(U)}.
		\end{CD}
	\end{equation*}
	To ease the notation, in the sequel  we shall often omit the subscript $L$ in the
function spaces when it is clear from the context.
	
	\noindent
	Then we define the set of {\em $\bar{D}$-polyanalytic functions of order $2$} or {\em polyanalytic functions of order $2$}, for short, which is defined as
	$$
	\mathcal{AP}_2^L(U):=\bar{D}(\mathcal{SH}_L(U))=\Set{\bar{D}f | f\in\mathcal{SH}_L(U)}.
	$$
	This function space is motivated by
	the following factorization of the Fueter mapping theorem
	$$
	D^2(\bar{D}f(q))=0 \ \ \ \ {\rm for \ all} \ \ f\in  \mathcal{SH}_L(U),
	$$
	so that the function spaces of the polyanalytic fine structure of order 2 and Dirac type are:
	\begin{equation}\label{poly*}
		\begin{CD}
			\textcolor{black}{\mathcal{SH}_L(U)} @>\bar{D}>> \textcolor{black}{\mathcal{AP}^L_2(U)}  @>\ \  D >>\textcolor{black}{\mathcal{AM}_L(U)}.
		\end{CD}
	\end{equation}
	We point out that complex polyanalytic functions are widely studied, see e.g. \cite{Balk1,Balk2}, and they were first investigated by Kolossov in elasticity problems \cite{K1}. These functions have found applications in signal analysis, in the context of Gabor frames, quantum mechanics and other fields, see \cite{AF} for an overview.
	
This discussion shows that the Laplace operator's factorization outlined above reveals two categories of functions which, along with their corresponding functional calculi on the $S$-spectrum, constitute the quaternionic fine structures of Dirac type.

	\medskip

	{\em Contents of the paper.}
	Several properties of quaternionic functions of the fine structures on the $S$-spectrum have been formulated with a focus on their relevance in operator theory.
	To ensure comprehensive coverage, Section \ref{SECT2} provides a detailed exposition of the basic definitions of quaternionic fine structure functions of Dirac type, along with their key properties.
	This section is the starting point for an exhaustive development of the function theory.
	
	\medskip
	In Section \ref{SECTTRE}, we show a relation between the truncated $*$-Taylor and spherical expansions of slice hyperholomorphic functions.
	This connection is established through the global operator of slice hyperholomorphic functions, introduced in \cite{CGS} as an additional way to define slice hyperholomorphicity; this operator is a linear first-order differential operator with nonconstant coefficients.
	
	\medskip
	In Section \ref{REPFPRML}, we present the representation formulas for the quaternionic fine structure of Dirac type. This section holds particular significance in the paper as these formulas encapsulate the essence of each class of functions within the quaternionic fine structure, and play a pivotal role in determining numerous properties of these functions.
	
	\medskip
	Section \ref{HARMSERIES} contains the Taylor series expansion of axially harmonic functions. Specifically, harmonic regular series are derived through the $*$-Taylor series expansion of slice hyperholomorphic functions, while harmonic spherical series result from the spherical expansion of slice hyperholomorphic functions. Section \ref{Sect:6} covers a similar study for Laurent expansions.
	
	\medskip
	The subsequent four sections adhere to a similar structure, focusing on axially Fueter series and polyanalytic functions of order 2. More precisely, in
	Sections \ref{FUETERSERIES}, \ref{Sect:8} we delve into both the Taylor  and Laurent series expansion of axially Fueter regular functions, encompassing both axially Fueter regular series and regular Fueter spherical series.
	Likewise, in Section \ref{POLYSERIES}, \ref{Sect:10} we study polyanalytic regular series of order 2 and polyanalytic spherical series of order 2.

	\medskip
	Finally, in Section \ref{CONCREMK}, we clarify the most recent applications in operator theory that stem from the systematic study of these functions. The purpose of this section is to highlight the significance that the class of functions within the fine structure of Dirac type holds in the advancement of operator theory in both the quaternionic and Clifford settings.

	\section{Functions of the quaternionic fine structure of Dirac type}\label{SECT2}
	
	In this section, we revisit the two fundamental concepts of hyperholomorphicity for qua\-ter\-nio\-nic-valued functions arising from the Fueter mapping theorem. The first concept encompasses the class of slice hyperholomorphic functions, while the second one involves Fueter regular functions
	(or functions in the kernel of the Dirac operator, when considering functions with values in a Clifford algebra).
	Furthermore, the factorization of the second Fueter map $T_{F2}$ leads to two additional classes of functions of the fine structures of Dirac type; some of the properties of these functions are spread in various publications.
	As our focus here is on a systematic study of all these classes of functions,
	we provide a self-contained summary of the main results necessary for our subsequent analysis.
	
	\medskip
	We recall some basic facts on the algebra of quaternions $\mathbb{H}$, see \eqref{mathH}.
	The real part of a quaternion $q=q_0+q_1e_1+q_2e_2+q_3e_3$ is denoted as $ \hbox{Re}(q)=q_0$, while the imaginary part is  $\underline{q}:=q_1e_1+q_2e_2+q_3e_3$. \\
The conjugate of a generic quaternion $q$ is defined as $ \bar{q}:=q_0- \underline{q}$, and the modulus is given by $|q|= \sqrt{\bar{q}q}= \sqrt{q_0^2+q_1^2+q_2^2+q_3^2}$. The inverse of a nonzero quaternion is given by $ q^{-1}:= \frac{\bar{q}}{|q|^2}$, and is called Kelvin inverse. The sphere of unit purely imaginary quaternions is the set
	$$
	\mathbb{S}:= \{ \underline{q}=q_1e_1+q_2e_2+q_3e_3 \ \  | \ \  q_1^2+q_2^2+q_3^2=1\}.
	$$
	We observe that if we consider $I \in \mathbb{S}$ then we have $ I^2=-1$. Therefore an element of the unit sphere $\mathbb{S}$ behaves like an imaginary unit, and so for any $I\in\mathbb{S}$ we can define an isomorphic copy of the complex numbers as
	$$ \mathbb{C}_I:= \{u+Iv \ \  | \ \ u, v \in \mathbb{R}\}.$$
	We associate to a non-real quaternion $q=q_0+ \underline{q}=q_0+ I_{\underline{q}}| \underline{q}|$, where $ I_{\underline{q}}= \frac{\underline{q}}{| \underline{q}|} \in \mathbb{S}$, the 2-dimensional sphere (in short, 2-sphere) defined by
	$$ [q]:= \{q_0+I| \underline{q}| \ \  | \ \ I \in \mathbb{S}\}.$$
	The set of quaternions can be decomposed into complex planes, namely
	
	\begin{equation}
		\label{books}
		\mathbb{H}= \bigcup_{I \in \mathbb{S}} \mathbb{C}_I.
	\end{equation}
	This decomposition is called book structure.

	\subsection{Slice hyperholomorphic functions}
	Slice hyperholomorphic functions admit multiple definitions that are not fully equivalent. In this context, we adopt a notion that proves most suitable for the Fueter mapping theorem and for the spectral theory on the $S$-spectrum, see \cite{CGK}, but we shall present two definitions. As we shall discuss below, on some specific open sets the two definitions give rise to the same function class and so there is no need to specify the definition in use, in other cases one has to specify the context.

We begin by introducing a few notions that will play a crucial role in the sequel, in particular in understanding the various classes of functions linked to the factorization of the operator $T_{F2}$ in the Fueter mapping theorem.

	\begin{definition}
		Let  $U \subseteq \mathbb{H}$.
		\begin{itemize}
			\item We say that $U$ is {\em axially symmetric} if $[q]\subset U$  for any $q \in U$.
			\item We say that $U$ is a {\em slice domain} if $U\cap\rr\neq\emptyset$ and if $U_{I}:=U\cap\cc_{I}$ is a domain in $\cc_{I}$ for any $I\in\mathbb{S}$.
		\end{itemize}
	\end{definition}

	\begin{definition}[Slice Cauchy domain]
		An axially symmetric open set $U\subset \hh$ is called a slice Cauchy domain, if $U\cap\cc_I$ is a Cauchy domain in $\cc_I$ for any $I\in\mathbb{S}$. More precisely, $U$ is a slice Cauchy domain if for any $I\in\mathbb{S}$ the boundary ${\partial( U\cap\cc_I)}$ of $U\cap\cc_I$ is the union of a finite number of non-intersecting piecewise continuously differentiable Jordan curves in $\cc_{I}$.
	\end{definition}
	
The definition below is due to Gentili and Struppa and is based on an idea due to Cullen:

\begin{definition}[Slice hyperholomorphic functions]\label{sho}
Let $U\subseteq\mathbb{H}$ be an open set.
A function $f:U\to \hh$ is called left slice hyperholomorphic
if for all $I\in\mathbb{S}$ the restriction $f_I: U_I\to\mathbb{H}$ has continuous partial derivatives and is holomorphic on $U_I$, namely
$$
\overline{\partial}_If_I(x+yI):=\frac{1}{2}\left(\frac{\partial}{\partial x}+I\frac{\partial}{\partial y}\right)f_I(x+yI)=0.
$$
The function $f$ is called right slice hyperholomorphic
if for all $I\in\mathbb{S}$ the restriction $f_I:U_I\to\mathbb{H}$ has continuous partial derivatives and satisfies
$$
f_I(x+yI)\overline{\partial}_I:=\frac{1}{2}\left(\frac{\partial}{\partial x}f_I(x+yI)+\frac{\partial }{\partial y} f_I(x+yI) I\right)=0.
$$\end{definition}

For the applications to operator theory the next definition works better:

	\begin{definition}[Slice hyperholomorphic functions]\label{sh}
		Let $U\subseteq\hh$ be an axially symmetric open set and let $\mathcal{U} = \{ (u,v)\in\rr^2: u+ \mathbb{S} v\subset U\}$. A function $f:U\to \hh$ is called a left
		slice function, if it is of the form
		\begin{equation}
			\label{form}
			f(q) = \alpha(u,v) + I\beta(u,v)\qquad \text{for } q = u + I v\in U
		\end{equation}
		with two functions $\alpha, \beta: \mathcal{U}\to \mathbb{H}$ that satisfy the compatibility conditions
		\begin{equation}
			\label{EO}
			\alpha(u,v)=\alpha(u,-v), \qquad \beta(u,v)=- \beta(u,-v), \qquad \forall (u,v) \in \mathcal{U}.
		\end{equation}
		If in addition $\alpha$ and $\beta$ satisfy the Cauchy-Riemann-equations
		\begin{equation}
			\label{CR}
			\partial_u \alpha(u,v)- \partial_v \beta(u,v)=0, \quad \partial_v \alpha(u,v)+ \partial_u \beta(u,v)=0.
		\end{equation}
		then $f$ is called left slice hyperholomorphic.
		A function $f:U\to \hh$ is called a right slice function if it is of the form
	\begin{equation}
		\label{form1}
		f(q) = \alpha(u,v) + \beta(u,v)I\qquad \text{for } q = u + I v\in U
	\end{equation}
		with two functions $ \alpha$ and $ \beta: \mathcal{U}\to \mathbb{H}$ that satisfy \eqref{EO}.
		If in addition $\alpha$ and $\beta$ satisfy the Cauchy-Riemann-equations (\ref{CR}), then $f$ is called right slice hyperholomorphic.

	We denote the sets of left and right slice hyperholomorphic functions on $U$ by $\mathcal{SH}_L(U)$ and $\mathcal{SH}_R(U)$, respectively. When no confusion arises, we refer to this class of functions simply as $\mathcal{SH}(U)$.
	
	\end{definition}
	\begin{definition}
		Let $U$ be an axially symmetric open set in $\mathbb{H}$.
		A left or right slice hyperholomorphic function of the form \eqref{form} or \eqref{form1} and  such that $\alpha$ and $\beta$ are real-valued functions is called intrinsic slice hyperholomorphic. The set of intrinsic slice hyperholomorphic functions is denoted by $\mathcal{N}(U)$.
	\end{definition}
	In the theory of slice hyperholomorphic functions, a crucial tool is the representation formula, as highlighted in \cite{CSS}. This formula asserts that any function slice hyperholomorphic according to Definition \ref{sho} on an axially symmetric slice domain can be entirely characterized by its values on two complex planes, using the quaternionic book structure \eqref{books}.

	\begin{theorem}[Representation formula]
		\label{rapp0}
		Let $U$ be an axially symmetric slice domain in $ \mathbb{H}$ and let $I\in\mathbb{S}$. A function $f:U \to \mathbb{H}$ is  a left slice hyperholomorphic  function (according to Definition \ref{sho}) on $U$ if and only if for any quaternion $q=u+I_{\underline{q}}v \in U$ we have
		$$
		f(q)= \frac{1}{2} \left[f(u+Iv)+f(u-Iv) \right]+ \frac{I_{\underline{q}}I}{2}\left[f(u-Iv)-f(u+Iv) \right],
		$$
and the two functions
$$
\alpha(u,v)=\frac{1}{2} \left[f(u+Iv)+f(u-Iv) \right],
$$
$$
\beta(u,v)= \frac{I}{2}\left[f(u-Iv)-f(u+Iv) \right]
$$
depend on $u,v$ but not on $I_{\underline{q}}$.
		A function $f:U\to\mathbb{H}$ is a right slice hyperholomorphic function on $U$ if and only if for any $q=u+I_{\underline{q}}v \in U$ we have
		\begin{equation}\label{SF2}
			f(q)=\frac{1}{2}\Big[f(u-Iv)+f(u+Iv)\Big]
			+\frac{1}{2}\Big[f(u-Iv)-f(u+Iv)\Big]I\,I_{\underline{q}},
		\end{equation}
and the two functions
$$
\alpha(u,v)=\frac{1}{2} \left[f(u+Iv)+f(u-Iv) \right],
$$
$$
\beta(u,v)= \frac{1}{2}\left[f(u-Iv)-f(u+Iv) \right]I
$$
depend on $u,v$ but not on $I_{\underline{q}}$.
	\end{theorem}
	
	\begin{remark}
The Representation Formula is automatically satisfied by slice hyperholomorphic functions according to Definition \ref{sh} since, in fact, the formula holds more in general for all (left or right) slice functions defined on axially symmetric open sets.\\
The Representation Formula shows that, on axially symmetric slice domains, slice hyperholomorphic functions according to Definition \ref{sho} are also such according to Definition \ref{sh}. We can conclude that on axially symmetric slice domains the two classes of functions in Definitions \ref{sho} and \ref{sh} coincide and so it is not necessary to specify which is the definition in use. In the sequel, when we write a slice hyperholomorphic function in form of a slice function, we evidently are referring to Definition \ref{sh}, even when it is not explicitly stated.
\end{remark}
	Independently of the chosen definition, the class of functions $\mathcal{SH}_L(U)$ (resp. $\mathcal{SH}_R(U)$) is a right (resp. left) linear space over $\mathbb H$ with respect to the sum of functions and the multiplication on the right  (resp. left) by a quaternion. \\
Next, one may wonder if $\mathcal{SH}_L(U)$ is closed under multiplication and one of the main differences between the holomorphic and the slice hyperholomorphic functions emerges: the pointwise product does not preserve the slice hyperholomorphicity. See the simple example:
	
	\begin{example}
		Let $f(q)=qa$ and $g(q)=qb$ be two left slice hyperholomorphic functions with $a$, $b \in \mathbb{H} \setminus \mathbb{R}$. Then the pointwise product is given by
		$$f(q) \cdot g(q)=qaqb.$$
		It is clear that the above function is not slice hyperholomorphic.
	\end{example}
	For this reason, a suitable product  preserving the slice hyperholomorphicity, called the $*$-product, has been introduced. For more information we refer the reader to the books \cite{CSS,CSS2,GSS}.
	
	\begin{definition}
	\label{staraxial}
		 Let $f=\alpha+I\beta$, $g=\alpha_1+I\beta_1 $ be two left slice functions. We define the left $*$-product as
		$$ f*_L g=(\alpha \alpha_1-\beta \beta_1)+I (\alpha \beta_1+\beta \alpha_1).$$
		If $f=\alpha+\beta I$, $g=\alpha_1+\beta_1 I$ are two right slice functions, their right $*$-product is defined as
		$$ f*_R g=(\alpha \alpha_1-\beta \beta_1)+(\alpha \beta_1+\beta \alpha_1)I .$$
	\end{definition}
	
	\begin{proposition}
	\label{closedp}
Let $U \subseteq \mathbb{H}$ be an axially symmetric open set. The set $\mathcal{SH}_L(U)$ is closed with respect to the $*_L$-multiplication. Similarly, $\mathcal{SH}_R(U)$  is closed with respect to the $*_R$-multiplication.
	\end{proposition}
	
	The left (resp. right) $*$-product is associative and distributive but is not commutative. The $*$-product coincides with the pointwise product if at least one of the two functions is intrinsic slice hyperholomorphic. Moreover in this case the $*$-product is also commutative, namely we have
	$$ f*_Lg=fg=g*_L f, \qquad \left( f*_Rg=fg=g*_R f \right).$$
	\begin{example}
		We  suppose that $ \{ a_i \}_{i \geq 0}$, $ \{b_i \}_{i \geq 0}$ are sequences of elements of $\mathbb{H}$.
		
		Let  $f(q)= \sum_{i=0}^{n} q^i a_i$ and $g(q)= \sum_{i=0}^{m} q^i b_i$, be two left slice hyperholomorphic polynomials. Then the left $*$-product is given by
		$$ (f*_Lg)(q)= \sum_{j=0}^{n+m} q^j \left(\sum_{i+k=j} a_i b_k \right).$$
		
		Let $f_1(q)= \sum_{i=0}^{n} a_i q^i $ and $g_1(q)= \sum_{i=0}^{m} b_i q^i $ be two right slice hyperholomorphic polynomials.
		Then the right $*$-product is given by
		$$ (f_1 *_R g_1)(q)= \sum_{j=0}^{n+m} \left(\sum_{i+k=j} a_i b_k \right) q^j.$$
	\end{example}
	\begin{definition}
		Let $U \subseteq \mathbb{H}$ be an axially symmetric open set. Let $f=\alpha+I\beta \in \mathcal{SH}_L(U)$. The slice hyperholomorphic conjugate is defined as $f^c=\bar{\alpha}+I \bar{\beta}$. The symmetrization of a function $f$ is given by $f^s= f*_L f^c=f^c*_L f$. Similarly for $f=\alpha+\beta I \in \mathcal{SH}_R(U)$ the slice hyperholomorphic conjugate is given by $f^c=\bar{\alpha}+\bar{\beta}I$ and its symmetrization is given by $f^s= f*_R f^c=f^c*_R f$.
	\end{definition}
	We note that the function $f^s$ is intrinsic slice hyperholomorphic.
	\begin{example}
		Let $\{a_i\}_{i \geq 0} \subseteq \mathbb{H}$ and $f(q)= \sum_{i=0}^{n} q^i a_i$. Then the slice hyperholomorphic conjugate of $f$ is given by
		$$
		f^c(q)=\sum_{i=0}^{n} q^i \bar{a}_i,
		$$
while its symmetrization is
		$$f^s(q)=\sum_{i=0}^{2n} q^i c_i, \quad c_i:= \sum_{r=0}^{i}a_r \overline{a}_{i-r}.$$
	\end{example}
	
The inverse of a slice hyperholomorphic function, in general, does not give a slice hyperholomorphic function while the notion of $*$-inverse is defined below.
	
	\begin{definition}
		Let $U \subseteq \mathbb{H}$ be an axially symmetric open set. Let $f \in \mathcal{SH}_L(U)$ with $f \neq 0$. The left slice hyperholomorphic inverse of $f$ is defined on $U\setminus \mathcal{Z}_{f^s}$ as
		\begin{equation}
			\label{investar}
			f^{-*_L}:= (f^s)^{-1} f^c,
		\end{equation}
		where
		$$
		\mathcal{Z}_{f^s}:= \{q \in U \, : \, f^s(q)=0\}.
		$$
		Let $f \in \mathcal{SH}_R(U)$ with $f \neq 0$. The right slice hyperholomorphic inverse of $f$ is defined on $U\setminus \mathcal{Z}_{f^s}$  by
		$$
		f^{-*_R}:= f^c(f^s)^{-1}.
		$$
	\end{definition}
	The left (resp. right) slice hyperholomorphic inverse satisfies the following properties, see \cite[Corollary 2.1.20]{CGK}.
	\begin{lemma}
		\label{investar1}
		Let $U \subseteq \mathbb{H}$ be an axially symmetric open set. The following statements hold:
		\begin{itemize}
			\item[1)] Let $f \in \mathcal{SH}_L(U)$ (resp. $f \in \mathcal{SH}_R(U)$) with $f \neq 0$. The left (resp. right) slice hyperholomorphic inverse is defined on $U \setminus \mathcal{Z}_{f^s}$.
			Moreover the left (resp. right) slice hyperholomorphic inverse satisfies the following relations
			$$ f^{-*_L}*_L f= f*_Lf^{-*_L}=1, \qquad\left( \hbox{resp.} \, \, f^{-*_R}*_R f=f *_R f^{-*R}=1 \right).$$
			\item[2)] Let $f \in \mathcal{N}(U)$ and $f \neq 0$, then $f^{-*_L}=f^{-*_R}=f^{-1}$.
		\end{itemize}
	\end{lemma}

	Slice hyperholomorphic functions can be expressed via a Cauchy integral formula.
	The key distinction with respect to the case of holomorphic functions is that in the slice hyperholomorphic context the Cauchy kernel contains a quadratic expression.
	
	\begin{definition}
		\label{Cauchykernels}
		Let $p$, $q \in \mathbb{H}$, with $q \notin [p]$. Then the left slice hyperholomorphic Cauchy kernel $S_L^{-1}(p,q)$ is the function
		\[
		\begin{split}
			S_L^{-1}(p,q)=(q^2-2p_0 q+|p|^2)^{-1}(\bar{p}-q).
		\end{split}
		\]
		The right slice hyperholomorphic Cauchy kernel $S_R^{-1}(p,q)$ is the function
		\[
		\begin{split}
			S_R^{-1}(p,q)=(\bar{p}-q)(q^2-2p_0 q+|p|^2)^{-1}.
		\end{split}
		\]
	\end{definition}
The next proposition shows that the Cauchy kernels can be written in two forms.
	\begin{proposition}
		\label{relform}
		If $q, p\in\hh$ with $q\not\in [p]$, then
		\begin{gather}\label{secondAAEQ}
			-(q^2 -2q {\rm Re} (p)+|p|^2)^{-1}(q-\overline p)=(p-\bar q)(p^2-2{\rm Re} (q)p+|q|^2)^{-1}
		\end{gather}
		and
		\begin{gather}\label{secondAAEQ1}
			(p^2-2{\rm Re}(q)p+|q|^2)^{-1}(p-\bar q)=-(q-\bar p)(q^2-2{\rm Re}(p)q+|p|^2)^{-1} .
		\end{gather}
	\end{proposition}
In view of this result we have:
	\begin{proposition}
		Let $p$, $q \in \mathbb{H}$, with $q \notin [p]$. Then the left slice hyperholomorphic Cauchy kernel $S_L^{-1}(p,q)$  can be written in two  equivalent forms:
		\[
		\begin{split}
			&S_L^{-1}(p,q)=(q^2-2p_0 q+|p|^2)^{-1}(\bar{p}-q), \qquad \rm{(I)}
			\\
			& S_L^{-1}(p,q)=(p- \bar{q})(p^2-2q_0 p+|q|^2)^{-1}, \qquad  \rm{(II)}.
		\end{split}
		\]
		Similarly,
		the right slice hyperholomorphic Cauchy kernel $S_R^{-1}(p,q)$ can be written in two  equivalent forms:
		\[
		\begin{split}
			&
			S_R^{-1}(p,q)=(\bar{p}-q)(q^2-2p_0 q+|p|^2)^{-1}, \qquad \rm{(I)}
			\\
			&
			S_R^{-1}(p,q)=(p^2-2q_0 p+|q|^2)^{-1}(p- \bar{q}), \qquad \rm{(II)}.
		\end{split}
		\]
	\end{proposition}
	
\begin{remark}
\label{NstarLL}
In the sequel, when dealing with the kernels $S_L^{-1}(p,q)$, $S_R^{-1}(p,q)$ seen as functions in two variables, it will be useful to observe that by Proposition \ref{relform} we have $(q-p)^{-*_{q,L}}=(q-p)^{-*_{p,R}}$.\\
We also recall the following simple, yet important, observation which will be used a number of times in the sequel. Denoting by $*_{q,L}$ and $*_{p,R}$ the left $*$-product in $q$ and the right $*$-product in $p$, respectively, we have the obvious equality:
	\begin{equation}\label{starLeR}
		(q-p)^{n*_{q,L}}=\sum_{r=0}^n {n\choose r} q^r(-p)^{n-r}=(q-p)^{n*_{p,R}}.
	\end{equation}	
Note that \eqref{starLeR} remains valid when $q$ is replaced by $\bar q$, since the definition of the $*$-product (see Definition~\ref{staraxial}) applies to slice functions.
\end{remark}

	\begin{remark}
		The two interchangeable expressions for the Cauchy kernels exhibit significant distinctions, particularly in the realm of operator theory. Specifically, the first form is more suitable for quaternionic operators $T=T_0+T_1e_1+T_2e_2+T_3e_3$ with noncommuting components $T_\ell$ for $\ell=0,...,3$, whereas the second form assumes a pivotal role in the case of operators with commuting components, see the last section of this paper for more details.
	\end{remark}
	We now recall the Cauchy formulas (see \cite{CSS}):
	\begin{theorem}[Cauchy formulas]\index{Cauchy formulas}
		\label{Cauchygenerale}
		Let $U\subset\mathbb{H}$ be a bounded slice Cauchy domain, let $I\in\mathbb{S}$ and set  $dp_I=dp (-I)$.
		If $f$ is a (left) slice hyperholomorphic function on a set that contains $\overline{U}$ then
		\begin{equation}\label{cauchynuovo}
			f(q)=\frac{1}{2 \pi}\int_{\partial (U\cap \mathbb{C}_I)} S_L^{-1}(p,q)\, dp_I\,  f(p),\qquad\text{for any }\ \  q\in U.
			\pagebreak[2]
		\end{equation}
		If $f$ is a right slice hyperholomorphic function on a set that contains $\overline{U}$,
		then
		\begin{equation}\label{Cauchyright}
			f(q)=\frac{1}{2 \pi}\int_{\partial (U\cap \mathbb{C}_I)}  f(p)\, dp_I\, S_R^{-1}(p,q),\qquad\text{for any }\ \  q\in U.
		\end{equation}
		These integrals  depend neither on $U$ nor on the imaginary unit $I\in\mathbb{S}$.
	\end{theorem}

Slice hyperholomorphic functions  also satisfy an identity principle, but the statement is different according to the type of functions we consider.
		\begin{theorem}[Identity Principle]
			\label{Ide}
			Let $U$ be a slice domain in $ \mathbb{H}$ and $f$, $g: U \to \mathbb{H}$ be left (resp. right) slice hyperholomorphic functions according to Definition \ref{sho}. Then if $f=g$ on a non-empty subset of $U_I=U \cap \mathbb{C}_I$, $I \in \mathbb{S}$, having an accumulation point then $f=g$ on $U$.
		\end{theorem}
Another version of the principle is below, where the notation $\mathbb C^+_I$ denotes the upper half-plane associated with the unit $I\in\mathbb S$.
	\begin{theorem}[Identity Principle]
			\label{Ide1}
			Let $U$ be an axially symmetric domain in $ \mathbb{H}$ and $f$, $g: U \to \mathbb{H}$ be left (resp. right) slice hyperholomorphic functions according to Definition \ref{sh}. Then if $f=g$ on non-empty subsets of $U \cap \mathbb{C}^+_I$ and of $U \cap \mathbb{C}^+_K$, $I,K \in \mathbb{S}$, $I\not=K$,  both with an accumulation point, then $f=g$ on $U$.
		\end{theorem}

	It is well known that a  slice hyperholomorphic $f$
	on an open ball $B(p_0,R)$ centered at the real point $p_0$ and of radius $R>0$,
	can be written as
	$$ f(q)= \sum_{n=0}^{\infty} (q-p_0)^n a_n, \qquad \{a_n\}_{n \in \mathbb{N}_0} \subset \mathbb{H},$$
	and the series converges on $B(p_0,R)$.
	
	However, if a function is slice hyperholomorphic in a neighbourhood of a non-real quaternionic point then its series expansion is more delicate.
	In the literature two possible expansions are available. The first one was developed in \cite{GS}, by using the $*$-product, and it is the natural generalization of the Taylor series in the non-commutative setting. However, this expansion has issues with respect to its convergence set, and in \cite{S} a different series expansion is studied. In this latter case, a second degree quaternionic polynomial, called of spherical type, is the building-block of the series.
\\
	We now give the precise definitions and the main properties of the two different series expansions around a generic quaternion.
	For the sake of simplicity, we will state the results only for left slice hyperholomorphic functions,
	but they can be suitably reformulated in the case of right slice hyperholomorphic functions. From now on, if no confusion arises, we shall omit to specify that we are considering the left case.
	\begin{definition}
		\label{reg10}
		Let $U$ be a domain in $ \mathbb{H}$ and $f: U \to \mathbb{H}$ be slice hyperholomorphic on $U$.
		We say that the function $f$ admits $*$-Taylor expansion at the point $p \in U$ if it can be written as
		\begin{equation}
			\label{regser}
			f(q)=\sum_{n=0}^{\infty}(q-p)^{n*_{q,L}} a_n, \qquad \{a_n\}_{n \in \mathbb{N}_0} \subset \mathbb{H}
		\end{equation}
		in a suitable subset of $\mathbb H$ containing $p$.
	\end{definition}
		Sometimes the $*$-Taylor series defined in \eqref{regser} is referred to as a regular series.
		As we shall discuss below, a drawback of considering a series of this nature is its lack of convergence in a Euclidean neighborhood, as discussed in \cite{GS}. To describe the convergence set we need some more terminology.

	With the symbol $\sigma$ we denote the distance defined by
	\begin{equation}
		\label{distance}
		\sigma(q,p)=\begin{cases}
			|q-p|, \qquad \hbox{if} \quad p,q \quad \hbox{lie on the same complex plane} \quad \mathbb{C}_I,\\
			\sqrt{(q_0-p_0)^2+ (|\underline{q}|+| \underline{p}|)^2},\quad \hbox{otherwise}.
		\end{cases}
	\end{equation}
The set  defined by $$ \Sigma(p, R)= \{q \in \mathbb{H} \,: \, \sigma(q,p)<R\}$$
is called $\sigma$-ball with centre $p$ and radius $R>0$.	

	\begin{theorem}[\cite{GS}]
Let $f:\ U\subseteq\mathbb{H}\to\mathbb{H}$ be a slice hyperholomorphic function admitting a series expansion of the form  \eqref{regser} at $p\in U$. Then the series \eqref{regser} is convergent in a suitable $ \Sigma(p, R)\subseteq\mathbb{H}$.
	\end{theorem}
	
	A $\sigma$-ball is not an open set in $\mathbb{H}$ in the Euclidean topology, unless $p \in \mathbb{R}$.
If $p$ is such that $|\underline{p}|<R$ then the $\sigma$-ball $\Sigma(p,R)$ has nonempty interior. If $|\underline{p}|\geq R$ then $\Sigma(p,R)$ reduces to a disc in the complex plane $\mathbb C_I$ containing $p$.
	
However, this problem has recently been overcome with the introduction of the notion of slice topology, see \cite{DRSY,DRS,DJGS}.
	\begin{definition}
		\label{slicetopo}
		 The slice topology on $\mathbb{H}$ is defined as
		$$
		\tau_s(\mathbb{H}):= \{ U \subset \mathbb{H}\, : \, U \, \hbox{is slice-open} \}.
		$$
		A subset $U$ is called slice-open if
		$$
		U_I:= U \cap \mathbb{C}_I,
		$$
		is open in the slice $ \mathbb{C}_I$ for any $I \in \mathbb{S}$.
	\end{definition}
We denote by $\tau(X)$ the Euclidean topology over $X=\mathbb C_I, \mathbb H$. On each $ \mathbb{C}_I$ the slice topology is such that
	$$ \tau_s(\mathbb{C}_I)=\tau(\mathbb{C}_I), \qquad \forall I \in \mathbb{S}.$$
	
	 Let $\tau_{\sigma}$ be the topology induced by the distance $\sigma$. Globally we have the following behaviour:
	\begin{proposition}[\cite{DRS,DJGS}]
		The slice topology on $\mathbb H$ is finer than the topologies $\tau_{\sigma}$ and $\tau$. Precisely, we have
		$$ \tau \subsetneq \tau_{\sigma} \subsetneq \tau_s.$$
	\end{proposition}

\begin{remark} The theory of slice hyperholomorphic functions addressed in \cite{DRS} with the approach of slice topology is more general than the one based on the Euclidean topology.
		In  particular \cite{DRS} extends the representation formula to this framework.
	\end{remark}

	In order to state the result of the convergence of the $*$-Taylor series \eqref{regser} we need to fix some notations. Let $r \in \mathbb{R}^+$ and $p \in \mathbb{C}_I$; we denote the disc with centre $p$ and radius $r$ by
	$$ D_I(p,r):= \{q \in \mathbb{C}_I \ | \ |q-p| \leq r\},$$
	and we set
	\begin{equation}
		\label{setcov}
\begin{split}
		\tilde{P}(p,r):&= \{x+Jy \in \mathbb{H} \ | \ J \in \mathbb{S} \, , \, x\pm yI \in D_I(p,r) \}\cup D_I(p,r)\\
&= \Omega(p,r)\cup D_I(p,r).
\end{split}
	\end{equation}
	
	In \cite{GS} the authors prove the following result that was eventually generalized in \cite{DRSY,DRS,DJGS}:
	\begin{theorem}
		Let $U\subseteq\mathbb{H}$ and $f$ be left slice hyperholomorphic around $p \in \mathbb{H}$. Then the $*$-Taylor series \eqref{regser} converges absolutely and uniformly on the compact subsets of $\tilde{P}(p,r) \subset U$, where $1/r=\lim\sup_{n\to\infty} |a_n|^{1/n}$, and does not converge at any point in $\mathbb H\setminus \overline{\tilde{P}(p,r)}$.
	\end{theorem}
However, when one needs to work with slice hyperholomorphic functions on Euclidean open sets, it is convenient to set $\Omega(p,r)=\{q=x+Jy \ | \ x+Iy\in D_I(p,r)\cap D_I(\bar p,r)\}$ and we have, see \cite{GS}:
\begin{corollary}
If $\Omega(p,r)\not=\emptyset$ the $*$-Taylor series \eqref{regser} centered at $p$ defines a slice hyperholomorphic function on $\Omega(p,r)$.
\end{corollary}

		The results below are proved  for left slice hyperholomorphic functions, since the corresponding  results
		for right slice hyperholomorphic functions follow with analogous considerations.
	\\
	Firstly, we prove  that the slice hyperholomorphic Cauchy kernel admits a $*$-Taylor expansion at a quaternionic point $p$ and converges in a suitable set.
	\begin{proposition}
		Let $p$, $q \in \mathbb{H}$. Then we have
		\begin{equation}
			\label{ex1}
			S_{L}^{-1}(p,q)=- \sum_{n=0}^{\infty} (1-q+p)^{n*_{q,L}}, \qquad q \in \Omega(p+1,1).
		\end{equation}
	\end{proposition}
	\begin{proof}
For the sake of simplicity, below we indicate simply by $*$ all the $*_{q,L}$-products in the variable $q$. First we show that the series
		$$ \sum_{n=0}^{\infty} (1-q+p)^{*n}= \sum_{n=0}^{\infty} (-1)^n(q-(p+1))^{*n},$$
converges for $q \in \Omega(p+1,1)$, see \eqref{setcov}. Here and in the sequel the imaginary unit $I\in\mathbb{S}$ used to define $\Omega(p+1,1)$, see \eqref{setcov}, is chosen so that $p\in\mathbb{C}_I$; with this choice the restriction to $\mathbb{C}_I$ of the $*$-power $(q-(p+1))^{*n}$ is the ordinary power $(q_{\pm I}-1-p)^{n}$. By Theorem \ref{rapp0} we have
$$ |(q-(p+1))^{*n}| \leq |(q_{I}-1-p)^n|+|(q_{-I}-1-p)^n|,$$
where $ q_{\pm I}=x \pm yI$. Since both the series
$$ \sum_{n=0}^{\infty} |(q_{\pm I}-p-1)^n|$$
are convergent for $|q_{\pm I}-p-1|<1$, it follows that
$$ \sum_{n=0}^{\infty}|(1-q+p)^{*n}|,$$
is convergent. This implies that also the series in \eqref{ex1} is convergent in $\Omega(p+1,1)$.
\\ Now, we prove the equality in \eqref{ex1}. We set
\begin{equation}
\label{ex2}
S(p,q):= \sum_{n=0}^{\infty} (1-q+p)^{*n}.
\end{equation}
		We compute the left $*$-product in $q$, still denoted by $*$, of the above formula with $1-q+p$ and we get
		\begin{eqnarray}
			\nonumber
			S(p,q)*(1-q+p)&=& \sum_{n=0}^{\infty} (1-q+p)^{*(n+1)}\\
			\nonumber
			&=& \sum_{n=1}^{\infty} (1-q+p)^{*n}\\
			\nonumber
			&=& \sum_{n=0}^{\infty} (1-q+p)^{*n}-1\\
\nonumber
			&=& S(p,q)-1.
		\end{eqnarray}
This implies that
		\begin{eqnarray}
			\label{ex5}
			1&=& S(p,q)-S(p,q)*(1-q+p)\\
			\nonumber
			&=& S(p,q)*(1-1+q-p)\\
			\nonumber
			&=& S(p,q)*(q-p).
		\end{eqnarray}
	We note that $(q-p)^{-*_{q,L}}$ is defined since $q\in\Omega(p+1,1)$ and so $q\not\in [p]$; thus, by left $*$-multiplying  \eqref{ex5} on the right by $(q-p)^{-*}$, we obtain
\begin{equation}
\label{S1}
 S(p,q)=(q-p)^{-*}=-S^{-1}_L(p,q),
\end{equation}
		where $S^{-1}_L(p,q)$ is the slice Cauchy kernel in the second form, see Definition \ref{Cauchykernels}. By plugging \eqref{S1} in formula \eqref{ex2} we get the result.
	\end{proof}

	In \cite{S} the author studied another, different expansion of a slice hyperholomorphic function around a point $p$ written in terms of the following quadratic polynomial which is slice hyperholomorphic in $q$:
	
	\begin{equation}
		\label{set1}
		Q_p^n(q):= ((q-p_0)^2+p_1^2)^n=(q^2-2p_0q+|p|^2)^n, \qquad q \in \mathbb{H},
	\end{equation}
	where $n\in\mathbb N$, $ p=p_0+Ip_1 $, with $p_0 \in \mathbb{R}$, and $I \in \mathbb{S}$.
	\begin{definition}
		\label{sheri}
		Let $U$ be an axially symmetric domain in $\mathbb{H}$ and let $f$ be a left slice hyperholomorphic function on $U$.
		We say that the function $f$ admits a spherical expansion at an arbitrarily fixed point $p \in U$ if it can be written as
		\begin{equation}
			\label{spherical}
			f(q)= \sum_{n=0}^{\infty} Q_p^n(q) [a_{2n}+(q-p) a_{2n+1}],
		\end{equation}
		where $\{a_n\}_{n \in \mathbb{N}_0} \subset \mathbb{H}$, and the series converges in a suitable neighbourhood of $p$ contained in $U$.
	\end{definition}
	
	\begin{remark}
		We observe that by using the quadratic equation
		$$ q^2-2q_0 q+|q|^2=0,$$
		which holds for every $q\in \mathbb{H}$,
		we can rewrite the polynomial $Q_p^n(q)$ as
		\begin{equation}
			\label{obs}
			Q_p^n(q)=(|p|^2-2(p_0-q_0)q-|q|^2)^n.
		\end{equation}
	\end{remark}
	
	The spherical series converges in a Euclidean neighbourhood of $p$ described below, see \cite[Prop. 2.4]{S}. To this end, we introduce a suitable terminology: let $p=p_0+Ip_1 \in \mathbb{H}$ with $p_0 \in \mathbb{R}$, $p_1>0$ and $I \in \mathbb{S}$. The set
\begin{equation}\label{Cassiniball} {U}(p,R)=\{q \in \mathbb{H} \,| \quad |(q-p_0)^2+p_1^2| < R^2\}
\end{equation}
		 is called Cassini ball centered at $p$.
	\\ If $p_1=0$ the Cassini ball is a 4-dimensional ball with center at $p_0$ and radius $R>0$.
		
	\begin{proposition}
		\label{c0}
		Let $ \{a_n\}_{n \in \mathbb{N}_0} \subset \mathbb{H}$ and suppose that
		$$ \limsup_{n \to \infty}| a_n|^{\frac{1}{n}}=\frac{1}{R}$$
		for some $R>0$. Let $p=p_0+Ip_1 \in \mathbb{H}$ with $p_0 \in \mathbb{R}$, $p_1>0$ and $I \in \mathbb{S}$. Then,
		the spherical series \eqref{spherical} converges absolutely and uniformly on compact subsets of the Cassini ball $ {U}(p,R)$ centered at $p$,
		where the series (\ref{spherical}) defines a left slice hyperholomorphic function.
	\end{proposition}
	\begin{remark}
		The function $ \delta(p,q)=\sqrt{|(q-p_0)^2+p_1^2|}$ is a pseudo-distance in $ \mathbb{H}$ and it is called Cassini pseudo-distance. It turns out to be continuous with respect to the Euclidean topology in $ \mathbb{H}$, see \cite[Prop. 6.2]{GPS}.
	\end{remark}

	\begin{remark}
We stress that	although the $*$-Taylor series and the spherical series, where they are both defined, are two expansions of a left slice hyperholomorphic function around a generic quaternion $p \in \mathbb{H}$, they are deeply different.
		The $*$-Taylor series has building blocks of the form $(q-p)$ namely polynomials of degree $1$,  while the spherical series involves powers of a second degree polynomial. Moreover the $*$-Taylor series converges with respect to a topology that is finer than the Euclidean topology used for the convergence of the spherical series.
	\end{remark}
	
	\subsection{Slice hyperholomorphic Laurent series}
	
	The theory of Laurent expansions can be introduced for slice hyperholomorphic functions, see \cite{CSS,GSS}, and to describe their convergence set we need more notations.
Specifically, we define the following $4$-dimensional spherical shell:
	\begin{equation}
		\label{shell}
		A(p,R_1, R_2):= \{q \in \mathbb{H}\, | \, R_1< |q-p|<R_2\}, \qquad p \in \mathbb{H},
	\end{equation}
	where $R_1$, $R_2 \in [0, \infty]$ and $R_1<R_2$. In particular, when $p=0$ we have:
	\begin{proposition}
		Let $\{a_n\}_{n \in \mathbb{Z}}$ be a sequence in $\mathbb{H}$. We set
		$$ R_{1}= \limsup_{n \to \infty} | a_{-n}|^{\frac{1}{n}}, \quad \hbox{and} \quad \frac{1}{R_2}=\limsup_{n \to \infty} |a_n|^{\frac{1}{n}},$$
		and we assume $R_1<R_2$. Then the Laurent series
\begin{equation}\label{Lauzero}
f(q)=\sum_{n=0}^\infty q^n a_n+ \sum_{n=1}^{\infty} q^{-n}a_{-n}
\end{equation}
converges absolutely and uniformly on every compact subset of $A(0,R_1,R_2)$.
	\end{proposition}
  Conversely, a slice hyperholomorphic function in $A(0,R_1, R_2)$ can be written as Laurent series in a neighbourhood of the origin. In fact we have:
	\begin{theorem}
		\label{classLau}
		Let $f: A(0,R_1, R_2) \to \mathbb{H}$ be a slice hyperholomorphic function. Then there exists a sequence $ \{a_n\}_{n \in \mathbb{Z}} \subseteq \mathbb{H}$ such that
		$$ f(q)=\sum_{n=0}^\infty q^n a_n+ \sum_{n=1}^{\infty} q^{-n}a_{-n},\qquad q\in A(0,R_1, R_2).$$
	\end{theorem}
More in general, we can have Laurent series centered at a point $p\in\mathbb H$ which are defined as follows.

	\begin{definition}
		Let $p\in\mathbb H$ and $ \{a_n\}_{n \in \mathbb{Z}}$ in $ \mathbb{H}$. The series
		\begin{equation}
			\label{Laup}
			\sum_{n \in \mathbb{Z}} (q-p)^{n*_{q,L}}a_{n},
		\end{equation}
		is called $*$-Laurent series centered at $p$.
	\end{definition}
	To discuss the convergence of the $*$-Laurent series we recall a result in  \cite{GPS}.
	\begin{proposition}
		Let $p \in \mathbb{H}$ and $ \{a_n\}_{n \in \mathbb{Z}} \subseteq \mathbb{H}$. Let us set
		$$ R_{1}= \limsup_{n \to \infty} | a_{-n}|^{\frac{1}{n}}, \quad \hbox{and} \quad \frac{1}{R_2}=\limsup_{n \to \infty} |a_n|^{\frac{1}{n}},$$
		and let us assume $R_1<R_2$ and $p\in\mathbb C_I$. Then the $*$-Laurent series \eqref{Laup} converges absolutely and uniformly on every compact subset of $A(p,R_1,R_2)\cap \mathbb C_I$.
	\end{proposition}
	To further study the convergence of the $*$-Laurent series the authors of \cite{GPS} introduced the pseudo-distance
	$$ \tau(q,p):=\begin{cases}
		|q-p|, \qquad \hbox{if} \quad q,p \quad \hbox{lie on the same complex plane} \ \mathbb{C}_I,\\
		\sqrt{(q_0-p_0)^2+\left(|\underline{q}|-|\underline{p}|\right)^2}, \quad \hbox{otherwise}.
	\end{cases}$$
	By recalling the distance $\sigma$ in \eqref{distance}, we now set
	$$ \Sigma (p,R_1, R_2):= \{q \in \mathbb{H} \, | \, \tau(q,p)>R_1 \, , \, \sigma(q,p)<R_2\}, \qquad p \in \mathbb{H},$$
	where $R_1$, $R_2 \in [0, + \infty]$ are such that $R_1<R_2$.
	
A slice hyperholomorphic function around a generic quaternion $p$ can be written in terms of the $*$-Laurent series, see \cite[Theorem 4.9]{GPS}, as follows:
	\begin{theorem}
		\label{Lau3}
		Let $U$ be an axially symmetric domain in $ \mathbb{H}$ and $f:U \to \mathbb{H}$ be a slice hyperholomorphic function. Let $p\in\mathbb H$ and $R_1,R_2>0$ be such that $\Sigma(p,R_1,R_2) \subseteq U$. Then there exists a sequence $ \{a_n\}_{n \in \mathbb{Z}}$ in $ \mathbb{H}$ such that
		\begin{equation}
			\label{Lau1}
			f(q)= \sum_{n=0}^{\infty}(q-p)^{n*_{q,L}} a_n+ \sum_{n=1}^{\infty} (q-p)^{-n*_{q,L}}a_{-n},
		\end{equation}
		in $\Sigma(p,R_1,R_2)$.
	\end{theorem}

	\begin{remark}
		If $a_{n}=0$, for $n<0$, then the result in \eqref{Lau1} gets back to the $*$-Taylor expansion at the point $p$ introduced in \eqref{regser}.
	\end{remark}
	
	The $*$-Laurent expansion of a slice hyperholomorphic function at a point $p$ can be written in terms of the derivative of the slice hyperholomorphic Cauchy kernel as follows:
	
	\begin{proposition}
		\label{Lau5}
Let $p$, $q \in \mathbb{H}$ such that $q \notin [p]$. Then for $n \geq 1$ we have
	\begin{equation}
	\label{starL}
	(q-p)^{-n*_{q,L}}= \frac{(-1)^{n}}{(n-1)!} \partial_{q_0}^{n-1} S_{L}^{-1}(p,q), \qquad n \geq 1,
\end{equation}
		where $S^{-1}_L(p,q)$ is the Cauchy kernel written in second form, see Definition \ref{Cauchykernels}.
Moreover, given $f$ slice hyperholomorphic in a neighborhood of $p$ we can write its $*$-Laurent series \eqref{Lau1} as
		\begin{equation}
			\label{Lau4}
			f(q)= \sum_{n=0}^{\infty} (q-p)^{n*_{q,L}}a_n + \sum_{n=1}^{\infty} \frac{(-1)^n}{(n-1)!} \partial_{q_0}^{n-1} S^{-1}_L(p,q)a_{-n},
		\end{equation}
where it is convergent.
	\end{proposition}
	\begin{proof}
	We start proving \eqref{starL}.
		By \cite[page 91]{CGK} we know that
\begin{equation}
\label{new1}
(q-p)^{-n*_{q,L}}=(q^2-2qp_0+|p|^2)^{-n} (q-\bar{p})^{n*_{q,L}}.
\end{equation}
		Using the fact that
		$$ \partial_{q_0}^{n-1}\left[(q^2-2p_0q+|p|^2)^{-1} (\bar{p}-q)\right]=(-1)^n (n-1)!(q^2-2qp_0+|p|^2)^{-n} (q-\bar{p})^{n*_{q,L}},$$
	Proposition~\ref{relform}, and the second form  of the slice hyperholomorphic Cauchy kernel, we obtain
		\begin{eqnarray*}
			(q-p)^{-*n_{q,L}}&=& \frac{(-1)^n}{(n-1)!} \left[(-1)^n (n-1)!(q^2-2qp_0+|p|^2)^{-n} (q-\bar{p})^{n*_{q,L}}\right]\\
			&=& \frac{(-1)^n}{(n-1)!} \partial_{q_0}^{n-1} \left[(q^2-2qp_0+|p|^2)^{-1} (\bar{p}-q)\right]\\
			&=& \frac{(-1)^n}{(n-1)!} \partial_{q_0}^{n-1} \left[(p-\bar{q})(p^2-2q_0p+|q|^2)^{-1}\right]\\
			&=& \frac{(-1)^n}{(n-1)!} \partial_{q_0}^{n-1} S_{L}^{-1}(p,q).
		\end{eqnarray*}
By plugging \eqref{starL} into \eqref{Lau1} we get \eqref{Lau4}.
	\end{proof}
	
	\begin{proposition}
		Let us consider $I$, $J \in \mathbb{S}$, $ p \in \mathbb{C}_I$ and $q=x_0+y_0J$, then we have
		\begin{equation}
			\label{ineqL}
			|(q-p)^{n*_{q,L}}| \leq 2 \max_{w=x_0 \pm y_0I}|(w-p)^{n}|
		\end{equation}
		and
		\begin{equation}
			\label{ineqL1}
			|(q-p)^{-n*_{q,L}}| \leq \frac{2}{\min_{w=x_0 \pm y_0I}|(w-p)^n|}.
		\end{equation}
	\end{proposition}
	\begin{proof}
		The inequality in \eqref{ineqL} follows by Theorem \ref{rapp0} and by standard computations, so we prove \eqref{ineqL1}. By definition of slice hyperholomorphic inverse the function $h(q):=(q-p)^{-*n_{q,L}}$ is slice hyperholomorphic in $ \mathbb{H} \setminus [p]$, where $[p]$ is the 2-sphere associated to $p$. Hence by the representation formula, see Theorem \ref{rapp0}, we have
		$$ h(q)=\alpha+I_{\underline{q}}I \beta,$$
		where
		$$ \alpha:= \frac{1}{2} \left[ h(x+yI)+h(x-yI)\right], \qquad \beta:=-\frac{1}{2} \left[ h(x+yI)-h(x-yI)\right].$$
		It is clear that $\alpha$ and $\beta$ satisfy the estimates
		$$ |\alpha|, |\beta| \leq \max_{w=x_0 \pm y_0 I} \left\{ \frac{1}{|w-p|^n}\right\}=\frac{1}{\min_{w=x_0 \pm y_0I}|(w-p)^n|}.$$
		These inequalities imply that
		$$ |h(q)| \leq |\alpha|+|I_{\underline{q}}I\beta| \leq  \frac{2}{\min_{w=x_0 \pm y_0I}|(w-p)^n|},$$
and the statement follows.
	\end{proof}
	The inequalities proved in the above result (see also \cite{GSS} for related results on optimality) pave the way to the study of the convergence of the $*$-Laurent series in appropriate open sets.
In fact, as it happens for the $*$-Taylor expansion around $p$, the set of convergence of the $*$-Laurent expansion is not open in $\mathbb{H}$, unless $p$ is real. However, it is an open set in the slice topology, see Definition \ref{slicetopo}.
\\
Let us consider $I \in \mathbb{S}$ and $0<R_1<R_2$. For any fixed $p \in \mathbb{C}_I$, the shell in $\mathbb{C}_I$ with centre $p$ and radii $R_1$ and $R_2$ is denoted by
	$$ S_I(p,R_1, R_2):= \{q \in \mathbb{C}_I \, :  \, R_1 < |q-p| < R_2\}.$$
	We also set
\begin{equation}\label{tildeS}
\begin{split}
\tilde{S}(p,R_1, R_2):&= \{x+yJ \in \mathbb{H} \, : \, J \in \mathbb{S} \,, \, \hbox{and} \quad x \pm yI \in S_I(p,R_1, R_2) \} \, \cup\,  S_I(p,R_1, R_2) \\
&=\Omega (p,R_1, R_2)\cup\,  S_I(p,R_1, R_2).
\end{split}
\end{equation}

	\begin{theorem}
		Let $ U \subseteq\mathbb{H}$ be a domain and let $f:U \to \mathbb{H}$ be a slice hyperholomorphic function. Then for any $p \in \mathbb{H}$ and $r_1$, $r_2>0$ such that the shell $ \Omega(p, r_1, r_2)$ is contained in $U$ there exists $\{a_n\}_{n\in\mathbb Z}\subseteq \mathbb{H}$ such that
		\begin{equation}
			\label{seriesL}
			f(q)= \sum_{n \in \mathbb{Z}} (q-p)^{n*_{q,L}}a_n, \qquad q \in \Omega(p, r_1, r_2),
		\end{equation}
	with $r_1\geq R_1$, $r_2\leq R_2>0$ and $R_1$, $R_2$ given by
		\begin{equation}
			\label{coeffL}
			\limsup_{n \to \infty}  |a_{-n}|^{\frac{1}{n}}=R_1, \qquad  \limsup_{n \to \infty}  |a_{n}|^{\frac{1}{n}}= \frac{1}{R_2}.
		\end{equation}
	\end{theorem}
	\begin{proof}
		We consider the function $g$ defined by
		$$ g(q):=\sum_{n=0}^{\infty} (q-p)^{n*_{q,L}}a_n+ \sum_{n=1}^{\infty} (q-p)^{-n*_{q,L}}a_{-n}.$$
		The convergence of the first series at the right hand side follows as in \cite[Thm. 8.8]{DJGS}. We study the convergence of the second series. By the inequality \eqref{ineqL1} we have
		\begin{equation}
			\label{auxL}
			|(q-p)^{-n*_{q,L}}a_{-n}| \leq 2 \left( |(q_I-p)^{-n} a_{-n}|+|(q_{-I}-p)^{-n}a_{-n}| \right),
		\end{equation}
		where $p\in\mathbb{C}_I$, $q_I:= x+yI$ and $q_{-I}=x-yI$. The hypothesis on the coefficients \eqref{coeffL} and the fact that $q \in \Omega(p, r_1, r_2)$ yield that the series
		$$ \sum_{n=1}^{\infty} |(q_I-p)^{-n} a_{-n}|, \quad \hbox{and} \quad \sum_{n=1}^{\infty} |(q_{-I}-p)^{-n}a_{-n}|,$$
		are convergent. By \eqref{auxL} we get that
		$$ \sum_{n=1}^{\infty} (q-p)^{-*n_{q,L}}a_{-n}$$
		is convergent. Since the function $g: \Omega(p, r_1, r_2) \to \mathbb{H}$ given by
		$$ g(q)=\sum_{n \in \mathbb{Z}} (q-p)^{n*_{q,L}}a_{n},$$
		is slice hyperholomorphic also $f$ is slice hyperholomorphic and
		$$f=g, \quad \hbox{on} \quad S_I(p,r_1, r_2).$$
		By the Identity principle, see Theorem \ref{Ide}, we conclude that $f=g$ and this proves the result.
	\end{proof}

We now discuss another Laurent expansion, see \cite{GPS}, whose building blocks are in terms of a second degree polynomial, see \eqref{set1}. The advantage is that this type of series, called spherical Laurent  series expansion, converges in open sets of $\mathbb{H}$.

	\begin{definition}
		Let $p \in \mathbb{H}$. For any $ \{a_n\}_{n \in \mathbb{Z}}$ in $ \mathbb{H}$ the series
		\begin{equation}
			\label{sphLa}
			\sum_{n \in \mathbb{Z}} Q_p^n(q) [a_{2n}+(q-p)a_{2n+1}],
		\end{equation}
		is called the spherical Laurent series centered at $p$.
	\end{definition}


	The convergence set of the spherical Laurent series is described in the following result, see \cite[Theorem 6.4]{GPS}. To this end we introduce the following notion of Cassini shell $U(p,R_1, R_2)$, where $p\in\mathbb H$ and $0<R_1<R_2$:
		\begin{equation}\label{Cassinishell}
U(p,R_1, R_2):= \{q \in \mathbb{H}\, : \, R_1^2< |(q-p_0)^2+p_1^2| < R_2^2\}.
\end{equation}

	\begin{proposition}
		Let $p=p_0+Ip_1 \in \mathbb{H}$, with $p_0 \in \mathbb{R}$, $p_1 \in \mathbb{R}$ and $I \in \mathbb{S}$. For $ \{a_n\}_{n \in \mathbb{Z}} \subseteq \mathbb{H}$ we set
		$$ R_1:= \limsup_{n \to \infty}|a_{-n}|^{\frac{1}{n}}, \quad \hbox{and} \quad \frac{1}{R_2}:= \limsup_{n \to \infty} |a_n|^{\frac{1}{n}}.$$
		Then the spherical Laurent expansion \eqref{sphLa} converges absolutely and uniformly on the compact subsets of the Cassini shell $U(p,R_1, R_2)$.
	\end{proposition}
	
We now provide the spherical Laurent expansion of a slice hyperholomorphic function around $p \in \mathbb{H}$, see \cite[Theorem 7.2]{GPS}
	
	\begin{theorem}
		Let $U$ be an axially symmetric open set and $p \in \mathbb{H}$. Let us assume that $f:U \to \mathbb{H}$ is a slice hyperholomorphic function. Let $R_1$, $R_2 \in [0, \infty]$ with $R_1<R_2$ be such that the Cassini shell $U(p,R_1,R_2)$ is contained in $U$. Then there exists $ \{a_n\}_{n \in \mathbb{Z}} \subseteq \mathbb{H}$ such that
		\begin{equation}
			\label{Lau2}
			f(q)= \sum_{n \in \mathbb{Z}} Q_p^n(q)[a_{2n}+(q-p)a_{2n+1}] ,
		\end{equation}
		for $q\in U(p,R_1,R_2)$.
	\end{theorem}
	
	
	\begin{remark}
		If $a_n=0$ for $n<0$ in \eqref{Lau2} we re-obtain the spherical expansion defined in \eqref{spherical}.
	\end{remark}
	
	\begin{remark}
The spherical Laurent expansion, where it is convergent, can be written in terms of the left slice hyperholomorphic Cauchy kernel:
	\begin{eqnarray*}
		f(q)&=& \sum_{n=0}^{\infty} Q_p^{n}(q) \left[a_{2n}+(q-p)a_{2n+1} \right]\\
		&& + \sum_{n=1}^{\infty}\frac{1}{\left[(n-1)! \right]^2} \left(\partial_{q_0}^{n-1}S^{-1}_L(p,q) \right) *_{q,L} \left( \partial_{q_0}^{n-1} S^{-1}_L(\bar{p},q) \right) \left[a_{-2n}+(q-p)a_{-2n+1}\right].
	\end{eqnarray*}
This follows by Remark \ref{NstarLL} and \eqref{new1}.
	\end{remark}

	\begin{remark}
	In the sequel, whenever the sets $\Omega(p,R)$, $\Omega(p,R_1,R_2)$, $U(p,R)$, and $U(p,R_1,R_2)$ are considered, they are understood to be nonempty.

	\end{remark}

	\subsection{Cauchy-Fueter regular functions}

	Another important class of functions in the hypercomplex setting is that of Fueter regular functions. This class has undergone extensive study for about a century, starting with the early works of Moisil, see \cite{Moisil}, Fueter and his school \cite{Fueter}, and as evidenced by many works, see \cite{red, green,BLUEBOOK} and the references therein.

	\begin{definition}
		Let $U \subseteq \mathbb{H}$ be an open set and $f:U \to \mathbb{H}$ be a function in $ \mathcal{C}^1(U)$. The function $f$ is called left (resp. right) Cauchy-Fueter regular if
		\begin{eqnarray*}
			&& Df(q)= (\partial_{q_0}+ \partial_{\underline{q}})f(q)= \left(\partial_{q_0}+ e_1\partial_{q_1}+e_2\partial_{q_2}+e_3\partial_{q_3}\right)f(q)=0\\
			&& \left(\hbox{resp.} \, \, f(q)D=0 \right).
		\end{eqnarray*}
		The operator $ D$ is called Cauchy-Fueter operator.
	\end{definition}
In the sequel we shall say that a function is Fueter regular instead of Cauchy-Fueter regular, for short.	\\
The conjugate Cauchy-Fueter operator is defined by
\begin{eqnarray*}
			&& \bar{D}f(q)= (\partial_{q_0}- \partial_{\underline{q}})f(q)= \left(\partial_{q_0}- e_1\partial_{q_1}-e_2\partial_{q_2}-e_3\partial_{q_3}\right)f(q).
		\end{eqnarray*}
\begin{remark}
The operator $\partial_{\underline{q}}$ can be written as
\begin{equation}
	\label{Frela}
	\partial_{\underline{q}}= \frac{\underline{q}}{| \underline{q}|^2}\left(\mathbb{E}_{\underline{q}}+ \Gamma_{\underline{q}}\right),
\end{equation}
see \cite{green}, where $ \mathbb{E}_{\underline{q}}$ is the Euler operator
\begin{equation}
	\label{Euler}
	\mathbb{E}_{\underline{q}}= \sum_{k=1}^{3} q_k \partial_{q_k}.
\end{equation}	
and $\Gamma_{\underline{q}}$ is the Gamma operator
\begin{equation}
	\label{GG}
	\Gamma_{\underline{q}}=- \sum_{j=1}^{3} \sum_{k=j+1}^{3} e_j e_k \left(q_j \partial_{q_k}- q_k \partial_{q_j} \right).
\end{equation}	
\end{remark}
	Well-known examples of Fueter regular functions are given by the so-called Fueter variables, defined as
	$$ \xi_1(q):= q_1-e_1q_0, \qquad \xi_2(q)=q_2-e_2q_0, \qquad \xi_3(q)=q_3-e_3q_0.$$
	Let $m \in \mathbb{N}_0$ and $\sigma_m$ be the set of all triples $ \nu=[m_1, m_2, m_3]$ of non-negative integers such that $m_1+m_2+m_3=m$. If $ \nu \in \sigma_m$ we define the Fueter polynomials as
	\begin{equation}
		\label{Fueterpoly}
		P_{\nu}(q)= \frac{1}{m!} \sum \left( q_0 e_{\lambda_{1}}-q_{\lambda_{1}} \right)...\left( q_0 e_{\lambda_{m}}-q_{\lambda_{m}} \right),
	\end{equation}
	where the above sum is extended to all $m$-tuples $(\lambda_{1},..., \lambda_{m})$ such that $1 \leq \lambda_{1},\ldots, \lambda_{m} \leq 3$ and such that the number of $\lambda_j$ equal to $h$ is exactly $m_h$, for $h=1,2,3$.
	
	An interesting subclass of regular functions, that will be of great importance in this work, is given by the so-called axially Fueter regular functions.
	\begin{definition}[Axially Fueter regular functions]\label{AXREG}
		Let $U$ be an axially symmetric open set in $ \mathbb{H}$. We say that the function $f: U \to \mathbb{H}$ is axially (left)
		Fueter regular if it is left Fueter regular and of the form
		$$ f(q_0+ \underline{q})=A(q_0, | \underline{q}|)+ \underline{\omega} B(q_0, |\underline{q}|), \qquad \underline{\omega}:= \frac{\underline{q}}{| \underline{q}|},$$
		where $A$ and $B$ are quaternionic-valued functions that satisfy the conditions \eqref{EO}. We denote the set of left axially Fueter regular functions as $ \mathcal{AM}_L(U)$ or, in short, $\mathcal{AM}(U)$  when no confusion arises.
	\end{definition}
	The above definition can be obviously adapted for right axially Fueter regular functions. Moreover, they can be further generalized to the case of Clifford algebra-valued functions.
	\begin{remark}
		It is worthwhile noting that when dealing with slice hyperholomorphic functions, the imaginary units are often denoted by $I, J, K$. However, in the case of Fueter regular functions or functions in the kernel of Dirac operators, the symbol $\underline{\omega}$ is more commonly employed for elements in $\mathbb S$ and we shall often use it.
	\end{remark}
\begin{examples}
\begin{enumerate}
\item	The Cauchy kernel
	\begin{equation}
		\label{Cauchykernel}
		E(q)=\frac{\bar q}{|q|^4},\ \ \ q\not=0
	\end{equation}
	is an example of an axially Fueter regular function.
	\item
	Clifford-Appell polynomials defined as:
	\begin{equation}
		\label{capoly}
		\mathcal{Q}_{n}(q)= \frac{2}{(n+1)(n+2)} \sum_{j=0}^{n} (n-j+1) q^{n-j} \bar{q}^j, \qquad n \geq 0
	\end{equation}
are axially Fueter regular functions, first studied in \cite{CFM, CMF} and further investigated in \cite{DDG1,DKS}.
\end{enumerate}
\end{examples}
\begin{remark}
	The polynomials $ \mathcal{Q}_n(q)$ form an Appell sequence with respect to the hypercomplex derivative (or conjugate Cauchy-Fueter operator) namely they satisfy
	\begin{equation}\label{appProp}
		\frac{\bar{D}}{2} \mathcal{Q}_n(q)=\left(\frac{\partial_{q_0}- \partial_{\underline{q}}}{2} \right)\mathcal{Q}_n(q)=n \mathcal{Q}_{n-1}(q).
	\end{equation}
	Moreover, the Clifford-Appell polynomials satisfy the inequality
	\begin{equation}
		\label{estCliff}
		| \mathcal{Q}_n(q)| \leq |q|^n.
	\end{equation}
\end{remark}	
	
	\subsection{Two classes of quaternionic polyanalytic functions}
	
	Another possible generalization of the notion of holomorphic function is given by the concept of polyanalytic function of order $n$, i.e. null-solutions of the powers of the Cauchy-Riemann operator. In the complex setting this class of functions is widely studied, see for instance \cite{Balk1, Balk2}, and it has several applications for example in elasticity problems, see \cite{K1,M1}, time-frequency analysis, see \cite{A, A1, AF}, duality theorems, see \cite{CDDS}, and integral transforms, see \cite{V1}.
	\\ The class of polyanalytic functions has been generalized in the quaternionic setting, see \cite{ADS, ADS2019} in the context of slice analysis and earlier in the context of Fueter regular functions, see \cite{B1976}.
	
	\begin{definition}[Poly slice hyperholomorphic functions (or Slice polyanalytic functions)]
		\label{polyslice}
		Let $n \in \mathbb{N}$ and $ U \subseteq \mathbb{H}$ be an axially symmetric open set. We say that a slice function $ f(q)=\alpha(u,v)+I \beta(u,v)$ in $\mathcal{C}^n(U)$
		is left slice polyanalytic (or slice polyanalytic, for short) of order $n$ if the functions $\alpha$ and $ \beta$ satisfy the even-odd conditions \eqref{EO} and the poly-Cauchy-Riemann equation
		$$ \left( \frac{\partial}{\partial u} +I \frac{\partial}{\partial v}\right)^n\left(\alpha(u,v)+I \beta(u,v) \right)=0, \qquad \forall I \in \mathbb{S}.$$
		The definition of right slice polyanalytic functions of order $n$ can be easily adapted.
		\\ The set of left (resp. right) slice polyanalytic functions of order $n$ is denoted by $\mathcal{SP}_{n}^L(U)$ (resp. $ \mathcal{SP}_{n}^R(U)$).
	\end{definition}
	
	\begin{remark}
		If we take $n=1$ in Definition \ref{polyslice} we obtain the definition of slice hyperholomorphic function, see Definition \ref{sh}.
	\end{remark}
	
	\begin{remark}
		The representation formula, see Theorem \ref{rapp0}, is valid for left slice polyanalytic functions,
		see \cite{ADS}, since they are particular slice functions.
	\end{remark}
	
Slice polyanalytic functions can be decomposed in terms of slice hyperholomorphic functions. This is called polyanalytic decomposition and it is stated in the next result, see \cite{ADS,ADS2019}.
	\begin{theorem}
		\label{polydeco}
		A function $f$ is slice polyanalytic of order $n$ if and only if there exist left slice hyperholomorphic functions $f_0$,...,$f_{n-1}$ such that we can decompose the function $f$ as
		$$ f(q)=\sum_{k=0}^{n-1} \bar{q}^k f_k(q).$$
	\end{theorem}
	\begin{example}
		\label{new} Any function of the form $ f(q)=\sum_{k=0}^{n-1} \bar{q}^k q^{n-k}a_k$, $a_k\in\mathbb H$ is an example of a slice polyanalytic function of order $n$ and, in particular, the Clifford-Appell polynomials $ \mathcal{Q}_n(q)$ in \eqref{capoly}, see \cite{ACDS}. These polynomials are particularly interesting since they are also axially Fueter regular and thus they are in the intersection of these two classes of functions.
	\end{example}

\begin{remark}\label{VROPER}
A connection between the left slice polyanalytic functions  and the left slice hyperholomorphic functions has been pointed out in \cite{ADS1}, and it is given by the so-called left global operator. This operator has proven to be a versatile tool in various situations, and finds applications across various aspects of hypercomplex analysis. It is defined as
	\begin{equation}
		\label{global}
		V_{q,L}:= \frac{\partial}{\partial q_0}+\frac{\underline{q}}{| \underline{q}|^2} \left(\sum_{\ell=1}^{3}q_\ell \frac{\partial}{\partial q_{\ell}} \right) , \qquad q \in \mathbb{H}\setminus \mathbb{R}.
	\end{equation}
	More precisely, we have
	\begin{equation*}
		\begin{CD}
			\mathcal{SP}_{n+1}^L(U) @>V^n_{q,L}>>\mathcal{SH}_L(U)\\
		\end{CD}
	\end{equation*}
	
An analogous connection holds between the right slice polyanalytic functions  and the right slice hyperholomorphic functions, through the right global operator defined by
		$$  V_{q,R}:= \frac{\partial}{\partial q_0}+ \left(\sum_{\ell=1}^{3}q_\ell \frac{\partial}{\partial q_{\ell}} \right)\frac{\underline{q}}{| \underline{q}|^2} , \qquad q \in \mathbb{H}\setminus \mathbb{R}.$$
A global operator written differently is considered in \cite{GP}.
	\end{remark}
	
	\begin{remark}
		\label{Nreg}
		In \cite{CGS} the authors discuss under which topological assumptions on the domain of definition, functions in the kernel of the global operator can be related to slice hyperholomorphic functions in the sense of Definition \ref{sh} or \ref{sho}.
	\end{remark}
	
	Another notion of polyanalyticity can be given in the quaternionic context considering the powers of the Fueter operator, see \cite{B1976, BD1978}, and can be stated as follows.
	
	\begin{definition}[Polyanalytic Fueter regular functions]
		Let $U \subseteq \mathbb{H}$ be an open set. A function $f:U \to \mathbb{H}$ in $\mathcal{C}^n(U)$ is left (resp. right) polyanalytic Fueter regular of order $n$ on $U$ if
		$$ D^nf(q)=0, \qquad \left(\hbox{resp.} \, f(q)D^n=0 \right) \ \ \   \forall q \in U.$$
	\end{definition}
	
	By Theorem \ref{polydeco}, left slice polyanalytic functions admit a decomposition in terms of slice hy\-per\-holo\-mor\-phic functions and powers of $\bar{q}$. Similarly, also the left (resp. right) polyanalytic Fueter regular functions can be decomposed in terms of left (resp. right) Fueter regular functions and powers of $q_0$, in fact we have:
	
	\begin{theorem}
		\label{polydec}
		Let $U \subseteq \mathbb{H}$ be an open set. A function $f:U \to \mathbb{H}$ is left (resp. right) polyanalytic Fueter regular of order $n$ if and only if there exist unique left (resp. right) Fueter regular functions $f_0$,..., $f_{n-1}$ such that
		$$ f(q)= \sum_{k=0}^{n-1} q_0^k f_k(q).$$
		
	\end{theorem}
	
	\begin{example}
	The counterpart of the kernel to be used in the Cauchy formula for polyanalytic Fueter regular functions is:
	
	\begin{equation}
		\label{polycauchy}
		\sum_{j=0}^{n-1} (-1)^j \frac{q_0^j}{j!} \frac{\bar{q}}{|q|^4},\ \ \ q\not=0.
	\end{equation}
	\end{example}

	We note that a connection between the left slice polyanalytic and left polyanalytic Fueter regular functions has been established in \cite{ADS1, DK3}.
	
	\subsection{Integral representations of the functions of the fine structure of Dirac type}\label{INTRAPP}
	
A crucial result in hypercomplex analysis and in this work is the Fueter mapping theorem, see \cite{Fueter}, that we shall reformulate using modern terminology.

	\begin{theorem}[Fueter theorem]
		\label{Fueter}
		Let $\Omega$ be a domain in the upper-half complex plane $\mathbb{C}^+$ and let
		$$ U_\Omega:= \{q=q_0+ \underline{q} \ | \ q_0+i | \underline{q}| \in \Omega \},$$
		be the open set  in $ \mathbb{H}$ induced by $\Omega$. Let $f_0(z)=\alpha(x,y)+i \beta(x,y)$, with $z=x+iy$, be a holomorphic function in $\Omega$. Then the so-called slice operator, defined by
		\begin{equation}
			\label{sliceop}
			T_{F1}(f_0):= \alpha(q_0, | \underline{q}|)+ \frac{\underline{q}}{|\underline{q}|} \beta(q_0, | \underline{q}|),
		\end{equation}
		maps the holomorphic function $f_0(z)$ into a slice hyperholomorphic (intrinsic) function $f(q)=T_{F1}(f_0)$. Moreover, the application of the map $T_{F2}=\Delta$, where $\Delta$ is the Laplace operator in four real variables,  applied to the slice hyperholomorphic function $f(q)$ gives an axially Fueter regular function.
	\end{theorem}
	
	As previously mentioned in the introduction, when examining the relationship between slice hyperholomorphic functions and Fueter regular functions, we focus specifically on the so-called second Fueter map $T_{F2}$ of the Fueter mapping theorem. It is crucial to highlight that all slice hyperholomorphic functions defined on an axially symmetric open set $U$ are mapped into Fueter regular functions of axial type.
The converse is also true, since the Fueter map is surjective on the set of axial regular functions, see	\cite{CSSOinverse}. Below, for simplicity, we shall write $U$ instead of $U_\Omega$.
	
	\begin{remark}
		The Fueter mapping theorem was extended to more general algebras, including Clifford algebras $\mathbb R_n$, by M. Sce. In this context, the second map is given by the operator
		$$
		T_{FS2}:=\Delta_{n+1}^{\frac{n-1}{2}}.
		$$
		In this case the result is called Fueter-Sce mapping theorem.
		The case when $n$ is even and the operator $T_{FS2}$ contains a fractional power of the Laplacian has been further explored by T. Qian. For more detailed information, we refer to the survey \cite{Q1} and to  the book \cite{CSS3}, which includes translations of M. Sce's works on hypercomplex analysis and provides a comprehensive overview of recent advances in the field.
	\end{remark}
	
	\begin{remark}
		
		In the Fueter mapping theorem, the hypothesis to consider the upper-half complex plane can be relaxed. Specifically, our interest lies in the second map of the construction, where we consider slice hyperholomorphic functions defined on axially symmetric open sets that, in general, may intersect the real line. This issue is addressed by considering functions of the form
		$$
		f(q) = \alpha(u,v) + I\beta(u,v)\qquad \text{for } q = u + I v\in U
		$$
		where the functions $\alpha, \beta: \mathcal{U}\to \mathbb{H}$ , satisfy the compatibility condition (\ref{EO}).
	\end{remark}

	In recent years,  the Fueter theorem was used as a tool to provide integral representations of various classes of functions using the Cauchy formula for slice hyperholomorphic functions. In fact the map
	\begin{equation}
		\label{scheme}
		\begin{CD}
			\textcolor{black}{\mathcal{SH}_L(U)}  @>\ \  T_{F2}=\Delta >>\textcolor{black}{\mathcal{AM}_L(U)}.
		\end{CD}
	\end{equation}
	applied to the Cauchy kernels of slice hyperholomorphic functions leads to the new kernels
	\begin{equation}
		\label{FK}
		F_L(p,q):=\Delta S^{-1}_L(p,q)=-4(p- \bar{q})(p^2-2q_0p+|q|^2)^{-2},  \ \ {\rm for}\  q, p\in\hh\  {\rm with} \ q\not\in [p]
	\end{equation}
	and
	$$ F_R(p,q):=\Delta S^{-1}_R(p,q)=-4 (p^2-2q_0p+|q|^2)^{-2}(p- \bar{q}),
	\ \ {\rm for}\  q, p\in\hh\  {\rm with} \ q\not\in [p],
	$$
	called the left (resp. right) Fueter kernels, see \cite{CSS0}.
	These kernels allow us to write the Fueter (or more in general the Fueter-Sce) mapping theorem in integral form via the Cauchy formula.
	We observe that $F_L(p,q)$ (resp. $F_R(p,q)$) is right slice hyperholomorphic (resp. left slice hyperholomorphic)  in the variable $p$, and is left axially Fueter regular (resp. right axially Fueter regular)  in the variable $q$, for $q\not\in [p]$.
	\begin{remark}
		Although the first form of the Cauchy kernel is more suitable for the definition of the $S$-functional calculus in its noncommutative formulation, see \cite{CGK,CSS}, it does not lead to easy computations when we apply the Laplace operator to it. To this purpose, one needs to use the second form of the Cauchy kernel.
	\end{remark}
	
	The combination of the Fueter theorem and of the Cauchy formula of slice hyperholomorphic functions leads to the following integral representation of axially Fueter regular functions first introduced in \cite{CSSO}.
	\begin{theorem}
		\label{Fueterint}
		Let $U\subset\mathbb{H}$ be a bounded slice Cauchy domain, let $I\in\mathbb{S}$ and   $dp_I=dp (-I)$.
		Let $f$ be a left (resp. right)  slice hyperholomorphic function on a set that contains $\overline{U}$.
		Then, for $q \in U$, the left  (resp. right) axially Fueter regular function $\Delta f(q)$ has the following integral representation
		\begin{eqnarray*}
			&&\Delta f(q)= \frac{1}{2 \pi} \int_{\partial(U \cap \mathbb{C}_I)}F_L(p,q) dp_I f(p),\\
			&&  \left( \hbox{resp.} \, \Delta f(q)=\frac{1}{2 \pi} \int_{\partial(U \cap \mathbb{C}_I)}f(p) dp_I F_R(p,q) \right),
		\end{eqnarray*}
where $F_L$, $F_R$ are in \eqref{FK}.
		The above integrals are independent of the open set $U$, the imaginary unit $I\in \mathbb{S}$, and the kernel of $\Delta$.
		
	\end{theorem}
	
\begin{remark}
In \cite{CGK} it was proved that the integral representations of Theorem \ref{Fueterint} do not depend on the choice of $f \in \mathcal{SH}_L(U)$. This result was originally proved more in general for operators, and so the proof can be easily adapted to the case of a quaternionic variable.
\end{remark}
	The above integral representation leads to the definition of the so-called $F$-functional calculus, see \cite{CSSO}, which is a monogenic functional calculus in the spirit of McIntosh and collaborators, see \cite{JM}. Following the Riesz-Dunford functional calculus, see \cite{RD}, for which the resolvent equation allows one to have a product rule, one may look for a resolvent equation of the quaternionic $F$-functional calculus. This problem was studied in \cite{FABJONA},
	and, more in general, in \cite{CDS, CDS1}.  We note that for the $F$-functional calculus the situation is more involved because it is defined through an integral transform.  The next definition, which is relevant in our current study, is crucial to get the product rule for the $F$-functional calculus, see \cite{CDPS}.

	\begin{definition}
		We call {\em quaternionic fine structures on the $S$-spectrum} the set of functions and the associated functional calculi
		induced by the factorization of the (second) Fueter map.
		In particular, {\em  the fine structures on the $S$-spectrum of Dirac type} are the ones obtained by the factorization of the second Fueter map in terms of the Cauchy-Fueter operator $D$ and of its conjugate $\bar D$.
	\end{definition}
	According to the above definition, in the quaternionic setting there are only two possible factorizations of the second Fueter map, i.e. the Laplacian in four real variables
	$$\Delta= \bar{D} D \quad \hbox{and} \quad \Delta=D \bar{D},$$
and these two factorizations lead to different fine structures.
	\\ In \cite{CDPS} we considered  the factorization $\Delta= \bar{D} D$ which leads to
	\begin{equation}
		\label{facto1}
		\begin{CD}
			\textcolor{black}{\mathcal{SH}_L(U)} @>D >> \textcolor{black}{\mathcal{AH}_L(U)}  @>\ \   \bar{D} >>\textcolor{black}{\mathcal{AM}_L(U)},
		\end{CD}
	\end{equation}where $ \mathcal{AH}_L(U)$ is the set of axially harmonic functions defined by
	$$ \mathcal{AH}_L(U):=D(\mathcal{SH}_L(U))=\{g\in\mathcal C^\infty(U)\ |\ g=Df \, , \, f\in\mathcal{SH}_L(U)\}.$$
	A function $g\in \mathcal{AH}_L(U)$ is harmonic by virtue of the Fueter mapping theorem, indeed $\Delta g=\Delta Df=0$.
	The application of the Fueter operator $D$ to the Cauchy kernel gives the following integral representation  of the axially harmonic functions, see \cite[Theorem 4.16]{CDPS}.
	
	\begin{theorem}
		\label{harmint}
		Let $U\subset\mathbb{H}$ be a bounded slice Cauchy domain, $I\in\mathbb{S}$ and let $dp_I=dp (-I)$.
		Let $f$ be a left (resp. right)  slice hyperholomorphic function on a set that contains $\overline{U}$.
		Then, for $q \in U$, the axially harmonic function $g=Df$ has the integral representation
		\begin{eqnarray*}
			&& g(q)=Df(q)=-\frac{1}{\pi} \int_{\partial(U \cap \mathbb{C}_I)} \mathcal{Q}_{c,p}^{-1}(q)dp_I f(p)\\
			&&\left(\hbox{resp.} \quad f(q)D=-\frac{1}{\pi} \int_{\partial(U \cap \mathbb{C}_I)} f(p) dp_I \mathcal{Q}_{c,p}^{-1}(q)\right),
		\end{eqnarray*}
		where the kernel is given by
		\begin{equation}
			\label{psudoCauchy}
			\mathcal{Q}_{c,p}^{-1}(q):=(p^2-2q_0p+|q|^{2})^{-1},
		\end{equation}
		and is obtained by applying the operator $D$ to the slice hyperholomorphic Cauchy kernels, i.e.:
		\begin{equation}\label{2punto*}
		D S^{-1}_L(p,q)=S^{-1}_R(p,q)D=-2(p^2-2q_0p+|q|^{2})^{-1}=-2\mathcal{Q}_{c,p}^{-1}(q).
		\end{equation}
		The above integrals are independent of the open set $U$, the imaginary unit $I\in \mathbb{S}$, and the kernel of $D$.
	\end{theorem}
	As we proved in \cite{CDPS}, the above integral representation does not depend on the choice of $f\in \mathcal{SH}_L(U)$ such that $g=Df$. This result was originally proven for operators, and thus it adapts to quaternions by taking the operator of multiplication (on the left or on the right)  with a quaternion.
	\begin{remark}
		The kernel $ \mathcal{Q}_{c,p}^{-1}(q)$ can be written in terms of the left Fueter kernel $F_L$, see \eqref{FK}. Precisely, by \cite[Thm. 7.3.1]{CGK} we have
		\begin{equation}
			\label{connection}
			\mathcal{Q}_{c,p}^{-1}(q)= \frac{1}{4} \left(-F_L(p,q)p+qF_L(p,q) \right)
		\end{equation}
		In an analogous way, it is possible also to write $ \mathcal{Q}_{c,p}^{-1}(q)$ in terms of the right Fueter kernel $F_R$.
	\end{remark}
	
		Now, we focus on the second fine structure of Dirac type in the quaternionic setting, i.e. the factorization $\Delta= D \bar{D}$. In this case we have:
	\begin{equation}
		\label{facto2}
		\begin{CD}
			\textcolor{black}{\mathcal{SH}_L(U)} @>\bar{D}>> \textcolor{black}{\mathcal{AP}_2^L(U)}  @>\ \  D >>\textcolor{black}{\mathcal{AM}_L(U)},
		\end{CD}
	\end{equation}
	where $ \mathcal{AP}_2^L(U)$ is the set of axially polyanalytic functions of order $2$, and it is defined by
	
	$$ \mathcal{AP}_2^L(U):=\bar D(\mathcal{SH}_L(U))=\{h\in\mathcal C^\infty(U)\, |\, h=\bar{D}f \, , \, f\in\mathcal{SH}_L(U)\}.$$

	A function $h\in \mathcal{AP}_2^L(U)$ is polyanalytic of order 2  by virtue of the Fueter mapping theorem, indeed $D^2 h=D^2 \bar D f= D\Delta f=0$.
	By applying the conjugate Fueter operator to the Cauchy kernel we get the following integral representation of axially polyanalytic functions of order 2, see \cite{Polyf1, Polyf2}.
	\begin{theorem}
		\label{polyint}
		Let $U\subset\mathbb{H}$ be a bounded slice Cauchy domain, let $I\in\mathbb{S}$ and set  $dp_I=dp (-I)$.
		Let $f$ be a left (resp. right)  slice hyperholomorphic function on a set that contains $\overline{U}$.
		Then, for $q \in U$, the left (resp. right) polyanalytic function $h=\bar{D}f$  of order 2 has the following integral representation
		\begin{eqnarray*}
			&&h(q)=\bar{D}f(q)=\frac{1}{2 \pi} \int_{\partial(U \cap \mathbb{C}_I)} P_2^L(p,q) dp_I f(p)\\
			&& \left(\hbox{resp.} \quad f(q)\bar{D}=\frac{1}{2 \pi} \int_{\partial(U \cap \mathbb{C}_I)} f(p) dp_I P_2^R(p,q)\right).
		\end{eqnarray*}
		$$ $$
		where the kernel $ P_2^L(p,q)$ is given by
		\begin{eqnarray}
			\label{polykernel}
			&& P_2^L(p,q):=\bar{D}S_L^{-1}(p,q)=F_L(p,q) (q_0-p),\\
			\nonumber
			&&\left( \hbox{resp.} \, P_2^R(p,q):=S^{-1}_R(p,q) \bar{D}=(q_0-p)F_R(p,q) \right).
		\end{eqnarray}
		The above integrals are independent of the open set $U$, the imaginary unit $I\in \mathbb{S}$, and the kernel of $\bar{D}$.
		
	\end{theorem}
		We point out that the independence of the above integral representation on the choice of $f\in \mathcal{SH}_L(U)$ such that $h=\bar Df$ was proved in  \cite{Polyf2} (in the more general case of operators).
	
	We observe that we can write the kernel $P_2^L(p,q)$ (resp. $P_2^R(p,q)$) in terms of the slice Cauchy kernel and of the pseudo-Cauchy kernel as described below.
	\begin{proposition}
		\label{corr}
		Let $p$, $q \in \mathbb{H}$ such that $q \notin[p]$. Then we can write the polyanalytic kernel $P_2^L(p,q)$ (resp. $P_2^R(p,q)$) as
		$$ P_2^L(p,q)=2 \left[\partial_{q_0}S_L^{-1}(p,q)+ \mathcal{Q}_{c,p}^{-1}(q)\right],$$
		$$ \left( \hbox{resp.} \quad P_2^R(p,q)=2 \left[\partial_{q_0}S_R^{-1}(p,q)+ \mathcal{Q}_{c,p}^{-1}(q)\right] \right).$$
	\end{proposition}
	\begin{proof}
		We compute the first derivative with respect to $q_0$ of the slice hyperholomorphic kernel $S^{-1}_L$, so that we have
		\begin{equation}
			\label{splitS5}
			\partial_{q_0}S^{-1}_L(p,q)= - \mathcal{Q}_{c,p}^{-1}(q)+ 2 (p-\bar{q})(p-q_0) \mathcal{Q}_{c,p}^{-2}(q).
		\end{equation}
		This implies that
		\begin{eqnarray*}
			2 \left[\partial_{q_0}S_L^{-1}(p,q)+\mathcal{Q}_{c,p}^{-1}(q)\right]&=&4(p-\bar{q})(p-q_0) \mathcal{Q}_{c,p}^{-2}(q)\\
			&=&4 (p-\bar{q}) \mathcal{Q}_{c,p}^{-2}(q)p-4q_0 (p-\bar{q}) \mathcal{Q}_{c,p}^{-2}(q)\\
			&=& P_2^L(p,q),
		\end{eqnarray*}
as stated.	The result for the right kernel $ P_2^R(p,q)$ follows by similar computations.
	\end{proof}

	In the next result we prove that the kernels in the integral representation of the axially Fueter regular, harmonic and polyanalytic of order $2$ functions can be written in terms of the slice hyperholomorphic Cauchy kernel $S^{-1}_L$.
	\begin{proposition}\label{P274}
		Let $p \in \mathbb{H}$ and $q \in \mathbb{H} \setminus \mathbb{R}$. Then we can write the kernels $\mathcal{Q}_{c,p}^{-1}(q)$, $P_2^L(p,q)$, $F_L(p,q)$ (resp. $P_2^R(p,q)$, $F_R(p,q)$) as
		\begin{equation}
			\label{splitS}
			\mathcal{Q}_{c,p}^{-1}(q)=- \frac{\underline{q}^{-1}}{2} \left( S^{-1}_L(p, \bar{q})-S^{-1}_L(p,q) \right),
		\end{equation}
		\begin{equation}
			\label{splitS1}
			P_2^L(p,q)=2 \partial_{q_0}S_L^{-1}(p,q)-\underline{q}^{-1} \left( S^{-1}_L(p, \bar{q})-S^{-1}_L(p,q) \right),
		\end{equation}
		$$ \left(resp.  \quad P_2^R(p,q)=2 \partial_{q_0}S_R^{-1}(p,q)-\underline{q}^{-1} \left( S^{-1}_R(p, \bar{q})-S^{-1}_R(p,q) \right)\right),$$
		\begin{equation}
			\label{splitS2}
			F_L(p,q)=-2 (\underline{q})^{-1} \partial_{q_0}S^{-1}_L(p,q)- (\underline{q})^{-2}\left[S_{L}^{-1}(p, \bar{q})-S^{-1}_L(p,q) \right],
		\end{equation}
		$$ \left(\hbox{resp.} \quad F_R(p,q)=-2 (\underline{q})^{-1} \partial_{q_0}S^{-1}_R(p,q)- (\underline{q})^{-2}\left[S_{R}^{-1}(p, \bar{q})-S^{-1}_R(p,q) \right] \right).$$
	\end{proposition}
	\begin{proof}
		
		We start by proving \eqref{splitS}. By the definition of the second form of the left slice hyperholomorphic Cauchy kernel and the fact that $\mathcal{Q}_{c,p}(q)= \mathcal{Q}_{c,p}(\bar{q})$ we have that
		\begin{eqnarray}
			\label{splitS3}
			\mathcal{Q}_{c,p}^{-1}(q)&=&- \frac{\underline{q}^{-1}}{2} \left( -q \mathcal{Q}_{c,p}^{-1}(q)+\bar{q} \mathcal{Q}_{c,p}^{-1}(q) \right)\\
			\nonumber
			&=&- \frac{\underline{q}^{-1}}{2} \left[ (p-q) \mathcal{Q}_{c,p}^{-1}(\bar{q})-(p-\bar{q}) \mathcal{Q}_{c,p}^{-1}(q) \right]\\
			\nonumber
			&=&- \frac{\underline{q}^{-1}}{2} \left(S^{-1}_L(p, \bar{q})-S^{-1}_L(p,q) \right).
		\end{eqnarray}
Formula \eqref{splitS1} follows by Proposition \ref{corr} and \eqref{splitS}. The result for the right kernel follows in a similar way.

Now, we prove formula \eqref{splitS2}. By \eqref{splitS5} and \eqref{splitS3} we have
		\begin{eqnarray*}
			F_L(p,q)&=&-4 (p- \bar{q}) \mathcal{Q}_{c,p}^{-2}(q)\\
			&=&4 (\underline{q})^{-1} \left(-q_0p+ \bar{q}p+ \bar{q}q- \bar{q}q_0\right) \mathcal{Q}_{c,p}^{-2}(q)\\
			&=&4 (\underline{q})^{-1} \left(p^2-2q_0p+|q|^2-p^2+q_0p+ \bar{q}p- \bar{q}q_0 \right) \mathcal{Q}_{c,p}^{-2}(q)\\
			&=&4 (\underline{q})^{-1} \left[\mathcal{Q}_{c,p}(q)-(p- \bar{q})(p-q_0)\right] \mathcal{Q}_{c,p}^{-2}(q)\\
			&=&4 (\underline{q})^{-1} \mathcal{Q}_{c,p}^{-1}(q)-4 (\underline{q})^{-1} (p- \bar{q})(p-q_0) \mathcal{Q}_{c,p}^{-2}(q)\\
			&=& 2 (\underline{q})^{-1} \mathcal{Q}_{c,p}^{-1}(q)-4 (\underline{q})^{-1} (p- \bar{q})(p-q_0) \mathcal{Q}_{c,p}^{-2}(q)+2 (\underline{q})^{-1} \mathcal{Q}_{c,p}^{-1}(q)\\
			&=&-2 (\underline{q})^{-1} \partial_{q_0}S^{-1}_L(p,q)- (\underline{q})^{-2}\left[S_{L}^{-1}(p, \bar{q})-S^{-1}_L(p,q) \right].
		\end{eqnarray*}
		
		The result for the right $F$-kernel follows using  similar arguments.

	\end{proof}
	
To show some Leibniz-like formulas for the Laplace operator, the Cauchy-Fueter operator and its conjugate we need some preliminary results.

	\begin{proposition}
	\label{holser1}
	Let $f(x+iy)=\alpha(x,y)+i \beta(x,y)$ be a holomorphic function in $ \Omega \subset \mathbb{C}$ such that $\Omega$ is symmetric with respect to the $x$-axis and $ \Omega \cap \mathbb{R} \neq \emptyset$. We assume that the functions $\alpha$ and $\beta$ satisfy the even-odd conditions $\alpha(x,y)=\alpha(x,-y)$ and $\beta(x,y)=-\beta(x,-y)$. Then in a suitable neighbourhood of $\Omega \cap \mathbb{R}$ we have
	\begin{equation}
		\label{A1}
		\alpha(x,y)= \sum_{j=0}^{\infty} \frac{(-1)^j y^{2j}}{(2j)!} \partial_x^{2j}[\alpha(x,0)],
	\end{equation}
	\begin{equation}
		\label{A2}
		\beta(x,y)= \sum_{j=0}^{\infty} \frac{(-1)^j y^{2j+1}}{(2j+1)!} \partial_x^{2j+1}[\alpha(x,0)].
	\end{equation}
	Moreover the series \eqref{A1} and \eqref{A2} converge uniformly on the compact subsets of $\Omega$.
\end{proposition}
\begin{proof}
	The function $f(x+iy)$ is real-analytic in $y$ around any point $x \in \Omega \cap \mathbb{R}$. Thus we have
	\begin{equation}
		\label{holTaylor}
		f(x+iy)= \sum_{j=0}^{\infty} \frac{y^j}{j!} \partial_{y}^j[f(x)].
	\end{equation}
	Using the fact that $f$ is holomorphic we get
	$$ \partial_y [f(x+iy)]=i \partial_x [f(x+iy)].$$
	Moreover since $\alpha$ and $\beta$ satisfy the even-odd conditions we have that $f(x)=\alpha(x,0)$. Thus we can write the expression \eqref{holTaylor} as
	\begin{eqnarray*}
		f(x+iy)&=& \sum_{j=0}^{\infty} \frac{(iy)^j}{j!} \partial_x^j  \left[f(x) \right]\\
		&=& \sum_{j=0}^{\infty} \frac{(iy)^j}{j!} \partial_x^j  \left[\alpha(x,0) \right]\\
		&=& \sum_{j=0}^{\infty} \frac{(-1)^j y^{2j}}{(2j)!} \partial_x^{2j}  \left[\alpha(x,0) \right]+i\sum_{j=0}^{\infty} \frac{(-1)^j y^{2j+1}}{(2j+1)!} \partial_x^{2j+1}  \left[\alpha(x,0) \right]
		\\&=&\alpha(x,y)+i\beta(x,y).
	\end{eqnarray*}
	The uniform convergence of the above series is guaranteed by the fact that the function $f$ is holomorphic.
\end{proof}
The above result can be easily extended to slice hyperholomorphic functions. Note that  $x + i y$ denotes a complex variable $z$, whereas in the result below, the quaternionic variable $q$ is written as $u + I v$.

\begin{lemma}
	\label{holser}
	Let $U\subseteq\mathbb{H}$ be an axially symmetric open set that intersects the real line. Let us assume that $f(q)=\alpha(u,v)+I\beta(u,v)$ is a slice hyperholomorphic function in $U$. Then in a suitable neighbourhood of $U \cap \mathbb{R}$ we can write the functions $\alpha(u,v)$ and $\beta(u,v)$ as
	$$ \alpha(u,v)= \sum_{j=0}^{\infty} \frac{(-1)^j v^{2j}}{(2j)!} \partial_u^{2j}[\alpha(u,0)],$$
	$$ \beta(u,v)= \sum_{j=0}^{\infty} \frac{(-1)^j v^{2j+1}}{(2j+1)!} \partial_u^{2j+1}[\alpha(u,0)].$$
\end{lemma}
\begin{proof}
	The result follows by applying the slice operator, see \eqref{sliceop}, to the holomorphic function $f(x+iy)=\alpha(x,y)+i \beta(x,y)$, where $\alpha$ and $\beta$ are as in \eqref{A1} and \eqref{A2}. See also \cite[Lemma 4.5]{CDP25}.
\end{proof}

\begin{lemma}
\label{NNN}
Let $q=u+Iv$, with $u$, $v \in \mathbb{R}$ and $I \in \mathbb{S}$.
	Let $U\subseteq\mathbb{H}$ be an axially symmetric open set that intersects the real line. Let us assume that $f(q)=\alpha(u,v)+I\beta(u,v)$ is a slice hyperholomorphic function in $U$. Then we have
$$ \lim_{v \to 0} Df(q)=-2 \partial_uf(u).$$
\end{lemma}
\begin{proof}
If $q \notin \mathbb{R}$, setting $v=|\underline{q}|$ and $I_{\underline{q}}=\underline{q}/|\underline{q}|$, we can write the action of the Cauchy-Fueter operator to $f$ as
$$ Df(q)= \left( \partial_u \alpha(u,v)-\partial_v \beta(u,v)- \frac{2}{v} \beta(u,v) \right)+ I_{\underline{q}} \left(\partial_u \beta(u,v)+ \partial_v \alpha(u,v)\right),$$
see e.g. \cite{green}.
Using the fact that $\alpha$ and $ \beta$ satisfy the Cauchy-Riemann conditions we have
$$ Df(q)=-\frac{2}{v} \beta(u,v).$$
To show the result we have to take the limit as $v$ tends to zero of the above expression. By Lemma \ref{holser} we have
\begin{eqnarray*}
	\lim_{v \to 0} Df(q)&=&-2 \lim_{v \to 0} \frac{\beta(u,v)}{v}\\
	&=& -2 \lim_{v \to 0} \sum_{j=0}^{\infty} \frac{(-1)^j v^{2j}}{(2j+1)!} \partial_u^{2j+1}[\alpha(u,0)]\\
	&=&-2 \partial_u \alpha(u,0).
\end{eqnarray*}
Finally, by using Theorem \ref{rapp0} we get
\begin{eqnarray*}
	\lim_{v \to 0} Df(q)&=& -2 \partial_u \alpha(u,0)\\
	&=& - \partial_u \left[ f(u)+f(u) \right]\\
	&=&-2 \partial_u f(u).
\end{eqnarray*}
\end{proof}

	\begin{theorem}
		\label{Leib}
		Let $g$ be a left slice hyperholomorphic function in an axially symmetric open set $U\subseteq\mathbb{H}$. Then, for $q\in U$, we have
		\begin{equation}
			\label{prodlapla}
			\Delta (qg(q))=q \Delta g(q) +2 D g(q),
		\end{equation}
		\begin{eqnarray}
			\label{prodhar1}
			D(qg(q))&=&\bar{q}Dg(q)-2g(q)\\
			\label{Nprodhar}
			&=& q Dg(q)-2g(\bar{q}).
		\end{eqnarray}
		\begin{eqnarray}
			\label{dbar2}
			\bar{D}(qg(q))&=& 4g(q)+2q \partial_{q_0} g(q)- \bar{q} Dg(q)\\
			\label{polyprod}
			&=&2 g(q)+2 g(\bar{q}) +q \bar{D}g(q).
		\end{eqnarray}
	\end{theorem}
	\begin{proof}
To prove \eqref{prodlapla} we use Leibniz formula for the derivatives $ \partial_{q_0}$ and $ \partial_{q_i}$, with $1 \leq i \leq 3$, and we have
		\begin{equation}
			\label{NN1}
			\partial_{q_0}^{2}(qg(q))=q \partial_{q_0}^2g(q)+2 \partial_{q_0}g(q),
		\end{equation}
		and
		\begin{equation}
			\label{NN2}
			\partial_{q_i}^{2}(qg(q))=q \partial_{q_i}^2g(q)+2 e_i\partial_{q_i}g(q).
		\end{equation}
		Using \eqref{NN1} and \eqref{NN2} we deduce
		\begin{eqnarray*}
			\Delta (qg(q))&=& q\partial_{q_0}^2(g(q))+ q\sum_{i=1}^{3} \partial_{q_i}^2(g(q))+2 \partial_{q_0}g(q)+2 \sum_{i=1}^{3}e_i \partial_{q_i}g(q)\\
			&=&q \Delta g(q) +2 D g(q).
		\end{eqnarray*}
		
		We now show formula \eqref{prodhar1}. By \cite{Dixan} we know that $ \partial_{\underline{q}}(\underline{q}g(q))=-3g(q)-\underline{q} (\partial_{\underline{q}}g(q))-2 \mathbb{E}_{\underline{q}}g(q)$, where $\mathbb{E}_{\underline{q}}$ is the Euler operator defined in \eqref{Euler}. This implies that
		\begin{eqnarray}
			\nonumber
			D(qg(q))&=& \partial_{q_0}(qg(q))+q_0 \partial_{\underline{q}} g(q)+ \partial_{\underline{q}}(\underline{q}g(q))\\
			\nonumber
			&=& q \partial_{q_0}g(q)+q_0 \partial_{\underline{q}} g(q)-2g(q)-\underline{q} \partial_{\underline{q}}g(q)-2 \mathbb{E}_{\underline{q}}g(q)\\
			\nonumber
			&=&  q \partial_{q_0}g(q)+\bar{q} Dg(q)-2g(q)-2 \mathbb{E}_{\underline{q}}g(q)-\bar{q} \partial_{q_0}g(q)\\
			\label{NN4}
			&=& \bar{q} Dg(q) -2g(q)+(2 \underline{q} \partial_{q_0}g(q)-2 \mathbb{E}_{\underline{q}}g(q)).
		\end{eqnarray}
Using the fact that $g$ is a left slice hyperholomorphic function, and so it is in the kernel of the left global operator $V_{q,L}$ in \eqref{global}, see Remark also \ref{Nreg}, we have  for $q \notin \mathbb{R}$
		\begin{eqnarray}
			\nonumber
			2 \left( \underline{q} \partial_{q_0}g(q)-\mathbb{E}_{\underline{q}}g(q) \right)&=&2 \underline{q} \left(\partial_{q_0}g(q)- \underline{q}^{-1} \mathbb{E}_{\underline{q}}g(q) \right)\\
			\nonumber
			&=&2 \underline{q}\left(\partial_{q_0}g(q)+ \frac{\underline{q}}{| \underline{q}|^2} \mathbb{E}_{\underline{q}}g(q) \right)\\
\nonumber
			&=& 2\underline{q} V_{q,L}(g)(q)\\
						\label{NN3}
            &=&0.
		\end{eqnarray}
		By plugging \eqref{NN3} into \eqref{NN4} we get the result. Now we consider $q \in \mathbb{R}$. Since the function $qg(q)$ is slice hyperholomorphic, by Lemma \ref{NNN} we have that
		$$ \lim_{v \to 0} D(qg(q))=-2 \partial_{q_0} (q_0 g(q_0))=-2 g(q_0)-2q_0 \partial_{q_0} g(q_0).$$
		This proves formula \eqref{prodhar1} in the case $q \in \mathbb{R}$. In this case,
		the equivalence between \eqref{prodhar1} and \eqref{Nprodhar} is
		immediate. It remains to prove the equivalence between \eqref{prodhar1} and
		\eqref{Nprodhar} for $q\notin\mathbb{R}$. This is equivalent to showing
		that
		\begin{equation}\label{NN7}
			\underline{q} Dg(q)= g(\bar{q})-g(q).
		\end{equation}
		We write $g(q)=\alpha(u,v)+I \beta(u,v)$, with $q=u+Iv$ and we know from e.g. \cite[page 74]{Dixan} that
		$$ Dg(q)= \left(\partial_u \alpha(u,v)-\partial_v \beta(u,v)- \frac{2}{| \underline{q}|} \beta(u,v) \right)+ \frac{\underline{q}}{|\underline{q}|} (\partial_{u} \beta(u,v)+ \partial_v \alpha(u,v)), \qquad q \notin \mathbb{R}.$$
		Since the function $g$ is left slice hyperholomorphic, by using the Cauchy-Riemann conditions see \eqref{CR}, we get
		\begin{equation}\label{NN6}
			Dg(q)=- \frac{2}{| \underline{q}|}\beta(u,v).
		\end{equation}
		The even-odd conditions satisfied by $\alpha(u,v)$ and $ \beta(u,v)$ give
		\begin{equation}\label{NN5}
			g(\bar{q})-g(q)=-2 \frac{\underline{q}}{| \underline{q}|} \beta(u,v).
		\end{equation}
		Since \eqref{NN6} and \eqref{NN5} coincide we get the equality \eqref{Nprodhar}.
\\Now, we show \eqref{dbar2}. By using the fact that $D+ \bar{D}=2 \partial_{q_0} $ and by \eqref{prodhar1} we get
		\begin{eqnarray*}
			\bar{D}(qg(q))&=& 2 \partial_{q_0}(qg(q))-D(qg(q))\\
			&=&4g(q)+2q \partial_{q_0}g(q)-\bar{q}Dg(q).
		\end{eqnarray*}
		To show \eqref{polyprod} we use once again the fact that $D+ \bar{D}=2 \partial_{q_0} $ and by formula \eqref{Nprodhar}, we have
		\begin{eqnarray*}
			\bar{D}(qg(q))&=& 2 \partial_{q_0}(qg(q))-D(qg(q))\\
			&=&2g(q)+2q\partial_{q_0} g(q)+ 2 g(\bar{q})-qDg(q)\\
			&=& 2 g(\bar{q})+2g(q)+q\bar{D}g(q).
		\end{eqnarray*}

	\end{proof}

\begin{remark}
	The Leibniz-like formulas in Theorem \ref{Leib} can be generalized by replacing the monomial $q$ with a quaternionic polynomial of any degree, with real coefficients, see \cite[Thm. 9.3]{CDPS} for the Laplace operator, see \cite[Thm. 3.10]{DPS} for the Cauchy-Fueter operator and \cite[Thm. 3.9]{CPS} for the conjugate Cauchy-Fueter operator, respectively.
\end{remark}

	\section{Relation between the $*$-Taylor expansion and the spherical expansion}\label{SECTTRE}
	
	In this section we study a relation between the two different expansions in series of a slice hyperholomorphic function around a quaternion $p$. A first result is provided in \cite[Thm. 8.1]{GPS}, however the relation described in this paper is different.
	\\ We start by studying the regularity in the variable $p$ of the building blocks of the spherical series, see \eqref{set1}.
	
	\begin{lemma}
		Let $p$, $q \in \mathbb{H}$. Then the function $ Q_p^n(q)$, for fixed $n \in \mathbb{N}$, is left slice polyanalytic of order $n+1$ with respect to the variable $p$.
	\end{lemma}
	\begin{proof}
		The result follows from direct calculations and the polyanalytic decomposition, see Theorem \ref{polydeco}. Indeed
\[
\begin{split}
&Q_p^n(q)=(q^2-2p_0q+|p|^2)^n=\sum_{k=0}^n {n\choose k}(q^2+|p|^2)^{k}(-2p_0q)^{n-k}=\\
&=\sum_{k=0}^n \sum_{\ell=0}^{k} {n\choose k}{{k}\choose \ell}q^{2(k-\ell)}|p|^{2\ell}(-2qp_0)^{n-k}=\sum_{k=0}^n \sum_{\ell=0}^{k} {n\choose k}{{k}\choose \ell}(-1)^{n-k}q^{n+k-2\ell}(p+\bar p)^{n-k}(p\bar p)^{\ell}\\
\end{split}
\]
from which, grouping $\bar p$ on the left with maximum power $n$, the statement follows.
	\end{proof}
	
		In the sequel, we shall make use of the right global operator $V_{p,R}$, see in Remark \ref{VROPER}.
		To ease the notation we denote it simply as $V_p$, and we write:
		\begin{equation}
			\label{Rgho}
			V_p:= \frac{\partial}{\partial p_0}+ \left(\sum_{\ell=1}^{3}p_\ell \frac{\partial}{\partial p_{\ell}} \right) \frac{\underline{p}}{| \underline{p}|^2}, \qquad p \in \mathbb{H}\setminus \mathbb{R}.
		\end{equation}

	\begin{lemma}
		\label{out}
		Let $U\subseteq \mathbb{H}$ be an open set and let $f(p)$ be a quaternionic-valued function of class $\mathcal{C}^1(U)$.
		Then, for $p\in U\setminus(U\cap\mathbb{R})$, we have
		\begin{equation}
			\label{leib}
			V_p\left( f(p)p \right)= V_p(f(p)) p.
		\end{equation}
		In general, for every $n \in \mathbb{N}$ and $p\in U\setminus(U\cap\mathbb{R})$ we have that
		\begin{equation}
			\label{gen}
			V_p(f(p) p^{n})= V_p(f(p)) p^{n}.
		\end{equation}
	\end{lemma}
	\begin{proof}
		From the Leibniz rule for the real derivatives we get
		\begin{eqnarray*}
			V_p \left( f(p) p \right) &=& \frac{\partial}{\partial p_0}\left( f(p) p \right) + \left(\sum_{\ell=1}^{3}p_\ell \frac{\partial}{\partial p_{\ell}} (f(p)p )\right) \frac{\underline{p}}{| \underline{p}|^2}\\
			&=& \frac{\partial}{\partial p_0}\left( f(p) \right)p + f(p) \\
			&&+ f(p) \left( \sum_{\ell=1}^{3}p_\ell e_{\ell} \right) \frac{\underline{p}}{| \underline{p}|^2}+\left(\sum_{\ell=1}^{3}p_\ell \frac{\partial}{\partial p_{\ell}} (f(p)) \right) \frac{\underline{p}}{| \underline{p}|^2}p\\
			&=&\frac{\partial}{\partial p_0}\left( f(p) \right)p +\left(\sum_{\ell=1}^{3}p_\ell \frac{\partial}{\partial p_{\ell}} (f(p)) \right) \frac{\underline{p}}{| \underline{p}|^2}p\\
			&=&V_p(f(p))p.
		\end{eqnarray*}
		Now, we prove formula \eqref{gen} by induction on $n$. For $n=1$ the result follows by \eqref{leib}. We suppose that the statement is true for $n$ and we prove it for $n+1$. By using the inductive hypothesis and \eqref{leib} we get
		\begin{eqnarray*}
			V_p\left( f(p) p^{n +1} \right) &=& V_p \left( f(p) p^{n} p \right)  \\
			&=& V_p \left(f(p) p^{n} \right)p \\
			&=& V_p\left(f(p) \right)p^{n +1}.
		\end{eqnarray*}
	\end{proof}

	The repeated application of the right global operator $V_p$
	 to the building blocks $ Q_p^n(q)$, see \eqref{set1}, of the spherical series expansions
	leads to a remarkable relation described below.

	\begin{theorem}
		Let $m$, $n \in\mathbb N$. Then for $q \in \mathbb{H}$, $p \in \mathbb{H}\setminus\mathbb R$ we have
		\begin{equation}
			\label{app1}
			V_p^m \left(Q_p^n(q) \right)= (-2)^m \frac{n!}{(n-m)!} Q_p^{n-m}(q)(q-p)^{m*_{q,L}}, \qquad n \geq m.
		\end{equation}
	\end{theorem}
	\begin{proof}
 We show the result by induction on $m$. We start from the case $m=1$. By applying the derivative $ \partial_{p_0}$ to $Q_p^n(q)$, and then  $ \partial_{p_{\ell}}$, with $ \ell=1,2,3$, we get
		
		\begin{equation}
			\label{c1bis}
			\frac{\partial}{\partial p_0} Q_p^n(q)=2nQ_p^{n-1}(q)(p_0-q),
		\end{equation}

		\begin{equation}\
			\label{c2}
			\frac{\partial}{\partial p_{\ell}} Q_p^n(q)=2n  p_{\ell}Q_{p}^{n-1}(q).
		\end{equation}
		Hence \eqref{c1bis} and \eqref{c2} yield
		\begin{eqnarray}
						\label{conn0}
			V_p\left(Q_p^n(q) \right)&=& 2nQ_p^{n-1}(q)(p_0-q)+ 2 n \left( \sum_{\ell=1}^{3} p_{\ell}^2 \right)Q_{p}^{n-1}(q) \frac{\underline{p}}{| \underline{p}|^2}  \\
			\nonumber
			&=&2nQ_p^{n-1}(q)(p_0-q)+2n Q_{p}^{n-1}(q)\underline{p} \\
\nonumber
			&=&-2n Q_p^{n-1}(q) (q-p).
		\end{eqnarray}
		Let us suppose that the statement holds for $m$ and we prove it for $m+1$. For the sake of simplicity, below we indicate simply by $*$ all the $*_{q,L}$-products. First we observe that $Q_p^n(q)$ commutes with $q$ and its powers. By the inductive hypothesis, the binomial theorem (see \eqref{starLeR}), Lemma \ref{out} and formula \eqref{conn0} we get
		\begingroup\allowdisplaybreaks
		\begin{align}
			\nonumber
			V_p^{m+1}\left(Q_p^n(q) \right)&= V_p V_p^m\left(Q_p^n(q) \right) \\
			\nonumber
			&= 2^m \frac{n!}{(n-m)!}  V_p \left[ Q_p^{n-m}(q)(p-q)^{*m} \right] \\
			\nonumber
			&= 2^m \frac{n!}{(n-m)!} V_p \left[ Q_p^{n-m}(q) \sum_{\ell=0}^{m} \binom{m}{\ell} q^\ell p^{m-\ell} (-1)^{\ell} \right]\\
			\nonumber
			&= 2^m \frac{n!}{(n-m)!} \sum_{\ell=0}^{m} \binom{m}{\ell} q^\ell (-1)^{\ell}  V_p \left[ Q_p^{n-m}(q) p^{m-\ell}  \right]\\
			\nonumber
			&=2^m \frac{n!}{(n-m)!} \sum_{\ell=0}^{m} \binom{m}{\ell} q^\ell (-1)^{\ell}  V_p \left[ Q_p^{n-m}(q) \right]p^{m-\ell}\\
			\nonumber
			&= 2^{m+1} \frac{n!}{(n-m)!}(n-m) \sum_{\ell=0}^{m} \binom{m}{\ell} q^\ell (-1)^{\ell}  Q_p^{n-m-1}(q) (p-q) p^{m-\ell}\\
			\nonumber
			&= 2^{m+1} \frac{n!}{(n-m-1)!} Q_p^{n-m-1}(q) \sum_{\ell=0}^{m} \binom{m}{\ell} q^\ell (-1)^{\ell}  (p-q) p^{m-\ell}\\
			\nonumber
			&= 2^{m+1} \frac{n!}{(n-m-1)!} Q_p^{n-m-1}(q) \left(-q \sum_{\ell=0}^{m} \binom{m}{\ell} q^\ell p^{m-\ell}(-1)^{\ell} \right. \\
			\nonumber
			& \, \, \, \, \, \,  \left. +\sum_{\ell=0}^{m} \binom{m}{\ell} q^\ell p^{m-\ell}(-1)^{\ell} p \right)\\
			\nonumber
			&= 2^{m+1} \frac{n!}{(n-m-1)!} Q_p^{n-m-1}(q)\left[-q(p-q)^{*m}+(p-q)^{*m}p \right]\\
			\nonumber
			&= 2^{m+1} \frac{n!}{(n-m-1)!} Q_p^{n-m-1}(q)\left[(p-q)^{*m}*(p-q) \right]\\
			\label{global01}
			&= 2^{m+1} \frac{n!}{(n-m-1)!} Q_p^{n-m-1}(q)\left[(p-q)^{*(m+1)} \right],
		\end{align}\endgroup
which proves the statement.
	\end{proof}
	
	The application of the right global operator $V_p^n$ to the second part of the spherical series gives the following result.
	
	\begin{theorem}
		Let $m$, $ n \in \mathbb{N}$. Then for $q\in \mathbb{H}$, $p \in \mathbb{H}\setminus \mathbb{R}$ we have
		\begin{equation}
			\label{app2}
			V_p^m \left[ Q_p^n(q)(q-p) \right]=(-2)^m  \frac{n!}{(n-m)!} Q_p^{n-m}(q) (q-p)^{(m+1)*_{q,L}}, \qquad n \geq m.
		\end{equation}
	\end{theorem}
	\begin{proof}
		We show the result by induction on $m$. If $m=1$, by Lemma \ref{out}, the fact that $Q_p^n(q)$ commutes with $q$, and formula \eqref{conn0} we get
		\begin{eqnarray*}
			V_p \left[ Q_p^n(q)(q-p) \right] &=& V_p \left[ Q_p^n(q)q \right]-V_p \left[ Q_p^n(q)p \right] \\
			&=& q V_p \left[ Q_p^n(q)\right]-V_p \left[ Q_p^n(q) \right]p\\
			&=& 2n Q_p^{n-1}(q) q(p-q)- 2n Q_p^{n-1}(q) (p-q)p\\
			&=& 2n Q_p^{n-1}(q)(-p^2+2qp-q^2)\\
			&=& -2n Q_p^{n-1}(q)(q-p)^{*2}.
		\end{eqnarray*}
		Supposing that the statement is true for $m$, following arguments similar to those  used to prove \eqref{global01}, we can prove it for $m+1$.
	\end{proof}
	
	As a particular case of the previous results we have the following.
	
	\begin{corollary}
		\label{appn}
		Let $q\in \mathbb{H}$ and $p \in \mathbb{H}\setminus \mathbb{R}$. For $n \in\mathbb N$ we have that
		\begin{equation}
			\label{p1}
			V_p^n( Q_p^n(q))=c_n(q-p)^{n*_{q,L}},
		\end{equation}
		and
		\begin{equation}
			\label{p2}
			V_p^n\left(Q_p^n(q)(q-p)\right)=c_n (q-p)^{(n+1)*_{q,L}},
		\end{equation}
		where $c_n:= 2^n n! (-1)^n$.
	\end{corollary}
	\begin{proof}
		The result follows by taking $m=n$ in \eqref{app1} and \eqref{app2}. For $n=0$ both formulas are trivially true, since $c_0=1$, $V^0_p$ is the identity and $Q^0_p(q)=1$.
	\end{proof}
	
	Now we have all the tools to describe a relation between the truncated regular and spherical series. We prove the next theorem for $p \in \mathbb{H}\setminus\mathbb {R}$, since for $p\in\mathbb R$ the regular and the spherical series coincide.
	
	\begin{theorem}
		Let $p \in \mathbb{H}\setminus\mathbb {R}$, $ \{a_n \}_{n \in \mathbb{N}_0} \subseteq \mathbb{H}$ be a given sequence. Define a sequence $ \{b_n\}_{n \in \mathbb{N}_0} \subseteq \mathbb{H}$ by
			\begin{equation}\label{coeff}
			\begin{cases}
				b_0=c_0a_0\\
				b_n= c_n a_{2n}+c_{n-1}a_{2n-1},\ n\geq 1
			\end{cases}
		\end{equation}
		where $c_{n}:= 2^n n! (-1)^n$.
		Then for every $N\in\mathbb N_0$ and every $q\in\mathbb H$ we have
		$$  \sum_{n=0}^{N} (q-p)^{n*_{q,L}} b_n= \sum_{n=0}^{N} V_p^n \left(Q_p^n(q)\right) a_{2n}+ \sum_{n=0}^{N-1} V_p^n\left(Q_p^n(q)(q-p)\right) a_{2n+1},$$
the last sum being empty for $N=0$.
	\end{theorem}

	\begin{proof}
For the sake of simplicity, below we indicate simply by $*$ all the $*_{q,L}$-products. Using Corollary \ref{appn} and the relations \eqref{coeff}, we can write
		\begin{eqnarray*}
			\sum_{n=0}^N (q-p)^{*n} b_n&=& c_0a_0+ \sum_{n=1}^{N} c_n (q-p)^{*n} a_{2n}\\
			&& + \sum_{n=1}^{N} c_{n-1} (q-p)^{*n}a_{2n-1}\\
			&=&\sum_{n=0}^{N} c_n (q-p)^{*n} a_{2n}+\sum_{n=0}^{N-1} c_n (q-p)^{*(n+1)} a_{2n+1}\\
			&=& \sum_{n=0}^{N} V_p^n \left(Q_p^n(q)\right) a_{2n}+ \sum_{n=0}^{N-1} V_p^n\left(Q_p^n(q)(q-p)\right) a_{2n+1}.
		\end{eqnarray*}
		This proves the assertion.
	\end{proof}

	\section{Representation formulas of the quaternionic fine structure of Dirac-type}\label{REPFPRML}

	In this section we show that functions which are axially harmonic or  polyanalytic  of order 2 or axially Fueter regular
	satisfy a representation formula in the spirit of the one in Theorem \ref{rapp0} for slice hyperholomorphic functions.
	This is indeed a peculiarity of these sets of functions, obtained by applying  to slice hyperholomorphic functions the operators $D$, $\bar{D}$ or $\Delta$, respectively.

	\subsection{Harmonic case}
	
	The application of the Cauchy-Fueter operator to the monomial $q^n$ leads to the formula
	\begin{equation}
		\label{for1}
		D (q^n)=-2 n H_{n-1}(q), \qquad n \in\mathbb N,
	\end{equation}
	where
	\begin{equation}
		\label{two}
		H_n(q)= \frac{1}{n+1} \sum_{k=1}^{n+1}  \bar{q}^{k-1} q^{n+1-k}, \qquad n \in\mathbb N_0,
	\end{equation}
	see \cite{B, CDPS}.
	The polynomials $H_n(q)$ have remarkable properties, described below.
	\begin{lemma}
		Let $q \in \mathbb{H}$ then the polynomials $H_n(q)$ satisfy the following properties:
		\begin{itemize}
			\item[1)] $H_n(q)$ are axially harmonic.
			\item[2)] $H_n(q)$ are left slice polyanalytic of order $n+1$.
			\item[3)] $H_n(q)$ are an Appell-sequence with respect to the real derivative $\partial_{q_0}$, i.e.
			\begin{equation}
				\label{AppellH}
				\partial_{q_0} H_n(q)= n H_{n-1}(q), \qquad n \in\mathbb N.
			\end{equation}
			\item[4)] $H_n(q)$ satisfy the inequality
			\begin{equation}
				\label{ineqhar}
				| H_n(q)| \leq |q|^n.
			\end{equation}
			\item[5)] If we assume that $q \notin \mathbb{R}$ then
			\begin{equation}
				\label{splittingH}
				H_{n-1}(q)=-\frac{(\underline{q})^{-1}}{2n} \left[\bar{q}^n-q^n\right], \qquad n \in\mathbb N.
			\end{equation}
		\end{itemize}
	\end{lemma}
	\begin{proof}
		\begin{itemize}
			\item[1)] The first statement is a direct consequence of the Fueter mapping theorem, indeed by formula \eqref{for1} we have
			$$\Delta H_n(q)=- \frac{1}{2(n+1)} \Delta D q^{n+1}=0.$$
			\item[2)] By a change of index in the definition of the polynomials $H_n(q)$ we have
			$$ H_n(q)=\frac{1}{n+1} \sum_{\ell=0}^{n} \bar{q}^{\ell} q^{n- \ell} .$$
			Since $ \{q^{n- \ell }\}_{\ell=0}^{n}$ are slice hyperholomorphic,  Theorem \ref{polydeco} gives the result.
			\item[3)] Since $\partial_{q_0}$ is a real derivative we can apply the Leibniz rule, and so we get
			\begin{eqnarray*}
				\partial_{q_0} H_n(q)&=& \frac{1}{n+1} \sum_{k=1}^{n+1} \partial_{q_0} (q^{n+1-k} \bar{q}^{k-1}) \\
				&=& \frac{1}{n+1} \left(\sum_{k=1}^{n+1} (n+1-k)q^{n-k} \bar{q}^{k-1}+ \sum_{k=2}^{n+1} (k-1) q^{n+1-k} \bar{q}^{k-2} \right)\\
				&=& \frac{1}{n+1} \left(\sum_{k=1}^{n} (n+1-k)q^{n-k} \bar{q}^{k-1}+ \sum_{k=1}^{n} k q^{n-k} \bar{q}^{k-1} \right)\\
				&=& \sum_{k=1}^{n} q^{n-k} \bar{q}^{k-1}\\
				&=& nH_{n-1}(q).
			\end{eqnarray*}
			\item[4)]  By the definition of the polynomials $H_n(q)$, see \eqref{two}, we have
			$$ |H_n(q)| \leq \frac{1}{n+1} \sum_{k=1}^{n+1} |q|^{n+1-k} |q|^{k-1}=|q|^n.$$
			\item[5)] We know that
			$$ a^n-b^n=(a-b) \sum_{k=1}^{n} a^{n-k}b^{k-1},$$
			if $a$ and $b$ commute. By choosing $a=q$ and $b=\bar{q}$ in the definition of the polynomials $H_n(q)$ we have that
			\begin{eqnarray*}
				H_{n-1}(q)= \frac{1}{n} \sum_{k=1}^{n}  \bar{q}^{k-1} q^{n-k}=- \frac{(\underline{q})^{-1}}{2n}(\bar{q}^n-q^n).
			\end{eqnarray*}
		\end{itemize}
	\end{proof}
	
	It turns out that axially harmonic functions in $\mathcal{AH}(U)$ admit a representation formula.

	\begin{theorem}[Representation formula for axially harmonic functions]
		\label{axhrap}
		Let $U \subseteq \mathbb{H}$ be an axially symmetric domain. Let $f: U \to \mathbb{H}$ be a slice hyperholomorphic function according to Definition \ref{sh}. We assume that $q=u+I_{\underline{q}}v \notin \mathbb{R}$. Then, for any $I\in\mathbb{S}$, the axially harmonic function $g=Df$ can be written as
		\begin{equation}
			\label{rappH}
			g(q)=Df(q)= - (\underline{q})^{-1} I_{\underline{q}}I \left[f(u-Iv)-f(u+Iv)\right],
		\end{equation}
and $g$ is constant on each $[q]\subset \mathbb{H}\setminus\mathbb{R}$.
		If $q=u \in \mathbb{R}$ we have
		\begin{equation}
			\label{realharm}
			g(u)=Df(u)=-2 \frac{\partial}{\partial u} f(u).
		\end{equation}
	\end{theorem}
	\begin{proof}
		By hypothesis we can write the function $f$ as
		$$ f(q)= \alpha(u,v)+I_{\underline{q}}\beta(u,v), \qquad q=u+I_{\underline{q}} v,$$
		where $\alpha$ and $\beta$ satisfy the conditions in \eqref{EO} and \eqref{CR}.
		
		If $ q \notin \mathbb{R}$ we can apply the Cauchy-Fueter operator as in \cite{green}, or Corollary 3.3 in \cite{P1},
		so that we obtain
		$$ g(q)=Df(q)= \left( \partial_u \alpha(u,v)-\partial_v \beta(u,v)- \frac{2}{v} \beta(u,v) \right)+ I_{\underline{q}} \left(\partial_u \beta(u,v)+ \partial_v \alpha(u,v)\right).$$
		Using the fact that $\alpha$ and $ \beta$ satisfy the Cauchy-Riemann conditions and Theorem \ref{rapp0}, we get
		\begin{eqnarray*}
			Df(u+Iv)&=&- \frac{2}{v} \beta(u,v)\\
			&=& - | \underline{q}|^{-1}I \left[ f(u-Iv)-f(u+Iv)\right]\\
			&=& - (\underline{q})^{-1} I_{\underline{q}} I\left[ f(u-Iv)-f(u+Iv)\right].
		\end{eqnarray*}
As is well-known, see \cite{CSS} the expression $I\left[ f(u-Iv)-f(u+Iv)\right]=2\beta(u,v)$, see Theorem \ref{rapp0}, and depends on $u,v$ only, and $- (\underline{q})^{-1} I_{\underline{q}}=-1/|\underline{q}|=-1/v$ so $g$ is constant on $[q]$.
		If $q=u \in \mathbb{R}$ the formula follows by Lemma \ref{NNN}.
	\end{proof}
	\begin{remark}\label{rmk43}
		If $U \subseteq \mathbb{H}$ is an axially symmetric domain and $f,f_1: U\to\mathbb H$ are  slice hyperholomorphic functions such that $g=Df=Df_1$ the representation formula \eqref{rappH} is not affected. Indeed by \cite[Thm. 5.10]{CDPS} we can write $f_1(q)=f(q)+a$, where $a\in\mathbb H$ and it is readily seen that the right hand side of \eqref{rappH} is unchanged.
	\end{remark}
\begin{remark} The paper \cite{P1} studies formulas relating the Cauchy-Riemann operator, the
spherical Dirac operator, the differential operator characterizing slice regularity, and the
spherical derivative of a slice function. Our formula \eqref{rappH} can be deduced from the formula \cite[Proposition 3.2]{P1}, point (b) which makes use of the spherical derivative, but our interpretation here is different.
\end{remark}
	The representation formula for axially harmonic functions in $\mathcal{AH}(U)$ can alternatively be deduced from the integral representation in Theorem \ref{harmint}.
In the next result	the strategy of proof is based on the computation of residues, as is done in \cite{CSS} to prove the Cauchy formula.
	
	\begin{proposition}
		\label{harmrap}
		Let $U\subset\mathbb{H}$ be a bounded slice Cauchy domain, let $I\in\mathbb{S}$ and set  $dp_I=dp (-I)$.
		Suppose that $f$ is a slice hyperholomorphic function on a set that contains $\overline{U}$.
		Then, for $ q=u+I_{\underline{q}}v \notin \mathbb{R}$, with $I_{\underline{q}}=\frac{\underline{q}}{v}$, the left axially harmonic function $g=Df\in\mathcal{AH}_L(U)$ can be written as
		$$ g(q)=Df(q)=-\frac{1}{\pi} \int_{\partial(U \cap \mathbb{C}_I)} \mathcal{Q}_{c,p}^{-1}(q) dp_I f(p)= -(\underline{q})^{-1} I_{\underline{q}}I \left[f(u-Iv)-f(u+Iv)\right], $$
		where
		$
		\mathcal{Q}_{c,p}^{-1}(q)=(p^2-2q_0p+|q|^{2})^{-1}.
		$
	\end{proposition}
	
	\begin{proof}
		We  start by observing that the zeros of the polynomial $p^2-2q_0p+|q|^2$ are either a real point or a 2-sphere. If the zeros are not real, on any slice $ \mathbb{C}_I$ we have two zeros $p_{1,2}=u \pm I v$. In this case we can calculate the residues in the points $p_1$ and $p_2$ on the plane $ \mathbb{C}_I$. We suppose that $I \neq I_{\underline{q}}$ and we start with $p_1$. We set
		$$ p=u+Iv+ \varepsilon e^{I \theta},$$
		and we have that
		\begin{eqnarray*}
			p^2-2q_0p+|q|^2&=&u^2-v^2+ \varepsilon^2 e^{2I \theta}+2u \varepsilon e^{I \theta}+2Iv \varepsilon e^{I \theta}+2 uvI\\
			&&-2 u^2-2u vI-2u \varepsilon e^{I \theta}+ u^2+v^2\\
			&=& \varepsilon^2 e^{2 I \theta}+2I v \varepsilon e^{I \theta}\\
			&=& \varepsilon e^{I \theta} \left( \varepsilon e^{I \theta}+2I v\right).
		\end{eqnarray*}
		Since $dp_I= \varepsilon e^{I \theta} d \theta$ we have
		\begin{eqnarray*}
			\pi I_1^{\varepsilon}&=& \int_{0}^{2 \pi} \left(\varepsilon^2e^{2I \theta}+2Iv \varepsilon e^{I \theta}\right)^{-1} \varepsilon e^{I \theta} d \theta f(u+I v)\\
			&=&\int_{0}^{2 \pi}  \varepsilon^{-1} e^{-I \theta}\left(\varepsilon e^{I \theta}+2Iv \right)^{-1} \varepsilon e^{I \theta} d \theta f(u+I v)\\
			&=& \int_{0}^{2 \pi}  \left(\varepsilon e^{I \theta}+2Iv \right)^{-1}  d \theta f(u+I v).
		\end{eqnarray*}
		Now, for $\varepsilon \to 0$ and $q \notin \mathbb{R}$ we get
		\begin{eqnarray*}
			\pi I_1^{0}&=& \int_{0}^{2 \pi} (2I v)^{-1} d \theta f(u+Iv)\\
			&=& -I \pi\frac{1}{v} f(u+Iv).
		\end{eqnarray*}
		Recalling that $ I_{\underline{q}}:= \frac{\underline{q}}{v}$ we have
		
		\begin{eqnarray*}
			I_1^{0}&=&-I\frac{1}{v} f(u+Iv)\\
			&=&-(\underline{q})^{-1} I_{\underline{q}} I f(u+Iv).
		\end{eqnarray*}
		With similar computations we show that the residue in $p_2$ is
		$$ I_2^0= (\underline{q})^{-1} I_{\underline{q}} I  f(u-Iv).$$
		Finally, by the residue theorem we have
		\begin{eqnarray*}
			g(q)=Df(q)=
			-\frac{1}{\pi} \int_{\partial(U \cap \mathbb{C}_I)} \mathcal{Q}_{c,p}^{-1}(q) dp_I f(p)&=&-(I_1^0+I_2^0)\\
			&=&-(\underline{q})^{-1}I_{\underline{q}}I\left[f(u-Iv)-f(u+Iv)\right].
		\end{eqnarray*}
	\end{proof}
	It is also possible to write an axially harmonic function in terms of the Gamma-operator in \eqref{GG}.

	\begin{proposition}
		\label{rappa1}
		Let $U \subseteq \mathbb{H}$ be an axially symmetric domain. Let $f: U \to \mathbb{H}$ be a slice hyperholomorphic function. Then for $q \notin \mathbb{R}$, the axially harmonic function  $g=Df\in \mathcal{AH}_L(U)$ can be written as
		$$
		g(q)=Df(q)=- (\underline{q})^{-1} \Gamma_{\underline{q}}f(q).
		$$
	\end{proposition}
	
	\begin{proof}
		Assume that function $f$ is left slice hyperholomorphic in the variable $q$; then it is in the kernel of the left global operator $V_{q,L}$, see \eqref{global}. Precisely we have that
		$$ V_{q,L}(f)(q)=\left( \frac{\partial}{\partial q_0}+ \frac{\underline{q}}{|\underline{q}|^2}\mathbb{E}_{\underline{q}}\right)f(q)=0.$$
		This implies
		\begin{equation}
			\label{Eulerglobal}
			\partial_{q_0}f(q)=- \frac{\underline{q}}{|\underline{q}|^2}\mathbb{E}_{\underline{q}}f(q).
		\end{equation}
		Hence by \eqref{Frela} we conclude that
		\begin{eqnarray*}
			Df(q)&=& \partial_{q_0}f(q)+ \partial_{\underline{q}}f(q)\\
			&=& \frac{\underline{q}}{| \underline{q}|^2} \Gamma_{\underline{q}} f(q)\\
			&=&-(\underline{q})^{-1} \Gamma_{\underline{q}}f(q).
		\end{eqnarray*}
		
	\end{proof}

	\subsection{Axially polyanalytic functions of order $2$}
Firstly, we consider the conjugate Fueter operator applied to the monomial $q^n$ which leads to the expression
	\begin{equation}
		\label{poly4}
		\bar{D}q^n=2n \mathcal{P}_{2, n-1}(q), \qquad n \in\mathbb N,
	\end{equation}
	where
	\begin{equation}\label{poly4b}
		\mathcal{P}_{2,n}(q)=q^n+ \frac{1}{n+1} \sum_{k=1}^{n+1} \bar{q}^{k-1} q^{n+1-k},
	\end{equation}
	see \cite{Polyf1, Polyf2}. We prove some properties of the functions $\mathcal{P}_{2,n}(q)$.
	\begin{theorem}
		\label{polypro}
		Let $n \in \mathbb{N}$ and $q \in \mathbb{H}$. Then the functions $\mathcal{P}_{2,n}(q)$
are such that
		\begin{itemize}
			\item[1)] $\mathcal{P}_{2,n}(q)$ are homogeneous of degree $n$.
			\item[2)] $\mathcal{P}_{2,n}(q)$ are left axially polyanalytic of order 2 and left slice polyanalytic of order $n+1$.
			\item[3)] $\mathcal{P}_{2,n}(q)$ can be written in terms of the polynomials $H_n(q)$, see \eqref{two}, and the Clifford-Appell polynomials $ \mathcal{Q}_{n}(q)$, see \eqref{capoly}, i.e.,
			\begin{equation}
				\label{poly1}
				\mathcal{P}_{2,n}(q)=(n+2) \mathcal{Q}_{n}(q)-q_0n \mathcal{Q}_{n-1}(q),
			\end{equation}
			and
			\begin{equation}
				\label{crossR}
				\mathcal{P}_{2,n}(q)=q^n+H_n(q).
			\end{equation}
			\item[4)] For $q \in \mathbb{H}\setminus \mathbb{R}$ we can write  $\mathcal{P}_{2,n}(q)$ as
			\begin{equation}
				\label{poly5}
				\mathcal{P}_{2,n}(q)=q^{n} - \frac{(\underline{q})^{-1}}{2(n+1)}(\bar{q}^{n+1}-q^{n+1}).
			\end{equation}
		\end{itemize}
	\end{theorem}
	\begin{proof} We separately prove the various points.
		\begin{itemize}
			\item[1)]  The first point follows directly from the definition. Indeed, for $t \in \mathbb{R}$, we have $ \mathcal{P}_{2,n}(tq)=t^n\mathcal{P}_{2,n}(q)$.
			\item[2)] From the Fueter mapping theorem, see  Theorem \ref{Fueter}, we get that the functions $\mathcal{P}_{2,n}(q)$ are left axially polyanalytic of order 2. Then, the fact that $\mathcal{P}_{2,n}(q)$ are left slice polyanalytic of order $n+1$ is a consequence of Theorem \ref{polydeco}.
			\item [3)] Formula \eqref{poly1} can be deduced by \cite{Polyf2} while formula \eqref{crossR} follows from the definition of the polynomials $H_n(q)$.
			\item [4)] For $q \notin \mathbb{R}$, formula \eqref{poly5} follows by inserting formula \eqref{splittingH} in \eqref{crossR}.
		\end{itemize}
	\end{proof}

	\begin{corollary}
		Combining formulas \eqref{poly1} and \eqref{crossR} we can write the polynomials $H_n(q)$ in terms of the Clifford-Appell polynomials
		$$ H_n(q)=  (n+2) \mathcal{Q}_n(q)-q_0 n \mathcal{Q}_{n-1}(q)- q^n.$$
	\end{corollary}
	
	\begin{remark}
	In quaternionic analysis if a function is both slice hyperholomorphic and axially Fueter regular then it is a constant.
		In the polyanalytic case this is not true.
		For instance the polynomials $ \mathcal{P}_{2,1}(q)$ are (left) axially polyanalytic of order 2
		and are also (left) slice polyanalytic of order 2, see the second point of Theorem \ref{polypro}.
	\end{remark}

	\begin{theorem}(Representation formula for axially polyanalytic functions of order $2$).
		\label{polyrap}
		Let $U \subseteq \mathbb{H}$ be an axially symmetric domain and let $f: U \to \mathbb{H}$ be a slice hyperholomorphic function, see Definition \ref{sh}. Then if $q=u+I_{\underline{q}}v \notin \mathbb{R}$, for any $I\in\mathbb S$, the function $h=\bar D f\in\mathcal{AP}_2^L(U)$ can be written as
		\begin{equation}\label{RapprPo}
\begin{split}
			h(q)=\bar{D}f(q)=
& \left[\frac{\partial}{\partial u}f(q_I)+\frac{\partial}{\partial u}f(q_{-I})\right] + I_{\underline{q}}I\left[\frac{\partial}{\partial u}f(q_{-I})-\frac{\partial}{\partial u} f(q_I)\right]\\
&+(\underline{q})^{-1}I_{\underline{q}}I\left[f(q_{-I})-f(q_I)\right],
\end{split}
		\end{equation}
	where $q_{\pm I}=u\pm I v$. If $q =u\in \mathbb{R}$ we have
\begin{equation}
\label{realbar}
h(u)=\bar{D}f(u)=4 \frac{\partial}{\partial u} f(u).
\end{equation}
	\end{theorem}
	\begin{proof}
		Let us consider $q \notin \mathbb{R}$. Since $D+\bar{D}=2\frac{\partial}{\partial q_0}$ and $f$ is slice hyperholomorphic, $\frac{\partial}{\partial q_0}f(q)=f'(q)$. Therefore we have
		$$ h(q)=\bar{D}f(q)=2f'(q)-Df(q).$$
		Thus by using Theorem \ref{rapp0} and Theorem \ref{axhrap}, we get
		\begin{eqnarray*}
			h(q)=\bar{D}f(q)&=& \left[\frac{\partial}{\partial u}f(u+Iv)+\frac{\partial}{\partial u}f(u-Iv)\right]+ I_{\underline{q}}I\left[\frac{\partial}{\partial u}f(u-Iv)-\frac{\partial}{\partial u} f(u+Iv)\right]\\
			&& +(\underline{q})^{-1}I_{\underline{q}}I\left[f(u-Iv)-f(u+Iv)\right],
		\end{eqnarray*}
		and this gives the statement.
		Now, we assume that $q=u \in \mathbb{R}$. By \eqref{realharm} we get
		\begin{eqnarray*}
			\lim_{v \to 0} h(u+Iv)=\lim_{v \to 0} \bar{D}f(u+Iv)&=&2  \lim_{v \to 0} f'(u+Iv)- \lim_{v \to 0} Df(u+Iv)\\
			&=& 2 \frac{\partial }{\partial u} (\alpha(u,0))+2 \frac{\partial }{\partial u} (\alpha(u,0))\\
			&=& 4 \frac{\partial }{\partial u}  f(u).
		\end{eqnarray*}
	\end{proof}
	\begin{remark}
	Under the assumptions of the previous theorem, if $f_1: U\to\mathbb H$ is another slice hyperholomorphic function such that $h=\bar Df_1$,
		then the representation formula \eqref{RapprPo} remains unaffected. Indeed, by \cite[Thm. 3.14]{Polyf2} we know that $f_1(q)=f(q)+a$, where $a\in\mathbb H$.
		It is evident that the right-hand side of \eqref{RapprPo} is unchanged by this modification.

	\end{remark}
	Also in the polyanalytic case it is possible to obtain the representation  formula by computing the integral that characterizes the axially polyanalytic functions of order $2$, see Theorem \ref{polyint}. Here we state the result with a sketch of the proof.

	\begin{proposition}
		\label{polyrap1}
		Let $U\subset\mathbb{H}$ be a bounded slice Cauchy domain, let $I\in\mathbb{S}$ and set  $dp_I=dp (-I)$.
		Suppose that $f$ is a slice hyperholomorphic function on a set that contains $\bar{U}$,
		then for any $q =u+I_{\underline{q}}v\notin \mathbb{R}$, the function $h=\bar D f\in\mathcal{AP}_2^L(U)$ can be written as
		\begin{eqnarray*}
			h(q)=\bar Df(q)=\frac{1}{2 \pi}\int_{\partial (U \cap \mathbb{C}_I)} P_{2}^L(p,q) dp_If(p)&=& \left[\frac{\partial}{\partial u}f(q_I)+\frac{\partial}{\partial u}f(q_{-I})\right]\\
			&& + I_{\underline{q}}I\left[\frac{\partial}{\partial u}f(q_{-I})-\frac{\partial}{\partial u} f(q_I)\right]\\
			&&+(\underline{q})^{-1}I_{\underline{q}}I\left[f(q_{-I})-f(q_I)\right],
		\end{eqnarray*}
	where $q_{\pm I}=u\pm I v$ and the kernel $P_{2}^L(p,q)$ is defined in \eqref{polykernel}.
	\end{proposition}
	\begin{proof}
		By Proposition \ref{corr} we have
		\begin{eqnarray*}
			h(q)=\bar Df(q)=\frac{1}{2 \pi}\int_{\partial (U \cap \mathbb{C}_I)} P_{2}^L(p,q) dp_If(p)&=&2 \frac{\partial}{\partial q_0}\left[ \frac{1}{2 \pi}\int_{\partial (U \cap \mathbb{C}_I)} S^{-1}_L(p,q)dp_I f(p)\right]\\
			&&-\left(- \frac{1}{\pi}\int_{\partial (U \cap \mathbb{C}_I)} \mathcal{Q}_{c,p}^{-1}(q) dp_I f(p)\right).
		\end{eqnarray*}
		Finally by \cite[Theorem 2.8.2]{CSS} and Proposition \ref{harmrap} we get the final result.
	\end{proof}
	
	It is also possible to write the representation formula of axially polyanalytic functions of order 2 in terms of the Euler and Gamma operators as follows.
	
	\begin{proposition}
		\label{rpoli}
		Let $U \subseteq \mathbb{H}$ be an axially symmetric domain. Let $f: U \to \mathbb{H}$ be a slice hyperholomorphic function. Then the  axially polyanalytic function  (of order $2$) $h=\bar D f$ can be written as
		$$ h(q)=\bar{D}f(q)= (\underline{q})^{-1} \left[ 2 \mathbb{E}_{\underline{q}}f(q)+ \Gamma_{\underline{q}}f(q) \right], \qquad q\not\in\mathbb R.$$
	\end{proposition}
	\begin{proof}
Let us write $q=q_0+\underline{q}$.
		Since $ D+ \bar{D}=2 \frac{\partial}{\partial q_0}$ we get that
		$$ \bar{D}f(q)= 2 \frac{\partial}{\partial q_0}f(q)-Df(q).$$
		Since the function $f$ is slice hyperholomorphic the equality  \eqref{Eulerglobal} holds. So by Proposition \ref{rappa1} we conclude that
		\begin{eqnarray*}
			\bar{D}f(q)&=& 2 \frac{\partial}{\partial q_0}f(q)-Df(q)\\
			&=&(\underline{q})^{-1} \left[2 \mathbb{E}_{\underline{q}}f(q)+ \Gamma_{\underline{q}}f(q) \right].
		\end{eqnarray*}
	\end{proof}

	\subsection{Axially regular case}
	
	In this section we deal with axially regular functions. We begin by proving that any axially Cauchy-Fueter regular function $\breve f=\Delta f$, for some slice hyperholomorphic function  $f$, satisfies a representation formula in terms of  $f$. The function $f$ is called Fueter primitive of $\breve f$. The fact that $f$ exists is due to the surjectivity of the second map
	in the  Fueter mapping theorem onto the set of axially regular functions, see \cite{CSSOinverse}.
	\begin{theorem}[Representation formula for axially Fueter regular functions]
		\label{axmrap}
		Let $U \subseteq \mathbb{H}$ be an axially symmetric domain, and $f: U\to \mathbb H$ be a slice hyperholomorphic function according to Definition \ref{sh}. We assume that $q=u+I_{\underline{q}}v \notin \mathbb{R}$. Then, for any $I \in \mathbb{S}$, the axially Fueter regular function $\breve f=\Delta f$ can be written as
		\begin{equation} \label{rapm}
			\breve f(q)=\Delta f(q)=  |\underline{q}|^{-2}I_{\underline{q}} \left( 2|\underline{q}|\frac{\partial}{\partial u}f(q)
			- I\left(f(q_{-I})-f(q_{I})\right)\right),
		\end{equation}
		where $q_{\pm I}=u\pm Iv$. If $q =u\in \mathbb{R}$ we have
\begin{equation}
\label{deltareal}
 \Delta f(u)= -2 \frac{\partial^{2}}{\partial u^{2}} \left[f(u)\right].
\end{equation}
	The formulas do not depend on the choice of the Fueter primitive $f$.
\end{theorem}
	\begin{proof}
		We start by considering $q \notin \mathbb{R}$. By \cite{CSSOinverse} we know that $\breve f=\Delta f$ and following \cite{CSS3,SC} (or \cite{GHS}) we can compute $\Delta f$, where $f(q)=f(u+I_{\underline{q}}v)=\alpha(u,v) +I_{\underline{q}}\beta(u,v)$ is a slice hyperholomorphic function, so
		$$ \Delta f =2 \frac{\partial \alpha (u,v)}{\partial v} \frac{1}{v}+2 I_{\underline{q}} \left( \frac{\partial \beta(u,v)}{\partial v} \frac{1}{v}-\frac{\beta(u,v)}{v^2}\right).$$
		Now, by using the Cauchy-Riemann conditions we get
		\begin{eqnarray}
			\nonumber
			\Delta f(q)&=& - \frac{2}{v} \frac{\partial \beta(u,v)}{\partial u}+ 2 I_{\underline{q}} \frac{1}{v} \frac{\partial \alpha(u,v)}{\partial u}-2I_{\underline{q}} \frac{\beta(u,v)}{v^2}\\
			\nonumber
			&=& 2 \frac 1v I_{\underline{q}}\left[ \frac{\partial}{\partial u} \alpha(u,v)+ I_{\underline q}\frac{\partial}{\partial u} \beta(u,v) -\frac 1v \beta(u,v)\right]\\
			\label{rapm1}
			&=& 2 |\underline{q}|^{-1} I_{\underline{q}} \frac{\partial}{\partial u} f(q) -|\underline{q}|^{-2}  I_{\underline{q}} I \left[f(q_{-I})-f(q_I) \right].
		\end{eqnarray}
		We now show the independence of the formula \eqref{rapm} from the choice of the Fueter primitive $f$ of $\Delta f$. Let $f_1$ be another Fueter primitive of $\Delta f$. We know from \cite[Corollary 1]{CPS1} or  \cite[Thm. 4.11]{DDG1} that $f_1$ differs from $f$ by a slice hyperholomorphic function in the kernel of the Laplacian, namely by an affine function. Thus $f_1(q)=f(q)+ qa+b$, $a,b\in\mathbb H$ and so
		\[
		\begin{split}
			\Delta f_1(q)&=|\underline{q}|^{-2}I_{\underline{q}} \left( 2|\underline{q}|\frac{\partial}{\partial u}f_1(q)
			- I\left[f_1(q_{-I})-f_1(q_{I})\right]\right)\\
			&=|\underline{q}|^{-2}I_{\underline{q}} \left( 2|\underline{q}|\frac{\partial}{\partial u}f(q)+2a|\underline{q}|
			- I\left[f(q_{-I})-f(q_{I})+q_{-I}a- q_Ia\right]\right)\\
			&=|\underline{q}|^{-2}I_{\underline{q}} \left( 2|\underline{q}|\frac{\partial}{\partial u}f(q)+2a|\underline{q}|
			- I\left(f(q_{-I})-f(q_{I})-2I|\underline{q}|a \right)\right)\\
		\end{split}
		\]
		which coincides with the right hand side of \eqref{rapm}.\\
		Now, we consider $q \in \mathbb{R}$ and we compute the limit for $v\to 0$  of $\Delta f(q)$. By Lemma \ref{holser} we have
		\begin{eqnarray}
			\nonumber
			\lim_{v \to 0} \frac{\partial \beta(u,v)}{\partial u} \frac{1}{v}&=& \lim_{v \to 0} \sum_{j=0}^{\infty} \frac{(-1)^j v^{2j}}{(2j+1)!} \frac{\partial^{2j+2}}{\partial u^{2j+2}} \left[\alpha(u,0)\right]\\
			\label{auxR2}
			&=& \frac{\partial^{2}}{\partial u^{2}} \left[\alpha(u,0)\right].
		\end{eqnarray}
		By using once again Lemma \ref{holser} we have
		\begin{eqnarray*}
			\frac{1}{v} \frac{\partial \alpha(u,v)}{\partial u}- \frac{\beta(u,v)}{v^2} &=& \frac{1}{v} \frac{\partial}{\partial u} \alpha(u,0)-\frac{1}{v^2}\left( v \frac{\partial}{\partial u}\alpha(u,0)\right)\\
			&&+ \sum_{j=1}^{\infty} \frac{(-1)^j v^{2j-1}}{(2j)!} \frac{\partial^{2j+1}}{\partial u^{2j+1}} \alpha(u,0)- \sum_{j=1}^{\infty} \frac{(-1)^j v^{2j-1}}{(2j+1)!} \frac{\partial^{2j+1}}{\partial u^{2j+1}} \alpha(u,0)\\
			&=& \sum_{j=1}^{\infty} \frac{(-1)^j v^{2j-1}}{(2j)!} \frac{\partial^{2j+1}}{\partial u^{2j+1}} \alpha(u,0)- \sum_{j=1}^{\infty} \frac{(-1)^j v^{2j-1}}{(2j+1)!} \frac{\partial^{2j+1}}{\partial u^{2j+1}}\alpha(u,0).
		\end{eqnarray*}
		This implies that
		\begin{equation}
			\label{auxR1}
			\lim_{v \to 0} \left( \frac{1}{v} \frac{\partial \alpha(u,v)}{\partial u}- \frac{\beta(u,v)}{v^2}\right)=0.
		\end{equation}
		Finally by \eqref{auxR1} and \eqref{auxR2} we get
		\begin{eqnarray*}
			\lim_{v \to 0} \Delta f(q)&=& -2 \lim_{v \to 0} \frac{\partial \beta(u,v)}{\partial u} \frac{1}{v}+2 \lim_{v \to 0} I_{\underline{q}} \left( \frac{1}{v} \frac{\partial \alpha(u,v)}{\partial u}- \frac{\beta(u,v)}{v^2}\right)\\
			&=& -2 \frac{\partial^{2}}{\partial u^{2}} \left[\alpha(u,0)\right]\\
			&=& -2 \frac{\partial^{2}}{\partial u^{2}} \left[f(u)\right]
		\end{eqnarray*}
		and the statement follows.
	\end{proof}
	Other representation formulas for axially Fueter regular functions can be found via their integral representation, see Theorem \ref{Fueterint}. Here we consider the meaningful case $q \notin \mathbb{R}$.
	\begin{proposition}
		Let $U\subset\mathbb{H}$ be a bounded slice Cauchy domain, let $I\in\mathbb{S}$ and set  $dp_I=dp (-I)$.
		Then, for $q \notin \mathbb{R}$, an axially Fueter regular function $\breve f=\Delta f$, where $f\in\mathcal{SH}(V)$, $V\supset \bar{U}$, can be written as
		\begin{equation}
\begin{split}
		\breve{f}(q)=\Delta f&=	\frac{1}{2 \pi}\int_{\partial(U \cap \mathbb{C}_I)} F_L(p,q)dp_I f(p)\\
&= -2(\underline{q})^{-1} \frac{\partial}{\partial u}f(q)
			+(\underline{q})^{-2}I_{\underline{q}}I\left[f(q_{-I})-f(q_{I})\right],
\end{split}
		\end{equation}
		where $F_L$ is the kernel in \eqref{FK} and $q_{\pm I}=u\pm Iv$. The formula is independent of the choice of the Fueter primitive $f$.
	\end{proposition}
	\begin{proof}
		By formula \eqref{connection} and Theorem \ref{polyint} we deduce that
		\begin{eqnarray*}
			-4\int_{\partial(U \cap \mathbb{C}_I)} \mathcal{Q}_{c,p}^{-1}(q) dp_I f(p) &=&  \int_{\partial (U \cap \mathbb{C}_I)} F_L(p,q) p dp_I f(p)-q \int_{\partial(U \cap \mathbb{C}_I)} F_L(p,q) dp_I f(p)\\
			&=&\int_{\partial (U \cap \mathbb{C}_I)} F_L(p,q) (p-q_0) dp_I f(p)-\underline{q} \int_{\partial(U \cap \mathbb{C}_I)} F_L(p,q) dp_I f(p) \\
			&=& -\int_{\partial (U \cap \mathbb{C}_I)} P_2^L(p,q) dp_I f(p)-\underline{q} \int_{\partial(U \cap \mathbb{C}_I)} F_L(p,q) dp_I f(p).
		\end{eqnarray*}
		So we get
		\begin{eqnarray*}
			\underline{q} \int_{\partial(U \cap \mathbb{C}_I)} F_L(p,q) dp_I f(p) =  -\int_{\partial (U \cap \mathbb{C}_I)} P_2^L(p,q) dp_I f(p)+4\int_{\partial(U \cap \mathbb{C}_I)} \mathcal{Q}_{c,p}(q)^{-1} dp_I f(p).
		\end{eqnarray*}
		Hence by Proposition \ref{harmrap} and Proposition \ref{polyrap1} we have
		\begin{eqnarray*}
			\underline{q} \int_{\partial(U \cap \mathbb{C}_I)} F_L(p,q) dp_I f(p) &=& -2 \pi \left[\frac{\partial}{\partial u}f(q_I)+\frac{\partial}{\partial u}f(q_{-I})\right]- 2 \pi I_{\underline{q}}I\left[\frac{\partial}{\partial u} f(q_{-I})-\frac{\partial}{\partial u}f(q_I)\right] \\
			&&+2 \pi (\underline{q})^{-1}I_{\underline{q}}I\left[f(q_{-I})-f(q_I)\right].
		\end{eqnarray*}
		Finally we obtain
		\begin{eqnarray*}
			\frac{1}{2 \pi}\int_{\partial(U \cap \mathbb{C}_I)} F_L(p,q)dp_I f(p) &=& -(\underline{q})^{-1} \left[\frac{\partial}{\partial u}f(q_{I})+\frac{\partial}{\partial u}f(q_{-I})\right] \\
			&& -(\underline{q})^{-1}I_{\underline{q}}I\left[ \frac{\partial}{\partial u}f(q_{-I})- \frac{\partial}{\partial u}f(q_{I})\right]
			\\
			&&
			+(\underline{q})^{-2}I_{\underline{q}}I\left[f(q_{-I})-f(q_{I})\right].
		\end{eqnarray*}
	The independence of	the choice of $f$ follows from Theorem \ref{Fueterint} applied to $f$.
		
	\end{proof}
	
	We now write the representation formula of $\breve f$ in terms of the Euler and Gamma operators.
	
	\begin{proposition}
		Let $U \subseteq \mathbb{H}$ be an axially symmetric slice domain and let $\breve{f}\in\mathcal{AM}(U)$ be such that $\breve{f}=\Delta f$  with ${f}\in\mathcal{SH}(U)$. Then the axially regular function $\breve f$ can be written as
		$$ \breve{f}(q)=\Delta f(q)= | \underline{q}|^{-2} \left[2 \mathbb{E}_{\underline{q}}f(q)- \Gamma_{\underline{q}}f(q) \right], \qquad q \notin \mathbb{R}.$$
	\end{proposition}
	\begin{proof}
		By Theorem \ref{axmrap} and  Theorem \ref{polyrap}  we have
		$$ \Delta f(q)= -(\underline{q})^{-1} \bar{D}f(q)+2(\underline{q})^{-2}I_{\underline{q}}I \left(f(q_{-I})-f(q_I)) \right),$$
		and by Theorem \ref{axhrap} we get
		$$ \Delta f(q)= -(\underline{q})^{-1} \bar{D}f(q)-2(\underline{q})^{-1}Df(q).$$
		Hence by Proposition \ref{rappa1} and Proposition \ref{rpoli} we can write
		$$ \Delta f(q)=-(\underline{q})^{-2} \left[2 \mathbb{E}_{\underline{q}}f(q)-\Gamma_{\underline{q}}f(q) \right].$$
		Since $(\underline{q})^{-2}=-| \underline{q}|^{-2}$ we get the final result.
	\end{proof}

The following result will play an important role when applying the operators $D$, $\Delta$, and $\bar D$ to spherical series expansions.
\begin{corollary}
	\label{harmcompact}
	Let $U\subseteq\mathbb H$ be an axially symmetric domain and let $f:U\to\mathbb H$ be a slice hyperholomorphic function. Then, for $q\in U\setminus\mathbb R$, we have
	\begin{equation}
		\label{compactH}
		Df(q)=(\underline q)^{-1}\left[f(\bar q)-f(q)\right].
	\end{equation}
	
			\begin{equation}
		\label{compactPoly}
		\bar{D}f(q)=2\partial_{q_0}f(q)+(\underline{q})^{-1}\left[f(q)-f(\bar{q})\right].
	\end{equation}

	\begin{equation}
		\label{compactF}
		\Delta f(q)=-2(\underline q)^{-1}\partial_{q_0}f(q)+(\underline q)^{-2}\left[f(q)-f(\bar q)\right].
	\end{equation}
	If $q=u\in U\cap\mathbb R$, then we have \eqref{realharm}, \eqref{realbar} and \eqref{deltareal}, respectively.
\end{corollary}

\begin{proof}
The result follows directly from Theorems~\ref{axhrap},  \ref{polyrap} and \ref{axmrap} by choosing $I=I_{\underline q}$ and using $I_{\underline q}^2=-1$ , so that $f(u-I_{\underline q}v)=f(\bar q)$, $f(u+I_{\underline q}v)=f(q)$, $|\underline q|^{-1}I_{\underline q}=-(\underline q)^{-1}$ and $|\underline q|^{-2}=-(\underline q)^{-2}$.
\end{proof}

We note that, so far, we have considered the case of {\em left} slice hyperholomorphic (resp. axially harmonic, polyanalytic or regular) functions, but the discussion can be adapted with suitable modifications to treat the {\em right} slice hyperholomorphic case.

	\section{Taylor series: axially harmonic functions}\label{HARMSERIES}

	In this section we tackle the problem of writing a Taylor series of an axially harmonic function $Df \in \mathcal{AH}_L(U)$ at a point $p$, using the fact that the slice hyperholomorphic function $f$ around $p$ can be expanded either using its $*$-Taylor or its spherical series.
We deduce that $Df$ admits two different series expansions around $p$ that we call harmonic regular and harmonic spherical expansions. As it happens for slice hyperholomorphic functions the two expansions have different sets of convergence.
	
	\subsection{Harmonic regular series}
	To discuss an expansion for an axially harmonic function at a point $p \in \mathbb{H}$, we begin by giving the following definition:
	
	\begin{definition}
		\label{reg1}
		Let $ U \subseteq \mathbb{H}$ be an axially symmetric domain. Suppose $f$ is a slice hyperholomorphic function (as in Definition \ref{sh}), which admits a $*$-Taylor expansion at $p$ that converges in a suitable subset of $U$. Then the function $g = Df$, which is axially harmonic, is said to have a harmonic regular series at an arbitrarily fixed point $p$.
		
	\end{definition}
	
	In order to give a precise expression of a harmonic regular series we compute the action of $D$ on the building block $(q-p)^{*_{p,R}}=(q-p)^{*_{q,L}}$ (see \eqref{starLeR}) of the $*$-Taylor expansion. \\
	To ease the notation, throughout this subsection, we denote the $*_{q,L}$-products simply as $ * $.
	
	\begin{lemma}
		\label{genbeg}
		Let $p \in \mathbb{H}$, $ n\geq 1$ and $D$ be the Fueter operator in the variable $q$. Then:
		\begin{equation}
			\label{dapp}
			D(q-p)^{*n}=-2 \sum_{k=1}^{n} \left[(q-p)^{*(n-k)} * (\bar{q}-p)^{*(k-1)}\right],
		\end{equation}
	\end{lemma}
	\begin{proof}
		We prove the result by induction on $n$. Since
		$$ D(q-p)=-2,$$
		the result is trivial for $n=1$. So we suppose that the statement is valid for $n$ and we prove that it holds for $n+1$.
		Using the inductive hypothesis and formula \eqref{Nprodhar} we get
		\begingroup\allowdisplaybreaks
		\begin{eqnarray*}
			D[(q-p)^{*(n+1)}]&=& D[(q-p)^{*n}*(q-p)]\\
			&=& D[q(q-p)^{*n}-(q-p)^{*n}p]\\
			&=&D[q(q-p)^{*n}]- D[(q-p)^{*n}]p\\
			&=&-2(\bar{q}-p)^{*n}-2q \sum_{k=1}^{n} (q-p)^{*(n-k)}* (\bar{q}-p)^{*(k-1)}\\
			&& +2 \sum_{k=1}^{n} (q-p)^{*(n-k)}* (\bar{q}-p)^{*(k-1)}p\\
			&=&-2(\bar{q}-p)^{*n}-2\sum_{k=1}^{n} (q-p)^{*(n-k)}*(q-p)*(\bar{q}-p)^{*(k-1)}\\
			&=&-2(\bar{q}-p)^{*n}-2\sum_{k=1}^{n} (q-p)^{*(n+1-k)}* (\bar{q}-p)^{*(k-1)}\\
			&=&-2\sum_{k=1}^{n+1} (q-p)^{*(n+1-k)}* (\bar{q}-p)^{*(k-1)}.
		\end{eqnarray*}
		\endgroup
		This proves the result.
	\end{proof}
	\begin{remark}
		If we take $p=0$ in \eqref{dapp}
(note that we have to compute first the $*$-powers in $p$ and then to evaluate the expression so obtained) we reobtain the expression in \eqref{for1}.
	\end{remark}
	By setting
		\begin{equation}
			\label{Harmopoly}
			\widetilde{H}_n(q,p):=\frac{1}{n+1} \sum_{k=1}^{n+1} (q-p)^{*(n+1-k)} * (\bar{q}-p)^{*(k-1)}, \quad n\geq 0,
		\end{equation}
	formula \eqref{dapp} can be written as
		\begin{equation}
			\label{appreg1}
			D(q-p)^{*n}=-2n \widetilde{H}_{n-1}(q,p).
		\end{equation}

	The main properties of the functions $\widetilde{H}_n(q,p)$, which are polynomials in the variable $p$, are summarized in the following result.
	\begin{proposition}
		\label{harmpoly}
		Let $n \in \mathbb{N}$ and $q$, $p \in \mathbb{H}$. The polynomials $\widetilde{H}_n(q,p)$ satisfy the following properties:
		\begin{itemize}
			\item[1)] $\widetilde{H}_n(q,p)$ are left axially harmonic in $q$ and right slice hyperholomorphic in $p$.
			\item[2)] $\widetilde{H}_n(q,p)$ are left slice polyanalytic functions of order $n+1$ in $q$.
			\item[3)]  $\widetilde{H}_n(q,p)$ form an Appell sequence with respect to the derivative $\partial_{q_0}$, i.e.
			\begin{equation}
				\label{partial1}
				\partial_{q_0} \widetilde{H}_n(q,p)=n \widetilde{H}_{n-1}(q,p), \qquad n \geq 1.
			\end{equation}
			\item[4)] If $q \notin \mathbb{R}$, then $\widetilde{H}_n(q,p)$ can be written as
			\begin{equation}
				\label{closedH1}
				\widetilde{H}_{n-1}(q,p)=- \frac{\underline{q}^{-1}}{2n} \left[(\bar{q}-p)^{*n}-(q-p)^{*n}\right].
			\end{equation}
		\end{itemize}
	\end{proposition}
	\begin{proof}
		\begin{itemize}
			\item[1)]  Formula \eqref{appreg1} yields
			$$ \widetilde{H}_n(q,p)=- \frac{1}{2(n+1)} D (q-p)^{*(n+1)}$$
			and since the function $ (q-p)^{*(n+1)}$ is left slice hyperholomorphic in the variable $q$, by the Fueter mapping theorem we get
			$$ \Delta  \widetilde{H}_n(q,p)=- \frac{1}{2(n+1)} \Delta D (q-p)^{*(n+1)}=0.$$
			From the formula \eqref{starLeR} we immediately deduce the right slice hyperholomorphicity in the variable $p$.
			\item[2)] A change of index in the sum \eqref{Harmopoly} gives
			$$ \widetilde{H}_n(q,p)= \frac{1}{n+1}\sum_{\ell=0}^{n} (q-p)^{*(n-\ell)}* (\bar{q}-p)^{*\ell}.$$
			By the definition of the $*$-product and the binomial theorem we have
			$$ \widetilde{H}_n(q,p)= \frac{1}{n+1} \sum_{\ell=0}^{n} \sum_{i=0}^{\ell} \binom{\ell}{i} \bar{q}^{i} (q-p)^{*(n- \ell)} (-p)^{\ell-i},$$
			and, collecting in the sum the terms for $\ell=0$ and $i=0$:
			\begin{eqnarray*}
				\widetilde{H}_n(q,p)&=& \frac{1}{n+1} \left( \sum_{\ell=1}^{n} \sum_{i=0}^{\ell} \binom{\ell}{i} \bar{q}^{i} (q-p)^{*(n- \ell)} (-p)^{\ell-i}+(q-p)^{*n} \right)
				\\
				&=& \frac{1}{n+1} \left( \sum_{\ell=1}^{n} \sum_{i=1}^{\ell} \binom{\ell}{i} \bar{q}^{i} (q-p)^{*(n- \ell)} (-p)^{\ell-i}\right.
				\\
				&+&
				\left.(q-p)^{*n}
				+ \sum_{\ell=1}^{n} (-1)^\ell (q-p)^{*(n-\ell)}p^\ell \right)
				\\
				&=&\frac{1}{n+1} \left( \sum_{\ell=1}^{n} \sum_{i=1}^{\ell} \binom{\ell}{i} \bar{q}^{i} (q-p)^{*(n- \ell)} (-p)^{\ell-i}+ \sum_{\ell=0}^{n} (-1)^\ell (q-p)^{*(n-\ell)}p^\ell \right).
			\end{eqnarray*}
			We then collect the terms for $i=1$, $\ell=1$ and we get
			\begingroup\allowdisplaybreaks
			\begin{eqnarray*}
				\widetilde{H}_n(q,p)&=& \frac{1}{n+1} \left( \sum_{\ell=2}^{n} \sum_{i=2}^{\ell} \binom{\ell}{i} \bar{q}^{i} (q-p)^{*(n- \ell)} (-p)^{\ell-i}+ \bar{q}(q-p)^{*(n-1)} \right.\\
				&&\left.+ \sum_{\ell=0}^{n} (-1)^\ell (q-p)^{*(n-\ell)}p^\ell+\sum_{\ell=2}^{n} \bar{q} \binom{\ell}{1} (-1)^{\ell-1}(q-p)^{*(n-\ell)} p^{\ell-1} \right)\\
				&=& \frac{1}{n+1} \left( \sum_{\ell=2}^{n} \sum_{i=2}^{\ell} \binom{\ell}{i} \bar{q}^{i} (q-p)^{*(n- \ell)} (-p)^{\ell-i} \right.\\
				&&\left.+ \sum_{\ell=0}^{n} (-1)^\ell (q-p)^{*(n-\ell)}p^\ell+\sum_{\ell=1}^{n} \bar{q} \binom{\ell}{1} (-1)^{\ell-1}(q-p)^{*(n-\ell)} p^{\ell-1} \right).
			\end{eqnarray*}
			\endgroup
			By iterating the process we get
			\begingroup\allowdisplaybreaks
			\begin{eqnarray*}
				\widetilde{H}_n(q,p)&=& \frac{1}{n+1} \left( \sum_{\ell=2}^{n} \sum_{i=2}^{\ell} \binom{\ell}{i} \bar{q}^{i} (q-p)^{*(n- \ell)} (-p)^{\ell-i} \right.\\
				&&\left.+ \sum_{k=0}^{n} (-1)^k (q-p)^{*(n-k)}p^k +\sum_{k=1}^{n} \bar{q} \binom{k}{1} (-1)^{k-1}(q-p)^{*(n-k)} p^{k-1} \right.\\
				&& \left. +\sum_{k=2}^{n} \bar{q}^2 \binom{k}{2} (-1)^{k-2}(q-p)^{*(n-k)} p^{k-2}...+ \bar{q}^n\right)\\
				&=& \frac{1}{n+1}\sum_{\ell=0}^{n} \bar{q}^\ell \sum_{k=\ell}^{n} \binom{k}{\ell} (-1)^{k- \ell} (q-p)^{*(n-k)} p^{k- \ell}.
			\end{eqnarray*}
			\endgroup
			So we deduce
			$$  \widetilde{H}_n(q,p)= \sum_{\ell=0}^{n} \bar{q}^\ell h_\ell(q), \qquad h_\ell(q):=\frac{1}{n+1}\sum_{k=\ell}^{n} \binom{k}{\ell} (-1)^{k- \ell} (q-p)^{*(n-k)} p^{k- \ell},$$
			and since the functions $h_{\ell}(q)$ are slice hyperholomorphic in $q$, by Theorem \ref{polydeco} we conclude that the polynomials $\widetilde{H}_n(q,p)$ are left slice polyanalytic of order $n+1$.
			\item[3)] To prove the statement we need to compute the action of $\partial_{q_0}$ on the $*$-product. We first show that for $\ell, s \geq 1$ we have
			\begingroup\allowdisplaybreaks
			\begin{eqnarray}
				\label{Leif}
				\partial_{q_0} \left[(q-p)^{*\ell}*(\bar{q}-p)^{*s}\right]&=&\ell (q-p)^{*(\ell-1)}*(\bar{q}-p)^{*s}\\
				\nonumber
				&&+s(q-p)^{*\ell}* (\bar{q}-p)^{*(s-1)} .
			\end{eqnarray}
			\endgroup
			Formula \eqref{starLeR} and the Leibniz formula yield
			\begingroup\allowdisplaybreaks
			\begin{eqnarray*}
				\partial_{q_0} \left[(q-p)^{*\ell}*(\bar{q}-p)^{*s}\right]&=& \partial_{q_0}\left[\left(\sum_{i=0}^{\ell} \binom{\ell}{i} q^{\ell-i}(-p)^i \right) *\left(\sum_{j=0}^{s} \binom{s}{j} \bar{q}^{s-j} (-p)^j \right)\right]\\
				&=& \sum_{i=0}^{\ell} \sum_{j=0}^{s}\binom{\ell}{i} \binom{s}{j} \partial_{q_0} \left( q^{\ell-i}\bar{q}^{s-j}\right)(-p)^{j+i}\\
				&=& \sum_{i=0}^{\ell} \sum_{j=0}^{s}\binom{\ell}{i} \binom{s}{j} (\ell-i) q^{\ell-i-1}\bar{q}^{s-j}(-p)^{j+i}\\
				&&+\sum_{i=0}^{\ell} \sum_{j=0}^{s}\binom{\ell}{i} \binom{s}{j} (s-j) q^{\ell-i}\bar{q}^{s-j-1}(-p)^{j+i}\\
				&=& \ell\sum_{i=0}^{\ell} \sum_{j=0}^{s}\binom{\ell-1}{i} \binom{s}{j}  q^{\ell-i-1}\bar{q}^{s-j}(-p)^{j+i}\\
				&&+s\sum_{i=0}^{\ell} \sum_{j=0}^{s}\binom{\ell}{i} \binom{s-1}{j}  q^{\ell-i}\bar{q}^{s-j-1}(-p)^{j+i}\\
				&=& \ell\left(\sum_{i=0}^{\ell} \binom{\ell-1}{i} q^{\ell-i-1}(-p)^i \right) *\left(\sum_{j=0}^{s} \binom{s}{j} \bar{q}^{s-j} (-p)^j \right)\\
				&& +s\left(\sum_{i=0}^{\ell} \binom{\ell}{i} q^{\ell-i}(-p)^i \right) *\left(\sum_{j=0}^{s} \binom{s-1}{j} \bar{q}^{s-j-1} (-p)^j \right)\\
				&=&\ell (q-p)^{*(\ell-1)}* (\bar{q}-p)^{*s}+s(q-p)^{*\ell}* (\bar{q}-p)^{*(s-1)},
			\end{eqnarray*}
			\endgroup
			for $\ell, s \geq 1$.
			Now, we prove the desired formula \eqref{partial1}. By using \eqref{Leif} (with $\ell:=n+1-k$ and $s:=k-1$) we get the statement, in fact:
			\begingroup\allowdisplaybreaks
			\begin{eqnarray*}
				\partial_{q_0} \widetilde{H}_n(q,p)&=& \frac{1}{n+1} \sum_{k=1}^{n+1} (n+1-k) (q-p)^{*(n-k)}*(\bar{q}-p)^{*(k-1)}\\
				&&+ \frac{1}{n+1}\sum_{k=2}^{n+1} (k-1) (q-p)^{*(n+1-k)}*(\bar{q}-p)^{*(k-2)}
				\\
				&=&\frac{1}{n+1}\sum_{k=1}^{n} (n+1-k) (q-p)^{*(n-k)}*(\bar{q}-p)^{*(k-1)}\\
				&&+\frac{1}{n+1} \sum_{k=1}^{n} k (q-p)^{*(n-k)}*(\bar{q}-p)^{*(k-1)}\\
				&=& n \widetilde{H}_{n-1}(q,p).
			\end{eqnarray*}
			\endgroup
			\item[4)] By using \eqref{appreg1}, the binomial theorem, and formulas \eqref{for1} and \eqref{splittingH} we have:
			\begin{eqnarray*}
				\widetilde{H}_{n-1}(q,p)&=&- \frac{1}{2n} D(q-p)^{*n}\\
				&=&- \frac{1}{2n} D \left( \sum_{k=0}^{n} \binom{n}{k} q^k (-p)^{n-k}  \right)\\
				&=& \frac{1}{n} \sum_{k=1}^{n} k \binom{n}{k} H_{k-1}(q) (-p)^{n-k}\\
				&=& - \frac{1}{2n} \sum_{k=0}^{n} \binom{n}{k} \underline{q}^{-1}(\bar{q}^k-q^k) (-p)^{n-k}\\
				&=&- \frac{ \underline{q}^{-1}}{2n} \left[(\bar{q}-p)^{*n}-(q-p)^{*n} \right].
			\end{eqnarray*}
		\end{itemize}
	\end{proof}
	
	\begin{remark}
		If we consider $p=0$ in \eqref{closedH1} we get back the closed expression for the polynomials $H_n(q)$, see formula \eqref{splittingH}.
	\end{remark}

	In the next result we show that the application of the Cauchy-Fueter operator to the $*$-Taylor series does not affect its set of convergence. We recall that the $*$-Taylor series converges in $\Omega(p,r)$, with $p \in \mathbb{H}$ and $r \in \mathbb{R}$, see \eqref{setcov}.

	\begin{theorem}
		\label{regh}
		Let $U \subseteq \mathbb{H}$ and let $f$ be a slice hyperholomorphic function in $U$ and assume that $f$ admits the $*$-Taylor expansion at $p \in U$
		\begin{equation}
			\label{NNf}
			f(q)=\sum_{n=0}^\infty (q-p)^{*n}a_n
		\end{equation}
convergent (absolutely and uniformly on the compact subsets) in $\Omega(p,R) \subset U$, where $ \frac{1}{R}= \limsup_{n \to \infty} |a_n|^{\frac{1}{n}}$. Then
		\begin{equation}
			\label{reghar}
			Df(q)= \sum_{n=0}^{\infty} \tilde{H}_{n}(q,p) b_{n}, \qquad q \in \Omega(p,R), \quad b_n:=-2(n+1) a_{n+1}.
		\end{equation}

	\end{theorem}
	\begin{proof}
		For the sake of simplicity, below we indicate the $*_{q,L}$-products by $*$.
		First we apply the Cauchy-Fueter operator to \eqref{NNf} term by term, and using formula \eqref{appreg1} we get
		\begingroup\allowdisplaybreaks
		\begin{eqnarray}
			\nonumber
			Df(q)&=& \sum_{n=0}^{\infty} D(q-p)^{*n} a_n\\
			\nonumber
			&=&-2 \sum_{n=1}^{\infty} n \widetilde{H}_{n-1}(q,p) a_n\\
			\label{Nser}
			&=& -2 \sum_{n=0}^{\infty} \widetilde{H}_{n}(q,p) (n+1) a_{n+1}.
		\end{eqnarray}
		\endgroup
		We now focus on the convergence of the series. If $q \in \Omega(p,R)\cap\mathbb{R}$ the convergence of the series is immediate, so we assume that $q \notin \mathbb{R}$. Since $(q-p)^{*(n+1)}$ is a slice hyperholomorphic function in $q$, by Theorem \ref{axhrap} and formula \eqref{appreg1} we have
		\begin{eqnarray*}
			\nonumber
			-2 \widetilde{H}_{n}(q,p) (n+1) a_{n+1}&=& D (q-p)^{*(n+1)}a_{n+1}\\
			\nonumber
			&=& -(\underline{q})^{-1}I_{\underline{q}}I \left[(q-p)_{|q=q_{-I}}^{*(n+1)}a_{n+1}-(q-p)_{|q=q_{I}}^{*(n+1)}a_{n+1} \right]\\
			&=& -(\underline{q})^{-1}I_{\underline{q}}I \left[(q_{-I}-p)^{(n+1)}a_{n+1}-(q_I-p)^{(n+1)}a_{n+1} \right],
		\end{eqnarray*}
		where $p\in\mathbb{C}_I$ and $q_I=u+Iv$, $q_{-I}=u-Iv\in [q]$ are on that same plane. So we deduce that
		$$ |2 \widetilde{H}_{n}(q,p) (n+1) a_{n+1}| \leq |\underline{q}|^{-1} \left[\left|(q_{-I}-p)^{(n+1)}a_{n+1}\right|+\left|(q_I-p)^{(n+1)} a_{n+1} \right| \right].$$
	Since both the series
		$$ \sum_{n=0}^{\infty}\left|(q_{I}-p)^{(n+1)}a_{n+1}\right|,\qquad \sum_{n=0}^{\infty}\left|(q_{-I}-p)^{(n+1)}a_{n+1}\right|,$$
		are convergent in $\Omega(p,R)$, also the series \eqref{Nser} is convergent.
		
	\end{proof}
	
	\begin{remark}
		\label{remharm}
		If we consider $p=0$ in \eqref{reghar} we get
		\begin{equation}
			\label{zeroharm}
			g(q)=Df(q)= -2 \sum_{n=0}^{\infty} (n+1)H_n(q) a_{n+1},
		\end{equation}
		since $ \widetilde{H}_n(q,0)=H_n(q)$. Formula \eqref{zeroharm} is the Taylor expansion  of $Df$ in a neighbourhood of the origin.
	\end{remark}
	
	The harmonic regular series can be written also in terms of the harmonic polynomials $H_n(q)$, as proved in the next result:
	
	\begin{proposition} Let $f$ be a slice hyperholomorphic function in a neighbourhood of $p$ that admits at $p$ the $*$-Taylor expansion \eqref{NNf} with coefficients $\{a_n\}_{n \in \mathbb{N}_0} \subseteq \mathbb{H}$, convergent (absolutely and uniformly on the compact subsets) in $\Omega(p,R)$, where $\frac1R=\limsup_{n\to\infty}|a_n|^{1/n}$. Then the regular harmonic series of $Df$, for $q\in\Omega(p,R)$, can be written as
		$$ \sum_{n=0}^{\infty} \tilde{H}_{n}(q,p) b_{n}=2 \sum_{n=0}^{\infty} \sum_{k=0}^{n-1} n\binom{n-1}{k} H_{k}(q) p^{n- k-1} (-1)^{n- k} a_{n}, $$
		where $ b_n:=-2(n+1) a_{n+1}$, ${n \in \mathbb{N}_0}$. Moreover, if $q \notin \mathbb{R}$ we have
		\begin{equation}
			\label{closedH}
			\sum_{n=0}^{\infty} \tilde{H}_{n}(q,p) b_{n}=\underline{q}^{-1} \sum_{n=0}^{\infty}\left[ (\bar{q}-p)^{*n}-(q-p)^{*n} \right]a_{n}.
		\end{equation}
	\end{proposition}
	\begin{proof}
The assumption on the convergence on the compact subsets of $\Omega(p,R)$ allows us to apply $D$ term by term, and to rearrange the resulting double series. Thus, using \eqref{starLeR}, we can write
		\begin{equation}
			\label{fact}
			f(q)=  \sum_{n=0}^{\infty} \sum_{k=0}^{n} \binom{n}{k} q^k p^{n-k}(-1)^{n-k}a_{n}.
		\end{equation}
	Applying the Fueter operator $D$ to \eqref{fact}, and then using \eqref{for1} and  Theorem \ref{regh} we get
		\begin{eqnarray*}
			\sum_{n=0}^{\infty} \tilde{H}_{n}(q,p) b_{n}&=&Df(q)\\
			&=& -2 \sum_{n=0}^{\infty} \sum_{k=1}^{n} \binom{n}{k} k H_{k-1}(q) p^{n-k} (-1)^{n-k}a_{n} \\
			&=& -2 \sum_{n=0}^{\infty}n \sum_{k=1}^{n} \binom{n-1}{k-1} H_{k-1}(q) p^{n-k} (-1)^{n-k}a_{n} \\
			&=& 2 \sum_{n=0}^{\infty}n \sum_{k=0}^{n-1} \binom{n-1}{k} H_k(q) p^{n-k-1} (-1)^{n-k}a_{n}.
		\end{eqnarray*}
		We now suppose that $q \notin \mathbb{R}$. By formula \eqref{closedH1} and the binomial theorem (applied to the $*$-product of slice functions) we get
		\begin{eqnarray*}
			\sum_{n=0}^{\infty} \tilde{H}_{n}(q,p) b_{n}&=& -2 \sum_{n=0}^{\infty} \tilde{H}_{n}(q,p)(n+1)a_{n+1}\\
			&=& (\underline{q})^{-1} \sum_{n=0}^{\infty} \left[(\bar{q}-p)^{*(n+1)}-(q-p)^{*(n+1)} \right]a_{n+1}\\
			&=& (\underline{q})^{-1} \sum_{n=0}^{\infty} \left[(\bar{q}-p)^{*n}-(q-p)^{*n} \right]a_{n},
		\end{eqnarray*}
		as asserted.
	\end{proof}
	
	We provide an example of a specific function that can be written as a harmonic regular series.

	\begin{example}
In a neighborhood of the point $p+1$, the pseudo-Cauchy kernel, see \eqref{psudoCauchy}, can be written in terms of the harmonic polynomials $\widetilde{H}_n(q,p)$:
\begin{equation}
	\label{z1}
\mathcal{Q}_{c,p}^{-1}(q) = \sum_{n=0}^{\infty} \widetilde{H}_n(q,p+1)a_n, \qquad q \in \Omega(p+1,1),
\end{equation}
	where $\{ a_{n}\}_{n \in \mathbb{N}_0}:= \{(-1)^n (n+1)\}_{n \in \mathbb{N}_0}$.
	The convergence of the series $\sum_{n=0}^{\infty} \widetilde{H}_n(q,p+1)a_n$ follows with arguments similar to those used in the proof of Theorem \ref{regh}. We observe that by applying term by term the Fueter operator to \eqref{ex1} and using \eqref{appreg1} we have
	\begin{eqnarray*}
		D S^{-1}_L(p,q)&=& \sum_{n=0}^{\infty} (-1)^{n+1}D (q-p-1)^{*n}\\
		&=& 2 \sum_{n=1}^{\infty} (-1)^n n \widetilde{H}_{n-1}(q,p+1)\\
		&=& -2 \sum_{n=0}^{\infty} (-1)^n (n+1) \widetilde{H}_n(q,p+1).
	\end{eqnarray*}
	Since $\mathcal{Q}_{c,p}^{-1}(q)=- \frac 12 D S^{-1}_L(p,q)$, see \eqref{2punto*},  formula \eqref{z1} follows.
	\end{example}

	\subsection{Harmonic spherical series}
	
	In this section we prove that there exists another series expansion of an axially harmonic function in a neighborhood of $p$ whose set of convergence is open in the Euclidean topology.

	\begin{definition}\label{harmonic-spherical-series}
		Let $D$ be the Cauchy-Fueter operator and
		let $ U$ be an axially symmetric domain. If $f$ is a slice hyperholomorphic function in a neighbourhood of $p \in U$ where it admits a convergent spherical series expansion, then we say that $D f$ has a harmonic spherical series expansion in a neighbourhood of $p$.
	\end{definition}

		Definition \ref{harmonic-spherical-series} is motivated by Theorem \ref{harm11}. To prove this theorem, we introduce the following definition.
		\begin{definition}
\label{block}
			For $p \in \mathbb{H}$ and $q \in \mathbb{H} \setminus [p]$ we set
			\begin{equation}
				\label{key4}
				\mathcal{L}_n(q,p)=\begin{cases}
					\displaystyle\sum_{k=0}^{n-1} \overline{Q_p^{\,n-1-k}(q)}Q_p^k(q),
					\qquad  & n \geq 1\\
					0, & n=0\\
					\displaystyle	- \sum_{k=0}^{|n|-1} \overline{Q_p^{-k-1}(q)}Q_p^{k+n}(q), & n \leq -1
				\end{cases}
			\end{equation}
		\end{definition}
		The function $\mathcal{L}_n(q,p)$ is real-valued, since $\overline{\mathcal{L}_n(q,p)}=\mathcal{L}_n(q,p)$. We now establish a relation between $\mathcal{L}_n(q,p)$ and
		$Q_p^n(q)$, see \eqref{set1}.
		\begin{lemma}
			Let $p \in \mathbb{H}$ and $q \in \mathbb{H} \setminus [p]$, then we have
			\begin{equation}
				\label{laurentI}
				Q_p^n(\overline{q})-Q_p^n(q)=-4(q_0-p_0) \underline{q}\mathcal{L}_n(q,p), \qquad n \in \mathbb{Z}.
			\end{equation}
		\end{lemma}
		\begin{proof}
	The result is trivial for $n=0$.
By \eqref{set1} we have, for $n=1$,
			\begin{equation}
				\label{difference}
				Q_p(\bar q)-Q_p(q)=\bar q^{\,2}-q^2-2p_0(\bar q-q)=(\bar q-q)(\bar q+q)-2p_0(\bar q -q)=-4(q_0-p_0)\underline{q}.
			\end{equation}
Then, we prove the formula for $n> 1$.
	We use the well-known identity
			\begin{equation}
				\label{key1}
				A^n-B^n= (A-B) \sum_{k=0}^{n-1}A^{n-1-k}B^{k},
			\end{equation}
valid whenever $A$ and $B$ commute.
 Applying \eqref{key1} with $A:=Q_p(\bar q)$ and $B:=Q_p(q)$, we obtain
			\begin{equation}
				\label{important}
				Q_p^n(\bar{q})-Q_p^n(q)=-4(q_0-p_0) \underline{q}  \sum_{k=0}^{n-1} \overline{Q_p^{n-1-k}(q)}Q_p^k(q).
			\end{equation}

			Hence, by the definition of $\mathcal{L}_n(q,p)$, \eqref{laurentI} follows for
			$n\geq1$.
			
			We now consider the case $n\leq-1$. Set $m:=-n\geq1$. If $A$ and $B$
			commute, then
			\begin{equation}
				\label{rel}
				A^{-m}-B^{-m}=-(A-B) \sum_{k=0}^{m-1}A^{-k-1}B^{k-m}.
			\end{equation}
			We apply \eqref{rel} with $A:=Q_p(\bar q)$ and $B:=Q_p(q)$.
			Therefore,
			\begin{equation}
				Q_p^{-m}(\bar{q})-Q_p^{-m}(q)=4(q_0-p_0) \underline{q} \sum_{k=0}^{m-1} \overline{Q_p^{-k-1}(q)}Q_p^{k-m}(q).
			\end{equation}
			By the definition of $\mathcal{L}_{-m}(q,p)$,
			this proves \eqref{laurentI} also for $n\leq-1$.
			
		\end{proof}

		\begin{theorem}
			\label{harm11}
			Let  $f$ be a slice hyperholomorphic function in  an axially symmetric neighbourhood of $p\in\mathbb H$ containing $U(p,R)$ for some $R>0$ where it admits a spherical expansion
			\begin{equation}
				\label{ser0}
				f(q)= \sum_{n=0}^{\infty} Q_p^n(q) a_{2n}+\sum_{n=0}^{\infty} Q_p^n(q)(q-p) a_{2n+1}, \qquad \{a_n\}_{n \in \mathbb{N}_0} \subset \mathbb{H}
			\end{equation}
convergent (absolutely and uniformly on the compact subsets) in $U(p,R)$.
			Let $g$ be the axially harmonic function in a neighbourhood of $p$ given by $g=Df$.Then, $g$ has the following harmonic spherical expansion
			\begin{equation}
				\label{harsp}
				g(q)= -4  \sum_{n=1}^{\infty}\mathcal{L}_n(q,p)(q_0-p_0)[a_{2n}+(\bar{q}-p)a_{2n+1}]-2 \sum_{n=0}^{\infty}Q_p^n(q)a_{2n+1}
			\end{equation}
			which can also be written as
			\begin{equation}
				\label{harsp2}
				g(q)= -4  \sum_{n=1}^{\infty}\mathcal{L}_n(q,p)(q_0-p_0)[a_{2n}+(q-p)a_{2n+1}]-2 \sum_{n=0}^{\infty}Q_p^n(\bar{q})a_{2n+1}
			\end{equation}
			convergent (absolutely and uniformly on the compact subsets) in $U(p,R)$.
		\end{theorem}
		\begin{proof}
We apply the operator $D$ term by term to \eqref{ser0}. The interchange of $D$ and the summation will be justified later in the proof. For $q\notin\mathbb R$, by \eqref{compactH} we have
$$D[Q_p^n(q)]=(\underline q)^{-1}\left[Q_p^n(\bar q)-Q_p^n(q)\right].$$
Using \eqref{laurentI}, this gives
\begin{equation}
\label{DQ}
D[Q_p^n(q)]=-4(q_0-p_0)\mathcal L_n(q,p).
\end{equation}
Similarly,
$$D[Q_p^n(q)(q-p)]=(\underline q)^{-1}\left[Q_p^n(\bar q)(\bar q-p)-Q_p^n(q)(q-p)\right].$$
Since $q=\bar q+2\underline q$, we obtain
$$
\begin{aligned}
&Q_p^n(\bar q)(\bar q-p)-Q_p^n(q)(q-p)\\
&\qquad=\left[Q_p^n(\bar q)-Q_p^n(q)\right](\bar q-p)-2\underline q\,Q_p^n(q).
\end{aligned}
$$
Hence, by \eqref{laurentI},
\begin{equation}
\label{DQodd}
D[Q_p^n(q)(q-p)]=-4(q_0-p_0)\mathcal L_n(q,p)(\bar q-p)-2Q_p^n(q).
\end{equation}
Therefore, applying \eqref{DQ} and \eqref{DQodd} term by term to \eqref{ser0} yields \eqref{harsp}. Alternatively, writing $q=\bar q+2\underline q$ in the opposite way gives
\begin{equation}
\label{DQodd2}
D[Q_p^n(q)(q-p)]=-4(q_0-p_0)\mathcal L_n(q,p)(q-p)-2Q_p^n(\bar q),
\end{equation}
which yields \eqref{harsp2} (recall also that $\mathcal L_n(q,p)$ is real).
We now justify the interchange of $D$ and the summation. Since, for $q\in K$ compact subset of $U(p,R)$ we have $\sup_K |Q_p(q)|=\rho^2<R^2$ and $|\mathcal{L}_n(q,p)|\leq \sum_{k=0}^{n-1}
|\overline{Q_p^{\,n-1-k}(q)}|\, |Q_p^k(q)|= \sum_{k=0}^{n-1}
|Q_p^{\,n-1-k}(q)|\, |Q_p^k(q)|= \sum_{k=0}^{n-1}
|Q_p(q)|^{n-1}=n \rho^{2n-2}$ all the series in the previous calculations are convergent on the compact subsets of $U(p,R)$.
\\Finally, let $q=q_0\in\mathbb R$. Then $Q_p^n(q)$ and $Q_p^n(q)(q-p)$ are left slice hyperholomorphic in $q$ (the first one being intrinsic), so that Lemma \ref{NNN} gives $D(\cdot)=-2\partial_{q_0}(\cdot)$ at real points; since
\begin{equation}
\label{derv}
\partial_{q_0}Q_p^n(q_0)=2n(q_0-p_0)Q_p^{n-1}(q_0)
\end{equation}
we get
\begin{equation}
\label{real1}
D\left(Q_p^n(q_0)\right)=-4n(q_0-p_0)Q_p^{n-1}(q_0)
\end{equation}
\begin{equation}
\label{real2}
 D\left[Q_p^n(q_0)(q_0-p)\right]=-4n(q_0-p_0)Q_p^{n-1}(q_0)(q_0-p)-2Q_p^n(q_0).
\end{equation}
On the other hand, for $q=q_0\in\mathbb R$ one has $Q_p(q_0)\in\mathbb R$, whence $\mathcal{L}_n(q_0,p)=nQ_p^{n-1}(q_0)$ and $\bar q=q$. Therefore, the expressions in
\eqref{harsp} and \eqref{harsp2} reduce precisely to those obtained by applying $D$ term by term at the real points. Hence, \eqref{harsp} and \eqref{harsp2} hold also for $q=q_0\in U(p,R)\cap\mathbb R$.

		\end{proof}

		Our next goal is to study the convergence of the harmonic spherical series. To this end we recall the following result, see \cite[Lemma 2.3]{S}.
		
		\begin{lemma}
			\label{r1}
			Let $p_0$, $p_1 \in \mathbb{R}$, $p_1>0$, $q \in \mathbb{H}$ and $ r \geq 0$ be such that $ |(q-p_0)^2+p_1^2|=r^2$. If $p=p_0+Ip_1$, for some $I \in \mathbb{S}$,  then
			$$ \sqrt{r^2+p_1^2}-p_1 \leq |q-p| \leq \sqrt{r^2+p_1^2}+p_1.$$
		\end{lemma}
		From this result we derive further estimates needed to study the convergence of a harmonic spherical series.
		\begin{lemma}
			\label{newres}
			Let $p_0$, $p_1 \in \mathbb{R}$, $p_1>0$, $q \in \mathbb{H}$ and $ r \geq 0$ be such that $ |(q-p_0)^2+p_1^2|=r^2$. If $p=p_0+Ip_1$, for some $I \in \mathbb{S}$,  then
			
			\begin{equation}
				\label{newin2}
				|q-p_0| \leq \sqrt{r^2+p_1^2},
			\end{equation}
			\begin{equation}
				\label{est}
				|q_0-p_0| \leq  \sqrt{r^2+p_1^2},
			\end{equation}
			and
			\begin{equation}
				\label{newin}
			| \bar{q}-p| \leq  \sqrt{r^2+p_1^2}+p_1.
			\end{equation}
			
		\end{lemma}
		\begin{proof}
			The inequalities \eqref{newin2} and \eqref{est}  follow by the triangle inequality. Indeed we have
			$$ |q-p_0|^2=| (q-p_0)^2+p_1^2-p_1^2| \leq |(q-p_0)^2+p_1^2|+p_1^2=r^2+p_1^2,$$
			namely \eqref{newin2}.
			Since $|q_0-p_0| \leq |q-p_0|$ we get the estimate \eqref{est}. Finally, using $|\bar{q}-p_0|=|q-p_0|$, $|p-p_0|=|p_1|$, we have
			\begin{eqnarray*}
				| \bar{q}-p| & =& |\bar{q}-p+p_0-p_0|\\
				& \leq& |q-p_0|+|p-p_0|\\
				& \leq&  \sqrt{r^2+p_1^2}+p_1,
			\end{eqnarray*}
			which proves inequality \eqref{newin}.
		\end{proof}

		Recalling the definition of Cassini ball
		$ {U}(p, R)$, see \eqref{Cassiniball}, we have the next result.
		\begin{theorem}\label{AAADDD}
			\label{c1}
			With the notations in Theorem \ref{harm11},
			let $ \{a_n\}_{n \in \mathbb{N}_0} \subseteq \mathbb{H}$ be such that
			$$ \limsup_{n \to \infty} |a_n|^{\frac{1}{n}}= \frac{1}{R}, \qquad R>0.$$
			The harmonic spherical series \eqref{harsp} converges (absolutely and uniformly on the compact subsets) in
			${U}(p, R)$.
		\end{theorem}
		
		\begin{proof}
			Let $K$ be a compact subset of ${U}(p,R)$. By definition, if $q\in K$ then it satisfies $|(q-p_0)^2+p_1^2|\leq r^2$ for some $r$ such that $0<r < R$. Fix $R'$ with $r<R'<R$ so that $|a_n|\leq (1/R')^n$ for $n$ sufficiently large. This implies that for $n \geq 1$ we have
			$$|\mathcal{L}_n (q,p)| \leq n |Q_p(q)|^{n-1} \leq n r^{2n-2}.$$

			Then by Theorem \ref{harm11}, Lemma \ref{r1} and Lemma \ref{newres} we have
			
			\begin{eqnarray*}
				&&4|\mathcal{L}_n(q,p)| |q_0-p_0|[|a_{2n}|+|\bar{q}-p||a_{2n+1}|] + 2 |Q_p^n(q)||a_{2n+1}|\\
				& \leq &  \frac{4n}{r^2}\left( \sqrt{r^2+p_1^2}\right) \left(\frac{r}{R'}\right)^{2n}\left[1+\frac{1}{R'}\left(\sqrt{r^2+p_1^2}+p_1\right)\right]+2 \frac{1}{R'}\left(\frac{r}{R'}\right)^{2n}\\
				&:=& c_n.
			\end{eqnarray*}

			Hence the harmonic spherical series in \eqref{harsp} is dominated by $\sum_{n \in \mathbb{N}} c_n$ which is convergent and the statement follows.
		\end{proof}
		
	We can write the harmonic spherical series in terms of the harmonic polynomials $H_n$, see \eqref{two}, as follows:

\begin{theorem}
	\label{Tayorharmo0}
	Let $g$ be the axially harmonic function in a neighborhood of $p=p_0+Ip_1\in \mathbb{H}$, with $I \in \mathbb{S}$, $p_0$, $p_1 \in \mathbb{R}$, given by $g=Df$ where $f$ is a slice hyperholomorphic function. Assume that $f$ admits the spherical series expansion  \eqref{ser0} around $p$ with coefficients $\{a_n\}_{n \in \mathbb{N}_0} \subseteq \mathbb{H}$, convergent (absolutely and uniformly on the compact subsets) in a Cassini ball $U(p,R)$, $R>0$, as in Theorem \ref{harm11}. Then, the harmonic spherical series expansion  of $g$, for $q \in U(p,R)$, can be written as
	\begingroup\allowdisplaybreaks
	\begin{eqnarray}
		\label{Taylorharmo}
		g(q)&=&-2 \left( \sum_{n=1}^{\infty} \sum_{k=0}^{n-1} 2n\binom{n-1}{k}  H_{2k+1}(q-p_0) p_{1}^{2(n-k-1)}a_{2n}\right.\\
		\nonumber
		&&\left.+\sum_{n=0}^{\infty}\sum_{k=0}^{n} \binom{n}{k} (2k+1) H_{2k}(q-p_0) p_{1}^{2(n-k)}a_{2n+1}\right.\\
		\nonumber
		&&\left. -2 \sum_{n=1}^{\infty} \sum_{k=0}^{n-1}n \binom{n-1}{k}  H_{2k+1}(q-p_0)  p_{1}^{2(n-k-1)+1}Ia_{2n+1}\right).
	\end{eqnarray}
	\endgroup
	
\end{theorem}
\begin{proof}
The hypothesis that $f$ can be expanded as in \eqref{spherical}, absolutely and uniformly on the compact subsets of $U(p,R)$ allows us, below, to apply $D$ term by term, to expand each $Q_p^n(q)$ by the binomial theorem and to rearrange the resulting series. We then apply the Cauchy-Fueter operator to the leftmost summation in \eqref{spherical} using the binomial theorem and  \eqref{for1}. We get
	\begingroup\allowdisplaybreaks
	\begin{eqnarray}
		\nonumber
		\sum_{n=0}^{\infty}D\left( Q_p^n(q) \right)a_{2n} &=& D \left(\sum_{n=0}^{\infty} \sum_{k=0}^{n}\binom{n}{k} (q-p_0)^{2k} p_{1}^{2(n-k)} a_{2n}\right)\\
		\nonumber
		&=& \sum_{n=0}^{\infty}\sum_{k=0}^{n}\binom{n}{k}D\left((q-p_0)^{2k} \right) p_{1}^{2(n-k)}a_{2n}\\
		\nonumber
		&=& -2  \sum_{n=0}^{\infty}\sum_{k=1}^{n} \binom{n}{k} 2k H_{2k-1}(q-p_0) p_{1}^{2(n-k)}a_{2n}\\
		\nonumber
		&=&-4 \sum_{n=1}^{\infty} \sum_{k=0}^{n-1} \binom{n}{k+1} (k+1) H_{2k+1}(q-p_0) p_{1}^{2(n-k-1)}a_{2n}\\
		\label{f16}
		&=& -4 \sum_{n=1}^{\infty} \sum_{k=0}^{n-1} n\binom{n-1}{k}  H_{2k+1}(q-p_0) p_{1}^{2(n-k-1)}a_{2n}.
	\end{eqnarray}
	\endgroup
	We now apply the Cauchy-Fueter operator to the second summation of the expansion in series of $f$. By the binomial theorem and \eqref{for1} we get
	\begingroup\allowdisplaybreaks
	\begin{eqnarray}
		\nonumber
		\sum_{n=0}^{\infty}D\left(Q_p^n(q)(q-p)\right)a_{2n+1}&=& \sum_{n=0}^{\infty}\sum_{k=0}^{n} \binom{n}{k} D(q-p_0)^{2k+1} p_{1}^{2(n-k)}a_{2n+1}\\
		\nonumber
		&&-\sum_{n=0}^{\infty} \sum_{k=0}^{n} \binom{n}{k} D(q-p_0)^{2k} p_{1}^{2(n-k)+1}I a_{2n+1}\\
		\nonumber
		&=& -2\sum_{n=0}^{\infty}\sum_{k=0}^{n} \binom{n}{k} (2k+1) H_{2k}(q-p_0) p_{1}^{2(n-k)}a_{2n+1}\\
		\nonumber
		&&+2 \sum_{n=0}^{\infty}\sum_{k=1}^{n} \binom{n}{k} (2k) H_{2k-1}(q-p_0)  p_{1}^{2(n-k)+1}Ia_{2n+1}\\
		\nonumber
		&=&  -2 \sum_{n=0}^{\infty}\sum_{k=0}^{n} \binom{n}{k} (2k+1) H_{2k}(q-p_0) p_{1}^{2(n-k)}a_{2n+1}\\
		\label{f17}
		&&+4 \sum_{n=1}^{\infty} \sum_{k=0}^{n-1}n \binom{n-1}{k}  H_{2k+1}(q-p_0)  p_{1}^{2(n-k-1)+1}Ia_{2n+1}.
	\end{eqnarray}
	\endgroup
	Finally, the result follows by adding \eqref{f16} and \eqref{f17}.

\end{proof}

\begin{remark}
	\label{zeroharmo}
	In the hypothesis of Theorem \ref{Tayorharmo0}, by taking $p=0$ in \eqref{Taylorharmo} we obtain
	\begin{eqnarray}
		\label{zeroharmo11}
		g(q)&=& -2  \sum_{n=1}^{\infty} (2n)H_{2n-1}(q)a_{2n}-2 \sum_{n=0}^{\infty} (2n+1) H_{2n}(q) a_{2n+1}\\
		\nonumber
		&=& -2  \sum_{n=0}^{\infty} (2n+2)H_{2n+1}(q)a_{2n+2}-2 \sum_{n=0}^{\infty} (2n+1) H_{2n}(q) a_{2n+1}\\
		\nonumber
		\label{zeroharmo1}
		&=& -2  \sum_{n=0}^{\infty} (n+1) H_{n}(q)a_{n+1}.
	\end{eqnarray}
	Formula \eqref{zeroharmo11} is the Taylor expansion in a neighbourhood of the origin of a function in $ \mathcal{AH}_L(U)$, where $U$ is a slice Cauchy domain containing the origin. The same result was obtained by taking $p=0$ in the harmonic regular series, see Remark \ref{remharm}.
\end{remark}	

The harmonic spherical series can also be written in terms of $Q_p^n(q)$ and $Q_p^n(\overline{q})$ in the next result, whose proof is given by direct calculations although it could be deduced also using the corresponding Representation Formula.
		
	\begin{theorem}
		\label{harm}
		Let  $f$ be a slice hyperholomorphic function in  a neighbourhood of $p\in\mathbb H$ having spherical expansion
		\begin{equation}
			\label{ser00}
			f(q)= \sum_{n=0}^{\infty} Q_p^n(q) a_{2n}+\sum_{n=0}^{\infty} \left(Q_p^n(q)(q-p)\right) a_{2n+1}, \qquad a_n\in\mathbb{H}
		\end{equation}
convergent (absolutely and uniformly on the compact subsets) in $U(p,R)$, for some $R>0$.
Let $g$ be the axially harmonic function in a neighborhood of $p$ given by $g=Df$. Then, for $q\notin\mathbb{R}$, $g$ has the following harmonic spherical expansion
		\begin{equation}
	g(q) = (\underline{q})^{-1} \sum_{n=0}^{\infty}\left(  Q_p^n(\bar{q})-Q_p^n(q) \right) a_{2n}\\
		  +(\underline{q})^{-1} \sum_{n=0}^{\infty} \left(  Q_p^n(\bar{q})(\bar{q}-p)-Q_p^n(q)(q-p)\right)a_{2n+1}
,
		\end{equation}
		which can also be written as
		\begin{equation}	g(q) = (\underline{q})^{-1} \sum_{n=0}^{\infty} Q_p^n(\bar{q})\left(  a_{2n}+(\bar{q}-p)a_{2n+1}\right)\\
		  -(\underline{q})^{-1} \sum_{n=0}^{\infty} Q_p^n(q) \left(a_{2n}+(q-p)a_{2n+1}\right).
		\end{equation}
	\end{theorem}
	\begin{proof}

We apply  term by term the Cauchy-Fueter operator $D$ in the variable $q$ to formula \eqref{ser00}, by making use of the fact that, since $Q_p(q)=(q-p_0)^2+p_1^2$, we can write $Q_p^n(q)=\sum_{k=0}^{n} \binom{n}{k}(q-p_0)^{2k}p_1^{2(n-k)}$.
		
More specifically, let us begin by applying the Cauchy-Fueter operator $D$ to the leftmost summation in \eqref{ser0}. From formulas \eqref{for1} and \eqref{splittingH} we have that
		\begin{eqnarray*}
			\nonumber
			D\left(\sum_{n=0}^{\infty} Q_p^{n}(q)a_{2n}\right)&=& D\left(\sum_{n=0}^{\infty}\sum_{k=0}^{n} \binom{n}{k}(q-p_0)^{2k}p_1^{2(n-k)} a_{2n}\right)
\\
			\nonumber
			&=&(\underline{q})^{-1} \sum_{n=1}^{\infty} \sum_{k=0}^{n-1} \frac{n}{k+1}\binom{n-1}{k} \left[ (\bar{q}-p_0)^{2(k+1)}-(q-p_0)^{2(k+1)}\right] p_{1}^{2(n-k-1)}a_{2n}\\
			\nonumber
			&=&(\underline{q})^{-1} \sum_{n=1}^{\infty} \sum_{k=0}^{n-1}\binom{n}{k+1} \left[ (\bar{q}-p_0)^{2(k+1)}-(q-p_0)^{2(k+1)}\right] p_{1}^{2(n-k-1)}a_{2n}\\
			&=& (\underline{q})^{-1} \sum_{n=1}^{\infty} \sum_{k=1}^{n}\binom{n}{k} \left[ (\bar{q}-p_0)^{2k}-(q-p_0)^{2k}\right] p_{1}^{2(n-k)}a_{2n}\\
&=& (\underline{q})^{-1} \sum_{n=1}^{\infty} \sum_{k=0}^{n}\binom{n}{k} \left[ (\bar{q}-p_0)^{2k}-(q-p_0)^{2k}\right] p_{1}^{2(n-k)}a_{2n}\\
			&=& (\underline{q})^{-1} \sum_{n=1}^{\infty}\left[Q_p^n(\bar{q})-Q_p^n(q)\right]a_{2n}.
		\end{eqnarray*}
We then apply the Cauchy-Fueter operator $D$ to the rightmost summation in \eqref{ser0}. From formulas \eqref{prodlapla}, \eqref{for1} and \eqref{splittingH} we get
		\begingroup\allowdisplaybreaks
		\begin{eqnarray*}
			\nonumber
			&&D\left(\sum_{n=0}^{\infty} Q_p^{n}(q)(q-p)a_{2n+1}\right)=
D\left(\sum_{n=0}^{\infty} \sum_{k=0}^{n} \binom{n}{k}\left(q(q-p_0)^{2k}p_1^{2(n-k)}- (q-p_0)^{2k}p_1^{2(n-k)}p\right) a_{2n+1}\right)\\
			&=& (\underline{q})^{-1}\left(\sum_{n=0}^{\infty}\sum_{k=0}^{n} \binom{n}{k} (\bar{q}-p_0)^{2k+1} p_1^{2(n-k)}-\sum_{k=0}^{n} \binom{n}{k} (q-p_0)^{2k+1} p_1^{2(n-k)}\right)a_{2n+1}\\
			\nonumber
			&& -(\underline{q})^{-1}\left(\sum_{n=0}^{\infty} \sum_{k=0}^{n} \binom{n}{k} (\bar{q}-p_0)^{2k} p_1^{2(n-k)}-\sum_{k=0}^{n} \binom{n}{k} (q-p_0)^{2k} p_1^{2(n-k)}\right)p_1 Ia_{2n+1}\\
			\nonumber
			&=&(\underline{q})^{-1} \sum_{n=0}^{\infty}\left(\left[(\bar{q}-p_0)^2+p_1^2\right]^n(\bar{q}-p_0)-\left[(q-p_0)^2+p_1^2\right]^n(q-p_0)\right)a_{2n+1}\\
			\nonumber
			&& -(\underline{q})^{-1}\sum_{n=0}^{\infty}\left([(\bar{q}-p_0)^2+p_1^2]^n-[(q-p_0)^2+p_1^2]^n\right)p_1I a_{2n+1}\\
			&=&(\underline{q})^{-1}\sum_{n=0}^{\infty} \left(Q_p^n(\bar{q})(\bar{q}-p)-Q_p^n(q)(q-p)\right)a_{2n+1}.
		\end{eqnarray*}
		\endgroup
		Hence we have
		\begingroup\allowdisplaybreaks
		\begin{eqnarray*}
			Df(q)&=&    D\left(\sum_{n=0}^{\infty} Q_p^{n}(q)a_{2n}\right)+D\left(\sum_{n=0}^{\infty} Q_p^{n}(q)(q-p)a_{2n+1}\right)\\
			&=& (\underline{q})^{-1} \sum_{n=0}^{\infty}\left(  Q_p^n(\bar{q})-Q_p^n(q) \right) a_{2n}\\
			&&  +(\underline{q})^{-1} \sum_{n=0}^{\infty} \left(  Q_p^n(\bar{q})(\bar{q}-p)-Q_p^n(q)(q-p)\right)a_{2n+1},
		\end{eqnarray*}
		\endgroup
where, as in all the previous similar results, all the computations are justified by the uniform convergence on compact subsets of $U(p,R)$, and both series converge wherever \eqref{ser00} converges. The statement follows.
	\end{proof}

A natural question to ask is whether there is a connection between the harmonic regular series and the harmonic spherical series. The answer is affirmative and involves the use of the right global operator introduced in \eqref{Rgho}.

		\begin{theorem}
			\label{conna}
			Let $p \in \mathbb{H}\setminus\mathbb{R}$ and let $\{a_n\}_{n \in \mathbb{N}_0} \subseteq \mathbb{H}$ be a given sequence. Define another sequence $\{b_n\}_{n \in \mathbb{N}_0} \subseteq \mathbb{H}$ by
			\begin{equation}
				\label{rel11}
				b_{n-1}=-2n \left(c_n a_{2n}+c_{n-1}a_{2n-1} \right), \qquad n \geq 1,
			\end{equation}
			where $c_n:=2^n n! (-1)^n$. Then the following relations between the truncated harmonic regular series and the harmonic spherical series hold
			\begin{eqnarray}
				\nonumber
				\sum_{n=0}^{N}  \widetilde{H}_n(q,p) b_n&=& -4\sum_{n=0}^{N+1} V_p^n \left((q_0-p_0)\mathcal{L}_n(q,p) \right)a_{2n}-4\sum_{n=0}^{N}V_p^n((q_0-p_0)\mathcal{L}_n(q,p)(q-p))a_{2n+1}\\
								\label{11bis}
				&&-2 \sum_{n=0}^{N}  V_p^n \left(Q_p^n(\overline{q})\right)a_{2n+1}.
			\end{eqnarray}
			and
			\begin{eqnarray}
				\nonumber
							\nonumber
			\sum_{n=0}^{N}  \widetilde{H}_n(q,p) b_n&=& -4\sum_{n=0}^{N+1} V_p^n \left((q_0-p_0)\mathcal{L}_n(q,p) \right)a_{2n}-4\sum_{n=0}^{N} V_p^n((q_0-p_0)\mathcal{L}_n(q,p)(\overline{q}-p))a_{2n+1}\\
				\label{22bis}
			&&-2 \sum_{n=0}^{N}  V_p^n \left(Q_p^n(q)\right)a_{2n+1}.
			\end{eqnarray}
			where $N \in \mathbb{N}_0$ is fixed.
	\end{theorem}
	\begin{proof}
		By \eqref{appreg1} we deduce  that

$$
			\sum_{n=0}^{N}  \widetilde{H}_n(q,p) b_n= -\sum_{n=0}^{N} D(q-p)^{*(n+1)} \frac{b_n}{2(n+1)}= -\sum_{n=1}^{N+1} D(q-p)^{*n} \frac{b_{n-1}}{2n}.
$$

		The relation \eqref{rel11} and Corollary \ref{appn} give
		\begin{eqnarray*}
			\sum_{n=0}^{N}  \widetilde{H}_n(q,p) b_n&=&  \sum_{n=1}^{N+1} D\left(c_n (q-p)^{*n}\right)a_{2n}+ \sum_{n=1}^{N+1} D\left(c_{n-1} (q-p)^{*n}\right)a_{2n-1}\\
			&=&  \sum_{n=0}^{N+1} D\left(c_n (q-p)^{*n}\right)a_{2n}+ \sum_{n=0}^{N} D\left(c_{n} (q-p)^{*(n+1)}\right)a_{2n+1}\\
			&=& \sum_{n=0}^{N+1} D \left(V_p^n(Q_p^n(q)) \right)a_{2n}+ \sum_{n=0}^{N} D\left(V_p^n   \left[ Q_p^n(q)(q-p) \right]\right)a_{2n+1}.
		\end{eqnarray*}
		Now, by the definition of the Cauchy-Fueter operator and the fact that $V_p^n$ is a left linear quaternionic operator we get
		\begin{eqnarray}
			\nonumber
			\sum_{n=0}^{N}  \widetilde{H}_n(q,p) b_n&=& \sum_{n=0}^{N+1} V_p^n(\partial_{q_0}Q_p^n(q))a_{2n}+ \sum_{n=0}^{N}  V_p^n \left(\partial_{q_0} \left[ Q_p^n(q)(q-p) \right]\right)a_{2n+1}\\
			\nonumber
			&& + \sum_{n=0}^{N+1} \sum_{\ell=1}^3 e_\ell V_p^n(\partial_{q_\ell}Q_p^n(q))a_{2n}+ \sum_{n=0}^{N}\sum_{\ell=1}^{3}  e_{\ell} V_p^n \left(\partial_{q_\ell} \left[ Q_p^n(q)(q-p) \right]\right)a_{2n+1}\\
			\nonumber
			&=&\sum_{n=0}^{N+1} V_p^n(\partial_{q_0}Q_p^n(q))a_{2n}+ \sum_{n=0}^{N}  V_p^n \left(\partial_{q_0} \left[ Q_p^n(q)(q-p) \right]\right)a_{2n+1}\\
			\nonumber
			&&  +\sum_{n=0}^{N+1} \sum_{\ell=1}^3 V_p^n(e_\ell \partial_{q_\ell}Q_p^n(q))a_{2n}+ \sum_{n=0}^{N}\sum_{\ell=1}^{3}  V_p^n \left( e_{\ell} \partial_{q_\ell} \left[ Q_p^n(q)(q-p) \right]\right)a_{2n+1}\\
			\label{starr}
			&=& \sum_{n=0}^{N+1}V_p^n \left(D(Q_p^n(q) )\right)a_{2n}+ \sum_{n=0}^{N}  V_p^n \left(  D \left[ Q_p^n(q)(q-p) \right]\right)a_{2n+1}.
		\end{eqnarray}
		Now, by using \eqref{DQ} and \eqref{DQodd2} we get
		\begin{eqnarray}
			\nonumber
			\sum_{n=0}^{N}  \widetilde{H}_n(q,p) b_n&=& -4\sum_{n=0}^{N+1} V_p^n \left((q_0-p_0)\mathcal{L}_n(q,p) \right)a_{2n}-4\sum_{n=0}^{N}V_p^n((q_0-p_0)\mathcal{L}_n(q,p)(q-p))a_{2n+1}\\
			\label{a1bis}
			&&-2 \sum_{n=0}^{N}  V_p^n \left(Q_p^n(\overline{q})\right)a_{2n+1}.
		\end{eqnarray}
This proves \eqref{11bis}. To prove \eqref{22bis} we use \eqref{DQodd} and we get
		\begin{eqnarray}
			\nonumber
			\sum_{n=0}^{N}  \widetilde{H}_n(q,p) b_n&=& -4\sum_{n=0}^{N+1}V_p^n \left((q_0-p_0)\mathcal{L}_n(q,p) \right)a_{2n}-4\sum_{n=0}^{N} V_p^n((q_0-p_0)\mathcal{L}_n(q,p)(\overline{q}-p))a_{2n+1}\\
			\label{a2bis}
			&&-2 \sum_{n=0}^{N}  V_p^n \left(Q_p^n(q)\right)a_{2n+1}.
		\end{eqnarray}
	\end{proof}

	\section{Laurent series: axially harmonic functions}\label{Sect:6}
	
	In this section our goal is to find and study a Laurent series expansion of an axially harmonic function $g=Df\in  \mathcal{AH}_L(U)$ around a generic quaternion $p$. As we discussed in the case of the Taylor series, also for the Laurent series of $g$ we have two possibilities, namely we have  a $*$-Laurent or a spherical Laurent series expansion.
	The two expansions have different convergence sets. As in the previous sections, we shall omit specifying that we are considering the left case.

	\subsection{Laurent harmonic regular series}
	In this part of the section we investigate a first notion of Laurent series in the framework of axially harmonic functions around a generic quaternion $p$, namely the $*$-Laurent series.
	\begin{definition}
		
		Let $ \Omega \subseteq \mathbb{H} $ be an axially symmetric open set, and suppose \( f \) is a slice hyperholomorphic function in $ \Omega $ that has a $*$-Laurent expansion at an arbitrarily fixed point $p \in \mathbb{H}$ convergent in a subset of $\Omega$. Then we say that $g = Df$ possesses a Laurent harmonic regular series at $p$.
		
	\end{definition}
	
	We start by applying the Cauchy-Fueter operator $D$ to the function of two variables $p,q$ given by $(q-p)^{-n*_{q,L}}=(q-p)^{-n*_{p,R}}$.  We recall that $*_{q,L}$ is the $*$-product acting on the slice functions (not necessarily slice hyperholomorphic) of the form $(q-p)^{n *_{q,L}}$ or $(\bar q-p)^{n *_{q,L}}$.  For clarity, throughout this subsection, we will denote all $ *_{q,L} $-products simply by $ * $.
	\begin{theorem}
		\label{actionN1}
		Let $q$, $p \in \mathbb{H}$ such that $q \notin [p]$ and $D$ be the Fueter operator in the variable $q$. Then for $ n \geq 1$ we have
		$$ D(q-p)^{-*n}= 2 \left(\sum_{k=1}^{n} (\bar{q}-p)^{*(n-k)}*(q-p)^{*(k-1)}  \right) \mathcal{Q}_{c,p}^{-n}(q),$$
		where $\mathcal{Q}_{c,p}^{-n}(q)=(p^2-2q_0p+|q|^2)^{-n}$.
	\end{theorem}
	\begin{proof}
		We prove the result by induction on $n$. We begin by considering $n=1$. By formula \eqref{investar}, Proposition \ref{relform} and \eqref{psudoCauchy} we have
		$$D(q-p)^{-*}=-D S^{-1}_L(p,q)=2 \mathcal{Q}_{c,p}^{-1}(q).$$
		This proves the result for $n=1$. We now suppose that the statement is true for $n$ and we prove it for $n+1$. By \eqref{starL}, Remark \ref{NstarLL} and \eqref{Leif} and using the induction hypothesis we have the following chain of equalities:
		\begingroup\allowdisplaybreaks
		\begin{eqnarray*}
			D(q-p)^{-*(n+1)}&=& \frac{(-1)^{n+1}}{n!} \partial_{q_0}^{n} DS^{-1}_L(p,q)\\
			&=&- \frac{1}{n} \partial_{q_0} D\left[ \frac{(-1)^{n}}{(n-1)!} \partial_{q_0}^{n-1}  S^{-1}_L(p,q) \right]\\
			&=& - \frac{1}{n} \partial_{q_0} D(q-p)^{-*n}\\
			&=& - \frac{2}{n} \partial_{q_0} \left[\left( \sum_{k=1}^{n} (\bar{q}-p)^{*(n-k)}*(q-p)^{*(k-1)}  \right) \mathcal{Q}_{c,p}^{-n}(q)  \right]\\
			&=& - \frac{2}{n}  \left[\left( \sum_{k=1}^{n} (\bar{q}-p)^{*(n-k)}*(q-p)^{*(k-1)} \right) \partial_{q_0}\left(\mathcal{Q}_{c,p}^{-n}(q)\right)  \right]\\
			&&- \frac{2}{n}  \left[\partial_{q_0}\left( \sum_{k=1}^{n}  (\bar{q}-p)^{*(n-k)}*(q-p)^{*(k-1)} \right) \mathcal{Q}_{c,p}^{-n}(q)  \right]\\
			&=& - \frac{2}{n}  \left[\left( \sum_{k=1}^{n}  (\bar{q}-p)^{*(n-k)}*(q-p)^{*(k-1)} \right) \left(-2n(q_0-p)\mathcal{Q}_{c,p}^{-n-1}(q) \right)  \right]\\
			&&- \frac{2}{n}  \left[\left(\sum_{k=2}^{n} (k-1)(\bar{q}-p)^{*(n-k)}*(q-p)^{*(k-2)} \right) \mathcal{Q}_{c,p}^{-n}(q)  \right]\\
			&& - \frac{2}{n}  \left[\left( \sum_{k=1}^{n} (n-k) (\bar{q}-p)^{*(n-k-1)}*(q-p)^{*(k-1)}\right) \mathcal{Q}_{c,p}^{-n}(q)  \right].
		\end{eqnarray*}
		\endgroup
Using the fact that $ \mathcal{Q}_{c,p}(q)= (q-p)*(\bar{q}-p)$ we have
		\begingroup\allowdisplaybreaks
		\begin{eqnarray*}
			D(q-p)^{-*(n+1)}
			&=& 2 \left\{ \left[\left( \sum_{k=1}^{n} (\bar{q}-p)^{*(n-k)}*(q-p)^{*(k-1)} \right)(2q_0-2p) \right]\right.\\
			&&\left. - \frac{1}{n}  \left[\left(\sum_{k=2}^{n} (k-1)(\bar{q}-p)^{*(n-k+1)} *(q-p)^{*(k-1)} \right)  \right] \right.\\
			&& \left.- \frac{1}{n}  \left[ \left(\sum_{k=1}^{n} (n-k)(\bar{q}-p)^{*(n-k)}*(q-p)^{*k}  \right)   \right]\right\}\mathcal{Q}_{c,p}^{-n-1}(q)\\
			&=&2 \left[ \left( \sum_{k=1}^{n} (\bar{q}-p)^{*(n-k)}*(q-p)^{*(k-1)} \right)(2q_0-2p)\right.\\
			&& \left.- \frac{1}{n} \sum_{k=1}^{n-1} k (\bar{q}-p)^{*(n-k)}*(q-p)^{*k}- \frac{1}{n} \sum_{k=1}^{n-1} (n-k) (\bar{q}-p)^{*(n-k)}*(q-p)^{*k}  \right]\\
			&&\mathcal{Q}_{c,p}^{-n-1}(q)\\
			&=& 2 \left[ \left( \sum_{k=1}^{n} (\bar{q}-p)^{*(n-k)}*(q-p)^{*(k-1)} \right)(\bar{q}-p+q-p)\right.\\
			&& \left.- \sum_{k=1}^{n-1} (\bar{q}-p)^{*(n-k)}*(q-p)^{*k}\right]\mathcal{Q}_{c,p}^{-n-1}(q)\\
			&=& 2 \left[ \sum_{k=1}^{n} (\bar{q}-p)^{*(n-k+1)}*(q-p)^{*(k-1)}+ \sum_{k=1}^{n} (\bar{q}-p)^{*(n-k)}*(q-p)^{*k} \right.\\
			&& \left.- \sum_{k=1}^{n} (\bar{q}-p)^{*(n-k)}*(q-p)^{*k}+ (q-p)^{*n} \right] \mathcal{Q}_{c,p}^{-n-1}(q)  \\
			&=& \left(2 \sum_{k=1}^{n+1} (\bar{q}-p)^{*(n-k+1)}*(q-p)^{*(k-1)}\right) \mathcal{Q}_{c,p}^{-n-1}(q).
		\end{eqnarray*}
		\endgroup
		This proves the result.
	\end{proof}
Theorem \ref{actionN1} allows to write the action of the operator $D$ on $(q-p)^{-*n}$, $n\in\mathbb{N}$ as
		\begin{equation}
			\label{appLauN}
			D(q-p)^{-*n}= 2 \mathcal{H}_n(q,p) \mathcal{Q}_{c,p}^{-n}(q),
		\end{equation}
		where
		\begin{equation}
			\label{harmN}
			\mathcal{H}_n(q,p)= \sum_{k=1}^{n} (\bar{q}-p)^{*(n-k)}*(q-p)^{*(k-1)}.
		\end{equation}
	Next proposition contains the properties of the polynomials $\mathcal{H}_n(q,p)$.
	\begin{proposition}
		\label{harmLauN}
		Let $\mathcal{H}_n(q,p)$ be as in (\ref{harmN}).
		Then, for $q$, $p \in \mathbb{H}$ such that $q \notin [p]$ and $n\in\mathbb{N}$, we have that:
		\begin{itemize}
			\item[1)]  $\mathcal{H}_n(q,p) \mathcal{Q}_{c,p}^{-n}(q)$, where
			$\mathcal{Q}_{c,p}^{-1}(q)$ is defined in (\ref{psudoCauchy}), are axially harmonic functions in $q$.
			\item[2)] $ \mathcal{H}_n(q,p)$ are left slice polyanalytic of order $n+1$ in $q$ and right slice hyperholomorphic in $p$.
			\item[3)] If $q \in \mathbb{H} \setminus \mathbb{R}$ the polynomials $ \mathcal{H}_n(q,p)$ have the following closed expression:
			\begin{equation}
				\label{closedN1}
				\mathcal{H}_n(q,p)= \frac{(\underline{q})^{-1}}{2} \left[ (q-p)^{*n}- (\bar{q}-p)^{*n}\right].
			\end{equation}
		\end{itemize}
	\end{proposition}
	\begin{proof}

To prove assertion
			$1)$ we note that the function $(q-p)^{-*n}$ is slice hyperholomorphic in $q$, thus the Fueter mapping theorem and \eqref{appLauN} immediately give that $\mathcal{H}_n(q,p) \mathcal{Q}_{c,p}^{-n}(q)$ is axially harmonic in $q$.

To prove $2)$ we use arguments similar to the ones used to show the second point of Proposition \ref{harmpoly}. Moreover, since the $*$-product preserves right slice hyperholomorphicity in $p$ (see Proposition~\ref{closedp}), we obtain the second part of the assertion.
			
Finally, we use induction on $n$ to prove $3)$. The case $n=1$ is straightforward. We suppose that the statement is true for $n$ and we prove that it holds for $n+1$. By the inductive hypothesis and the definition of the left $*$-product with respect to $q$ we get
			\begin{eqnarray*}
				\mathcal{H}_{n+1}(q,p)&=& \sum_{k=1}^{n+1} (\bar{q}-p)^{*(n+1-k)}*(q-p)^{*(k-1)}\\
				&=& \sum_{k=1}^{n} (\bar{q}-p)^{*(n+1-k)}*(q-p)^{*(k-1)}+ (q-p)^{*n}\\
				&=&\bar{q} \mathcal{H}_n(q,p)- \mathcal{H}_n(q,p)p + (q-p)^{*n}\\
				&=& \frac{(\underline{q})^{-1} }{2} \left[\bar{q} \left(q-p \right)^{*n}-\bar{q}(\bar{q}-p)^{*n}-\left(q-p \right)^{*n}p+(\bar{q}-p)^{*n}p\right]+ (q-p)^{*n}.
			\end{eqnarray*}
			Now, we observe that
			\begin{eqnarray*}
				\nonumber
				\frac{(\underline{q})^{-1}\bar{q}}{2} (q-p)^{*n}+ (q-p)^{*n}&=& \frac{(\underline{q})^{-1}q_0}{2} (q-p)^{*n}- \frac{(\underline{q})^{-1}\underline{q}}{2}(q-p)^{*n}+ (q-p)^{*n}\\
				\nonumber
				&=&\frac{(\underline{q})^{-1}q_0}{2} (q-p)^{*n}+ \frac{(\underline{q})^{-1}\underline{q}}{2}(q-p)^{*n}\\
				&=& \frac{(\underline{q})^{-1} q}{2} (q-p)^{*n},
			\end{eqnarray*}
			which implies
			\begin{eqnarray*}
				\mathcal{H}_{n+1}(q,p)&=&  \frac{(\underline{q})^{-1}}{2} \left[- \bar{q} (\bar{q}-p)^{*n}+q (q-p)^{*n}-(q-p)^{*n}p+(\bar{q}-p)^{*n}p \right]\\
				&=& \frac{(\underline{q})^{-1}}{2} \left[ (q-p)^{*(n+1)}- (\bar{q}-p)^{*(n+1)}\right].
			\end{eqnarray*}
			This proves the result.
	\end{proof}
	\begin{remark}
		If we consider $p=0$ in \eqref{appLauN} we get
			\begin{equation}
			\label{neghar}
			D (q^{-n})=2 P_n(q)|q|^{-2n}, \qquad n \in \mathbb{N},
		\end{equation}
		where the polynomials $P_n(q)$ are given by
		
		\begin{equation}
			\label{harmoN}
			P_n(q)=\mathcal{H}_n(q,0)= \sum_{i=1}^{n} \bar{q}^{n-i} q^{i-1} .
		\end{equation}
	If $q \in \mathbb{H}\setminus \mathbb{R}$ the polynomials $P_n(q)$ have the following closed expression
	
		\begin{equation}
		\label{Nclosed}
		P_n(q)= \frac{(\underline{q})^{-1}}{2} (q^n-\bar{q}^n).
	\end{equation}
	
	\end{remark}

The Laurent harmonic regular series can be represented using the pseudo-Cauchy kernel $\mathcal{Q}^{-1}_{c,p}(q)$ defined in \eqref{psudoCauchy}, with a convergence set that matches that of the slice hyperholomorphic Laurent expansion.
	\begin{theorem}
		\label{harmL1}
		Let $f$ be a slice hyperholomorphic function in an open set $\Omega\subseteq\mathbb H$.
Let $p\in\mathbb H$  and assume that $f$ admits the $*$-Laurent expansion at an arbitrarily fixed point $p$
		\begin{equation}
			\label{NNL1}
			f(q)= \sum_{n=0}^{\infty} (q-p)^{*n}a_n+ \sum_{n=1}^{\infty} (q-p)^{-*n}a_{-n}, \qquad \{a_n\}_{n \in \mathbb{Z}} \subset \mathbb{H},
		\end{equation}
which is convergent (absolutely and uniformly on the compact subsets) in
$\Omega(p,R_1, R_2) \subseteq \Omega$, with $R_1,R_2$ defined by $ \frac{1}{R_2}= \limsup_{n \to \infty} |a_n|^{\frac{1}{n}}$ and $R_1= \limsup_{n \to \infty} |a_{-n}|^{\frac{1}{n}}$.  Then the Laurent harmonic series of $g=Df$ is given by
		\begin{equation}
			\label{LH}
			g(q)=Df(q)= \sum_{n=0}^{\infty} \widetilde{H}_n(q,p)b_n+ \sum_{n=1}^{\infty} \mathcal{H}_n(q,p) \mathcal{Q}_{c,p}^{-n}(q) b_{-n}, \qquad q \in \Omega(p,R_1, R_2),
		\end{equation}
		where $\{b_n\}_{n \geq 0}:=\{-2(n+1)a_{n+1}\}_{n \geq 0}$, $\{b_{-n}\}_{n \geq 1}:= \{2a_{-n}\}_{n \geq 1}$, and $\Omega(p,R_1, R_2)$ is as in \eqref{tildeS}.
		\end{theorem}
	\begin{proof}
		Formula \eqref{LH} is proved by using \eqref{appreg1} and \eqref{appLauN} which give, by applying $D$ term by term,
		\begin{eqnarray*}
			g(q)=Df(q)&=& \sum_{n=0}^{\infty} D(q-p)^{*n}a_{n}+ \sum_{n=1}^{\infty} D(q-p)^{-*n}a_{-n}\\
			&=& -2\sum_{n=0}^{\infty} (n+1)\widetilde{H}_{n}(q,p) a_{n+1}+2 \sum_{n=1}^{\infty} \mathcal{H}_n(q,p) \mathcal{Q}_{c,p}^{-n}(q) a_{-n}.
		\end{eqnarray*}
		Therefore the result follows by setting the coefficients as in the statement.
		To prove the convergence of the Laurent harmonic series, we note that the convergence of the first series is studied in Theorem \ref{regh}, thus we focus on the convergence of the second series. We use Theorem \ref{axhrap}. If $q \in \mathbb{R}$ the result is straightforward. Otherwise, we suppose that $ q \notin \mathbb{R}$. By formulas \eqref{appLauN} and \eqref{rappH} we have
		\begin{eqnarray*}
			\mathcal{H}_n(q,p) \mathcal{Q}_{c,p}^{-n}(q)a_{-n} &=& \frac{1}{2}	D(q-p)^{-*n}a_{-n}\\
			&=& -\frac{(\underline{q})^{-1} I_{\underline{q}}I}{2} \left[(q-p)_{|q=q_{-I}}^{-*n}-(q-p)_{|q=q_{I}}^{-*n} \right]a_{-n},
		\end{eqnarray*}
		where if $p\in\mathbb{C}_I$, we take $q_I=x+Iy$ and $q_{-I}=x-Iy$. Therefore we have
		$$ |\mathcal{H}_n(q,p) \mathcal{Q}_{c,p}^{-n}(q)a_{-n}| \leq \frac{|\underline{q}|^{-1}}{2} \left[|(q_{-I}-p)^{-n}a_{-n}|+|(q_{I}-p)^{-n}a_{-n}|\right].$$
		By the assumptions on the coefficients $ \{a_n\}_{n \in \mathbb{Z}}$ and the fact that $q \in  \Omega(p,R_1, R_2)$, see \eqref{tildeS}, we have that the series
		$$ \sum_{n=1}^{\infty} |(q_{\pm I}-p)^{-n}a_{-n}|,$$
		are convergent. Hence we obtain the convergence of the series \eqref{LH} where stated, in $\Omega(p, R_1, R_2)$.

	\end{proof}
	\begin{proposition}
		\label{harmL3}
		Let  $f$ be a slice hyperholomorphic function in an open set  $\Omega$ and assume that $f$ admits a $*$-Laurent expansion centered at an arbitrarily fixed point $p$ as in \eqref{Lau1} convergent in a suitable subset of $\Omega$ and with coefficients $ \{a_n\}_{n \in \mathbb{Z}} \subseteq \mathbb{H}$,  convergent (absolutely and uniformly on the compact subsets) in $\Omega(p,R_1,R_2)\subseteq\Omega$, with $R_1$, $R_2$ as in Theorem \ref{harmL1}. Then the Laurent harmonic regular series, for $q\in\Omega(p,R_1,R_2)$, can be written as
		\begin{equation}
			\label{harmL2}
			Df(q)= \sum_{n=0}^{\infty} \widetilde{H}_n(q,p)b_n- \sum_{n=1}^{\infty} \frac{1}{(n-1)!} \partial_{q_0}^{n-1} \mathcal{Q}_{c,p}^{-1}(q) b_{-n},
		\end{equation}
where $\{b_n\}_{n \geq 0}:=\{-2(n+1) a_{n+1}\}_{n \geq 0}$ and  $\{b_{-n}\}_{n \geq 1}:=\{2 (-1)^{n} a_{-n}\}_{n \geq 1}$.
	\end{proposition}
	\begin{proof}
By formula \eqref{Lau4}, and by \eqref{appreg1} we get
		\begin{eqnarray*}
			Df(q)&=& \sum_{n=0}^{\infty} D(q-p)^{n*_{q,L}} a_{n}+ \sum_{n=1}^{\infty} \frac{(-1)^n}{(n-1)!} D\partial_{q_0}^{n-1} S^{-1}_L(p,q)a_{-n}\\
			&=& -2\sum_{n=0}^{\infty}(n+1)\widetilde{H}_{n}(q,p) a_{n+1}-2 \sum_{n=1}^{\infty} \frac{(-1)^n}{(n-1)!}\partial_{q_0}^{n-1}  \mathcal{Q}_{c,p}^{-1}(q)a_{-n},
		\end{eqnarray*}
where we used \eqref{2punto*}.
		Thus, the result follows by setting the coefficients as in the statement.
		
	\end{proof}
	
	\begin{remark}
		If we take $p=0$ in \eqref{harmL2} we get a nice expression of the Laurent series in $\Omega(0, R_1, R_2)$ involving the Cauchy kernel $E(q)=\bar q/|q|^4$:
		$$ 	Df(q)= \sum_{n=0}^{\infty} H_n(q)b_n-\sum_{n=1}^{\infty} \frac{1}{(n-1)!} \partial_{q_0}^{n-1} (qE(q)) b_{-n}$$
	\end{remark}
\begin{remark}
By using \eqref{connection} and \eqref{splitS5} one can write the Laurent harmonic regular series in terms of the $F$-kernel and the slice hyperholomorphic Cauchy kernel, respectively.
	\\
	Note that in \cite[Lemma 4.8]{CDPS} the authors proved an expansion in series of the pseudo-Cauchy kernel $\mathcal{Q}_{c,p}^{-1}(q)$ in terms of the harmonic polynomials $H_n(q)$, i.e.:
	\begin{equation}
		\label{kernelHL}
		\mathcal{Q}_{c,p}^{-1}(q)= \sum_{n=0}^{\infty} (n+1) H_{n}(q) p^{-2-n}, \qquad |q|<|p|.
	\end{equation}
\end{remark}

	For the next considerations we need a preliminary result which is somewhat related with the calculations done in the proof of Theorem \ref{actionN1}.
	\begin{proposition}
		\label{powerN}
		Let $p$, $q \in \mathbb{H}$. Then for $n \in \mathbb{N}$ we have
		\begin{equation}
			\label{powerN1}
			\mathcal{Q}_{c,p}^{n}(q)=(q-p)^{*n}*(\bar{q}-p)^{*n},
		\end{equation}
and, for $q \notin [p]$,
\begin{equation}
\label{twoo}
\mathcal{Q}_{c,p}^{-n}(q)=(q-p)^{-*n}*(\bar{q}-p)^{-*n}.
\end{equation}
	\end{proposition}
	\begin{proof}
		We show formula \eqref{powerN1} by induction. For $n=1$ the result follows by definition of the $*$-left product in $q$. Now, we suppose that the statement is valid for $n$ and we prove it for $n+1$. By using the induction hypothesis we have
		\begin{eqnarray*}
			\mathcal{Q}_{c,p}^{n+1}(q)&=& \mathcal{Q}_{c,p}^{n}(q) \mathcal{Q}_{c,p}(q)\\
			&=& \left[ (q-p)^{*n}* (\bar{q}-p)^{*n}\right] (p^2-2q_0p+|q|^2)\\
			&=&\left[ (q-p)^{*n}*(\bar{q}-p)^{*n}\right] * (q-p)* (\bar{q}-p).
		\end{eqnarray*}
		Since $(q-p)* (\bar{q}-p)=(\bar{q}-p)*(q-p)$ we have
		$$ \mathcal{Q}_{c,p}^{n+1}(q)=(q-p)^{*(n+1)}* (\bar{q}-p)^{*(n+1)},$$
		and this proves formula \eqref{powerN1}. Finally, formula \eqref{twoo} follows by using Lemma \ref{investar1} and the equality
		$$ 1= \left[(q-p)^{*n}*(\bar{q}-p)^{*n}\right] \mathcal{Q}_{c,p}^{-n}(q).$$
	\end{proof}
Finally, we rewrite the Laurent harmonic regular series in terms of the harmonic polynomials $\widetilde{H}_n(q,p)$, see \eqref{Harmopoly}, and $H_n(q)$, see \eqref{two}.

	\begin{theorem}
Let $f$ be a slice hyperholomorphic function on an open set $\Omega\subseteq\mathbb{H}$ admitting the $*$-Laurent expansion \eqref{Lau1}, centered at an arbitrarily fixed point $p\in\mathbb{H}$, with coefficients $\{a_n\}_{n\in\mathbb{Z}}\subseteq\mathbb{H}$, converging (absolutely and uniformly on the compact subsets) in $\Omega(p,R_1,R_2)\subseteq\Omega$ intersected with $\{|q|<|p|\}$. Then, in its domain of convergence, the Laurent harmonic regular expansion is given by
\begin{equation}
\label{newS}
g(q) = Df(q)= \sum_{n=0}^{\infty} \widetilde{H}_n(q,p) b_n- \sum_{n=1}^{\infty} \sum_{k=n-1}^{\infty} \frac{1}{(n-1)!}\frac{(k+1)!}{(k-n+1)!} H_{k-n+1}(q) p^{-2-k} b_{-n},
\end{equation}
	where $\{b_n\}_{n \geq 0}:=\{-2(n+1) a_{n+1}\}_{n \geq 0}$ and  $\{b_{-n}\}_{n \geq 1}:=\{2 (-1)^n a_{-n}\}_{n \geq 1}$.
		Moreover, for $q \not \in \mathbb{R}$ we have
		\begin{equation}
			\label{harmL5}
			g(q)=Df(q)= (\underline{q})^{-1} \left( \sum_{n \in \mathbb{Z}} (\bar{q}-p)^{*n} a_{n}-\sum_{n \in \mathbb{Z}} (q-p)^{*n} a_{n} \right).
		\end{equation}
	\end{theorem}
	\begin{proof}
	Let $n\in\mathbb{N}$ be fixed. By repeated application of the Appell property for the polynomials $H_n(q)$, see \eqref{AppellH}, we obtain
	\[
	\partial_{q_0}^{\,n-1} H_k(q)
	= \frac{k!}{(k-n+1)!}\, H_{k-n+1}(q).
	\]
Assume now that $|q|<|p|$, so that the expansion \eqref{kernelHL} is valid, and that this region intersects the domain of convergence of the spherical Laurent series. Substituting the series expansion \eqref{kernelHL} into \eqref{harmL2}, we obtain
	\begin{eqnarray*}
		g(q)
		&=& \sum_{n=0}^{\infty} \widetilde{H}_n(q,p)\, b_n
		-  \sum_{n=1}^{\infty}\frac{1}{(n-1)!} \sum_{k=n-1}^{\infty}
		(k+1)\, \partial_{q_0}^{n-1} H_k(q)\, p^{-2-k}\, b_{-n} \\
		&=& \sum_{n=0}^{\infty} \widetilde{H}_n(q,p)\, b_n
		-  \sum_{n=1}^{\infty} \sum_{k=n-1}^{\infty}
		\frac{1}{(n-1)!}\frac{(k+1)!}{(k-n+1)!}\, H_{k-n+1}(q)\, p^{-2-k}\, b_{-n}.
	\end{eqnarray*}
		Now, we consider $q \notin \mathbb{R}$.
		Theorem \ref{harmL1} gives
		\begin{equation}
			\label{harmoN1}
			g(q)=- 2 \sum_{n=0}^{\infty} \widetilde{H}_n(q,p)(n+1)a_{n+1} +2\sum_{n=1}^{\infty} \mathcal{H}_n(q,p) \mathcal{Q}_{c,p}^{-n}(q) a_{-n},
		\end{equation}
		where $ \{a_n\}_{n \in \mathbb{Z}} \subseteq \mathbb{H}$.
		By \eqref{closedH} we have
		\begin{equation}
			\label{harmoN3}
			-2 \sum_{n=0}^{\infty} \widetilde{H}_n(q,p)(n+1)a_{n+1}=(\underline{q})^{-1} \sum_{n=0}^{\infty} \left[  (\bar{q}-p)^{*n}-(q-p)^{*n}\right]a_n.
		\end{equation}
		Now, we focus on the second series of \eqref{harmoN1}. By \eqref{closedN1} and Proposition \ref{powerN} we have
		\begin{eqnarray}
			\nonumber
			2\sum_{n=1}^{\infty} \mathcal{H}_n(q,p) \mathcal{Q}_{c,p}^{-n}(q) a_{-n}&=& (\underline{q})^{-1} \sum_{n=1}^{\infty}\left[  (q-p)^{*n}-(\bar{q}-p)^{*n}\right] \mathcal{Q}_{c,p}^{-n}(q)a_{-n}\\
			\label{harmoN2}
			&=&(\underline{q})^{-1}\sum_{n=1}^{\infty}\left[  (\bar{q}-p)^{-*n}-(q-p)^{-*n}\right]a_{-n}.
		\end{eqnarray}
		Finally, by plugging \eqref{harmoN3} and \eqref{harmoN2} into \eqref{harmoN1} we get
		\begin{eqnarray*}
			g(q)&=& (\underline{q})^{-1} \sum_{n=0}^{\infty} \left[  (\bar{q}-p)^{*n}-(q-p)^{*n}\right]a_n+(\underline{q})^{-1}\sum_{n=1}^{\infty}\left[  (\bar{q}-p)^{-*n}-(q-p)^{-*n}\right]a_{-n}\\
			&=& (\underline{q})^{-1} \left(\sum_{n \in \mathbb{Z}} (\bar{q}-p)^{*n} a_{n}-\sum_{n \in \mathbb{Z}} (q-p)^{*n} a_{n} \right),
		\end{eqnarray*}
namely the asserted equality \eqref{harmL5}.
	\end{proof}

	\subsection{Laurent harmonic spherical series}
	
	Similarly to the case of the Taylor expansion, axially harmonic functions have two Laurent expansions and besides the one discussed in the previous subsection, we have a harmonic spherical series in terms of the spherical polynomials $Q_p(q)$.
	
	\begin{definition}
		Let $D$ be the Cauchy-Fueter operator and let $\Omega$ be an axially symmetric open set. We suppose that $f$ is a slice hyperholomorphic function admitting a convergent spherical Laurent expansion centered at an arbitrarily fixed point $p \in \mathbb H$. Then  we say that $Df$ has a Laurent harmonic spherical series at the point $p$.
	\end{definition}
	
We now provide two explicit representations of the Laurent harmonic spherical series. Later, we will discuss their sets of convergence.

\begin{theorem}
	\label{harmospherL}
Let $f$ be a slice hyperholomorphic function in a neighbourhood of $p\in\mathbb H$ containing the Cassini shell $U(p,R_1,R_2)$, for some $0<R_1<R_2$, where it admits the spherical Laurent expansion
	\begin{equation}
		\label{serr0111}
		f(q)= \sum_{n \in \mathbb{Z}} Q_p^n(q) a_{2n}+\sum_{n \in \mathbb{Z}} \left(Q_p^n(q)(q-p)\right) a_{2n+1},\qquad a_n\in\mathbb H,
	\end{equation}
convergent (absolutely and uniformly on compact subsets) in $U(p,R_1,R_2)$.
Let $g=Df$ be the associated axially harmonic function. Then $g$ admits the following Laurent harmonic spherical expansion:
	\begin{equation}\label{sei17}
		g(q)=Df(q)=-4 \sum_{n \in \mathbb{Z}} \mathcal{L}_n(q,p)(q_0-p_0) \left[ a_{2n}+(\bar{q}-p)a_{2n+1}\right]-2 \sum_{n \in \mathbb{Z}} Q_p^{n}(q) a_{2n+1},
	\end{equation}
	or, alternatively,
	\begin{equation}
		\label{LauH1}
		g(q)=Df(q)=-4 \sum_{n \in \mathbb{Z}} \mathcal{L}_n(q,p)(q_0-p_0) \left[ a_{2n}+(q-p)a_{2n+1}\right]-2 \sum_{n \in \mathbb{Z}} Q_p^{n}(\bar{q}) a_{2n+1},
	\end{equation}
Both series converge (absolutely and uniformly on compact subsets) in $U(p,R_1,R_2)$. The functions $\mathcal{L}_n(q,p)$ are defined in Definition~\ref{block}.
\end{theorem}
\begin{proof}
We apply the operator $D$ term by term to the Laurent spherical expansion \eqref{serr0111}. The interchange of $D$ and the summation will be justified later in the proof. We first consider the case $q\notin\mathbb R$. By Corollary~\ref{harmcompact}, formulas \eqref{DQ}, \eqref{DQodd}, and \eqref{DQodd2} remain valid for every $n\in\mathbb Z$. Hence, substituting these identities into \eqref{serr0111}, we obtain \eqref{sei17}. Similarly, using
\eqref{DQodd2} instead of \eqref{DQodd}, we obtain \eqref{LauH1}.

We now justify the interchange of $D$ and the summation in \eqref{serr0111}. Let $K$ be a compact subset of $U(p,R_1,R_2)$. There exist $\rho_1,\rho_2$ such that $R_1<\rho_1<\rho_2<R_2$ and $\rho_1^2\leq |Q_p(q)|\leq \rho_2^2,$ for $q\in K.$
Therefore, for $n \geq 1$ we have $|\mathcal L_n(q,p)|\leq n\rho_2^{2n-2}$, whereas for $ n\ \leq -1$ we have $|\mathcal L_n(q,p)|\leq |n|\rho_1^{2n-2}$. All the series appearing in the previous calculations converge absolutely and uniformly on compact subsets of $U(p,R_1,R_2)$.
\\It remains to consider the real points. Let
$q=q_0\in U(p,R_1,R_2)\cap\mathbb R$. Since $Q_p(q_0)\in\mathbb R$, we have
$$
\mathcal L_n(q_0,p)=nQ_p^{n-1}(q_0),\qquad n\in\mathbb Z.
$$
Therefore, by \eqref{real1} and \eqref{real2}, which remain valid for $n\in\mathbb{Z}$, the expressions in \eqref{sei17} and \eqref{LauH1} reduce precisely to those obtained by applying $D$ term by term to the Laurent spherical expansion \eqref{serr0111}. Hence, \eqref{sei17} and
\eqref{LauH1} hold also on the real axis.
\end{proof}

Now, we prove that the harmonic spherical Laurent series converges in an open Euclidean set.
Recalling the definition of Cassini shell, see \eqref{Cassinishell}, we have:
\begin{proposition}
	\label{convL1}
	Let $f$ be a slice hyperholomorphic function in an open set $\Omega$,  having spherical Laurent expansion  at $p\in\mathbb H$ formally given by
	\eqref{serr0111}.
	Let the coefficients $ \{a_n\}_{n \in \mathbb{Z}} \subseteq \mathbb{H}$ be such that
	$$ R_1:= \limsup_{n \to \infty}  |a_{-n}|^{\frac{1}{n}}, \qquad \hbox{and} \qquad \frac{1}{R_2}:= \limsup_{n \to \infty}  |a_n|^{\frac{1}{n}},$$
	with $R_1<R_2$.	Then the harmonic spherical Laurent series \eqref{sei17} converges (absolutely and uniformly on the compact subsets) in the Cassini shell
	$ U(p,R_1,R_2)$.	
\end{proposition}
\begin{proof} Let $K$ be a compact subset of the Cassini shell $U(p,R_1,R_2)$. Then there exist $r_1,r_2$ such that $R_1<r_1<r_2<R_2$ and $r_1^{2n } <|Q_p^n(q)| <r_2^{2n}$ for $q \in K$. Choose $R'$ and $R''$ such that $R_1<R'<r_1<r_2<R''<R_2$. By the assumptions on the coefficients we have $|a_n| \leq (1/R'')^n$ and $|a_{-n}| \leq (R')^n$ for $n$ sufficiently large. For $n \geq 1$ and by Definition \ref{block} we have
$$|\mathcal{L}_{-n}(q,p)| \leq n r_1^{-2n-2},$$
while, by \eqref{est}, $|q_0-p_0|\leq \sqrt{r_2^2+p_1^2}$ on $K$, and $|\bar q-p|\leq \sqrt{r_2^2+p_1^2}+p_1$ by \eqref{newin}.
By Theorem \ref{harmospherL} we can write the harmonic spherical series as
	\begin{eqnarray*}
		Df(q)&=&-4 \sum_{n=1}^{\infty}\mathcal{L}_n(q,p)(q_0-p_0) \left[a_{2n}+(\overline{q}-p)a_{2n+1}\right]-2 \sum_{n=0}^{\infty} Q_p^n(q) a_{2n+1}\\
		&& -4 \sum_{n=1}^{\infty}\mathcal{L}_{-n}(q,p)(q_0-p_0)\left[a_{-2n}+(\bar{q}-p)a_{-2n+1} \right]- 2\sum_{n=1}^{\infty} Q_p^{-n}(q)a_{-2n+1}.
	\end{eqnarray*}
	The convergence of the first two series follows by Theorem \ref{c1}. We focus on the other two series. By Lemma \ref{newres} we get
	\begin{eqnarray*}
		&& 4 |\mathcal{L}_{-n}(q,p)| |q_0-p_0|\left[|a_{-2n}|+|\bar{q}-p||a_{-2n+1}| \right]+ 2 |Q_p^{-n}(q)||a_{-2n+1}|\\
		&& \leq  \frac{4n}{r_1^2R'} \left(\frac{R'}{r_1}\right)^{2n} \sqrt{r_2^2+p_1^2}\left[R'+ \left(\sqrt{r_2^2+p_1^2} +p_1\right)\right]+ \frac{2}{R'} \left(\frac{R'}{r_1}\right)^{2n}\\
		&&=  \left( \frac{2n}{r_1^2}\sqrt{r_2^2+p_1^2}\left[R'+ \left(\sqrt{r_2^2+p_1^2} +p_1\right)\right]+1 \right)\left(\frac{R'}{r_1}\right)^{2n}\frac{2}{R'}\\
		&&=:  s_n.
	\end{eqnarray*}
	By the ratio test the series $\sum_{n=1}^{\infty} s_n$ is convergent. This implies that the Laurent harmonic spherical series is convergent where stated.
\end{proof}

	We now give a version  of the spherical harmonic Laurent series in terms of the polynomials $P_n(q)$, see \eqref{harmoN}, and of the polynomials $H_n(q)$, see \eqref{two}. To prove this result, we first recall the following elementary identity, which holds for commuting  $a$ and $s$ with $|s|<|a|$:
	\begin{equation}
		\label{NN11}
	(s+a)^{-n}= \sum_{k=0}^\infty (-1)^k \binom{n+k-1}{k}s^k a^{-n-k}, \qquad |s|<|a|
	\end{equation}
	\begin{proposition}
		\label{spherHarmL}
		Let $f$ be a slice hyperholomorphic function that admits a spherical Laurent series expansion centered at an arbitrarily fixed point $p$ as in \eqref{serr0111} whose coefficients are $ \{a_n\}_{n \in \mathbb{Z}} \subseteq \mathbb{H}$, and is convergent (absolutely and uniformly on the compact subsets) in a Cassini shell $U(p,R_1,R_2)$, $0<R_1<R_2$. Assume moreover that the region $|p_1|<|q-p_0|$ has nonempty intersection with $U(p,R_1,R_2)$. Then, on that intersection, we can write the spherical harmonic Laurent series as
		\begin{eqnarray}
			\nonumber
			Df(q)&=& B_T(q,p)+2 \sum_{n=1}^{\infty} \sum_{k=0}^{\infty} (-1)^k \binom{n+k-1}{k}  P_{2(n+k)}(q-p_0)|q-p_0|^{-4(n+k)} p_1^{2k}a_{-2n}\\
			\nonumber
			&&+2 \sum_{n=1}^{\infty}\sum_{k=0}^{\infty} (-1)^k \binom{n+k-1}{k}  P_{2(n+k)-1}(q-p_0)|q-p_0|^{-4(n+k)+2} p_1^{2k}a_{-2n+1}\\
			\label{harmoseriesL}
			&&- 2\sum_{n=1}^{\infty} \sum_{k=0}^{\infty} (-1)^k \binom{n+k-1}{k}  P_{2(n+k)}(q-p_0)|q-p_0|^{-4(n+k)} p_1^{2k+1}Ia_{-2n+1},
		\end{eqnarray}
where $B_T(q,p)$ is the Taylor part given in \eqref{Taylorharmo} of the Laurent harmonic spherical series. If instead the region $|q-p_0|<|p_1|$ has nonempty intersection with $U(p,R_1,R_2)$, then on that intersection another representation of the Laurent harmonic spherical series is given by
				\begin{eqnarray}
					\nonumber
			Df(q)&=& B_T(q,p)+4\sum_{n=1}^{\infty} \sum_{k=0}^{\infty} (-1)^k \binom{n+k}{k}n H_{2k+1}(q-p_0) p_1^{-2(n+k+1)}a_{-2n}\\
			\nonumber
			&&- 2\sum_{n=1}^{\infty}\sum_{k=0}^{\infty} (-1)^k \binom{n+k-1}{k} (2k+1) H_{2k}(q-p_0) p_1^{-2(n+k)}a_{-2n+1}\\
			\label{111}
			&&- 4\sum_{n=1}^{\infty}\sum_{k=0}^{\infty}(-1)^k  \binom{n+k}{k}n H_{2k+1}(q-p_0)p_1^{-2(n+k+1)+1}Ia_{-2n+1}.
		\end{eqnarray}
	\end{proposition}
	\begin{proof}
By hypothesis, we can write the following series expansion of the function $f$, and all the calculations below are legitimate by our assumptions on the absolute and uniform convergence on the compact subsets of $U(p,R_1,R_2)$
		\begin{eqnarray}
			\label{Las}
			f(q)&=& \sum_{n=0}^{\infty} Q_p^n(q) a_{2n}+ \sum_{n=0}^{\infty} Q_p^n(q) (q-p)a_{2n+1}\\
			\nonumber
			&&+ \sum_{n=1}^{\infty}Q_p^{-n}(q)a_{-2n}+ \sum_{n=1}^{\infty}Q_p^{-n}(q)(q-p)a_{-2n+1}.
		\end{eqnarray}	
		The application of $D$ to the Taylor part of the above series has already been carried out in Theorem~\ref{Tayorharmo0}. We denote the resulting kernel by $B_T(q,p)$; see \eqref{Taylorharmo}. We focus on the remaining series. Let $p=p_0+Ip_1$ with $p_0$, $p_1 \in \mathbb{R}$. We assume that $|p_1|<|q-p_0|$ and that the above region has a nonempty intersection with the domain of convergence of the spherical Laurent series. By Newton's binomial theorem for negative integers (see \eqref{NN11}) we can write
		\begin{equation}
			\label{negLa}
			Q_p^{-n}(q)=\sum_{k=0}^{\infty} (-1)^k\binom{n+k-1}{k}  (q-p_0)^{-2(n+k)} p_1^{2k}.
		\end{equation}

We apply the operator $D$ to this last expression of $Q_p^{-n}(q)$ and by formula \eqref{neghar} we have

\begingroup\allowdisplaybreaks		
		\begin{eqnarray}
			\nonumber
			D \left( Q_p^{-n}(q)\right)&=& \sum_{k=0}^{\infty} (-1)^k \binom{n+k-1}{k} D(q-p_0)^{-2(n+k)} p_1^{2k}\\
			\label{one}
			&=& 2 \sum_{k=0}^{\infty} (-1)^k \binom{n+k-1}{k} |q-p_0|^{-4(n+k)} P_{2(n+k)}(q-p_0) p_1^{2k}.
		\end{eqnarray}
\endgroup
		Applying the Cauchy-Fueter operator to $ Q_p^{-n}(q)(q-p)$,
rewritten as
\[
\begin{split}
Q_p^{-n}(q)(q-p_0-Ip_1)&=\sum_{k=0}^{\infty} (-1)^k\binom{n+k-1}{k}  (q-p_0)^{-2(n+k)+1} p_1^{2k}\\
&- \sum_{k=0}^{\infty} (-1)^k\binom{n+k-1}{k}  (q-p_0)^{-2(n+k)} Ip_1^{2k+1},
\end{split}
\]
and using \eqref{prodhar1} and \eqref{neghar} we get
\begingroup\allowdisplaybreaks
		\begin{eqnarray}
			\nonumber
			D\left(Q_p^{-n}(q)(q-p) \right)&=& \sum_{k=0}^{\infty} (-1)^k \binom{n+k-1}{k} D(q-p_0)^{-2(n+k)+1} p_1^{2k}\\
			\nonumber
			&&- \sum_{k=0}^{\infty} (-1)^k \binom{n+k-1}{k} D(q-p_0)^{-2(n+k)} p_1^{2k+1}I\\
			\nonumber
			&=&2\sum_{k=0}^{\infty} (-1)^k \binom{n+k-1}{k} |q-p_0|^{-4(n+k)+2} P_{2(n+k)-1}(q-p_0) p_1^{2k}\\
			\label{zero}
			&&- 2\sum_{k=0}^{\infty} (-1)^k \binom{n+k-1}{k} |q-p_0|^{-4(n+k)} P_{2(n+k)}(q-p_0) p_1^{2k+1}I.
		\end{eqnarray}
\endgroup
Finally we obtain \eqref{harmoseriesL} by putting together \eqref{one} and \eqref{zero}.
\\ We now assume that $|q-p_0|<|p_1|$ and that this region intersects the domain of convergence of the spherical Laurent series. In this case, we may write $Q_p^{-n}(q)$ in the form
\begin{equation}
\label{negLa11}
Q_p^{-n}(q)=\sum_{k=0}^{\infty} (-1)^k \binom{n+k-1}{k} (q-p_0)^{2k} p_1^{-2(n+k)}.
\end{equation}

By applying the Fueter operator to $Q_p^{-n}(q)$ and to $Q_p^{-n}(q)(q-p)$, written as in \eqref{negLa11}, and by \eqref{two} we obtain
\begingroup\allowdisplaybreaks
\begin{eqnarray}
\nonumber
	D(Q_p^{-n}(q))&=& \sum_{k=0}^{\infty} (-1)^k \binom{n+k-1}{k} D(q-p_0)^{2k} p_1^{-2(n+k)}\\
	\nonumber
	&=&-4 \sum_{k=1}^{\infty} (-1)^k \binom{n+k-1}{k}k H_{2k-1}(q-p_0) p_1^{-2(n+k)}\\
	\nonumber
	&=& -4n \sum_{k=1}^{\infty} (-1)^k \binom{n+k-1}{k-1} H_{2k-1}(q-p_0) p_1^{-2(n+k)}\\
	\label{auxLaharm}
	&=&4n \sum_{k=0}^{\infty} (-1)^k \binom{n+k}{k} H_{2k+1}(q-p_0) p_1^{-2(n+k+1)},
\end{eqnarray}
\endgroup
and
\begingroup\allowdisplaybreaks
\begin{eqnarray}
	\nonumber
	D \left[Q_p^{-n}(q)(q-p)\right]&=& \sum_{k=0}^{\infty} (-1)^k \binom{n+k-1}{k} D(q-p_0)^{2k+1} p_1^{-2(n+k)}\\
	\nonumber
	&& - \sum_{k=0}^{\infty} (-1)^k \binom{n+k-1}{k} D(q-p_0)^{2k} p_1^{-2(n+k)+1}I\\
	\nonumber
	&=&-2\sum_{k=0}^{\infty} (-1)^k \binom{n+k-1}{k} (2k+1) H_{2k}(q-p_0) p_1^{-2(n+k)}\\
	\nonumber
	&&+ 4n\sum_{k=1}^{\infty}(-1)^k  \binom{n+k-1}{k-1} H_{2k-1}(q-p_0)p_1^{-2(n+k)+1}I\\
	\nonumber
	&=&-2\sum_{k=0}^{\infty} (-1)^k \binom{n+k-1}{k} (2k+1) H_{2k}(q-p_0) p_1^{-2(n+k)}\\
	\label{auxLaharm1}
	&&- 4n\sum_{k=0}^{\infty}(-1)^k  \binom{n+k}{k} H_{2k+1}(q-p_0)p_1^{-2(n+k+1)+1}I.
\end{eqnarray}
\endgroup
Finally formula \eqref{111} follows by putting together \eqref{auxLaharm} and \eqref{auxLaharm1}.		
	\end{proof}

	\begin{remark}
		If we take $p=0$ in \eqref{harmoseriesL}, by Remark \ref{zeroharmo} we have
		$$ B_T(q,0)=-2  \sum_{n=0}^{\infty} (n+1) H_{n}(q)a_{n+1}.$$
		The non zero terms in the other series appearing  in \eqref{harmoseriesL} are obtained for $k=0$, so that they become
		\begin{eqnarray*}
			&& 2 \sum_{n=1}^{\infty} |q|^{-4n}P_{2n}(q) a_{-2n}+2 \sum_{n=1}^{\infty} |q|^{-4n+2} P_{2n-1}(q) a_{-2n+1}\\
			&=& 2 \sum_{n=1}^{\infty} |q|^{-4n}P_{2n}(q) a_{-2n}+2 \sum_{n=1}^{\infty} |q|^{-2(2n-1)} P_{2n-1}(q) a_{-2n+1}\\
			&=& 2 \sum_{n=1}^{\infty} P_n(q) |q|^{-2n}a_{-n}.
		\end{eqnarray*}
Moreover, the hypothesis  $|p_1|<|q-p_0|$ in Proposition \ref{spherHarmL} is always satisfied if $p=0$. Thus we get the harmonic Laurent series centered at the origin:
$$ g(q)= -2\sum_{n=0}^{\infty} (n+1) H_n(q)a_{n+1}+2\sum_{n=1}^{\infty} P_n(q)|q|^{-2n}a_{-n}.$$
	\end{remark}
	
The harmonic spherical Laurent series admits a special representation in its domain of convergence when $q \notin \mathbb{R}$, as proved by direct calculations in the following result which can be obtained also via the corresponding Representation Formula.

	\begin{theorem}
		Let $q \in \mathbb{H} \setminus \mathbb{R}$. We assume that $f$ is a slice hyperholomorphic function that admits a spherical Laurent series expansion centered at an arbitrarily fixed point $p$ as in \eqref{serr0111} with coefficients $ \{a_n\}_{n \in \mathbb{Z}} \subseteq \mathbb{H}$, convergent (absolutely and uniformly on the compact subsets) in a Cassini shell $U(p,R_1,R_2)$, $0<R_1<R_2$. Then, for every $q\in U(p,R_1,R_2)\setminus\mathbb{R}$ with $|q-p_0|\neq|p_1|$, we can write the harmonic spherical Laurent series, where it is convergent, as
		\begin{eqnarray}
			\label{lastHarm}
			Df(q)&=&(\underline{q})^{-1} \sum_{n \in \mathbb{Z}} Q_p^n(\bar{q}) \left[a_{2n}+(\bar{q}-p)a_{2n+1}\right]\\
			\nonumber
			&&- (\underline{q})^{-1} \sum_{n \in \mathbb{Z}} Q_p^n(q) \left[a_{2n}+(q-p)a_{2n+1}\right].
		\end{eqnarray}
	\end{theorem}
	\begin{proof} We recall that $Df(q)$ is given in \eqref{harmoseriesL}. We first focus on the Taylor part of \eqref{harmoseriesL} and \eqref{111}, denoted by $B_T(q,p)$ and defined in \eqref{Taylorharmo}. In Theorem~\ref{harm}, we proved that if $q \notin \mathbb{R}$, then the Taylor part of the series for $Df(q)$ can be written as
\begin{eqnarray}
	\nonumber
B_T(q,p)&=& (\underline{q})^{-1} \sum_{n=0}^{\infty}\left(  Q_p^n(\bar{q})-Q_p^n(q) \right) a_{2n}\\
\label{HHL}
&&  +(\underline{q})^{-1} \sum_{n=0}^{\infty} \left(  Q_p^n(\bar{q})(\bar{q}-p)-Q_p^n(q)(q-p)\right)a_{2n+1}.
\end{eqnarray}

To complete the proof, we distinguish two cases. Let $p=p_0+Ip_1\in\mathbb{H}$, with $p_0,p_1\in\mathbb{R}$. We consider separately the regions $|q-p_0|>|p_1|$ and $|q-p_0|<|p_1|$, assuming in each case that the corresponding region intersects the domain of convergence of the spherical Laurent series.
\newline
\newline
\emph{Case I: $|p_1|<|q-p_0|$} We start by writing a closed expression of the first series in \eqref{harmoseriesL}. By \eqref{Nclosed} we have
		\begin{eqnarray}
			\nonumber
			&& 2 \sum_{n=1}^{\infty} \sum_{k=0}^{\infty} (-1)^k  \binom{n+k-1}{k} |q-p_0|^{-4(n+k)} P_{2(n+k)}(q-p_0) p_1^{2k}a_{-2n}\\
			\nonumber
			&&= (\underline{q})^{-1} \sum_{n=1}^{\infty}\sum_{k=0}^{\infty} (-1)^k \binom{n+k-1}{k} |q-p_0|^{-4(n+k)}\left[(q-p_0)^{2(n+k)}-(\bar{q}-p_0)^{2(n+k)} \right] p_1^{2k}a_{-2n}\\
			\nonumber
			&&= (\underline{q})^{-1} \sum_{n=1}^{\infty}\sum_{k=0}^{\infty} (-1)^k \binom{n+k-1}{k} \left[(\bar{q}-p_0)^{-2(n+k)}-(q-p_0)^{-2(n+k)} \right] p_1^{2k}a_{-2n}\\
			\label{Lasphe}
			&& \quad  = (\underline{q})^{-1}\sum_{n=1}^{\infty}\left[Q_p^{-n}(\bar{q})-Q_p^{-n}(q)\right]a_{-2n},
		\end{eqnarray}
Then, we focus on the second series in \eqref{harmoseriesL}. By \eqref{Nclosed} we have
	\begingroup\allowdisplaybreaks
		\begin{eqnarray}
			\nonumber
&&2 \sum_{n=1}^{\infty}\sum_{k=0}^{\infty} (-1)^k \binom{n+k-1}{k} |q-p_0|^{-4(n+k)+2} P_{2(n+k)-1}(q-p_0) p_1^{2k}a_{-2n+1}\\
			\nonumber
			&=& (\underline{q})^{-1}\sum_{n=1}^{\infty} \sum_{k=0}^{\infty} (-1)^k \binom{n+k-1}{k}|q-p_0|^{-4(n+k)+2} \cdot \\
			\nonumber
			&&\cdot \left((q-p_0)^{2(n+k)-1}-(\bar{q}-p_0)^{2(n+k)-1} \right)p_1^{2k}a_{-2n+1}\\
			\label{HH}
			&=& (\underline{q})^{-1} \sum_{n=1}^{\infty}\left[ Q_p^{-n}(\bar{q})(\bar{q}-p_0)-Q_p^{-n}(q)(q-p_0)\right] a_{-2n+1}.
		\end{eqnarray}
\endgroup

Finally, to compute a closed form for the third series, we note that by \eqref{Lasphe} we deduce
	\begingroup\allowdisplaybreaks		
		\begin{eqnarray}
\nonumber
&&-2 \sum_{n=1}^{\infty} \sum_{k=0}^{\infty} (-1)^k \binom{n+k-1}{k} |q-p_0|^{-4(n+k)} P_{2(n+k)}(q-p_0) p_1^{2k+1}I a_{-2n+1}\\
\label{HHH}
&&= -(\underline{q})^{-1}\sum_{n=1}^{\infty}\left[Q_p^{-n}(\bar{q})-Q_p^{-n}(q)\right]p_1Ia_{-2n+1}
		\end{eqnarray}
		\endgroup
		
		Hence the final result follows by putting together \eqref{HHL}, \eqref{Lasphe}, \eqref{HH} and \eqref{HHH}.
\newline
\newline
\emph{Case II: $|p_1|>|q-p_0|$}. We repeat the reasoning in Case I, with some suitable changes.
\newline
\newline	
We consider the first series in \eqref{111}. By \eqref{splittingH} and \eqref{NN11} we have

\begin{eqnarray}
\nonumber
&&4\sum_{n=1}^{\infty} \sum_{k=0}^{\infty} (-1)^k \binom{n+k}{k}n H_{2k+1}(q-p_0) p_1^{-2(n+k+1)}a_{-2n}\\
\nonumber
&=&- (\underline{q})^{-1} \sum_{n=1}^{\infty} \sum_{k=0}^{\infty} \frac{n(-1)^k}{k+1} \binom{n+k}{k} \left[ (\bar{q}-p_0)^{2(k+1)}-(q-p_0)^{2(k+1)} \right]p_1^{-2(n+k+1)}a_{-2n}\\
\nonumber
&=& -(\underline{q})^{-1} \sum_{n=1}^{\infty} \sum_{k=0}^{\infty} (-1)^k \binom{n+k}{k+1} \left[ (\bar{q}-p_0)^{2(k+1)}-(q-p_0)^{2(k+1)} \right]p_1^{-2(n+k+1)}a_{-2n}\\
\nonumber
&=&(\underline{q})^{-1} \sum_{n=1}^{\infty}\sum_{k=1}^{\infty} (-1)^k \binom{n+k-1}{k} \left[ (\bar{q}-p_0)^{2k}-(q-p_0)^{2k} \right]p_1^{-2(n+k)}a_{-2n}\\
\nonumber
&=&(\underline{q})^{-1} \sum_{n=1}^{\infty}\sum_{k=0}^{\infty} (-1)^k \binom{n+k-1}{k} \left[ (\bar{q}-p_0)^{2k}-(q-p_0)^{2k} \right]p_1^{-2(n+k)}a_{-2n}\\
\label{NN}
&=&(\underline{q})^{-1}\sum_{n=1}^{\infty} \left[Q_p^{-n}(\bar{q})-Q_p^{-n}(q)\right]a_{-2n}
\end{eqnarray}
In the case of the second series in \eqref{111}, using \eqref{NN11} and again \eqref{splittingH}, we obtain
\begin{eqnarray}
\nonumber
&&-2\sum_{n=1}^{\infty}\sum_{k=0}^{\infty} (-1)^k \binom{n+k-1}{k} (2k+1) H_{2k}(q-p_0) p_1^{-2(n+k)}a_{-2n+1}\\
\nonumber
&=& (\underline{q})^{-1} \sum_{n=1}^{\infty}\sum_{k=0}^{\infty} (-1)^k \binom{n+k-1}{k}\left[(\bar{q}-p_0)^{2k+1}-(q-p_0)^{2k+1} \right]p_1^{-2(n+k)}a_{-2n+1}\\
\label{NNNN}
&=&(\underline{q})^{-1}\sum_{n=1}^{\infty}\left[Q_p^{-n}(\bar{q})(\bar{q}-p_0)-Q_p^{-n}(q)(q-p_0) \right] a_{-2n+1}.
\end{eqnarray}	
Now, we focus on the third series in \eqref{111}. By \eqref{NN} and \eqref{NN11} we have	
	\begin{eqnarray}
\nonumber
&&-4\sum_{n=1}^{\infty}\sum_{k=0}^{\infty}(-1)^k  n\binom{n+k}{k} H_{2k+1}(q-p_0)p_1^{-2(n+k+1)+1}Ia_{-2n+1}\\
\label{NNH}
&=& -(\underline{q})^{-1}\sum_{n=1}^{\infty}\left[Q_p^{-n}(\bar{q})-Q_p^{-n}(q) \right]p_1Ia_{-2n+1}.
	\end{eqnarray}
Thus we get \eqref{lastHarm} inserting \eqref{NN}, \eqref{NNNN} and \eqref{NNH} in  \eqref{111}.	
This completes the proof.	
	\end{proof}

\section{Axially Fueter regular functions}\label{FUETERSERIES}

A Fueter regular function in a neighbourhood of a generic quaternion $p$ can be expanded in series using the Fueter polynomials centered at $p$ introduced in \eqref{Fueterpoly}. We recall this result below and we refer the reader interested in more details to consult e.g. \cite{Fueter1,GHS}.

\begin{theorem}
    Let $U$ be an open set and $h: U \subseteq \mathbb{H} \to \mathbb{H}$ be a Fueter regular function. Let us consider $p \in U$ and $\sigma <\hbox{dist}(p, \partial U) $. For $q$ such that $|q-p|< \delta<\sigma$, the function $h$ admits the series expansion
    $$ h(q)= \sum_{n=0}^{\infty} \sum_{\nu \in \sigma_n} P_{\nu}(q-p) a_{n}, \qquad \{a_n \}_{n \in \mathbb{N}_0} \subseteq \mathbb{H}.$$
\end{theorem}

In \cite{AKS3} the authors proved that an axially Fueter regular function in a neighbourhood of the origin admits an expansion in series in terms of Clifford-Appell polynomials in \eqref{capoly}.
Indeed we have:
\begin{proposition}\label{Propo72}
    Let $U  \subseteq \mathbb{H}$ be an axially symmetric slice domain containing the origin. Let $f$ be a slice hyperholomorphic function in $U$ with series expansion
    $$f(q)=\sum_{k=0}^{\infty}q^k a_k, \qquad \{a_k\}_{k \in \mathbb{N}_0} \subseteq \mathbb{H}$$
    convergent (absolutely and uniformly on the compact subsets) in the ball $B(0,R)$, where $R=R_{0}$ is the largest positive real number such that $B(0,R)$ is contained in U.
    Then, for $q\in B(0,R)$,  the axially Fueter regular function $\breve f=\Delta f$ can be written as
    \begin{equation}
        \label{closedmono}
        \breve{f}(q)=\Delta f(q)= \sum_{k=0}^{\infty} \mathcal{Q}_k(q) b_k,
    \end{equation}
with $\mathcal{Q}_k(q)$ as in \eqref{capoly} and $b_k:=-2(k+1)(k+2)a_{k+2}$.
\end{proposition}
We note that the absolute and uniform convergence on the compact subsets of $B(0,R)$ follows from the estimate \eqref{estCliff}.
We now prove another expression of the series expansion of a function which is axially Fueter regular in a neighborhood of the origin. To this end we need the following result.
\begin{proposition}
\label{closedLapla2}
    Let $n \geq 1$, then for $q \notin \mathbb{R}$ we have
    \begin{equation}
        \label{closedLapla}
        \Delta (q^n)= -(\underline{q})^{-1} \left[2n q^{n-1}+ (\underline{q})^{-1} \left(\bar{q}^n-q^n\right)\right],
    \end{equation}
    while, for $q \in \mathbb{R}$,
    \begin{equation}
        \label{realclosed}
        \Delta (q^n)=-2n(n-1)q^{n-2}.
    \end{equation}
\end{proposition}
\begin{proof}
    We consider $q \notin \mathbb{R}$ and we prove the result by induction on $n$. \\
Let us consider $n=1$ and write $q=q_0+ \underline{q}$; we get
$$
        \Delta(q)=- (\underline{q})^{-1} \left[2+ (\underline{q})^{-1}\left(\bar{q}-q\right)\right]=0$$
and the formula holds since it is evident that $\Delta(q)=0$.  We now suppose that formula \eqref{closedLapla} is valid for $n$ and we prove it for $n+1$. By the product formula for the Laplace operator, see \eqref{prodlapla},  \eqref{for1} and formula \eqref{splittingH}, we have
    \begin{eqnarray*}
        \Delta (q^{n+1})&=& \Delta(q^n) q+2 D(q^n)\\
        &=& \Delta(q^n) q-4n H_{n-1}(q)\\
        &=&\Delta(q^n)q+2 (\underline{q})^{-1}\left(\bar{q}^n-q^n \right).
    \end{eqnarray*}
    The induction hypothesis gives
    \begin{eqnarray*}
        \Delta (q^{n+1})&=&-2n (\underline{q})^{-1} q^{n}-(\underline{q})^{-2}(\bar{q}^n-q^n)q+2 (\underline{q})^{-1}\left(\bar{q}^n-q^n \right)\\
        &=& -2(n+1) (\underline{q})^{-1}q^n-q_0 (\underline{q})^{-2} \left(\bar{q}^n-q^n\right)+ (\underline{q})^{-1} (q^n+ \bar{q}^n)\\
        &=& -2(n+1) (\underline{q})^{-1}q^n-q_0 (\underline{q})^{-2} \left(\bar{q}^n-q^n\right)- \underline{q}(\underline{q})^{-2} (-q^n-\bar{q}^n)\\
        &=&-2(n+1) (\underline{q})^{-1}q^n-(\underline{q})^{-2}\left(q_0 \bar{q}^n-\underline{q}\bar{q}^n-q_0 q^n-\underline{q}q^n \right)\\
        &=& -2(n+1) (\underline{q})^{-1}q^n- (\underline{q})^{-2} \left(\bar{q}^n \bar{q}-q^n q\right)\\
        &=& -2(n+1) (\underline{q})^{-1}q^n- (\underline{q})^{-2} \left(\bar{q}^{n+1}-q^{n+1}\right).
    \end{eqnarray*}
    This proves formula \eqref{closedLapla} for a nonreal $q$.
    If $q\in\mathbb{R}$, formula \eqref{deltareal} gives $\Delta=-2\partial_{q_0}^2$ on the real axis for slice hyperholomorphic functions, whence \eqref{realclosed}.
\end{proof}

The latter result allows us to rewrite Proposition \ref{Propo72} in an alternative way.
\begin{corollary}
\label{closedLap}
 Let $U  \subseteq \mathbb{H}$ be an axially symmetric slice domain containing the origin. Let $f$ be a slice hyperholomorphic function in $U$ with series expansion
\begin{equation}
\label{newstar}
f(q)=\sum_{k=0}^{\infty}q^k a_k, \qquad \{a_k\}_{k \in \mathbb{N}_0} \subseteq \mathbb{H}
\end{equation}
   convergent absolutely and uniformly on the compact subsets of $B(0,R)$ where $R=R_{0}$ is the largest positive real number such that $B(0,R)$ is contained in $U$.
    Then, for $q\in B(0,R)$,  the axially Fueter regular function $\breve f=\Delta f$ can be written as
$$ \breve f(q)=\Delta f(q)=-2 (\underline{q})^{-1} \partial_{q_0}\left(\sum_{n=0}^{\infty} q^{n}a_n \right)-(\underline{q})^{-2} \sum_{n=0}^{\infty} \left(\bar{q}^n-q^n \right)a_{n}, \qquad q \notin \mathbb{R}.$$
\end{corollary}
\begin{proof}
By formula \eqref{closedLapla}, applied to the series expansion \eqref{newstar}, we obtain
\begin{eqnarray*}
\Delta f(q)&=&-2 (\underline{q})^{-1} \sum_{n=1}^{\infty}n q^{n-1}a_n-(\underline{q})^{-2} \sum_{n=1}^{\infty} \left(\bar{q}^n-q^n \right)a_{n}\\
&=& -2 (\underline{q})^{-1} \partial_{q_0}\left(\sum_{n=0}^{\infty} q^{n}a_n \right)-(\underline{q})^{-2} \sum_{n=0}^{\infty} \left(\bar{q}^n-q^n \right)a_{n}.
\end{eqnarray*}
\end{proof}

A byproduct of Proposition \ref{closedLapla2} is the possibility to write the Clifford-Appell polynomials, see \eqref{capoly}, in a more compact form.

\begin{corollary}
\label{closedAPP}
Let $n \geq 0$ and $ q \in \mathbb{H} \setminus \mathbb{R}$. We can write the Clifford-Appell polynomials as
$$ \mathcal{Q}_{n}(q)=\frac{(\underline{q})^{-1}}{2(n+1)(n+2)} \left[ 2(n+2)q^{n+1}+ (\underline{q})^{-1}\left(\bar{q}^{n+2}-q^{n+2} \right)\right].$$
\end{corollary}
\begin{proof} In \cite{DDG,DKS} it is shown that
\begin{equation}\label{appLapla} \Delta(q^{n})= - {2n(n-1)} \mathcal{Q}_{n-2}(q).\end{equation}
So the result follows by applying \eqref{closedLapla} and a change of indices.
\end{proof}

So far we discussed series expansions at the origin. We now address the problem of studying the series expansion  of an axially Fueter regular function in a neighbourhood of a generic quaternion $p \in \mathbb{H} $.
Using the Fueter theorem and considering an axially Fueter regular function as the image via the second Fueter map of a slice hyperholomorphic function, we can exploit the series expansion of the Fueter primitive. The expansion is either a $*$-Taylor series or a spherical series and, according to these two cases, we obtain the notions of  axially Fueter regular series and of regular Fueter spherical series, respectively, which will be the object of the next subsections.

\subsection{Axially Fueter regular series}

Our task here is to study axially Fueter regular series and to discuss their convergence. We shall consider slice hyperholomorphic functions in the sense of Definition \ref{sh}, even when not explicitly stated. In fact, we shall systematically consider the action of the second map in the Fueter construction which is based on functions of slice form.

\begin{definition}
    \label{axf}
    Let $U$ be an axially symmetric open set in $\mathbb{H}$ and $p \in U$. We say that $\breve f=\Delta f$ admits an axially Fueter regular series at $p$ if $f$ is a slice hyperholomorphic function in $p$ (according to Definition \ref{sh}) that admits a $*$-Taylor series centered at an arbitrarily fixed point $p$, convergent in a set contained in $U$.
\end{definition}
As we did in the previous sections, we shall indicate by $*$ all $*_{q,L}$-products. The following result is crucial to give a compact expression for an axially Fueter regular series.

\begin{theorem}
    \label{lapla0}
    Let $p,q \in \mathbb{H}$, and let $\Delta$ be the Laplace operator in four real variables with respect to the variable $q$. Then for $ n \geq 2$ we have
    \begin{equation}
        \label{lapla}
        \Delta (q-p)^{*n}= -4 \sum_{k=1}^{n-1} (n-k) (q-p)^{*(n-k-1)}* (\bar{q}-p)^{*(k-1)}.
    \end{equation}\end{theorem}
\begin{proof}
    We prove the result by induction on $n$. For $n=2$ we have
    $$ \Delta (q-p)^{*2}= \Delta (q^2+p^2-2qp)=-4$$
    thus, the result is trivial for $n=2$. Then, we suppose that the formula is true for $n$ and we prove it for $n+1$. From the product rule of the Laplace operator, see \eqref{prodlapla}, we deduce
    \begin{eqnarray*}
        \Delta  (q-p)^{*(n+1)}&=& \Delta[(q-p)^{*n}* (q-p)]\\
        &=&\Delta [q(q-p)^{*n}]- [\Delta (q-p)^{*n}]p\\
        &=& q \Delta [(q-p)^{*n}]+2 D[(q-p)^{*n}]-[\Delta (q-p)^{*n}]p.
    \end{eqnarray*}
The inductive hypothesis and Lemma \ref{genbeg} yield
    \begin{eqnarray}
        \nonumber
        \Delta  (q-p)^{*(n+1)}&=& -4q\sum_{k=1}^{n-1} (n-k) (q-p)^{*(n-k-1)}* (\bar{q}-p)^{*(k-1)}\\
        \nonumber
        &&-4 \sum_{k=1}^{n} (q-p)^{*(n-k)}* (\bar{q}-p)^{*(k-1)}\\
        \label{doublestar}
        &&+4\sum_{k=1}^{n-1} (n-k) (q-p)^{*(n-k-1)}* (\bar{q}-p)^{*(k-1)}p.
    \end{eqnarray}
  By the definition of the right $*$-product, and the fact that $(q-p)$ and $(\bar q-p)$ $*$-commute, we have that
    \begin{eqnarray}
        \nonumber
        &&q\sum_{k=1}^{n-1} (n-k) (q-p)^{*(n-k-1)}* (\bar{q}-p)^{*(k-1)}\\
        \nonumber
        &&-\sum_{k=1}^{n-1} (n-k) (q-p)^{*(n-k-1)}* (\bar{q}-p)^{*(k-1)}p\\
        \label{onestar}
        &=& \sum_{k=1}^{n-1} (n-k) (q-p)^{*(n-k)}* (\bar{q}-p)^{*(k-1)}.
    \end{eqnarray}
    Therefore, by substituting \eqref{onestar} in \eqref{doublestar}  we conclude that
    \begin{eqnarray*}
        \Delta  (q-p)^{*(n+1)}&=& -4\sum_{k=1}^{n} (n-k) (q-p)^{*(n-k)}* (\bar{q}-p)^{*(k-1)}\\
        && -4 \sum_{k=1}^{n} (q-p)^{*(n-k)}*(\bar{q}-p)^{*(k-1)}\\
        &=& -4\sum_{k=1}^{n} (n+1-k) (q-p)^{*(n-k)}*(\bar{q}-p)^{*(k-1)},
    \end{eqnarray*}
which proves the result.
\end{proof}

\begin{remark}
    By taking $p=0$ in formula \eqref{lapla} we get the same result obtained in \cite[Theorem 3.2]{DKS}.
\end{remark}

We now introduce some new functions which are defined via the formula \eqref{lapla}, indeed we can write
\begin{equation}
    \label{tildeq}
    \Delta (q-p)^{*n}=-2n (n-1) \widetilde{Q}_{n-2}(q,p), \qquad n \geq 2,
\end{equation}
where the functions $\widetilde{Q}_{n}(q, p)$ are:
\begin{equation}
    \label{monopol}
    \widetilde{Q}_{n}(q, p):=2 \sum_{j=0}^{n} \frac{(n-j+1)}{(n+1)(n+2)} (q-p)^{*(n-j)}* (\bar{q}-p)^{*j}.
\end{equation}
Below we study some properties of the functions $\widetilde{Q}_{n}(q, p)$, $n\in\mathbb N_0$.
\begin{proposition}
    \label{der}
Let $p$, $q \in \mathbb{H}$,  $n \in \mathbb{N}_0$. Then the functions $\widetilde{Q}_{n}(q, p)$ satify the following properties:
\begin{itemize}
\item[1)]   $\widetilde{Q}_{n}(q, p)$ are left axially regular in the variable $q$ and right slice hyperholomorphic in the variable $p$.
\item[2)]  $\widetilde{Q}_{n}(q, p)$ are left slice polyanalytic of order $n+1$ in the variable $q$.
\item[3)]  $\widetilde{Q}_{n}(q, p)$ form an Appell sequence with respect to the conjugate Fueter operator $\bar{D}$ in $q$. Precisely, we have
$$ \frac{\bar{D}}{2} \widetilde{Q}_{n}(q,p)=n \widetilde{Q}_{n-1}(q,p), \qquad n \geq 1.$$
\item[4)]  For $q \notin \mathbb{R}$ and $n \geq 2$, $\widetilde{Q}_{n}(q, p)$ satisfy
\begin{equation}
\label{clsoedtilde}
\widetilde{Q}_{n-2}(q,p)= \frac{(\underline{q})^{-1}}{2n(n-1)} \left[2n (q-p)^{*(n-1)}+ (\underline{q})^{-1}\left[(\bar{q}-p)^{*n}-(q-p)^{*n}\right] \right],
\end{equation}
and
\begin{equation}
\label{NNNNN}
\widetilde{Q}_{n-2}(q,p)= \frac{(\underline{q})^{-1}}{n-1} \left[ (q-p)^{*(n-1)}- \widetilde{H}_{n-1}(q,p)\right],
\end{equation}
where $\widetilde{H}_n(q,p)$ are defined in \eqref{Harmopoly}.
\end{itemize}
\end{proposition}

\begin{proof}
We prove the statements listed above.
\begin{itemize}
\item [1)] By formula \eqref{tildeq} we have that
$$ \widetilde{Q}_{n}(q,p)= -\frac{1}{2(n+2)(n+1)}\Delta (q-p)^{*(n+2)}.$$
Since $(q-p)^{*(n+2)}$ is left slice hyperholomorphic in $q$, by the Fueter theorem we have that
$$ D \widetilde{Q}_{n}(q,p)=-\frac{1}{2(n+2)(n+1)}\Delta D(q-p)^{*(n+2)}=0.$$
The fact that $\widetilde{Q}_{n}(q,p)$  are right slice hyperholomorphic in $p$ follows from Proposition \ref{closedp}.
\item [2)] By using arguments similar to those used to prove the second point of Proposition \ref{harmpoly} we can write
$$ \widetilde{Q}_{n}(q,p)= \sum_{\ell=0}^{n} \bar{q}^\ell g_{\ell}(q), \quad g_{\ell}(q):= \frac{2}{(n+1)(n+2)} \sum_{k= \ell}^{n} (n- k+1) \binom{k}{\ell} (q-p)^{*(n-k)} p^{k- \ell} (-1)^{k- \ell}.$$
By the Proposition \ref{closedp} we get that the functions $g_{\ell}(q)$ are slice hyperholomorphic. So Theorem \ref{polydeco} implies that $\widetilde{Q}_{n}(q,p)$ are left slice polyanalytic of order $n+1$.
\item[3)]   Since the functions $\widetilde{Q}_{n}(q,p)$ are axially Fueter regular, we have
$\partial_{q_0}\widetilde{Q}_{n}(q,p)=-\partial_{\underline{q}}\widetilde{Q}_{n}(q,p)$. Therefore, from point 1) of this proposition we get
$$  \frac{\bar{D}}{2} \widetilde{Q}_n(q,p)=\frac{(\partial_{q_0}- \partial_{\underline{q}})\widetilde{Q}_n(q,p)}{2}=\partial_{q_0} \widetilde{Q}_n(q,p)=n \widetilde{Q}_{n-1}(q,p).$$

The last equality follows by arguments similar to those used to prove point 3) of Proposition \ref{harmpoly}.
\item[4)] By \eqref{tildeq}, \eqref{closedLapla} and the binomial theorem we have
\begin{eqnarray*}
\widetilde{Q}_{n-2}(q,p)&=&- \frac{1}{2n(n-1)} \Delta (q-p)^{*n}\\
&=& - \frac{1}{2n(n-1)} \sum_{k=0}^{n} \binom{n}{k} \Delta (q^k) (-p)^{n-k}\\
&=& \frac{(\underline{q})^{-1}}{n(n-1)} \sum_{k=1}^{n} k \binom{n}{k} q^{k-1}(-p)^{n-k}+ \frac{(\underline{q})^{-2}}{2n(n-1)} \sum_{k=0}^{n} \binom{n}{k} (\bar{q}^k-q^k) (-p)^{n-k}\\
&=& \frac{(\underline{q})^{-1}}{n-1} \sum_{k=1}^{n} \binom{n-1}{k-1} q^{k-1}(-p)^{n-k}+ \frac{(\underline{q})^{-2}}{2n(n-1)} \sum_{k=0}^{n} \binom{n}{k} (\bar{q}^k-q^k) (-p)^{n-k}\\
&=& \frac{(\underline{q})^{-1}}{2n(n-1)} \left[2n (q-p)^{*(n-1)}+ (\underline{q})^{-1}\left[(\bar{q}-p)^{*n}-(q-p)^{*n}\right] \right].
\end{eqnarray*}
Finally formula \eqref{NNNNN} follows by using \eqref{clsoedtilde} and \eqref{closedH1}:
\begin{eqnarray*}
    \widetilde{Q}_{n-2}(q,p)&=& \frac{(\underline{q})^{-1}}{2n(n-1)} \left[2n (q-p)^{*(n-1)}-2n \widetilde{H}_{n-1}(q,p)\right]\\
    &=& \frac{(\underline{q})^{-1}}{n-1} \left[ (q-p)^{*(n-1)}- \widetilde{H}_{n-1}(q,p)\right],
\end{eqnarray*}
\end{itemize}
which completes the proof.
\end{proof}

Now, we prove another series expansion of an axially Fueter regular function around a quaternion $p$.

\begin{theorem}
	    \label{conve2}
Let $U$ be an axially symmetric open set in $\mathbb{H}$. Let $f$ be a slice hyperholomorphic function on $U$ admitting the $*$-Taylor expansion at an arbitrary point $p\in U$.
$$ f(q)= \sum_{n=0}^{\infty} (q-p)^{n*_{q,L}}a_n,$$
convergent (absolutely and uniformly on the compact subsets) in $\Omega(p,R) \subset U$,
where $\frac{1}{R}=\limsup_{n \to \infty} |a_n|^{\frac{1}{n}}$. Then
    \begin{equation}
        \label{axm1}
        \breve f(q)=\Delta f(q)= \sum_{n=0}^{\infty} \widetilde{Q}_{n}(q, p) b_n, \qquad q \in \Omega(p,R), \quad b_n:= -2(n+2)(n+1) a_{n+2}.
    \end{equation}
\end{theorem}
\begin{proof}
We start proving formula \eqref{axm1}. Applying the second Fueter map to $(q-p)^{*n}$, see   \eqref{tildeq}, we get
\begingroup\allowdisplaybreaks
\begin{eqnarray*}
	\Delta f(q) &=& \sum_{n=2}^{\infty} -2n (n-1) \widetilde{Q}_{n-2}(q,p)a_{n} \\
	&=& -2\sum_{n=0}^\infty (n+2)(n+1) \widetilde{Q}_{n}(q,p)a_{n+2} \\
	&=& \sum_{n=0}^{\infty} \widetilde{Q}_{n}(q, p) b_n,
\end{eqnarray*}
\endgroup
where $b_n:= -2(n+2)(n+1) a_{n+2}$. Now, we prove the convergence. If $q \in \mathbb{R}$ the convergence is trivial and follows, e.g., by the root test. We suppose that $q \notin \mathbb{R}$. Since $(q-p)^{*(n+2)}$ is slice hyperholomorphic, by formulas \eqref{tildeq} and \eqref{rapm}, for $n \geq 2$, and by Theorem \ref{rapp0} we have
    \begin{eqnarray*}
       -2(n+2)(n+1) \widetilde{Q}_{n}(q,p)a_{n+2} &=& \Delta (q-p)^{*(n+2)}a_{n+2}\\
        &=&-(\underline{q})^{-1}I_{\underline{q}}I(n+2) \left[(q-p)_{|q=q_{I}}^{*(n+1)}-(q-p)_{|q=q_{-I}}^{*(n+1)}\right]a_{n+2}\\
        &&-(n+2)(\underline{q})^{-1} \left[(q-p)_{|q=q_{-I}}^{*(n+1)}+(q-p)_{|q=q_{I}}^{*(n+1)}\right]a_{n+2}\\
        && +(\underline{q})^{-2}I_{\underline{q}}I\left[(q-p)_{|q=q_{-I}}^{*(n+2)}-(q-p)_{|q=q_{I}}^{*(n+2)}\right]a_{n+2}.
    \end{eqnarray*}
    Hence by taking $q_{\pm I}=x \pm yI$ on $\mathbb{C}_I$ containing $p$ we have
    \begin{eqnarray*}
        |2(n+2)(n+1) \widetilde{Q}_{n}(q,p)a_{n+2}|& \leq & 2(n+2)|\underline{q}|^{-1} \left[\left|(q_{-I}-p)^{(n+1)}a_{n+2} \right| +\left|(q_{I}-p)^{(n+1)} a_{n+2} \right|\right]\\
        && + (n+2)|\underline{q}|^{-2} \left[\left|(q_{-I}-p)^{(n+2)}a_{n+2} \right| +\left|(q_{I}-p)^{(n+2)} a_{n+2} \right|\right].
    \end{eqnarray*}
    Since $ q \in \Omega(p,R)$ and $\frac{1}{R}=\limsup_{n \to \infty} |a_n|^{\frac{1}{n}}=\limsup_{n \to \infty} (n |a_n|)^{\frac{1}{n}}$ we get that the series
    $$ \sum_{n=0}^{\infty}\left|(q_{\pm I}-p)^{(n+1)}(n+2)a_{n+2} \right|, \quad \hbox{and} \quad \sum_{n=0}^{\infty}\left|(q_{\pm I}-p)^{(n+2)}(n+2)a_{n+2} \right|,$$
    are convergent. So we obtain that the series \eqref{axm1} is convergent where stated.

\end{proof}

\begin{remark}
	\label{remmono}
	If we take $p=0$ in \eqref{axm1} we get
	\begin{equation}
		\label{zeromono}
		\breve f(q)=\Delta f(q)= -2 \sum_{n=0}^{\infty} (n+1)(n+2)\mathcal{Q}_n(q) a_{n+2},
	\end{equation}
	since $ \widetilde{Q}_n(q,0)=\mathcal{Q}_n(q)$. Formula \eqref{zeromono} is the Taylor expansion in a neighbourhood of the origin of the axially Fueter regular function $\breve f$ in \eqref{closedmono}. 	
\end{remark}

A regular series can also be written in terms of Clifford-Appell polynomials
\begin{proposition}
    \label{regca}
    Let $U$ be an axially symmetric open set in $\mathbb{H}$ and let $f$ be a slice hyperholomorphic function in $U $. Assume that $f$ admits a $*$-Taylor expansion centered at an arbitrarily fixed point $p\in U$, with coefficients  $ \{a_n\}_{n \in \mathbb{N}_0} \subseteq \mathbb{H}$ and convergent (uniformly  on the compact subsets) in $\Omega(p,R)\subseteq U$. Then the axially Fueter regular series of $\breve f=\Delta f$, where it is convergent, can be written as
\begin{equation}
\label{starrr}
\sum_{n=0}^{\infty} \widetilde{Q}_n(q,p)b_n=\sum_{n=0}^{\infty} \sum_{k=0}^{n} \binom{n}{k} \mathcal{Q}_k(q) p^{n-k} (-1)^{n-k} b_{n},
\end{equation}
    where $b_n:=-2(n+2)(n+1)a_{n+2}$. Moreover, for $q \notin \mathbb{R}$,  \eqref{starrr} rewrites as
\begin{equation}
\label{laplareal}
\sum_{n=0}^{\infty} \widetilde{Q}_n(q,p)b_n=-2 (\underline{q})^{-1} \partial_{q_0}\left(\sum_{n=0}^{\infty} (q-p)^{*n}a_n\right)-(\underline{q})^{-2}\sum_{n=0}^{\infty}\left[ (\bar{q}-p)^{*n}-(q-p)^{*n} \right]a_n.
\end{equation}
\end{proposition}
\begin{proof}
By the hypothesis on the function $f$ and the binomial theorem we can write
    $$f(q)= \sum_{n=0}^{\infty} \sum_{k=0}^{n} \binom{n}{k} q^k p^{n-k} (-1)^{n-k}a_n.$$
Using formula \eqref{appLapla} we get
       \begingroup\allowdisplaybreaks
    \begin{eqnarray}
    	\nonumber
        \breve f(q)=\Delta f(q) &=& \sum_{n=0}^{\infty} \sum_{k=0}^{n} \binom{n}{k} (-1)^{n-k} \Delta(q^k) p^{n-k} a_n\\
        \nonumber
        &=&-2 \sum_{n=0}^{\infty} \sum_{k=2}^{n} \binom{n}{k} k(k-1) \mathcal{Q}_{k-2}(q) p^{n-k}(-1)^{n-k}a_n \\
        \nonumber
        &=& -2\sum_{n=2}^{\infty} \sum_{k=0}^{n-2} \binom{n}{k+2} (k+2)(k+1) \mathcal{Q}_k(q) p^{n-2-k} (-1)^{n-k} a_n\\
        \nonumber
        &=& -2\sum_{n=2}^{\infty} \sum_{k=0}^{n-2} n(n-1)\binom{n-2}{k} \mathcal{Q}_k(q) p^{n-2-k} (-1)^{n-k} a_n\\
        \label{onee}
        &=& -2\sum_{n=0}^{\infty} \sum_{k=0}^{n} (n+2)(n+1)\binom{n}{k} \mathcal{Q}_k(q) p^{n-k} (-1)^{n-k} a_{n+2}.
    \end{eqnarray}
      \endgroup
The assertion follows from \eqref{onee} and \eqref{axm1}.
Moreover, if we consider $q \notin \mathbb{R}$, by formula \eqref{clsoedtilde} we have
       \begingroup\allowdisplaybreaks
\begin{eqnarray*}
\sum_{n=0}^{\infty} \widetilde{Q}_n(q,p)b_n&=& -2\sum_{n=0}^{\infty} (n+2)(n+1) \widetilde{Q}_n(q,p)a_{n+2}\\
&=&-2 (\underline{q})^{-1} \sum_{n=0}^{\infty} (n+2) (q-p)^{*(n+1)}a_{n+2}\\
&&-(\underline{q})^{-2} \sum_{n=0}^{\infty} \left[(\bar{q}-p)^{*(n+2)}-(q-p)^{*(n+2)} \right]a_{n+2}\\
&=&-2 (\underline{q})^{-1}\partial_{q_0} \left(\sum_{n=2}^{\infty}  (q-p)^{*n}a_{n}\right)-(\underline{q})^{-2} \sum_{n=2}^{\infty} \left[(\bar{q}-p)^{*n}-(q-p)^{*n} \right]a_n\\
&=& -2 (\underline{q})^{-1}\partial_{q_0} \left(\sum_{n=0}^{\infty}  (q-p)^{*n}a_{n}-(q-p)a_1-a_0\right)\\
&&-(\underline{q})^{-2}\left( \sum_{n=0}^{\infty} \left[(\bar{q}-p)^{*n}-(q-p)^{*n} \right]a_n- [(\bar{q}-p)-(q-p)]a_1 \right)\\
&=&  -2 (\underline{q})^{-1}\partial_{q_0} \left(\sum_{n=0}^{\infty}  (q-p)^{*n}a_{n}\right)+ 2 (\underline{q})^{-1}a_1\\
&& -(\underline{q})^{-2}  \sum_{n=0}^{\infty} \left[(\bar{q}-p)^{*n}-(q-p)^{*n} \right]a_n-2 (\underline{q})^{-1}a_1\\
&=& -2 (\underline{q})^{-1} \partial_{q_0}\left(\sum_{n=0}^{\infty} (q-p)^{*n}a_n\right)-(\underline{q})^{-2}\sum_{n=0}^{\infty}\left[ (\bar{q}-p)^{*n}-(q-p)^{*n} \right]a_n,
\end{eqnarray*}	
\endgroup
which ends the proof.
\end{proof}

Now, we provide an example of an axially Fueter regular series.
\begin{example}
The $F$-kernel, see \eqref{FK}, can be expressed in terms of the polynomials introduced in \eqref{monopol}:
	\begin{equation}
	\label{ex6}
	F_L(p,q)= \sum_{n=0}^{\infty} \widetilde{Q}_n(q,p+1)a_n, \qquad q \in \Omega(p+1,1),
\end{equation}
where $\{ a_n\}_{n  \in \mathbb{N}_0}:= \{2 (-1)^n (n+1)(n+2)\}_{n  \in \mathbb{N}_0}$. In fact, if we apply the second Fueter map to $S^{-1}_L(p,q)=-\sum_{n=0}^{+\infty}(1-q+p)^{*n}$, see \eqref{ex2}, using the fact that the series converges absolutely and uniformly on the compact sets in $\Omega(p+1,1)$ we can apply the Laplacian termwise.  By \eqref{tildeq} we get
\begin{eqnarray*}
	\Delta S^{-1}_L(p,q)&=& \sum_{n=0}^{\infty} (-1)^{n+1} \Delta (q-p-1)^{*n}\\
	&=&2 \sum_{n=2}^{\infty} (-1)^n n(n-1) \widetilde{Q}_{n-2}(q,p+1)\\
	&=& 2 \sum_{n=0}^{\infty} (-1)^n (n+1)(n+2) \widetilde{Q}_n(q,p+1),
\end{eqnarray*}
for $q\in\Omega(p+1,1)$.
Since $\Delta S^{-1}_L(p,q)=F_L(p,q)$, see \eqref{FK}, we get \eqref{ex6}.
\end{example}

\begin{remark}
	We observe that $F_L(0,q)=4 E(q)$ and $ \widetilde{Q}_n(q-1,0)=\widetilde{Q}_n(q,1)= \mathcal{Q}_n(q-1)$. Thus, if we take $p=0$ in \eqref{ex6} we get
	\begin{equation}
		\label{ex7}
		E(q)= \sum_{n=0}^{\infty} \frac{(-1)^n (n+1)(n+2)}{2} \mathcal{Q}_n(q-1)= \sum_{n=0}^{\infty} \frac{(n+1)(n+2)}{2} \mathcal{Q}_n(1-q),
	\end{equation}
	in $|q-1|<1$. We observe that the expansion \eqref{ex7} was obtained also in \cite{CFM}, by using a different method.
\end{remark}

\subsection{Regular Fueter spherical series}

We now show that a function axially Fueter regular in a Euclidean neighborhood of a generic quaternion $p$ admits another series expansion in a convenient set containing $p$.

\begin{definition}\label{aFuSe}
    Let $U$ be an axially symmetric slice domain. Let us assume that $f$ is slice hyperholomorphic in  $U$ and admits a spherical series expansion convergent in a suitable neighborhood of $p$ contained in $U$. Then we say that $ \breve f=\Delta f$ has a regular Fueter spherical series in a neighbourhood of $p$.
\end{definition}

The above definition is justified by the fact that, by applying the second Fueter map to the spherical series expansion of a function $f$, we obtain a formula for the regular Fueter spherical series.

\begin{definition}
	\label{Cndef}
	For $p\in\mathbb{H}$ and $q\in\mathbb{H}\setminus[p]$ we set
	\begin{equation}
		\label{keyC}
		\mathcal{C}_n(q,p):=
		\begin{cases}
			\displaystyle\sum_{k=0}^{n-2}(k+1)\,\overline{Q_p^{\,n-2-k}(q)}\;Q_p^{k}(q), & n\geq 2,\\[2mm]
			0, & n=0,1,\\[2mm]
			\displaystyle\sum_{k=0}^{|n|-1}(|n|-k)\;Q_p^{\,k+n-1}(q)\;\overline{Q_p^{-k-1}(q)}, & n\leq -1.
		\end{cases}
	\end{equation}
	For $n\geq 2$ the function $\mathcal{C}_n(q,p)$ is a polynomial in $q,\bar q$ and it is defined for
	every $q\in\mathbb{H}$.
\end{definition}

\begin{lemma}
Let $p \in \mathbb{H}$. For $q \in \mathbb{H} \setminus [p]$ we have
\begin{equation}
	\label{LaurentI2}
	\mathcal{L}_n(q,p)-n Q_p^{n-1}(q)=-4(q_0-p_0) \underline{q} \mathcal{C}_n(q,p), \qquad n \in \mathbb{Z}.
\end{equation}	
\end{lemma}
\begin{proof}
We first consider the case $n\geq 2$. By the definition of $\mathcal{L}_n(q,p)$, see \eqref{key4}, we have
\begin{align*}
\mathcal{L}_n(q,p)-nQ^{n-1}_p(q)&=\sum_{k=0}^{n-1}\left(\overline{Q_p(q)}^{\,n-1-k}-Q_p^{n-1-k}(q)\right)Q_p^k(q)\\
&=\sum_{k=0}^{n-2}\left(\overline{Q_p(q)}^{\,n-1-k}-Q_p(q)^{n-1-k}\right)Q_p^k(q).
\end{align*}
Using the identity \eqref{key1} and
	\begin{equation}
	\label{diffQC}
	\overline{Q_p(q)}-Q_p(q)=-4(q_0-p_0)\underline{q}.
\end{equation}
 we obtain
 $$ \mathcal{L}_n(q,p)-nQ^{n-1}_p(q)=-4(q_0-p_0)\underline{q}\sum_{k=0}^{n-2}\sum_{j=0}^{n-2-k}
 \overline{Q_p^{\,n-2-k-j}(q)}Q^{j+k}_p(q).$$
 Setting $\ell=j+k$, we observe that, for each
 $\ell=0,\ldots,n-2$, there are exactly $\ell+1$ pairs $(j,k)$ such
 that $j+k=\ell$. Hence
\begin{align*}
\mathcal{L}_n(q,p)-nQ^{n-1}_p(q)
&=-4(q_0-p_0)\underline{q}\sum_{\ell=0}^{n-2}(\ell+1)\overline{Q_p^{\,n-2-\ell}(q)}Q_p^\ell(q)\\
&=-4(q_0-p_0)\underline{q}\mathcal{C}_n(q,p).
\end{align*}
For $n=0$, the identity follows immediately from $\mathcal{L}_0(q,p)=\mathcal{C}_0(q,p)=0$. For $n=1$, we have $\mathcal{L}_1(q,p)=1$ and $\mathcal{C}_1(q,p)=0$, and hence the identity follows also in this case.
\\It remains to consider $n\leq-1$. Set $m:=-n\geq1$. By the definition of $\mathcal{L}_{-m}(q,p)$, see \eqref{key4}, we have
\begin{align*}
\mathcal{L}_{-m}(q,p)+mQ^{-m-1}_p(q)&=-\sum_{k=0}^{m-1}\overline{Q_p^{-k-1}(q)}Q_p^{k-m}(q)
+\sum_{k=0}^{m-1}Q_p^{-m-1}(q)\\
&=-\sum_{k=0}^{m-1}\left(\overline{Q_p^{-k-1}(q)}-Q_p^{-k-1}(q)\right)Q_p^{k-m}(q).
\end{align*}
Using the identity \eqref{rel} we get
$$ \mathcal{L}_{-m}(q,p)+mQ_p^{-m-1}(q)=-4(q_0-p_0)\underline{q}\sum_{k=0}^{m-1}\sum_{j=0}^{k}
\overline{Q_p^{-j-1}(q)}Q_p^{j-m-1}(q)
$$
Interchanging the order of summation, we observe that, for each
$j=0,\ldots,m-1$, the index $k$ ranges from $j$ to $m-1$, and hence
there are exactly $m-j$ corresponding terms. Therefore,
\begin{align*}
\mathcal{L}_{-m}(q,p)+mQ_p^{-m-1}(q)
&=-4(q_0-p_0)\underline{q}\sum_{j=0}^{m-1}(m-j)Q_p^{j-m-1}(q)\overline{Q_p^{-j-1}(q)}\\
&=-4(q_0-p_0)\underline{q}\mathcal{C}_{-m}(q,p).
\end{align*}
Thus, the identity holds for every $n\in\mathbb{Z}$.
\end{proof}

\begin{theorem}
	\label{convreal}
Let $f$ be a slice hyperholomorphic function in a neighbourhood of $p \in \mathbb{H}$
		admitting the spherical expansion \eqref{spherical} with coefficients
		$\{a_n\}_{n\in\mathbb{N}_0}\subseteq\mathbb{H}$, converging (absolutely and uniformly on the compact subsets) in $U(p,R)$. Let $\breve{f}$ be the axially regular function in a neighborhood  of $p$ given by $\breve{f}=\Delta f$. Then, $\breve{f}$ has the following regular Fueter spherical expansion
\begin{eqnarray}
	\nonumber
	\breve f(q)=\Delta f(q)=&&\!\!\!\!\!\!\!\!-16(q_0-p_0)^2\sum_{n=0}^{\infty}\mathcal{C}_{n+2}(q,p)\left[a_{2n+4}+(q-p)a_{2n+5}\right]\\
	\label{res1bis}
	&&\!\!\!\!\!\!\!\!-8(q_0-p_0)\sum_{n=0}^{\infty}\mathcal{L}_{n+1}(q,p)\,a_{2n+3}
	-4\sum_{n=0}^{\infty}(n+1)\,Q^{n}_p(q)\left[a_{2n+2}+(q-p)a_{2n+3}\right].
\end{eqnarray}
	where $\mathcal{L}_n(q,p)$ and $\mathcal{C}_n(q,p)$ are given by \eqref{key4} and \eqref{keyC}.
\end{theorem}

\begin{proof}
We apply the operator $\Delta$ term by term to \eqref{spherical}. The interchange of $\Delta$ and the summation will be justified later in the proof. We first consider the case $q\notin\mathbb R$. By formula \eqref{compactF}, applied to $Q_p^n(q)$, we have
$$\Delta[Q_p^n(q)]=-2(\underline q)^{-1}\partial_{q_0}Q_p^n(q)+(\underline q)^{-2}\left[Q_p^n(q)-Q_p^n(\bar q)\right].
$$
Using \eqref{derv} together with \eqref{laurentI} we get
$$\Delta[Q_p^n(q)]=-4nQ_p^{n-1}(q)+4(q_0-p_0)(\underline q)^{-1}\left[\mathcal L_n(q,p)-nQ_p^{n-1}(q)\right].$$
By \eqref{LaurentI2} it follows that
	\begin{equation}
	\label{delta}
	\Delta\left(Q^n_p(q)\right)=-16(q_0-p_0)^2\,\mathcal{C}_n(q,p)-4n\,Q^{n-1}_p(q), \qquad n\in \mathbb{Z}.
\end{equation}
Similarly, applying formula \eqref{compactF} to $Q_p^n(q)(q-p)$ and using the same identities yields
\begin{eqnarray}
\nonumber
	\Delta\left(Q^n_p(q)(q-p)\right)&=&-16(q_0-p_0)^2\,\mathcal{C}_n(q,p)(q-p)-8(q_0-p_0)\,
	\mathcal{L}_n(q,p)\\
		\label{delta1}
	&&-4n\,Q^{n-1}_p(q)(q-p), \qquad n \in \mathbb{Z}.
\end{eqnarray}
We now justify the interchange of $\Delta$ and the summation. Let $K$ be a compact subset of $U(p,R)$. Thus, we have $\sup_K |Q_p(q)|=\rho^2<R^2$. This implies $|\mathcal L_n(q,p)|\leq n\rho^{2n-2}$, and $|\mathcal C_n(q,p)|\leq\frac{n(n-1)}{2}\rho^{2n-4}$. Therefore, all the series in the previous calculations are convergent on the compact subsets of $U(p,R)$.
\\Finally, let $q=q_0\in\mathbb R$. Then $Q_p^n(q)$ and $Q_p^n(q)(q-p)$ are left slice hyperholomorphic in $q$, with the first one being intrinsic. Formula \eqref{deltareal} gives $\Delta(\cdot)=-2\partial_{q_0}^2(\cdot)$ at real points. By \eqref{derv} we get
\begin{equation}
	\label{real1d}
\Delta[Q_p^n(q_0)]=-4nQ_p^{n-1}(q_0)-8n(n-1)(q_0-p_0)^2Q_p^{n-2}(q_0)
\end{equation}
\begin{eqnarray}
\nonumber
\Delta\left[Q_p^n(q_0)(q_0-p)\right]&=&
-4nQ_p^{n-1}(q_0)(q_0-p)-8n(n-1)(q_0-p_0)^2Q_p^{n-2}(q_0)(q_0-p)\\
	\label{real2d}
&&-8n(q_0-p_0)Q_p^{n-1}(q_0).
\end{eqnarray}
On the other hand, for $q=q_0\in\mathbb R$ one has $Q_p(q_0)\in\mathbb R$, and hence
\begin{equation}
\mathcal{L}_n(q_0,p)=nQ_p^{n-1}(q_0), \qquad \mathcal C_n(q_0,p)=\frac{n(n-1)}2Q_p^{n-2}(q_0).
\end{equation}
Therefore, the right-hand side of \eqref{res1bis} reduces precisely to the expression obtained by applying $\Delta$ term by term at the real points.
\end{proof}

We now prove that the regular Fueter spherical series of $\breve f=\Delta f$ (see Definition \ref{aFuSe}) has the same radius and the same set of convergence as the slice hyperholomorphic and  the harmonic spherical series, see  Proposition \ref{c0} and  Theorem \ref{c1}.

\begin{theorem}
	\label{conve}
	Under the hypothesis of Theorem \ref{convreal}, let $ \{a_n\}_{n \in \mathbb{N}_0} \subset \mathbb{H}$ and suppose that
	$$ \limsup_{n \to \infty} |a_n|^{\frac{1}{n}}= \frac{1}{R},$$
	for some $R>0$. The regular Fueter spherical series \eqref{res1bis} converges (absolutely and uniformly on the compact subsets) in
	$U(p, R)$
	where $p \in \mathbb{H}$.
\end{theorem}
\begin{proof}
	Let us consider a compact set $K$ in ${U}(p, R)$. Then if $q \in K$ by definition we have $|Q_p^n(q)| \leq r^{2n}$  for $r$ such that $0<r<R$. Fix $R'$ with $r<R'<R$ so that $|a_n| \leq (1/R')^n$ for $n$ sufficiently large.
	Moreover, by \eqref{keyC} and \eqref{key4},
	$$|\mathcal{C}_{n+2}(q,p)|\leq \sum_{k=0}^{n}(k+1)|Q_p(q)|^{n}\leq \frac{(n+1)(n+2)}{2}\,r^{2n},\qquad
	|\mathcal{L}_{n+1}(q,p)|\leq (n+1)\,r^{2n},$$
	so that the three sums of \eqref{res1bis} are dominated exactly by the same quantities which are estimated below.
	We start by estimating the terms in the first sum of the regular Fueter spherical series, see \eqref{res1bis}. By the inequality \eqref{est} we have
	\begin{eqnarray*}
		&& 16| \mathcal{C}_{n+2}(q,p)(q_0-p_0)^2[a_{2n+4}+(q-p)a_{2n+5}]|\\
		&&\leq {\frac{8(n+1)(n+2)}{R'^5}}\left(\frac{r}{R'}\right)^{2n} (r^2+p_1^2) \left[R'+ \left(\sqrt{r^2+p_1^2}+p_1\right)\right]:= b_n.
	\end{eqnarray*}
	Now, we estimate the terms in the second and third series in \eqref{res1bis}. By Lemma \ref{r1} and Lemma \ref{newres} we have
	$$ 4| (n+1)Q_p^{n}(q)[a_{2n+2}+(q-p)a_{2n+3}]| \leq {\frac{4(n+1)}{R'^3}} \left( \frac{r}{R'}\right)^{2n} \left[R'+ \left(\sqrt{r^2+p_1^2}+p_1\right)\right]:=c_n.$$
	and
	$$ 8| \mathcal{L}_{n+1}(q,p)(q_0-p_0)a_{2n+3}| \leq {\frac{8(n+1)}{R'^3}}\sqrt{r^2+p_1^2} \left(\frac{r}{R'}\right)^{2n}=d_n.$$
	Therefore the regular Fueter spherical series is dominated by $ \sum_{n =0}^\infty \left(b_n+c_n+d_n\right)$, that is convergent by the ratio test, and this concludes the proof.
\end{proof}

The regular Fueter spherical series can be also written in terms of the Clifford-Appell polynomials. To state this result we need to recall the definition of the Pochhammer symbol:
\begin{equation}
    \label{poch}
    (a)_n=\frac{\Gamma(a+n)}{\Gamma(a)}, \qquad a \in \mathbb{R}, \,  n \in \mathbb{N}.
\end{equation}
\begin{theorem}
    \label{Cliffordapp}
    Let $f$ be a slice hyperholomorphic function in a neighbourhood of $p=p_0+Ip_1$ that admits a spherical series as in \eqref{spherical} with coefficients are $\{a_n\}_{n \in \mathbb{N}_0} \subseteq \mathbb{H}$, convergent (absolutely and uniformly on the compact subsets) in $U(p,R)$. Then the regular Fueter spherical series of $\breve f=\Delta f$, for $q\in U(p,R)$, can be written as
    \begin{eqnarray}
        \label{Taylormono}
        \breve f(q)=\Delta f(q)&=&-2 \left[ \sum_{n=1}^{\infty} \sum_{k=0}^{n-1}(2k+1)_2 \binom{n}{k+1} \mathcal{Q}_{2k}(q-p_0) p_{1}^{2(n-k-1)} a_{2n} \right. \\
        \nonumber
        && + \left.  \left( \sum_{n=1}^{\infty}\sum_{k=0}^{n-1} \binom{n}{k+1} (2k+2)_2  \mathcal{Q}_{2k+1}(q-p_0)p_1^{2(n-k-1)} \right. \right.\\
        \nonumber
        && \left. \left. -\sum_{n=1}^{\infty}\sum_{k=0}^{n-1} \binom{n}{k+1} (2k+1)_2 \mathcal{Q}_{2k}(q-p_0) p_1^{2(n-k)-1}I \right) a_{2n+1} \right],
    \end{eqnarray}
where $p=p_0+Ip_1$, with $I \in \mathbb{S}$, $p_0$ $p_1 \in \mathbb{R}$.
\end{theorem}
\begin{proof}
By the hypothesis on the function $f$, we can write
\begingroup\allowdisplaybreaks
\begin{equation}
	\label{NEW}
	f(q)= \sum_{n=0}^{\infty} \big[(q-p_0)^2+p_1^2\big]^n a_{2n}
	+ \sum_{n=0}^{\infty}\big(\big[(q-p_0)^2+p_1^2\big]^n (q-p)\big) a_{2n+1},
\end{equation}
\endgroup
where the series is convergent.
We then apply the second Fueter map to the first summation of the expansion of the function $f$ and using the binomial theorem and \eqref{appLapla} we get
    \begingroup\allowdisplaybreaks
    \begin{eqnarray}
        \nonumber
        \Delta [(q-p_0)^{2}+p_1^2]^n&=&\sum_{k=0}^{n} \binom{n}{k} \Delta((q-p_0)^{2k}) p_1^{2(n-k)}\\
        \nonumber
        &=&-2 \sum_{k=1}^{n} \binom{n}{k} (2k)(2k-1) \mathcal{Q}_{2(k-1)}(q-p_0) p_1^{2(n-k)}\\
        \label{f8}
        &=& -2 \sum_{k=0}^{n-1} (2k+1)_2\binom{n}{k+1} \mathcal{Q}_{2k}(q-p_0) p_1^{2(n-k-1)}.
    \end{eqnarray}
    \endgroup
    Applying the same procedure to the second sum in the expansion \eqref{NEW} we obtain
    \begingroup\allowdisplaybreaks
    \begin{eqnarray}
        \nonumber
        \Delta \left((q-p_0)^2+p_1^2)^n(q-p)\right)
        &=&\Delta \left( \sum_{k=0}^{n} \binom{n}{k} (q-p_0)^{2k+1}p_1^{2(n-k)}\right)\\
        \nonumber
        &&- \Delta \left(\sum_{k=0}^{n} \binom{n}{k} (q-p_0)^{2k}p_1^{2(n-k)+1}\right)I\\
        \nonumber
        &=& -2\sum_{k=1}^{n} \binom{n}{k} (2k)(2k+1) \mathcal{Q}_{2k-1}(q-p_0)p_1^{2(n-k)}\\
        \nonumber
        &&+2\sum_{k=1}^{n} \binom{n}{k} (2k)(2k-1) \mathcal{Q}_{2(k-1)}(q-p_0) p_1^{2(n-k)+1}I\\
        \nonumber
        &=&  -2\sum_{k=0}^{n-1} \binom{n}{k+1} (2k+2)_2 \mathcal{Q}_{2k+1}(q-p_0)p_1^{2(n-k-1)}\\
        \label{f9}
        &&+2\sum_{k=0}^{n-1} \binom{n}{k+1} (2k+1)_2 \mathcal{Q}_{2k}(q-p_0) p_1^{2(n-k)-1}I.
    \end{eqnarray}
    \endgroup
    By putting together \eqref{f8} and \eqref{f9} we get the assertion.
\end{proof}

\begin{remark}
	\label{zeroFueter}
If we set $p=0$ in \eqref{Taylormono}, the nonzero terms correspond to $k=n-1$, and thus we have
      \begingroup\allowdisplaybreaks
	\begin{eqnarray*}
		\breve f(q)=\Delta f(q)&=&-2 \sum_{n=1}^{\infty} (2n-1)_2 \mathcal{Q}_{2(n-1)}(q) a_{2n}-2 \sum_{n=1}^{\infty} (2n)_2 \mathcal{Q}_{2n-1}(q) a_{2n+1}\\
		&=& -2 \sum_{n=0}^{\infty} (2n+2) (2n+1) \mathcal{Q}_{2n}(q) a_{2n+2}-2 \sum_{n=0}^{\infty} (2n+3)(2n+2) \mathcal{Q}_{2n+1}(q)a_{2n+3}\\
		&=& -2 \sum_{n=0}^{\infty} (n+2)(n+1) \mathcal{Q}_n(q) a_{n+2}.
	\end{eqnarray*}
   \endgroup
	Thus, we get back to the Taylor expansion of an axially Fueter regular function in a neighbourhood of the origin, see formula \eqref{closedmono}. We observe that the same result was obtained in Remark \ref{remmono} by taking $p=0$ in the axially Fueter regular series.
\end{remark}	

When $q$ is not real, the regular Fueter spherical series rewrites in a special form that is
described in the next result.

\begin{proposition}
	\label{closedreal}
	Let $f$ be a slice hyperholomorphic function in a neighbourhood of $p\in\mathbb{H}$ that admits a spherical series of the form \eqref{spherical} with coefficients $\{a_n\}_{n \in \mathbb{N}_0} \subseteq \mathbb{H}$,  convergent (absolutely and uniformly on the compact subsets) in $U(p,R)$. Then, for $q \in U(p,R)\setminus\mathbb{R}$, the regular Fueter spherical series of $\breve f=\Delta f$, where it is convergent, can be written as
	\begingroup\allowdisplaybreaks
	\begin{eqnarray*}
		\breve f(q)=\Delta f(q)&=& -2 (\underline{q})^{-1} \partial_{q_0}\left(\sum_{n=0}^{\infty} Q_p^n(q)a_{2n}+ Q_p^n(q)(q-p)a_{2n+1} \right)\\
		&&+(\underline{q})^{-2}\left(\sum_{n=0}^{\infty} Q_p^n(q)a_{2n}+ Q_p^n(q)(q-p)a_{2n+1} \right)\\
		&& -(\underline{q})^{-2}\left(\sum_{n=0}^{\infty} Q_p^n(\bar{q})a_{2n}+ Q_p^n(\bar{q})(\bar{q}-p)a_{2n+1} \right).
	\end{eqnarray*}
	\endgroup
\end{proposition}
\begin{proof}
	The function $f$ admits a spherical series and so
	\begin{equation}
		\label{MM}
		f(q)= \sum_{n=0}^{\infty}	Q_p^n(q) a_{2n}+ \sum_{n=0}^{\infty}Q_p^n(q) (q-p) a_{2n+1},
	\end{equation}
	for $q$ in the convergence set, the convergence being absolute and uniform on the compact subsets of $U(p,R)$, so that $\Delta$ can be applied term by term.
	Applying the second Fueter map to the first summation in \eqref{MM} and using Proposition \ref{closedLapla2} we have
	\begingroup\allowdisplaybreaks
	\begin{eqnarray}
		\nonumber
		\Delta [(q-p_0)^{2}+p_1^2]^n&=&\sum_{k=0}^{n} \binom{n}{k} \Delta((q-p_0)^{2k}) p_1^{2(n-k)}\\
		\nonumber
		&=& -2 (\underline{q})^{-1} \partial_{q_0}\left(\sum_{k=0}^{n} \binom{n}{k} (q-p_0)^{2k}p_1^{2(n-k)}\right)\\
		\nonumber
		&&-(\underline{q})^{-2}\left(\sum_{k=0}^{n} \binom{n}{k} \left((\bar{q}-p_0)^{2k}-(q-p_0)^{2k}\right)p_1^{2(n-k)}\right)\\
		\label{auxcl0}
		&=& -2 (\underline{q})^{-1} \partial_{q_0}\left( Q_p^n(q) \right)-(\underline{q})^{-2} \left(Q_p^n(\bar{q})-Q_p^n(q) \right).
	\end{eqnarray}
	\endgroup
	Now, we focus on the second summation of \eqref{MM}. Using again Proposition \ref{closedLapla2}, we obtain
	\begin{eqnarray}
		\nonumber
		\Delta \left((q-p_0)^2+p_1^2)^n(q-p)\right)
		&=& \sum_{k=0}^{n} \binom{n}{k} \Delta \left((q-p_0)^{2k+1} \right)p_1^{2(n-k)}\\
		\nonumber
		&&-\sum_{k=0}^{n} \binom{n}{k} \Delta\left((q-p_0)^{2k}\right)p_1^{2(n-k)+1}I\\
		\nonumber
		&=&-2(\underline{q})^{-1}\partial_{q_0} \left(Q_p^n(q)(q-p_0)\right)-(\underline{q})^{-2} \left(Q_p^n(\bar{q})(\bar{q}-p_0) \right)\\
		\nonumber
		&&+(\underline{q})^{-2}\left(Q_p^n(q)(q-p_0) \right)+2(\underline{q})^{-1}\partial_{q_0} \left(Q_p^n(q)\right)p_1I\\
		\nonumber
		&& +(\underline{q})^{-2} \left(Q_p^n(\bar{q}) \right)p_1I-(\underline{q})^{-2}\left(Q_p^n(q)\right)p_1I\\
		\nonumber
		&=& -2(\underline{q})^{-1}\partial_{q_0} \left(Q_p^n(q)(q-p)\right)-(\underline{q})^{-2} \left(Q_p^n(\bar{q})(\bar{q}-p) \right)\\
		\label{auxcl1}
		&&+(\underline{q})^{-2}\left(Q_p^n(q)(q-p) \right).
	\end{eqnarray}
	Finally, by putting together \eqref{auxcl0} and \eqref{auxcl1} we get the result.
	
\end{proof}

A relation between the truncated regular series and the truncated regular Fueter spherical series can be established using the right global operator.

\begin{theorem}
   \label{connb}

Let $p \in \mathbb{H}\setminus\mathbb{R}$ and let $\{a_n\}_{n \in \mathbb{N}_0} \subseteq \mathbb{H}$ be a given sequence. Define another sequence $\{b_n\}_{n \in \mathbb{N}_0} \subseteq \mathbb{H}$ by
\begin{equation}
	\label{rel1}
	b_{n-2}=-2n(n-1) \left(c_n a_{2n}+c_{n-1}a_{2n-1} \right), \qquad n \geq 2,
\end{equation}
where $c_n:=2^n n! (-1)^n$. Then the following relation between the truncated regular and the truncated Fueter regular spherical series holds:
    \begin{eqnarray*}
        \sum_{n=0}^{N} \widetilde{Q}_n(q,p) b_n &=&   -16 \sum_{n=0}^{N+2}V_p^n[(q_0-p_0)^2 \mathcal{C}_n(q,p)]a_{2n}-4 \sum_{n=1}^{N+2}n V_p^n[Q_p^{n-1}(q)] a_{2n}\\
        && -16 \sum_{n=1}^{N+1} V_p^n[(q_0-p_0)^2 \mathcal{C}_n(q,p)(q-p)]a_{2n+1}-8 \sum_{n=1}^{N+1}V_p^n[(q_0-p_0) \mathcal{L}_n(q,p)]a_{2n+1}\\
        &&-4\sum_{n=1}^{N+1} n V_p^n[Q_p^{n-1}(q)(q-p)]a_{2n+1}.
    \end{eqnarray*}
 where $N \in \mathbb{N}_0$ is fixed.
\end{theorem}
\begin{proof}
Let us consider the $*$-series expansion of $f$ and let us compute $\Delta f$.
 By using formula \eqref{tildeq} and changing index in the sum we get
$$
        \sum_{n=0}^{N} \widetilde{Q}_n(q,p) b_n = - \sum_{n=0}^{N}  \Delta (q-p)^{(n+2)*_{p,R}}\frac{b_n}{2(n+1)(n+2)}= - \sum_{n=2}^{N+2} \Delta (q-p)^{n*_{p,R}} \frac{b_{n-2}}{2n(n-1)}.
$$
    Using the assumption of the coefficients given in \eqref{rel1} and Corollary \ref{appn} we get
    \begin{eqnarray*}
        \sum_{n=0}^{N} \widetilde{Q}_n(q,p) b_n &=&  \sum_{n=2}^{N+2} \Delta (q-p)^{n*_{p,R}} (c_n a_{2n}+c_{n-1}a_{2n-1})\\
        &=& \sum_{n=2}^{N+2} \Delta\left( c_n(q-p)^{n*_{p,R}}\right)  a_{2n}+ \sum_{n=1}^{N+1} \Delta \left( c_n(q-p)^{(n+1)*_{p,R}}\right)a_{2n+1}\\
        &=& \sum_{n=0}^{N+2} c_n \Delta (q-p)^{n*_{p,R}}a_{2n}+ \sum_{n=1}^{N+1} c_n \Delta (q-p)^{(n+1) *_{p,R}} a_{2n+1}\\
        &=& \sum_{n=0}^{N+2}\Delta V_p^n \left( Q_p^n(q) \right)a_{2n}+ \sum_{n=1}^{N+1} \Delta V_p^n \left[Q_p^n(q)(q-p) \right]a_{2n+1}.
    \end{eqnarray*}
    Finally since the Laplacian is a real operator, by formulas \eqref{delta} and \eqref{delta1} we obtain
    \begin{eqnarray*}
        \sum_{n=0}^{N} \widetilde{Q}_n(q,p) b_n&=&\sum_{n=0}^{N+2} V_p^n \Delta \left( Q_p^n(q) \right)a_{2n}+ \sum_{n=1}^{N+1}  V_p^n \Delta \left[Q_p^n(q)(q-p) \right]a_{2n+1}\\
        &=& -16 \sum_{n=0}^{N+2}V_p^n[(q_0-p_0)^2 \mathcal{C}_n(q,p)]a_{2n}-4 \sum_{n=1}^{N+2}n V_p^n[Q_p^{n-1}(q)]a_{2n}\\
        && -16 \sum_{n=1}^{N+1} V_p^n[(q_0-p_0)^2 \mathcal{C}_n(q,p)(q-p)]a_{2n+1}-8 \sum_{n=1}^{N+1}V_p^n[(q_0-p_0) \mathcal{L}_n(q,p)]a_{2n+1}\\
        &&-4\sum_{n=1}^{N+1} nV_p^n[Q_p^{n-1}(q)(q-p)]a_{2n+1}.
    \end{eqnarray*}
    By making a change of the index in the above summations we get the final result.
\end{proof}

\section{Laurent series expansion: axially Fueter regular functions}\label{Sect:8}

It is well known that Fueter regular functions admit a Laurent expansion, see \cite{red, green}. More precisely, a function $f$ which is Fueter regular in the annular domain $A=A(0,R_1,R_2)$, see \eqref{shell}, admits in $A$ the Laurent expansion
$$ f(q)= \sum_{n=0}^{\infty} \sum_{\nu \in \sigma_n} P_{\nu}(q)a_n+ \sum_{n=1}^{\infty}\sum_{\nu \in \sigma_n} (-1)^{n} \partial_{q_1}^{n_1}\partial_{q_2}^{n_2}\partial_{q_3}^{n_3} E(q)a_{-n},$$
where $E(q)$ is the Cauchy-Fueter kernel, $P_\nu (q)$ are the Fueter polynomials defined in \eqref{Fueterpoly} and $\sigma_n$ be the set of all triples $ \nu=[n_1, n_2, n_3]$ of non-negative integers such that $n_1+n_2+n_3=n$.\\
In this section we study Laurent expansions at a generic point $p$, both in the case of the $*$-powers and of  the spherical expansions. Recalling that any axially Fueter regular function admits a Fueter primitive $f$ which is slice hyperholomorphic, this can be achieved by exploiting the corresponding expansions of $f$.

\subsection{Laurent Fueter regular series}
We first consider the Laurent expansion of an axially Fueter regular function expanded at a generic quaternion $p$ in terms of the $*$-product. To this end, we give the following definition.
\begin{definition}
Let $\Omega \subseteq\mathbb{H}$ be an axially symmetric open set. We say that $\breve f=\Delta f$ has a Laurent Fueter regular series at a point $p \in \mathbb H$ if $f$ is a slice hyperholomorphic function in $\Omega$ that admits a $*$-Laurent series at $p$ convergent in a set contained in $\Omega$.
\end{definition}
It is then crucial to understand how the second Fueter operator acts on negative $*$-powers, and this is done in the next result. To ease the notation we shall write $*$ instead of $ *_{p,R}= *_{q,L}$.
\begin{theorem}
	\label{appLauN2}
Let $p, q \in \mathbb{H}$ be such that $q \notin [p]$, and let $\Delta$ denote the Laplace operator with respect to the variable $q$. Then for $n \geq 1$ we have
	\begin{equation}
		\label{appLauN1}
		\Delta (q-p)^{-*n}=-4 \left(\sum_{j=0}^{n}(n-j) (\bar{q}-p)^{*(n-j)}*(q-p)^{*j} \right) \mathcal{Q}_{c,p}^{-n-1}(q).
	\end{equation}
\end{theorem}
\begin{proof}
 By \eqref{FK} we have
	$$ \Delta (q-p)^{-*}=-\Delta S^{-1}_L(p,q)=4 (p- \bar{q}) \mathcal{Q}_{c,p}^{-2}(q),$$
	while, on the right hand side of the equality \eqref{appLauN1}, we have
	$$
	-4 \left(\sum_{j=0}^{1}(1-j) (\bar{q}-p)^{*(1-j)}*(q-p)^{*j} \right) \mathcal{Q}_{c,p}^{-2}(q)=4 (p- \bar{q})  \mathcal{Q}_{c,p}^{-2}(q),
	$$
and this proves the assertion for $n=1$.
	We suppose that formula \eqref{appLauN1} is true for $n$, and we prove it for $n+1$. By \eqref{Lau4}, \eqref{Leif} and the inductive hypothesis we have
	\begingroup\allowdisplaybreaks
	\begin{eqnarray*}
		\Delta (q-p)^{-*(n+1)}&=& - \frac{1}{n} \partial_{q_0}\left[ \frac{(-1)^{n}}{(n-1)!}  \partial_{q_0}^{n-1} \Delta S^{-1}_L(p,q)\right]\\
		&=&- \frac{1}{n} \partial_{q_0} \Delta (q-p)^{-*n}\\
		&=&  \frac{4}{n} \partial_{q_0}\left[ \left( \sum_{j=0}^{n} (n-j) (\bar{q}-p)^{*(n-j)}* (q-p)^{*j} \right) \mathcal{Q}_{c,p}^{-n-1}(q)\right]\\
		&=& \frac{4}{n} \left[ \left( \sum_{j=0}^{n} (n-j)^2 (\bar{q}-p)^{*(n-j-1)}* (q-p)^{*j} \right) \mathcal{Q}_{c,p}(q)\right.\\
		&& \left. + \left(\sum_{j=1}^{n} (n-j)j (\bar{q}-p)^{*(n-j)}*(q-p)^{*(j-1)} \right) \mathcal{Q}_{c,p}(q) \right.\\
		&& \left. -(n+1) \left( \sum_{j=0}^{n} (n-j) (\bar{q}-p)^{*(n-j)}* (q-p)^{*j} \right)(2q_0-2p)  \right]\mathcal{Q}_{c,p}^{-n-2}(q).
	\end{eqnarray*}
	\endgroup
Since $\mathcal{Q}_{c,p}(q)=(q-p)*(\bar{q}-p)$ and $ 2q_0-2p=(q-p)+(\bar{q}-p)$, and since we are considering the left $*$-product with respect to $q$, we obtain
	\begingroup\allowdisplaybreaks
	\begin{eqnarray*}
		\Delta (q-p)^{-*(n+1)}&=&\frac{4}{n} \left[ \left( \sum_{j=0}^{n} (n-j)^2 (\bar{q}-p)^{*(n-j)}* (q-p)^{*(j+1)} \right) \right.\\
		&& \left. + \left(\sum_{j=0}^{n-1} (n-1-j)(j+1) (\bar{q}-p)^{*(n-j)}*(q-p)^{*(j+1)} \right)  \right.\\
		&& \left. -(n+1) \left( \sum_{j=0}^{n} (n-j) (\bar{q}-p)^{*(n+1-j)}* (q-p)^{*j} \right) \right.\\
		&& \left. -(n+1) \left( \sum_{j=0}^{n} (n-j) (\bar{q}-p)^{*(n-j)}* (q-p)^{*(j+1)} \right) \right]\mathcal{Q}_{c,p}^{-n-2}(q)\\
		&=& \frac{4}{n} \left( n^2\sum_{j=0}^{n-1}  (\bar{q}-p)^{*(n-j)}* (q-p)^{*(j+1)} -n\sum_{j=0}^{n-1} j (\bar{q}-p)^{*(n-j)}* (q-p)^{*(j+1)} \right.\\
		&& \left. +\sum_{j=0}^{n-1} (n-j) (\bar{q}-p)^{*(n-j)}* (q-p)^{*(j+1)}-\sum_{j=0}^{n-1} (j+1) (\bar{q}-p)^{*(n-j)}* (q-p)^{*(j+1)} \right.\\
		&& \left. -\sum_{j=0}^{n} n(n-j) (\bar{q}-p)^{*(n-j)}* (q-p)^{*(j+1)}-\sum_{j=0}^{n} (n-j) (\bar{q}-p)^{*(n-j)}* (q-p)^{*(j+1)} \right.\\
		&& \left.- \sum_{j=0}^{n} (n-j) (\bar{q}-p)^{*(n+1-j)}*(q-p)^{*j}- \sum_{j=0}^{n} n(n-j)(\bar{q}-p)^{*(n+1-j)}*(q-p)^{*j}  \right)\\
		&&\mathcal{Q}_{c,p}^{-n-2}(q)\\
		&=& \frac{4}{n} \left(- \sum_{j=0}^{n-1}(j+1) (\bar{q}-p)^{*(n-j)}*(q-p)^{*(j+1)}+ \sum_{j=1}^{n} j (\bar{q}-p)^{*(n+1-j)}*(q-p)^{*j} \right.\\
		&& \left.-n \sum_{j=0}^{n} (\bar{q}-p)^{*(n+1-j)}*(q-p)^{*j}- \sum_{j=0}^{n}n(n-j)(\bar{q}-p)^{*(n+1-j)}*(q-p)^{*j} \right)\\
		&&\mathcal{Q}_{c,p}^{-n-2}(q)\\
		&=&-4 \left( \sum_{j=0}^{n+1} (n+1-j) (\bar{q}-p)^{*(n+1-j)}*(q-p)^{*j} \right)\mathcal{Q}_{c,p}^{-n-2}(q).
	\end{eqnarray*}
	\endgroup
	This proves the result.
\end{proof}
Using Theorem \ref{appLauN2} we can write
\begin{equation}
	\label{appLaplaN}
	\Delta (q-p)^{-*n}= -4\mathcal{M}_n(q,p) \mathcal{Q}_{c,p}^{-n-1}(q),
\end{equation}
where
\begin{equation}\label{calMn}
\mathcal{M}_n(q,p)=\sum_{j=0}^{n}(n-j) (\bar{q}-p)^{*(n-j)}*(q-p)^{*j} ,
\end{equation}
and with this notation we  get an interesting expression for the powers of the conjugate Fueter operator $\bar{D}^n$ applied to the $F$-kernel, see \eqref{FK}.
\begin{corollary}
	Let $n  \in \mathbb{N}$, and $q$, $p \in \mathbb{H}$ such that $q \notin [p]$. Then
	\begin{equation}
		\label{derivativeF}
		\bar{D}^{n} F_L(p,q)= 2^{n+2}n! (-1)^{n} \mathcal{M}_{n+1}(q,p) \mathcal{Q}_{c,p}^{-n-2}(q).
	\end{equation}
\end{corollary}
\begin{proof}
	By \eqref{starL}, Remark \ref{NstarLL} and \eqref{FK} we have
	$$ \Delta (q-p)^{-*n}= \frac{(-1)^{n}}{(n-1)!} \partial_{q_0}^{n-1} F_L(p,q).$$
	Since the kernel $F_L(p,q)$ is axially Fueter regular in $q$ we can replace $\partial_{q_0}^{n-1}$ with $\frac{\bar{D}^{n-1}}{2^{n-1}}$, and by \eqref{appLaplaN} we get the result.
\end{proof}
\begin{remark}
	If we consider $p=0$ in \eqref{derivativeF} we get a nicer formula. Indeed, since $F_L(0,q)=4E(q)$ and $ \mathcal{M}_n(q,0)=S_n(q)$ we have
	$$ \bar{D}^n E(q)=2^n n! (-1)^n S_{n+1}(q)|q|^{-2(n+2)}, \qquad q \neq 0,$$
	where
	\begin{equation}
		\label{lpoly}
		S_n(q)=\mathcal{M}_n(q,0)= \sum_{j=0}^{n} (n-j) \bar{q}^{n-j} q^j, \qquad n \geq 0.
	\end{equation}
\end{remark}

In the next result we list the main properties of the functions $ \mathcal{M}_n(q,p)$.
\begin{proposition} Let $n \in \mathbb{N}$ and $p$, $q \in \mathbb{H}$ such that $ q \notin [p]$. Then:
	\begin{itemize}
		\item[1)] $ \mathcal{M}_n(q,p) \mathcal{Q}_{c,p}^{-n-1}(q)$ is an axially Fueter regular function in $q$.
		\item[2)] $ \mathcal{M}_n(q,p)$ is a left slice polyanalytic function of order $n+1$ in $q$ and is right slice hyperholomorphic in $p$.
		\item[3)] If $q \in \mathbb{H} \setminus \mathbb{R}$, $ \mathcal{M}_n(q,p)$ rewrites as
		\begin{equation}
			\label{closedFN}
			\mathcal{M}_n(q,p)= \frac{(\underline{q})^{-1}}{4} \left[-2n (\bar{q}-p)^{*(n+1)}+ (\underline{q})^{-1}  \left( (q-p)^{*n}-(\bar{q}-p)^{*n} \right)\mathcal{Q}_{c,p}(q)\right]
		\end{equation}
	\end{itemize}
\end{proposition}
\begin{proof}
We prove the various statements listed above.
	\begin{itemize}
		\item[1)] By \eqref{appLaplaN}, the fact that the function $(q-p)^{-*n}$ is slice hyperholomorphic for $ q \notin [p]$ and the Fueter theorem  we get that the function $ \mathcal{M}_n(q,p) \mathcal{Q}_{c,p}^{-n-1}$ is axially Fueter regular in $q$.
		\item[2)] Arguments similar to those used in Proposition \ref{harmpoly} prove that the $ \mathcal{M}_n(q,p)$ are left slice polyanalytic functions of order $n+1$ in $q$. The fact that the $ \mathcal{M}_n(q,p)$ are right slice hyperholomorphic in $p$ follows from Proposition \ref{closedp}.
		\item[3)] We prove the result by induction on $n$. If $n=1$ we have
		$$ \mathcal{M}_1(q,p)=\bar{q}-p.$$
		On the right-hand side of formula \eqref{closedFN} we have
		\begin{eqnarray*}
		\frac{(\underline{q})^{-1}}{4}\left[-2 (\bar{q}-p)^{*2}+(\underline{q})^{-1} (q-\bar{q})(p^2-2q_0p+|q|^2) \right]&=&\frac{(\underline{q})^{-1}}{2}\left(2 \bar{q}p-\bar{q}^2-2q_0p+|q|^2 \right)\\
		&=&\frac{(\underline{q})^{-1}}{2}\left(-2 \underline{q}p+2| \underline{q}|^2+2 q_0 \underline{q} \right)\\
		&=&\bar{q}-p,
		\end{eqnarray*}
and this proves the formula for $n=1$. We suppose that formula \eqref{closedFN} is valid for $n$ and we prove it for $n+1$. By the definition of the left $*$-product in $q$ we have
			\begingroup\allowdisplaybreaks
		\begin{eqnarray*}
			\mathcal{M}_{n+1}(q,p)&=& \sum_{j=0}^{n+1} (n+1-j) (\bar{q}-p)^{*(n+1-j)}*(q-p)^{*j}\\
			&=& \sum_{j=0}^{n} (\bar{q}-p)^{*(n+1-j)}* (q-p)^{*j}+ \sum_{j=0}^{n} (n-j) (\bar{q}-p)^{*(n+1-j)}* (q-p)^{*j}\\
			&=& (\bar{q}-p)^{*(n+1)}+(q-p)*(\bar{q}-p)* \sum_{j=1}^{n} (\bar{q}-p)^{*(n-j)}* (q-p)^{*(j-1)}\\
			&&+ \sum_{j=0}^{n} (n-j) (\bar{q}-p)^{*(n+1-j)}* (q-p)^{*j}\\
			&=& (\bar{q}-p)^{*(n+1)}+ \mathcal{Q}_{c,p}(q) \mathcal{H}_n(q,p)+ \bar{q} \mathcal{M}_n(q,p)- \mathcal{M}_n(q,p)p,
		\end{eqnarray*}
		\endgroup
where the functions $ \mathcal{H}_n(q,p)$ are defined in \eqref{harmN}.
		By \eqref{closedN1} and the inductive hypothesis we have
			\begingroup\allowdisplaybreaks
		\begin{eqnarray}
			\nonumber
			\mathcal{M}_{n+1}(q,p)&=& (\bar{q}-p)^{*(n+1)}+  \frac{(\underline{q})^{-1}}{2} \left[  (q-p)^{*n}-(\bar{q}-p)^{*n} \right] \mathcal{Q}_{c,p}(q)\\
			\nonumber
			&& + \frac{\bar{q}(\underline{q})^{-1}}{4} \left[-2n (\bar{q}-p)^{*(n+1)}+ (\underline{q})^{-1}  \left( (q-p)^{*n}-(\bar{q}-p)^{*n} \right)\mathcal{Q}_{c,p}(q)\right]\\
			\nonumber
			&&-\frac{(\underline{q})^{-1}}{4} \left[-2n (\bar{q}-p)^{*(n+1)}+ (\underline{q})^{-1}  \left( (q-p)^{*n}-(\bar{q}-p)^{*n} \right)\mathcal{Q}_{c,p}(q)\right]p\\
			\nonumber
			&=&\frac{(\underline{q})^{-1}}{4} \left[ 4 \underline{q} (\bar{q}-p)^{*(n+1)}+2  \underline{q} (\underline{q})^{-1}\left( (q-p)^{*n}-(\bar{q}-p)^{*n} \right)\mathcal{Q}_{c,p}(q) \right.\\
			\nonumber
			&& \left. -2n \bar{q}(\bar{q}-p)^{*(n+1)}+ \bar{q}(\underline{q})^{-1}  \left( (q-p)^{*n}-(\bar{q}-p)^{*n} \right) \mathcal{Q}_{c,p}(q) \right.\\
			\nonumber
			&& \left. +2n (\bar{q}-p)^{*(n+1)}p- (\underline{q})^{-1} \left( (q-p)^{*n}-(\bar{q}-p)^{*n} \right)\mathcal{Q}_{c,p}(q) p \right]\\
			\nonumber
			&=&\frac{(\underline{q})^{-1}}{4} \left[ 2(q- \bar{q}) (\bar{q}-p)^{*(n+1)}+  (q- \bar{q}) (\underline{q})^{-1}\left( (q-p)^{*n}-(\bar{q}-p)^{*n} \right) \mathcal{Q}_{c,p}(q)\right.\\
			\nonumber
			&& \left. -2n (\bar{q}-p)^{*(n+2)}+ \bar{q}(\underline{q})^{-1}  \left( (q-p)^{*n}-(\bar{q}-p)^{*n} \right)\mathcal{Q}_{c,p}(q) \right.\\
			\nonumber
			&& \left.- (\underline{q})^{-1}  \left( (q-p)^{*n}-(\bar{q}-p)^{*n} \right)\mathcal{Q}_{c,p}(q)p \right]\\
			\nonumber
			&=&\frac{(\underline{q})^{-1}}{4} \left[ -2(n+1) (\bar{q}-p)^{*(n+2)} + q(\underline{q})^{-1}  \left( (q-p)^{*n}-(\bar{q}-p)^{*n} \right)\mathcal{Q}_{c,p}(q) \right.\\
			\nonumber
			&& \left.- (\underline{q})^{-1}  \left( (q-p)^{*n}-(\bar{q}-p)^{*n} \right)\mathcal{Q}_{c,p}(q)p - 2(\bar{q}-p)^{*(n+1)}p+2 q(\bar{q}-p)^{*(n+1)} \right]\\
			\nonumber
			&=&\frac{(\underline{q})^{-1}}{4} \left[ -2(n+1) (\bar{q}-p)^{*(n+2)}+ (\underline{q})^{-1}(q-p)^{*(n+1)}\mathcal{Q}_{c,p}(q)+\right.\\
			\nonumber
			&& \left. -(\underline{q})^{-1}q (\bar{q}-p)^{*n}\mathcal{Q}_{c,p}(q)+ (\underline{q})^{-1} (\bar{q}-p)^{*n}\mathcal{Q}_{c,p}(q) p- 2(\bar{q}-p)^{*(n+1)}p+2 q(\bar{q}-p)^{*(n+1)} \right]\\
			\nonumber
			&=&\frac{(\underline{q})^{-1}}{4} \left[ -2(n+1) (\bar{q}-p)^{*(n+2)}+ (\underline{q})^{-1}(q-p)^{*(n+1)}\mathcal{Q}_{c,p}(q)\right.\\
			\label{FN1}
			&& \left. - (\underline{q})^{-1}(\bar{q}-p)^{*n}*(q-p)\mathcal{Q}_{c,p}(q) +2(\bar{q}-p)^{*(n+1)}*(q-p) \right].
		\end{eqnarray}
		\endgroup
	Recalling that $\mathcal{Q}_{c,p}(q)=(q-p)*(\bar{q}-p)$ we deduce
		\begin{eqnarray}
			\nonumber
			&&- (\underline{q})^{-1}(\bar{q}-p)^{*n}*(q-p)\mathcal{Q}_{c,p}(q)+2 (\bar{q}-p)^{*(n+1)}*(q-p)\\
			\nonumber
			&=&- (\underline{q})^{-1}(\bar{q}-p)^{*(n+1)}*(q-p)^{*2}+2 (\bar{q}-p)^{*(n+1)}*(q-p)\\
			\nonumber
			&=& -(\underline{q})^{-1}(\bar{q}-p)^{*(n+1)}*(q-p)*  \left(q-p-2 \underline{q} \right)\\
			\nonumber
			&=& -(\underline{q})^{-1}(\bar{q}-p)^{*(n+1)}*(q-p)* (\bar{q}-p)\\
			\label{FN2}
			&=&-(\underline{q})^{-1} (\bar{q}-p)^{*(n+1)}\mathcal{Q}_{c,p}(q).
		\end{eqnarray}
		Finally, by plugging \eqref{FN2} into \eqref{FN1} we get
		$$ \mathcal{M}_{n+1}(q,p)=\frac{(\underline{q})^{-1}}{4} \left[-2(n+1) (\bar{q}-p)^{*(n+2)}+ (\underline{q})^{-1} \left( (q-p)^{*(n+1)}-(\bar{q}-p)^{*(n+1)} \right)\mathcal{Q}_{c,p}(q) \right],$$
	\end{itemize}	
which ends the proof.
\end{proof}

\begin{remark}
If we set $ p=0 $ in \eqref{appLaplaN}, we obtain
\begin{equation}
	\label{appN}
	\Delta (q^{-n})=-4 |q|^{-2(n+1)}S_{n}(q),
\end{equation}
where the polynomials $ S_n(q)$ are defined in \eqref{lpoly}.
If $ q \in \mathbb{H}\setminus \mathbb{R} $, the polynomials $S_n(q) $ admit the following closed-form expression:
\begin{equation}
	\label{closedN}
	S_n(q)= \frac{(\underline{q})^{-1}}{4}\left[-2n \bar{q}^{n+1}+(\underline{q})^{-1}|q|^2 \left(q^n-\bar{q}^n \right)\right].
\end{equation}

\end{remark}

Now we show that the Laurent Fueter regular series of $f$ and $\breve f=\Delta f$ have the same set of convergence and an interesting compact expression.

\begin{theorem}
	\label{LauReg}
	Let $\Omega \subseteq\mathbb{H}$ be an axially symmetric open set and let $f$ be a slice hyperholomorphic function in $\Omega$ that admits $*$-Laurent expansion at $p \in \mathbb{H}$
	\begin{equation}
	\label{NNL2}
	f(q)=\sum_{n=0}^{\infty} (q-p)^{*n}a_n+\sum_{n=1}^{\infty} (q-p)^{-*n}a_{-n}, \qquad \{a_n\}_{n \in \mathbb{Z}} \subseteq \mathbb{H},
	\end{equation}
convergent (absolutely and uniformly on the compact subsets) in
 $ \Omega(p,R_1, R_2) \subset \Omega$,
where $ \frac{1}{R_2}= \limsup_{n \to \infty} |a_n|^{\frac{1}{n}}$ and $R_1= \limsup_{n \to \infty}  |a_{-n}|^{\frac{1}{n}}$. Then we can write the Laurent Fueter regular series of $\breve f=\Delta f$, which is convergent in $ q \in \Omega(p,R_1, R_2)$, as
	\begin{equation}
		\label{closedL1}
	\breve f(q)=	\Delta f(q)= \sum_{n=0}^{\infty} \widetilde{Q}_n(q,p) c_n+ \sum_{n=1}^{\infty} \mathcal{M}_n(q,p) \mathcal{Q}_{c,p}^{-n-1}(q) c_{-n},
	\end{equation}
	where $\{c_n\}_{n \geq 0}:=\{-2(n+1)(n+2)a_{n+2}\}_{n \geq 0}$ and $\{c_{-n}\}_{n \geq 1}:= \{-4a_{-n}\}_{n \geq 1}$.
\end{theorem}

\begin{proof}
We apply the Laplace operator to \eqref{NNL2}. By \eqref{Lau4}, \eqref{tildeq} and \eqref{appLaplaN} we have
\begin{eqnarray*}
	\Delta f(q)&=& \sum_{n=0}^{\infty} \Delta (q-p)^{*n} a_n+ \sum_{n=1}^{\infty} \Delta (q-p)^{-*n}a_{-n}\\
	&=&-2 \sum_{n=0}^{\infty} (n+2)(n+1) \widetilde{Q}_n(q,p)a_{n+2}-4 \sum_{n=1}^{\infty} \mathcal{M}_n(q,p) \mathcal{Q}_{c,p}^{-n-1}(q)a_{-n} .
\end{eqnarray*}
The result follows by setting the coefficients $ \{c_n\}_{n \in \mathbb{Z}}$ as in the statement. We now study the convergence of both the series in \eqref{closedL1}. The convergence of the first series follows by Theorem \ref{conve2}, so we focus on the second series.  Since the case $q \in \mathbb{R}$ is immediate, we consider the case $q \notin \mathbb{R}$. By formulas \eqref{appLaplaN}, \eqref{rapm} and Theorem \ref{rapp0} we have
	\begin{eqnarray*}
		-4\mathcal{M}_n(q,p) \mathcal{Q}_{c,p}^{-n-1}(q)a_{-n}&=&\Delta (q-p)^{-*n}a_{-n}\\
		&=&n \underline{q}^{-1} I_{\underline{q}} \left[(q_{-I}-p)^{-*(n+1)}-(q_I-p)^{-(n+1)}\right]a_{-n}\\
		&&-n (\underline{q})^{-1}I\left[(q_{-I}-p)^{-(n+1)}-(q_I-p)^{-(n+1)}\right]a_{-n}\\
		&& -I_{\underline{q}} (\underline{q})^{-2} I \left[(q_{-I}-p)^{-n}-(q_I-p)^{-n}\right]a_{-n},
	\end{eqnarray*}
	where $q_I=x+yI$ and $q_{-I}=x-yI$ are chosen on $\mathbb{C}_I$, with $p\in\mathbb{C}_I$. This implies
	\begin{eqnarray*}
		| 4\mathcal{M}_n(q,p) \mathcal{Q}_{c,p}^{-n-1}(q)a_{-n}| & \leq & 2n | \underline{q}|^{-1}  \left[\left|(q_{-I}-p)^{-(n+1)}a_{-n}\right|+ \left|(q_I-p)^{-(n+1)}a_{-n} \right|\right]\\
		&& +| \underline{q}|^{-2} \left[\left|(q_{-I}-p)^{-n}a_{-n}\right|+ \left|(q_I-p)^{-n}a_{-n} \right|\right].
	\end{eqnarray*}
	By the hypothesis on the coefficients $ \{a_n\}_{n<0}$ and the fact that $ q \in \Omega(p,R_1, R_2)$ we get that the series
	$$ \sum_{n=1}^{\infty} |(q_{\pm I}-p)^{-*(n+1)}a_{-n}|, \qquad \hbox{and} \quad \sum_{n=1}^{\infty} |(q_{\pm I}-p)^{-*n}a_{-n}|$$
	are both convergent and this ends the proof.

\end{proof}

As for the harmonic regular series also the Laurent Fueter regular series can also be expressed in terms of the $F$-kernel, see \eqref{FK}, and of the slice hyperholomorphic Cauchy kernel.
\begin{proposition}
\label{LauF}
Let $f$ be a slice hyperholomorphic function in an axially symmetric open set $\Omega\subseteq\mathbb H$ that admits a $*$-Laurent expansion at $p\in\mathbb H$ as in \eqref{Lau1} with coefficients $ \{a_n\}_{n \in \mathbb{Z}} \subseteq \mathbb{H}$ and assume that it is  convergent (absolutely and uniformly on the compact subsets) in $\Omega(p,R_1, R_2)\subseteq\Omega$, with $R_1$, $R_2$ as in Theorem \ref{LauReg}. Then, for $q\in \Omega(p,R_1, R_2)$ the Laurent Fueter regular series of $\breve f=\Delta f$ can be written as
	\begin{equation}
		\label{closedL1b}
		\breve f(q)=\Delta f(q)=\sum_{n=0}^{\infty}\widetilde{Q}_n(q,p) b_n+ \sum_{n=1}^{\infty} \frac{1}{(n-1)!}\bar{D}^{n-1} F_L(p,q) b_{-n},
	\end{equation}
	with $\{b_n\}_{n \geq 0}:=\{-2(n+1)(n+2) a_{n+2}\}_{n \geq 0}$ and $\{b_{-n}\}_{n \geq 1}:= \{(-1)^{n} 2^{1-n}a_{-n}\}_{n \geq 1}$.
\end{proposition}
\begin{proof}
	Bearing in mind that we can apply the Laplacian term by term, by formula \eqref{Lau4}  we can write the expansion in series of $\breve f=\Delta f$ as
	$$ \breve f(q)=\Delta f(q)= \sum_{n=0}^{\infty} \Delta (q-p)^{*n} a_n + \sum_{n=1}^{\infty} \frac{(-1)^n}{(n-1)!} \partial_{q_0}^{n-1} \Delta S^{-1}_L(p,q)a_{-n}.$$
	By \eqref{tildeq} and \eqref{FK} we get
	$$ \Delta f(q)=-2 \sum_{n=0}^{\infty} (n+2)(n+1) \widetilde{Q}_n(q,p)a_{n+2} + \sum_{n=1}^{\infty} \frac{(-1)^n}{(n-1)!} \partial_{q_0}^{n-1} F_L(p,q)a_{-n}.$$
	Since the $F$-kernel $F_L(p,q)$ is axially Fueter regular in the variable $q$ we get that $\partial_{q_0}^{n-1} F_L(p,q)= \frac{1}{2^{n-1}} \bar{D}^{n-1} F_L(p,q)$. Finally by setting the coefficients $ \{b_n\}_{n \in \mathbb{Z}}$ as in the statement we obtain \eqref{closedL1b}.
\end{proof}
\begin{remark}
If we set $p=0$ in \eqref{closedL1b}, and use the identity $F_L(0,q)=4E(q)$, we obtain the following expression for the Laurent Fueter regular series in a neighbourhood of the origin:
	\begin{equation}
		\label{L2}
		\Delta f(q)= \sum_{n=0}^{\infty} \mathcal{Q}_n(q) b_n+ \sum_{n=1}^{\infty} \frac{4}{(n-1)!} \bar{D}^{n-1} E(q)b_{-n},
	\end{equation}
We also note that we can write the Laurent Fueter regular series in terms of the pseudo-Cauchy kernel and the slice hyperholomorphic Cauchy kernel by using  \eqref{connection} and \eqref{splitS2}, respectively.
\end{remark}

In \cite{CDS1} the authors prove that the $F$-kernel has the following series expansion in terms of Clifford-Appell polynomials:
\begin{equation}
	\label{Fseries}
	F_L(p,q)=-2 \sum_{n=0}^{\infty} (n+1)(n+2) \mathcal{Q}_n(q) p^{-3-n}, \qquad |q|<|p|.
\end{equation}
This paves the way to get another equivalent expression of the Laurent Fueter regular series.
\begin{proposition}
Let $f$ be a slice hyperholomorphic function in an axially symmetric open set $\Omega$ in $\mathbb H$ that admits a $*$-Laurent expansion at an arbitrarily fixed point $p\in\mathbb H$ as in \eqref{Lau1}, with coefficients $ \{a_n\}_{n \in \mathbb{Z}} \subseteq \mathbb{H}$,  convergent (absolutely and uniformly on the compact subsets) in $\Omega(p,R_1,R_2)\subseteq\Omega$. Assume moreover that the region $|q|<|p|$ has nonempty intersection with $\Omega(p,R_1,R_2)$. Then, on that intersection, the Laurent Fueter regular series of $\breve f$, can be written as
\begin{equation}
\label{newseries}
\breve{f}(q)= \Delta f(q)= \sum_{n=0}^{\infty} \widetilde{Q}_n(q,p) b_n-2 \sum_{n=1}^{\infty} \sum_{k=n-1}^{\infty} (k+1)_2 \binom{k}{n-1} \mathcal{Q}_{k-n+1}(q) p^{-3-k} b_{-n},
\end{equation}
	where $\{b_n\}_{n \geq 0}:=\{-2(n+1)(n+2) a_{n+2}\}_{n \geq 0}$ and $\{b_{-n}\}_{n \geq 1}:= \{(-1)^{n} 2^{1-n}a_{-n}\}_{n \geq 1}$.
	\\ Moreover, for $q \notin \mathbb{R}$ we have
	$$ \Delta f(q)=-2(\underline{q})^{-1} \partial_{q_0} \left( \sum_{n \in \mathbb{Z}
	} (q-p)^{*n}a_{n}\right)-(\underline{q})^{-2} \sum_{n \in \mathbb{Z}} \left( (\bar{q}-p)^{*n}-(q-p)^{*n} \right)a_n.$$
\end{proposition}
\begin{proof}
By using repeatedly the Appell property of the polynomials $ \mathcal{Q}_{k}(q)$, see \eqref{appProp}, we get
	\begin{equation}
		\label{appPro1}
		\bar{D}^{n-1} \mathcal{Q}_k(q)= 2^{n-1}\frac{k!}{(k-n+1)!} \mathcal{Q}_{k-n+1}(q), \qquad k \geq n-1.
	\end{equation}
	Now, we assume that the intersection of the region $|q|<|p|$ with the domain of convergence of the $*$-Taylor expansion is nonempty. From Proposition \ref{LauF}, the expansion in series of the $F$-kernel, see \eqref{Fseries}, and \eqref{appPro1} we get
	\begin{eqnarray*}
		\Delta f(q)&=& \sum_{n=0}^{\infty}\widetilde{Q}_n(q,p) b_n+ \sum_{n=1}^{\infty} \frac{1}{(n-1)!}\bar{D}^{n-1} F_L(p,q) b_{-n}\\
		&=&\sum_{n=0}^{\infty} \widetilde{Q}_n(q,p) b_n-2 \sum_{n=1}^{\infty} \sum_{k=0}^{\infty} \frac{(k+1)(k+2)}{(n-1)!} \left(\bar{D}^{n-1} \mathcal{Q}_k(q)\right)p^{-3-k}b_{-n}\\
		&=&\sum_{n=0}^{\infty} \widetilde{Q}_n(q,p) b_n-2 \sum_{n=1}^{\infty} \sum_{k=0}^{\infty} 2^{n-1}\frac{(k+1)(k+2)k!}{(k-n+1)!(n-1)!}  \mathcal{Q}_{k-n+1}(q)p^{-3-k}b_{-n}\\
		&=& \sum_{n=0}^{\infty} \widetilde{Q}_n(q,p) b_n-2 \sum_{n=1}^{\infty} \sum_{k=n-1}^{\infty} (k+1)_2 \binom{k}{n-1} \mathcal{Q}_{k-n+1}(q)p^{-3-k}2^{n-1}b_{-n}.
	\end{eqnarray*}
	This proves \eqref{newseries}.
	We consider $q \notin \mathbb{R}$. First we find a closed expression of the second series in \eqref{closedL1}. By Proposition \ref{powerN} and by \eqref{closedFN} we have
	\begin{eqnarray}
		\nonumber
		-4 \sum_{n=1}^{\infty} \mathcal{M}_n(q,p) \mathcal{Q}_{c,p}^{-n-1}(q)a_{-n}&=&2 (\underline{q})^{-1} \sum_{n=1}^{\infty} n (q-p)^{-*(n+1)} a_{-n}\\
		\nonumber
		&&- (\underline{q})^{-2} \sum_{n=1}^{\infty}  \left( (q-p)^{*n}-(\bar{q}-p)^{*n} \right)\mathcal{Q}_{c,p}^{-n}(q)a_{-n}\\
		\nonumber
		&=&-2 (\underline{q})^{-1} \partial_{q_0}\left(\sum_{n=1}^{\infty}  (q-p)^{-*n} a_{-n} \right)\\
		\label{FN3}
		&&- (\underline{q})^{-2} \sum_{n=1}^{\infty} \left( (\bar{q}-p)^{-*n}-(q-p)^{-*n} \right)a_{-n}.
	\end{eqnarray}
Hence by Theorem \ref{LauReg} and Proposition \ref{regca} we have
	\begin{eqnarray*}
		\Delta f(q)&=& -2 \sum_{n=0}^{\infty} (n+1)(n+2) \widetilde{Q}_n(q,p)a_{n+2 }-4 \sum_{n=1}^{\infty} \mathcal{M}_n(q,p) \mathcal{Q}_{c,p}^{-n-1}(q)a_{-n}\\
		&=&- 2 (\underline{q})^{-1} \partial_{q_0} \left( \sum_{n=0}^{\infty} (q-p)^{*n}a_n\right)- (\underline{q})^{-2} \sum_{n=0}^{\infty} \left[(\bar{q}-p)^{*n}-(q-p)^{*n} \right]a_n\\
		&&-2 (\underline{q})^{-1} \partial_{q_0}\left(\sum_{n=1}^{\infty}  (q-p)^{-*n} a_{-n} \right)- (\underline{q})^{-2} \sum_{n=1}^{\infty} \left( (\bar{q}-p)^{-*n}-(q-p)^{-*n} \right)a_{-n}\\
		&=& -2(\underline{q})^{-1} \partial_{q_0} \left( \sum_{n \in \mathbb{Z}
		} (q-p)^{*n}a_{n}\right)-(\underline{q})^{-2} \sum_{n \in \mathbb{Z}} \left( (\bar{q}-p)^{*n}-(q-p)^{*n} \right)a_n,
	\end{eqnarray*}
as stated.

\end{proof}

\subsection{Laurent Fueter regular spherical series}

To discuss the Laurent series written in terms of the polynomials of spherical type, we introduce the following notion.

\begin{definition}
	Let $ \Omega$ be an axially symmetric open set in $\mathbb H$ and let $f$ be a slice hyperholomorphic function in $\Omega$  that admits a spherical Laurent expansion centered at $p$ and convergent in a suitable set contained in $\Omega$. Then we say that $\breve f=\Delta f$ has a Laurent Fueter regular spherical series at $p$.
\end{definition}

The above definition is justified by the fact that applying the second Fueter map to the
expression of the spherical series of a function $f$ we get a formula for the regular Fueter
spherical series.

\begin{theorem}
	\label{LauFreal}
	Let $f$ be a slice hyperholomorphic function in a neighborhood of $p \in \mathbb{H}$ containing $U(p,R_1,R_2)$, for some $R_2>R_1>0$, where it admits a spherical Laurent expansion \eqref{Lau2},  convergent (absolutely and uniformly on the compact
	subsets) in a Cassini shell $U(p,R_1,R_2)$. Let $\breve{f}$ be the axially Fueter regular function in a neighbourhood of $p$ given by $\breve{f}=\Delta f$. Then, $\breve{f}$ has the following Laurent Fueter regular spherical series expansion
	\begin{eqnarray}
\nonumber
			\breve f(q)=\Delta f(q)&=&\!\!\!\!\!-16(q_0-p_0)^2\sum_{n \in \mathbb{Z}}\mathcal{C}_n(q,p)\left[a_{2n}+(q-p)a_{2n+1}\right]\\
					\label{res2}
			&&\!\!\!\!\!-8(q_0-p_0)\sum_{n \in \mathbb{Z}}\mathcal{L}_n(q,p)\,a_{2n+1}
			-4\sum_{n \in \mathbb{Z}}n\,Q^{n-1}_p(q)\left[a_{2n}+(q-p)a_{2n+1}\right],
	\end{eqnarray}
	where $\mathcal{L}_n(q,p)$ and $\mathcal{C}_n(q,p)$ are defined in \eqref{key4} and \eqref{keyC}, respectively.
\end{theorem}
\begin{proof}
Note that a Cassini shell with $R_1>0$ does not meet $[p]$, so that all the terms are well defined. We apply the operator $\Delta$ term by term to the Laurent spherical expansion \eqref{Lau2}. The interchange of $\Delta$ and the summation will be justified later in the proof. We first consider the case $q\notin\mathbb R$. By \eqref{compactF}, formulas \eqref{delta} and \eqref{delta1} remain valid for every $n\in\mathbb Z$. Hence, substituting these identities into \eqref{Lau2}, we obtain \eqref{res2}.
We now justify the interchange of $\Delta$ and the summation. Let $K$
be a compact subset of $U(p,R_1,R_2)$. There exist $\rho_1,\rho_2$ such that $R_1<\rho_1<\rho_2<R_2$, so $\rho_1^2\leq |Q_p(q)|\leq \rho_2^2$ for $q\in K$
For $n\geq 1$, we have
$|\mathcal L_n(q,p)|\leq n\rho_2^{2n-2}$, and $|\mathcal C_n(q,p)|
\leq \frac{n(n-1)}{2}\rho_2^{2n-4}$. Whereas, for $n\leq -1$, we have $|\mathcal L_n(q,p)|
\leq |n|\rho_1^{2n-2}$ and $|\mathcal C_n(q,p)|\leq \frac{|n|(|n|+1)}{2}\rho_1^{2n-4}$.
Therefore, all the series appearing in the previous calculations converge absolutely and uniformly on compact subsets of $U(p,R_1,R_2)$.
\\It remains to consider the real points. Let
$q=q_0\in U(p,R_1,R_2)\cap\mathbb R$. Then
$$
\mathcal L_n(q_0,p)=nQ_p^{n-1}(q_0),\qquad \mathcal C_n(q_0,p)=\frac{n(n-1)}{2}Q_p^{n-2}(q_0),
\qquad n\in\mathbb Z.
$$
Moreover, formulas \eqref{real1d} and \eqref{real2d} remain valid for every $n\in\mathbb{Z}$. Therefore, at the real points, the right-hand side of \eqref{res2} coincides with the series obtained by applying $\Delta$ term by term to \eqref{Lau2}.
\end{proof}

\begin{proposition}
	Under the hypotheses of Theorem \ref{LauFreal}, let $ \{a_n\}_{n \in \mathbb{Z}} \subseteq \mathbb{H}$ be such that
	$$ R_1:= \limsup_{n \to \infty}  |a_{-n}|^{\frac{1}{n}}, \qquad \hbox{and} \qquad \frac{1}{R_2}:= \limsup_{n \to \infty}  |a_n|^{\frac{1}{n}}, \qquad R_1<R_2.$$
	The Laurent Fueter regular spherical series centered at $p$ converges (absolutely and uniformly on the compact subsets) in the Cassini shell $ U(p,R_1,R_2)$,
	where $p=p_0+Ip_1 \in \mathbb{H}$, with $I \in \mathbb{S}$.	
\end{proposition}
\begin{proof}
Let $K$ be a compact subset of the Cassini shell $U(p,R_1,R_2)$. Then there exist $r_1,r_2$ such that $R_1<r_1<r_2<R_2$ and $r_1^{2n}<|Q_p^n(q)|<r_2^{2n}$ for $q\in K$. Choose $R'$ and $R''$ such that $R_1<R'<r_1<r_2<R''<R_2$. By the assumptions on the coefficients, we have $|a_n|\leq (1/R'')^n$ and $|a_{-n}|\leq (R')^n$ for all sufficiently large $n$. Using Theorem \ref{LauFreal} we can write the Laurent Fueter regular spherical series as
	\begingroup\allowdisplaybreaks
	\begin{eqnarray*}
		&&\Delta f(q)={-16 \sum_{n=0}^{\infty}  \mathcal{C}_{n+2}(q,p)(q_0-p_0)^2 \left[ a_{2n+4}+(q-p)a_{2n+5}\right]}\\
		&&-4 \sum_{n=0}^{\infty}(n+1) Q_p^{n}(q) \left[a_{2n+2}+(q-p)a_{2n+3} \right] {-8 \sum_{n=0}^{\infty} \mathcal{L}_{n+1}(q,p) (q_0-p_0)a_{2n+3}}\\
		&&{-16 \sum_{n=1}^{\infty}  \mathcal{C}_{-n}(q,p)(q_0-p_0)^2 \left[a_{-2n}+(q-p)a_{-2n+1} \right]} \\
		&&+4 \sum_{n=1}^{\infty}n Q_p^{-n-1}(q)[a_{-2n}+(q-p)a_{-2n+1}]{-8 \sum_{n=1}^{\infty} \mathcal{L}_{-n}(q,p)(q_0-p_0)a_{-2n+1}}.
	\end{eqnarray*}
	\endgroup
	The convergence of the first three series follows by Theorem \ref{conve}, so we focus on the remaining ones.
	By \eqref{keyC}, \eqref{key4} on $K$ we have,
	$$|\mathcal{C}_{-n}(q,p)|\leq \frac{n(n+1)}{2}\,r_1^{-2n-4},\qquad |\mathcal{L}_{-n}(q,p)|\leq n\,r_1^{-2n-2},$$
	so that in the estimates below one may replace $2\mathcal{C}_{-n}(q,p)$ by $n(n+1)Q^{-n-2}_p(q)$ and $\mathcal{L}_{-n}(q,p)$ by $nQ^{-n-1}_p(q)$, as far as moduli are concerned.
	By Lemma \ref{r1} and Lemma \ref{newres} we get 	
	\begin{eqnarray*}
		&&n(n+1)\left| Q_p^{-n-2}(q)(q_0-p_0)^2 \left[a_{-2n}+(q-p)a_{-2n+1} \right]\right|\\
		&& \leq \frac{n(n+1)}{R' r_1^4} \left(\frac{R'}{r_1}\right)^{2n} (r_2^2+p_1^2) \left(R'+\left(\sqrt{r_2^2+p_1^2}+p_1\right) \right)=:  A_n,
	\end{eqnarray*}
	and
	\begin{eqnarray*}
		n |Q_p^{-n-1}(q)[a_{-2n}+(q-p)a_{-2n+1}]|\leq  \frac{n}{R'r_1^2} \left(\frac{R'}{r_1}\right)^{2n}\left(R'+\left(\sqrt{r_2^2+p_1^2}+p_1\right) \right):=B_n,
	\end{eqnarray*}
	and
	\begin{eqnarray*} | nQ_p^{-n-1}(q)(q_0-p_0)a_{-2n+1}|\leq \frac{n}{R' r_1^2}\left(\frac{R'}{r_1}\right)^{2n}\sqrt{r_2^2+p_1^2}  :=C_n.
	\end{eqnarray*}
	By the ratio test we get that the series  $ \sum_{n=1}^{\infty} (A_n+B_n+C_n)$ is convergent and this proves the result.
\end{proof}

In suitable sets, it is possible to write the Laurent Fueter regular spherical series in terms of the functions $S_n(q)$ defined in \eqref{lpoly} or the Clifford-Appell polynomials $\mathcal{Q}_n(q)$, see \eqref{capoly}, as proved below.

\begin{proposition}
	\label{LauN}
	Let $q$, $p=p_0+Ip_1\in \mathbb{H}$. We assume that $f$ is a slice hyperholomorphic function that admits a spherical Laurent expansion at $p$ as in \eqref{Lau2}, with coefficients $\{a_n\}_{n \in \mathbb{Z}} \subseteq \mathbb{H}$,  convergent (absolutely and uniformly on the compact subsets) in $U(p,R_1,R_2)$, $0<R_1<R_2$. Assume moreover that the region $|p_1|<|q-p_0|$ has nonempty intersection with $U(p,R_1,R_2)$. Then, on that intersection, the Laurent Fueter regular spherical series of $\breve f$ can be written as
	\begin{eqnarray}
		\nonumber
	\breve f(q)&=&	\Delta f(q)= A_T(q,p)-4 \sum_{n=1}^{\infty} \sum_{k=0}^{\infty} (-1)^k\binom{n+k-1}{k}  S_{2(n+k)}(q-p_0)|q-p_0|^{-4(n+k)-2}p_1^{2k}a_{-2n}\\
		\nonumber
		&& -4 \sum_{n=1}^{\infty} \sum_{k=0}^{\infty} (-1)^k \binom{n+k-1}{k}S_{2(n+k)-1}(q-p_0)|(q-p_0)|^{-4(n+k)} p_1^{2k}a_{-2n+1}\\
		\label{Lauretsp}
		&&+4\sum_{n=1}^{\infty} \sum_{k=0}^{\infty} (-1)^k  \binom{n+k-1}{k} S_{2(n+k)}(q-p_0)|q-p_0|^{-4(n+k)-2} p_1^{2k+1}Ia_{-2n+1},
	\end{eqnarray}
If instead the region $|q-p_0|<|p_1|$ has nonempty intersection with $U(p,R_1,R_2)$ then, on that intersection, another way of writing the Laurent--Fueter regular spherical series of  $\breve f$ is given by:
\begin{eqnarray}
	\nonumber
	\breve f(q)=\Delta f(q)&=& A_T(q,p)+4 \sum_{n=1}^{\infty}\sum_{k=0}^{\infty} (-1)^k \binom{n+k}{k} n(2k+1) \mathcal{Q}_{2k}(q-p_0) p_1^{-2(n+1+k)}a_{-2n}\\
	\nonumber
	&& +4\sum_{n=1}^{\infty}\sum_{k=0}^{\infty} (-1)^k \binom{n+k}{k}n (2k+3) \mathcal{Q}_{2k+1}(q-p_0) p_1^{-2(n+k+1)}a_{-2n+1}\\
			\label{Lauretsp1}
	&&-4\sum_{n=1}^{\infty}\sum_{k=0}^{\infty} (-1)^k  \binom{n+k}{k}n (2k+1) \mathcal{Q}_{2k}(q-p_0) p_1^{-2(n+k)-1}Ia_{-2n+1}
\end{eqnarray}
	where $A_T(q,p)$ is its Taylor part given by \eqref{Taylormono}.
\end{proposition}
\begin{proof}
	By hypothesis we can write $f$ as in \eqref{Las}. The Taylor part of \eqref{Lauretsp} can be written as in Theorem \ref{Cliffordapp}. Thus, we focus on showing the claim for the part with $n<0$. Let $p=p_0+Ip_1$, with $p_0$, $p_1 \in \mathbb{R}$. We start by assuming that the region $|p_1| < |q - p_0|$ has nonempty intersection with the set of convergence of the spherical Laurent series, so that we can write $Q_p^{-n}(q)$ as in \eqref{negLa}. By applying the second Fueter map to \eqref{negLa} and by using \eqref{appN} we get
			\begingroup\allowdisplaybreaks
	\begin{eqnarray}
		\nonumber
		\Delta Q_p^{-n}(q)&=& \sum_{k=0}^{\infty} (-1)^k\binom{n+k-1}{k}  \Delta (q-p_0)^{-2(n+k)}p_1^{2k}\\
		\label{Ls1a}
		&=&-4 \sum_{k=0}^{\infty} (-1)^k\binom{n+k-1}{k}  S_{2(n+k)}(q-p_0)|q-p_0|^{-4(n+k)-2}p_1^{2k}.
	\end{eqnarray}
		\endgroup
By using \eqref{appN} we get
		\begingroup\allowdisplaybreaks
	\begin{eqnarray}
		\nonumber
		\Delta \left[Q_p^{-n}(q)(q-p)\right]&=&-4 \sum_{k=0}^{\infty} (-1)^k \binom{n+k-1}{k}S_{2(n+k)-1}(q-p_0)|q-p_0|^{-4(n+k)} p_1^{2k}\\
		\label{Ls2a}
		&&+4\sum_{k=0}^{\infty} (-1)^k  \binom{n+k-1}{k} S_{2(n+k)}(q-p_0)|q-p_0|^{-4(n+k)-2} p_1^{2k+1}I.
	\end{eqnarray}
	\endgroup
Formula \eqref{Lauretsp} follows by adding \eqref{Ls1a} and \eqref{Ls2a}.
	\\ We now suppose that the region $|q - p_0| < |p_1|$ has a nonempty intersection with the domain of convergence of the spherical Laurent series. The Taylor part $A_T(q,p)$ follows by Theorem \ref{Cliffordapp}, so we focus on the remaining part. Using the hypothesis, we can write $Q_p^{-n}(q)$ as in \eqref{negLa11}.
	Thus, by applying the second Fueter map and by using \eqref{appLapla} we get
			\begingroup\allowdisplaybreaks
	\begin{eqnarray}
		\nonumber
		\Delta Q_p^{-n}(q)&=& \sum_{k=0}^{\infty} (-1)^k\binom{n+k-1}{k}  \Delta (q-p_0)^{2k}p_1^{-2(n+k)}\\
		\nonumber
		&=&-2 \sum_{k=1}^{\infty} (-1)^k\binom{n+k-1}{k}  (2k)(2k-1)\mathcal{Q}_{2(k-1)}(q-p_0)p_1^{-2(n+k)}\\
		\nonumber
		&=& -4n \sum_{k=1}^{\infty} (-1)^k \binom{n+k-1}{k-1} (2k-1) \mathcal{Q}_{2(k-1)}(q-p_0) p_1^{-2(n+k)}\\
		\label{Ls1}
		&=& 4n \sum_{k=0}^{\infty} (-1)^k \binom{n+k}{k} (2k+1) \mathcal{Q}_{2k}(q-p_0) p_1^{-2(n+1+k)},
	\end{eqnarray}
	\endgroup
	and
		\begingroup\allowdisplaybreaks
	\begin{eqnarray}
		\nonumber
		\Delta \left[Q_p^{-n}(q)(q-p)\right]&=&-2 \sum_{k=1}^{\infty} (-1)^k \binom{n+k-1}{k} (2k+1)(2k) \mathcal{Q}_{2k-1}(q-p_0) p_1^{-2(n+k)}\\
		\nonumber
		&&+2\sum_{k=1}^{\infty} (-1)^k  \binom{n+k-1}{k} (2k)(2k-1) \mathcal{Q}_{2(k-1)}(q-p_0) p_1^{-2(n+k)+1}I\\
		\nonumber
		&=& -4n\sum_{k=1}^{\infty} (-1)^k \binom{n+k-1}{k-1} (2k+1) \mathcal{Q}_{2k-1}(q-p_0) p_1^{-2(n+k)}\\
		\nonumber
		&&+4n\sum_{k=1}^{\infty} (-1)^k  \binom{n+k-1}{k-1} (2k-1) \mathcal{Q}_{2(k-1)}(q-p_0) p_1^{-2(n+k)+1}I\\
		\label{Ls2}
		&=& 4n\sum_{k=0}^{\infty} (-1)^k \binom{n+k}{k} (2k+3) \mathcal{Q}_{2k+1}(q-p_0) p_1^{-2(n+k+1)}\\
		\nonumber
		&&-4n\sum_{k=0}^{\infty} (-1)^k  \binom{n+k}{k} (2k+1) \mathcal{Q}_{2k}(q-p_0) p_1^{-2(n+k)-1}I.
	\end{eqnarray}
	\endgroup
	Finally, the series \eqref{Lauretsp1} follows by summing \eqref{Ls1} and \eqref{Ls2} and this completes the proof.
\end{proof}
\begin{remark}
Taking $p=0$ in \eqref{Lauretsp}, by Remark \ref{zeroFueter} we have
	$$ A_T(q,0)=-2 \sum_{n=0}^{\infty} (n+2)(n+1) \mathcal{Q}_n(q) a_{n+2}.$$
The other series in \eqref{Lauretsp} when $p=0$ has the nonzero terms only when $k=n$, and thus we have
	\begin{eqnarray*}
		&&-4 \sum_{n=1}^{\infty} S_{2n}(q) |q|^{-4n-2} a_{-2n}-4 \sum_{n=1}^{\infty} S_{2n-1}(q) |q|^{-4n}a_{-2n+1}\\
		&&= -4 \sum_{n=1}^{\infty} S_{2n}(q) |q|^{-2(2n+1)} a_{-2n}-4 \sum_{n=1}^{\infty} S_{2n-1}(q) |q|^{-2(2n-1+1)}a_{-2n+1}\\
		&&=-4 \sum_{n=1}^{\infty} S_n(q) |q|^{-2(n+1)} a_{-n}.
	\end{eqnarray*}
	Moreover, for $p=0$ the condition $p_1<|q-p_0|$ is trivially satisfied and so we get to the Fueter regular Laurent series, where it is convergent, at the origin:
	
$$
		\breve f(q)=\Delta f(q)= -2\sum_{n=0}^{\infty}(n+2)(n+1) \mathcal{Q}_n(q) a_{n+2}-4 \sum_{n=1}^{\infty} S_n(q)|q|^{-2(n+1)}a _{-n}.
$$
\end{remark}

We now show that the Laurent Fueter regular spherical series assumes a special form if we consider a nonreal quaternion $q$.
Now, we describe the form of the Laurent–Fueter regular spherical series at $p$, which is valid outside the real axis.

\begin{proposition}
	Let $f$ be a slice hyperholomorphic function in an axially symmetric open set $\Omega$ that admits a spherical Laurent expansion at an arbitrarily fixed point $p=p_0+Ip_1\in\mathbb H$ as in \eqref{Lau2}, with coefficients $\{a_n\}_{n \in \mathbb{Z}} \subseteq \mathbb{H}$, the expansion being convergent (absolutely and uniformly on the compact subsets) in $U(p,R_1,R_2)$. Then, for $q\in U(p,R_1,R_2)\setminus\mathbb R$ with $|q-p_0|\not=|p_1|$, we can write the Laurent Fueter regular spherical series as
	\begingroup\allowdisplaybreaks
	\begin{eqnarray*}
		\breve{f}(q)=\Delta f(q)&=& -2 (\underline{q})^{-1} \partial_{q_0} \left[\sum_{n \in \mathbb{Z}} Q_p^n(q)\left(a_{2n}+(q-p)a_{2n+1}\right)\right]\\
		&& + (\underline{q})^{-2} \left[\sum_{n \in \mathbb{Z}} Q_p^n(q)\left(a_{2n}+(q-p)a_{2n+1}\right)\right]\\
		&&-(\underline{q})^{-2}\left[\sum_{n \in \mathbb{Z}} Q_p^n(\bar{q})\left(a_{2n}+(\bar{q}-p)a_{2n+1}\right)\right].
	\end{eqnarray*}
	\endgroup
\end{proposition}
\begin{proof}
	We first focus on the Taylor part of $\Delta f(q)$, which we denote by $A_T(q,p)$. By Proposition~\ref{LauN}, we have
	\begin{eqnarray}
		\nonumber
		A_T(q,p)&=&-2 (\underline{q})^{-1} \partial_{q_0}\left(\sum_{n=0}^{\infty} Q_p^n(q)a_{2n}+ Q_p^n(q)(q-p)a_{2n+1} \right)\\
		\nonumber
		&&+(\underline{q})^{-2}\left(\sum_{n=0}^{\infty} Q_p^n(q)a_{2n}+ Q_p^n(q)(q-p)a_{2n+1} \right)\\
		\label{LAU}
		&& -(\underline{q})^{-2}\left(\sum_{n=0}^{\infty} Q_p^n(\bar{q})a_{2n}+ Q_p^n(\bar{q})(\bar{q}-p)a_{2n+1} \right).
	\end{eqnarray}
	
	We now consider the two series in \eqref{Lauretsp} and \eqref{Lauretsp1}. Since these series converge in different regions, we split the proof into two cases. Let $p = p_0 + I p_1$. In the remainder of the proof, we assume that the regions $|p_1| < |q - p_0|$ and $|q - p_0| < |p_1|$ have a nonempty intersection with the domain of convergence of the spherical Laurent series.
	\newline
	\newline
	\emph{Case I: $|p_1|<|q-p_0|$}. Our goal is to find a closed expression of the summations in \eqref{Lauretsp} expressing $\breve{f}(q)$.
	\newline
	\newline
	We start from the first series in \eqref{Lauretsp}. By using \eqref{closedN} and \eqref{negLa} we get
	\begingroup\allowdisplaybreaks
	\begin{eqnarray}
		\nonumber
		&&-4 \sum_{n=1}^{\infty}\sum_{k=0}^{\infty} (-1)^k\binom{n+k-1}{k}  S_{2(n+k)}(q-p_0)|q-p_0|^{-4(n+k)-2}p_1^{2k}a_{-2n}\\
		\nonumber
		&=&-(\underline{q})^{-1} \sum_{n=1}^{\infty}\sum_{k=0}^{\infty} (-1)^k \binom{n+k-1}{k} \left[-4(n+k) (\bar{q}-p_0)^{2n+2k+1} \right.\\
		\nonumber
		&&\left. + (\underline{q})^{-1} |q-p_0|^2 \left( (q-p_0)^{2(n+k)}-(\bar{q}-p_0)^{2(n+k)} \right) \right]|q-p_0|^{-4(n+k)-2}p_1^{2k}a_{-2n}\\
		\nonumber
		&=& -2 (\underline{q})^{-1} \sum_{n=1}^{\infty}\sum_{k=0}^{\infty} (-1)^k \binom{n+k-1}{k} (-2(n+k)) (q-p_0)^{-2(n+k)-1}p_1^{2k}a_{-2n}\\
		\nonumber
		&& - (\underline{q})^{-2} \sum_{n=1}^{\infty}\sum_{k=0}^{\infty} (-1)^k \binom{n+k-1}{k} \left[ (\bar{q}-p_0)^{-2(n+k)}-(q-p_0)^{-2(n+k)} \right] p_1^{2k}a_{-2n}\\
		\label{Laumono}
		&=& -2 (\underline{q})^{-1} \sum_{n=1}^{\infty}\partial_{q_0} \left(Q_p^{-n}(q)\right)a_{-2n}-(\underline{q})^{-2}\sum_{n=1}^{\infty} \left[Q_p^{-n}(\bar{q})-Q_p^{-n}(q)\right]a_{-2n}.
	\end{eqnarray}
	\endgroup
	Now we focus on the second series in \eqref{Lauretsp}. By using \eqref{closedN} and \eqref{negLa} we obtain
	\begingroup\allowdisplaybreaks
	\begin{eqnarray}
		\nonumber
		&&-4 \sum_{n=1}^{\infty}\sum_{k=0}^{\infty} (-1)^k \binom{n+k-1}{k}S_{2(n+k)-1}(q-p_0)|(q-p_0)|^{-4(n+k)} p_1^{2k}a_{-2n+1}\\
		\nonumber
		&=& -(\underline{q})^{-1}\sum_{n=1}^{\infty}\sum_{k=0}^{\infty} (-1)^k \binom{n+k-1}{k}\left[-2(2(n+k)-1) (\bar{q}-p_0)^{2(n+k)} \right.\\
		\nonumber
		&& \left.+(\underline{q})^{-1} |q-p_0|^2 \left( (q-p_0)^{2(n+k)-1}-(\bar{q}-p_0)^{2(n+k)-1} \right) \right] |q-p_0|^{-4(n+k)} p_1^{2k}a_{-2n+1}\\
		\nonumber
		&=& -2(\underline{q})^{-1}\sum_{n=1}^{\infty}\sum_{k=0}^{\infty} (-1)^k \binom{n+k-1}{k}\left((-2(n+k)+1) (q-p_0)^{-2(n+k)} \right)p_1^{2k}a_{-2n+1}\\
		\nonumber
		&&-(\underline{q})^{-2} \sum_{n=1}^{\infty}\sum_{k=0}^{\infty}(-1)^k \binom{n+k-1}{k} \left[ (\bar{q}-p_0)^{-2(n+k)+1}-(q-p_0)^{-2(n+k)+1} \right]p_1^{2k}a_{-2n+1}\\
		\nonumber
		&=& -2 (\underline{q})^{-1}\sum_{n=1}^{\infty} \partial_{q_0} \left( Q_p^{-n}(q)(q-p_0)\right)a_{-2n+1}\\
		\label{Laumono1}
		&&-(\underline{q})^{-2}\sum_{n=1}^{\infty} \left[Q_p^{-n}(\bar{q})(\bar{q}-p_0)-Q_p^{-n}(q)(q-p_0)\right]a_{-2n+1}.
	\end{eqnarray}
	\endgroup
	A compact expression of the third series in \eqref{Lauretsp} follows by \eqref{Laumono}:
	\begin{eqnarray}
		\nonumber
		&&	-4\sum_{n=1}^{\infty} \sum_{k=0}^{\infty} (-1)^k  \binom{n+k-1}{k} S_{2(n+k)}(q-p_0)|q-p_0|^{-4(n+k)-2} p_1^{2k+1} Ia_{-2n+1}\\
		\label{Laumono2}
		&&=-2 (\underline{q})^{-1}\sum_{n=1}^{\infty} \partial_{q_0} \left( Q_p^{-n}(q)\right)p_1I a_{-2n+1}-(\underline{q})^{-2}\sum_{n=1}^{\infty} \left[Q_p^{-n}(\bar{q})-Q_p^{-n}(q)\right]p_1Ia_{-2n+1}.
	\end{eqnarray}
	Hence the final result follows by putting together \eqref{Laumono}, \eqref{Laumono1} and \eqref{Laumono2}.
	\newline
	\newline
	\emph{Case II: $|q-p_0|<|p_1|$}. Now, we aim to find a closed expression of  \eqref{Lauretsp1}.
	\newline
	\newline	
	First we find a compact expression of the first series in \eqref{Lauretsp1}. By Corollary \ref{closedAPP} and \eqref{negLa11} we have	
	\begingroup\allowdisplaybreaks
	\begin{eqnarray}
		\nonumber
		&&4 \sum_{n=1}^{\infty}\sum_{k=0}^{\infty} (-1)^k \binom{n+k}{k} n(2k+1) \mathcal{Q}_{2k}(q-p_0) p_1^{-2(n+1+k)}a_{-2n}\\
		\nonumber
		&&=4 (\underline{q})^{-1} \sum_{n=1}^{\infty}\sum_{k=0}^{\infty} (-1)^k \binom{n+k}{k}n (q-p_0)^{2k+1}p_1^{-2(n+1+k)}a_{-2n}\\
		\nonumber
		&& \, \, \, \, \, \, \, \,+4 (\underline{q})^{-2} \sum_{n=1}^{\infty}\sum_{k=0}^{\infty} (-1)^k \binom{n+k}{k}\frac{n}{2(2k+2)} \left[(\bar{q}-p_0)^{2(k+1)}-(q-p_0)^{2(k+1)}\right]p_1^{-2(n+1+k)}a_{-2n}\\
		\nonumber
		&&=2 (\underline{q})^{-1} \sum_{n=1}^{\infty}\sum_{k=0}^{\infty} (-1)^k \binom{n+k}{k}\frac{n}{k+1} \partial_{q_0}(q-p_0)^{2k+2}p_1^{-2(n+1+k)}a_{-2n}\\
		\nonumber
		&& \, \, \, \, \, \, \, \,+ (\underline{q})^{-2} \sum_{n=1}^{\infty}\sum_{k=0}^{\infty} (-1)^k \binom{n+k}{k+1} \left[(\bar{q}-p_0)^{2(k+1)}-(q-p_0)^{2(k+1)}\right]p_1^{-2(n+1+k)}a_{-2n}\\
		\nonumber
		&&=2 (\underline{q})^{-1} \sum_{n=1}^{\infty}\sum_{k=0}^{\infty} (-1)^k \binom{n+k}{k+1} \partial_{q_0}(q-p_0)^{2(k+1)}p_1^{-2(n+1+k)}a_{-2n}\\
		\nonumber
		&& \, \, \, \, \, \, \, \,+ (\underline{q})^{-2} \sum_{n=1}^{\infty}\sum_{k=0}^{\infty} (-1)^k \binom{n+k}{k+1} \left[(\bar{q}-p_0)^{2(k+1)}-(q-p_0)^{2(k+1)}\right]p_1^{-2(n+1+k)}a_{-2n}\\
		\nonumber
		&&=-2 (\underline{q})^{-1} \sum_{n=1}^{\infty}\sum_{k=1}^{\infty} (-1)^k \binom{n+k-1}{k} \partial_{q_0}(q-p_0)^{2k}p_1^{-2(n+k)}a_{-2n}\\
		\nonumber
		&& \, \, \, \, \, \, \, \,- (\underline{q})^{-2} \sum_{n=1}^{\infty}\sum_{k=0}^{\infty} (-1)^k \binom{n+k-1}{k} \left[(\bar{q}-p_0)^{2k}-(q-p_0)^{2k}\right]p_1^{-2(n+k)}a_{-2n}\\
		\nonumber
		&&=-2 (\underline{q})^{-1}\sum_{n=1}^{\infty} \partial_{q_0} \left(Q_p^{-n}(q)\right)a_{-2n}\\
		\label{LQ}
		&& \, \, \, \, \, \, \, \, \quad-(\underline{q})^{-2}\sum_{n=1}^{\infty} \left[Q_p^{-n}(\bar{q})-Q_p^{-n}(q)\right]a_{-2n}.
	\end{eqnarray}	
	\endgroup
	Now we focus on the second summation of \eqref{Lauretsp1}. By Corollary \ref{closedAPP} and \eqref{negLa} we get
	\begingroup\allowdisplaybreaks
	\begin{eqnarray}
		\nonumber
		&&4\sum_{n=1}^{\infty}\sum_{k=0}^{\infty} (-1)^k \binom{n+k}{k}n (2k+3) \mathcal{Q}_{2k+1}(q-p_0) p_1^{-2(n+k+1)}a_{-2n+1}\\
		\nonumber
		&&=2 (\underline{q})^{-1}\sum_{n=1}^{\infty}\sum_{k=0}^{\infty} (-1)^k \binom{n+k}{k}\frac{n}{k+1} (2k+3) (q-p_0)^{2k+2} p_1^{-2(n+1+k)}a_{-2n+1}\\
		\nonumber
		&&\, \,+(\underline{q})^{-2}\sum_{n=1}^{\infty}\sum_{k=0}^{\infty} (-1)^k \binom{n+k}{k} \frac{n}{k+1} \left[(\bar{q}-p_0)^{2k+3}-(q-p_0)^{2k+3} \right]p_1^{-2(n+1+k)}a_{-2n+1}\\
		\nonumber
		&=& 2 (\underline{q})^{-1}\sum_{n=1}^{\infty}\sum_{k=0}^{\infty} (-1)^k \binom{n+k}{k+1}\partial_{q_0}(q-p_0)^{2k+3} p_1^{-2(n+k+1)}a_{-2n+1}\\
		\nonumber
		&&\, \,+(\underline{q})^{-2}\sum_{n=1}^{\infty}\sum_{k=0}^{\infty} (-1)^k \binom{n+k}{k+1} \left[(\bar{q}-p_0)^{2k+3}-(q-p_0)^{2k+3} \right]p_1^{-2(n+k+1)}a_{-2n+1}\\
		\nonumber
		&=& -2 (\underline{q})^{-1}\sum_{n=1}^{\infty}\sum_{k=1}^{\infty} (-1)^k \binom{n+k-1}{k}\partial_{q_0}(q-p_0)^{2k+1} p_1^{-2(n+k)}a_{-2n+1}\\
		\nonumber
		&&\, \, - (\underline{q})^{-2}\sum_{n=1}^{\infty}\sum_{k=1}^{\infty} (-1)^k \binom{n+k-1}{k} \left[(\bar{q}-p_0)^{2k+1}-(q-p_0)^{2k+1} \right]p_1^{-2(n+k)}a_{-2n+1}\\
		\nonumber
		&=&-2 (\underline{q})^{-1} \sum_{n=1}^{\infty}\partial_{q_0} \left(Q_p^{-n}(q)(q-p_0) \right)a_{-2n+1}\\
		\label{LQ1}
		&&-(\underline{q})^{-2} \sum_{n=1}^{\infty}\left[ Q_p^{-n}(\bar{q})(\bar{q}-p_0)-Q_p^{-n}(q)(q-p_0) \right]a_{-2n+1}.
	\end{eqnarray}
	\endgroup
	
	Finally we focus on the third series of \eqref{Lauretsp1}. By \eqref{LQ} we have
	\begingroup\allowdisplaybreaks
	\begin{eqnarray}
		\nonumber
		&&-4\sum_{n=1}^{\infty}\sum_{k=0}^{\infty} (-1)^k  \binom{n+k}{k}n (2k+1) \mathcal{Q}_{2k}(q-p_0) p_1^{-2(n+k)-1}Ia_{-2n+1}\\
		\nonumber
		&&= - 4\sum_{n=1}^{\infty}\sum_{k=0}^{\infty} (-1)^k  \binom{n+k}{k}n (2k+1) \mathcal{Q}_{2k}(q-p_0) p_1^{-2(n+k+1)}p_1Ia_{-2n+1}\\
		\label{LQ2}
		&&=2 (\underline{q})^{-1} \sum_{n=1}^{\infty}\partial_{q_0} \left(Q_p^{-n}(q) \right)p_1Ia_{-2n+1}+(\underline{q})^{-2} \sum_{n=1}^{\infty}\left[ Q_p^{-n}(\bar{q})-Q_p^{-n}(q) \right]p_1Ia_{-2n+1}.
	\end{eqnarray}
	\endgroup
	Hence the result follows by putting together \eqref{LAU}, \eqref{LQ}, \eqref{LQ1} and \eqref{LQ2}.
	
\end{proof}

\section{Taylor series: axially polyanalytic functions  of order 2}\label{POLYSERIES}

In this section we complete the study of the function spaces that appear in the quaternionic fine structures. The conjugate Fueter operator applied to a slice hyperholomorphic function gives rise to an axially polyanalytic function of order 2. The goal of this section is to study the series expansions of these functions. As in the preceding cases, we shall exploit the $*$-Taylor expansion and the spherical series of a slice hyperholomorphic function to obtain series expansions for axially polyanalytic functions of order 2.

\subsection{Polyanalytic regular series of order 2}
The first type of series expansion that we consider in this section is the $*$-Taylor series.

\begin{definition}
    \label{polyreg}
    Let $ U \subseteq\mathbb{H}$ be an axially symmetric open set. Let $f$ be a slice hyperholomorphic function (according to Definition \ref{sh}) admitting a $*$-Taylor expansion at an arbitrarily fixed point $p$ convergent in a set contained in $U$. Then we say that the  axially polyanalytic function of order $2$ given by $h=\bar{D}f$ has a polyanalytic regular series of order $2$ at $p$.
\end{definition}
To write the polyanalytic regular series of order 2 of $h=\bar Df$ as above, we need to  compute the action of the conjugate Fueter operator applied to the building block of the $*$-Taylor expansion at the point $p$, namely $(q-p)^{n*_{q,L}}$. As we did previously, below we will denote all $ *_{q,L} $-products simply by $ *$.
.

\begin{proposition}
    \label{four0}
    Let $n \geq 1$ and $q$, $p \in \mathbb{H}$ and let $\bar{D}$ be the conjugate Fueter operator in the variable $q$. Then
    $$
    \bar{D} (q-p)^{*n}=2 \left( n(q-p)^{*(n-1)}+ \sum_{k=1}^{n} (q-p)^{*(n-k)}* (\bar{q}-p)^{*(k-1)} \right).
    $$
\end{proposition}
\begin{proof}
 We start by applying the derivative $ \partial_{q_0}$ to $(q-p)^{*n}$, so that by \eqref{starLeR} we get
	\begingroup\allowdisplaybreaks
    \begin{eqnarray}
        \nonumber
        \partial_{q_0}(q-p)^{*n} &=&  \partial_{q_0} \left( \sum_{k=0}^{n} \binom{n}{k} q^k p^{n-k}(-1)^{n-k} \right) \\
        \nonumber
        &=& \sum_{k=1}^{n} \binom{n}{k} k q^{k-1} p^{n-k} (-1)^{n-k} \\
        \nonumber
        &=& n \sum_{k=1}^{n} \binom{n-1}{k-1} q^{k-1} p^{n-k} (-1)^{n-k} \\
        \nonumber
        &=& n \sum_{k=0}^{n-1} \binom{n-1}{k} q^k p^{n-k-1}(-1)^{n-k-1}\\
        \label{partial}
        &=&n (q-p)^{*(n-1)}.
    \end{eqnarray}
    \endgroup
    By Lemma \ref{genbeg} we have
    $$ (\partial_{q_0}+\partial_{\underline{q}}) (q-p)^{*n}=D(q-p)^{*n}=-2 \sum_{k=1}^{n} (q-p)^{*(n-k)}*  (\bar{q}-p)^{*(k-1)}.$$
    This formula, together with \eqref{partial} imply that
    \begin{eqnarray}
        \nonumber
        \partial_{\underline{q}} (q-p)^{*n} &=& -\partial_{q_0} (q-p)^{*n}-2 \sum_{k=1}^{n} (q-p)^{*(n-k)}*  (\bar{q}-p)^{*(k-1)} \\
        \label{partialvec}
        &=& -n (q-p)^{*(n-1)}-2 \sum_{k=1}^{n} (q-p)^{*(n-k)}*  (\bar{q}-p)^{*(k-1)}.
    \end{eqnarray}
    Finally, by \eqref{partial} and \eqref{partialvec} we obtain
    \begin{eqnarray*}
        \bar{D}(q-p)^{*n} &=& (\partial_{q_0}- \partial_{\underline{q}})(q-p)^{*n} \\
        &=&2 \left(n (q-p)^{*(n-1)}+ \sum_{k=1}^{n} (q-p)^{*(n-k)}*  (\bar{q}-p)^{*(k-1)} \right).
    \end{eqnarray*}
    This concludes the proof.
\end{proof}
The previous result can be written using a more compact notation, namely
\begin{equation}
    \label{polyf}
    \bar{D}(q-p)^{*n}=2n \widetilde{P}_{2, n-1}(q,p),
\end{equation}
where we set
\begin{equation}
\label{polytilde1}
\widetilde{P}_{2, n}(q,p):=  (q-p)^{*n}+ \frac{1}{n+1} \sum_{k=1}^{n+1} (q-p)^{*(n+1-k)}* (\bar{q}-p)^{*(k-1)}.
\end{equation}

The main properties of the functions $\widetilde{P}_{2, n}(q,p)$ are summarized in the following result.

\begin{proposition}
Let $p$, $q \in \mathbb{H}$ and $n \in \mathbb{N}$. The functions $\widetilde{P}_{2, n}(q,p)$ satisfy the following properties:
\begin{itemize}
\item[1)]   $\widetilde{P}_{2, n}(q,p)$ are right slice hyperholomorphic functions in $p$ and left axially polyanalytic of order 2 in the variable $q$.
\item[2)]  $\widetilde{P}_{2, n}(q,p)$ are related to the axially harmonic polynomials $\widetilde{H}_n(q,p)$, see \eqref{Harmopoly}, by the formula
\begin{equation}
\label{polytilde}
    \widetilde{P}_{2, n}(q,p)=(q-p)^{*n}+ \widetilde{H}_n(q,p).
\end{equation}
\item[3)] $\widetilde{P}_{2, n}(q,p)$ are left slice polyanalytic of order $n+1$ in the variable $q$.
\item[4)] For $q \notin \mathbb{R}$ the functions $\widetilde{P}_{2, n}(q,p)$ rewrite in the form
\begin{equation}
\label{NNpoly}
 \widetilde{P}_{2,n}(q,p)=(q-p)^{*n}- \frac{(\underline{q})^{-1}}{2(n+1)} \left[ (\bar{q}-p)^{*(n+1)}-(q-p)^{*(n+1)}\right].
\end{equation}
\end{itemize}
\end{proposition}
\begin{proof}
\begin{itemize}
\item[1)]
By formula \eqref{polytilde1}, is clear that $\widetilde{P}_{2,n}(q,p)$ is right  slice hyperholomorphic in $p$, see Proposition \ref{closedp}. By the Fueter theorem we have that
\begin{eqnarray*}
D^2 \widetilde{P}_{2, n}(q,p)&=& D^2 \left( \frac{\bar{D}(q-p)^{*(n+1)}}{2(n+1)}\right)\\
&=& \frac{\Delta D (q-p)^{*(n+1)}}{2(n+1)}\\
&=&0.
\end{eqnarray*}\item[2)] Formula \eqref{polytilde} follows by \eqref{Harmopoly}, the definition of the functions  $\widetilde{H}_n(q,p)$ and the definition of $\widetilde{P}_{2, n}(q,p)$, see \eqref{polytilde1}.
\item [3)] By the second point of Proposition \ref{harmpoly}, we have that the polynomials $\widetilde{H}_{n}(q,p)$ are left slice polyanalytic of order $n+1$ in $q$. Hence we get the result by \eqref{polytilde} and the fact that the sum of the left slice hyperholomorphic function $(q-p)^{*n}$ and a left slice polyanalytic function of order $n+1$ is a left slice polyanalytic function of order $n+1$.
\item[4)] Formula \eqref{NNpoly} follows by combining \eqref{polytilde} and \eqref{closedH1}.
\end{itemize}

\end{proof}

We know that any axially polyanalytic function of order $n$ can be decomposed in terms of Fueter regular functions, see Theorem \ref{polydec}. In the next result we show that the functions $\widetilde{P}_{2,n}(q,p)$ can be decomposed in terms of the regular Fueter polynomials $\widetilde{Q}_n(q,p)$ introduced in \eqref{monopol}.

\begin{proposition}
    \label{split2}
    Let $ n  \in \mathbb{N}$. Then for $q$, $p \in \mathbb{H}$ we have
    \begin{equation}\label{poly7}
        \widetilde{P}_{2,n}(q,p)= (n+2) \widetilde{Q}_{n}(q,p)+n \widetilde{Q}_{n-1}(q,p)p-q_0n \widetilde{Q}_{n-1}(q,p).
    \end{equation}

\end{proposition}
\begin{proof}
We start by observing that
    \begin{eqnarray}
        \nonumber
        \bar{D}(q-p)^{*n}&=&  \left[\bar{D}(q-p)^{*n}-q_0 \Delta  (q-p)^{*n} \right]+q_0 \Delta (q-p)^{*n} \\
        \label{polya}
        &:=&  h_0(q,p)+q_0 h_1(q,p).
    \end{eqnarray}
    Proposition \ref{four0} and Theorem \ref{lapla0} yield
    \begin{eqnarray*}
        h_0(q,p) &=&  \bar{D} (q-p)^{*n}-q_0\Delta (q-p)^{*n} \\
        &=& 2 \left[n (q-p)^{*(n-1)}+ \sum_{k=1}^n (q-p)^{*(n-k)}* (\bar{q}-p)^{*(k-1)}\right.\\
        && \left. +2q_0 \sum_{k=1}^{n-1} (n-k) (q-p)^{*(n-k-1)}*(\bar{q}-p)^{*(k-1)}  \right].
    \end{eqnarray*}
    Since $2q_0=q+ \bar{q}$ we have
    \begin{eqnarray}
        \nonumber
        h_0(q,p) &=& 2 \left[n (q-p)^{*(n-1)}+ \sum_{k=1}^n (q-p)^{*(n-k)}* (\bar{q}-p)^{*(k-1)}\right.\\
        \nonumber
        && \left. + \sum_{k=1}^{n-1} (n-k) (q-p)^{*(n-k-1)}*q* (\bar{q}-p)^{*(k-1)} \right.\\
        \label{poly01}
        && + \left. \sum_{k=1}^{n-1} (n-k) (q-p)^{*(n-k-1)}*\bar{q}* (\bar{q}-p)^{*(k-1)} \right].
    \end{eqnarray}
    By adding and subtracting the same quantity we rewrite \eqref{poly01} as
    \begingroup\allowdisplaybreaks
    \begin{eqnarray*}
        h_0(q,p) &=& 2 \left[n (q-p)^{*(n-1)}+ \sum_{k=1}^n (q-p)^{*(n-k)}* (\bar{q}-p)^{*(k-1)}\right.\\
        &&\left. + \sum_{k=1}^{n-1} (n-k) (q-p)^{*(n-k-1)}* (q-p) * (\bar{q}-p)^{*(k-1)} \right. \\
        && \left. + \sum_{k=1}^{n-1} (n-k) (q-p)^{*(n-k-1)}* (\bar{q}-p) * (\bar{q}-p)^{*(k-1)} \right. \\
        && \left. +2 \sum_{k=1}^{n-1} (n-k) (q-p)^{*(n-k-1)}*p*(\bar{q}-p)^{*(k-1)} \right]\\
        &=&2 \left[ \sum_{k=1}^n (q-p)^{*(n-k)}* (\bar{q}-p)^{*(k-1)} \right.\\
        && \left. + \sum_{k=1}^{n-1} (n-k) (q-p)^{*(n-k)}*  (\bar{q}-p)^{*(k-1)}\right.\\
        && \left. + \sum_{k=0}^{n-1} (n-k) (q-p)^{*(n-k-1)}* (\bar{q}-p)^{*k} \right. \\
        && \left. +2 \sum_{k=1}^{n-1} (n-k) (q-p)^{*(n-k-1)}*(\bar{q}-p)^{*(k-1)}p \right]\\
        &=& 2 \left[ \sum_{k=1}^{n} (n-k+1) (q-p)^{*(n-k)}*(\bar{q}-p)^{*(k-1)}\right.\\
        && \left. + \sum_{k=0}^{n-1} (n-k) (q-p)^{*(n-k-1)}*(\bar{q}-p)^{*k} \right. \\
        && \left. +2 \sum_{k=1}^{n-1} (n-k) (q-p)^{*(n-k-1)}*(\bar{q}-p)^{*(k-1)}p \right].
    \end{eqnarray*}
    \endgroup

    Changing index in the second series and using Theorem \ref{lapla0} we get

    \begingroup\allowdisplaybreaks
    \begin{eqnarray}
        \nonumber
        h_0(q,p)&=& 2 \left[ \sum_{k=1}^{n} (n-k+1) (q-p)^{*(n-k)}* (\bar{q}-p)^{*(k-1)}\right.\\
        \nonumber
        && \left. + \sum_{k=1}^{n} (n-k+1) (q-p)^{*(n-k)}*(\bar{q}-p)^{*(k-1)} \right. \\
        \nonumber
        && \left. +2 \sum_{k=1}^{n-1} (n-k) (q-p)^{*(n-k-1)}(\bar{q}-p)^{*(k-1)}p \right]\\
        \nonumber
        &=& 4 \left[\sum_{k=1}^{n} (n-k+1) (q-p)^{*(n-k)}*(\bar{q}-p)^{*(k-1)} \right. \\
        \nonumber
        && \left. + \sum_{k=1}^{n-1} (n-k) (q-p)^{*(n-k-1)}(\bar{q}-p)^{*(k-1)}p \right]\\
        \nonumber
        &=& -\Delta \left( (q-p)^{*(n+1)} \right)- \Delta \left( (q-p)^{*n} \right)p.\\
        \label{polya2}
        &=&2n(n+1) \widetilde{Q}_{n-1}(q,p)+2 n (n-1) \widetilde{Q}_{n-2}(q,p)p.
    \end{eqnarray}
    \endgroup
    For the term $h_1(q,p)$ in \eqref{polya}, we can directly use Theorem \ref{lapla0} and we obtain
    \begin{equation}
        \label{polya1}
        h_1(q,p)= \Delta \left( (q-p)^{*n} \right)=-2n(n-1) \widetilde{Q}_{n-2}(q,p).
    \end{equation}
    By plugging \eqref{polya2} and \eqref{polya1} into \eqref{polya} we can write
    $$
    \bar{D}(q-p)^{*n}=2n \left[(n+1) \widetilde{Q}_{n-1}(q,p)+(n-1) \widetilde{Q}_{n-2}(q,p)p-q_0 (n-1) \widetilde{Q}_{n-2}(q,p) \right].
    $$
    Finally by \eqref{polyf} we arrive to the following:
    \begin{eqnarray}
        \nonumber
        \widetilde{P}_{2,n-1}(q,p)&=& \frac{\bar{D}(q-p)^{*n}}{2n}\\
        \nonumber
        &=&(n+1) \widetilde{Q}_{n-1}(q,p)+(n-1) \widetilde{Q}_{n-2}(q,p)p\\
        \label{fin}
        &&-q_0 (n-1) \widetilde{Q}_{n-2}(q,p).
    \end{eqnarray}
    The result follows by rearranging the indices in \eqref{fin}.
\end{proof}
\begin{remark}
    If we consider $p=0$ in \eqref{poly7} we get back to \eqref{poly1}.
\\
By means of Proposition \ref{split2} and formula \eqref{polytilde} we can write the harmonic functions $ \widetilde{H}_n(q,p)$ in terms of the Fueter regular polynomials $\widetilde{Q}_n(q,p)$. Precisely, for $n\in\mathbb N$, we have
$$ \widetilde{H}_n(q,p)=(n+2) \widetilde{Q}_{n}(q,p)+n \widetilde{Q}_{n-1}(q,p)p-q_0n \widetilde{Q}_{n-1}(q,p)- (q-p)^{*n}.$$
\end{remark}

\begin{theorem}
    \label{conve4}
    Let $ U$ be an axially symmetric open set in $\mathbb{H}$ and let  $f$ be a slice hyperholomorphic function in $U$  that admits the $*$-Taylor expansion at an arbitrarily fixed point $p\in U$
\begin{equation}
\label{seriesNN}
f(q)= \sum_{n=0}^\infty (q-p)^{*n}a_n,
\end{equation}
convergent (absolutely and uniformly on the compact subsets) in $\Omega(p,R) \subseteq U$,
    where $ \frac{1}{R}=\limsup_{n \to \infty}|a_n|^{\frac{1}{n}}$. Then
    \begin{equation}
        \label{ser3}
        \bar{D}f(q)= \sum_{n=0}^{\infty} \widetilde{P}_{2,n}(q,p) b_n, \qquad q \in \Omega(p,R), \quad b_n:=2(n+1) a_{n+1}.
    \end{equation}
\end{theorem}
\begin{proof}
We apply the operator $\bar{D}$ to the function \eqref{seriesNN}.  By \eqref{polyf} we can write
\begin{eqnarray*}
	\bar{D}f(q) &=& \sum_{n=0}^{\infty} \bar{D}(q-p)^{*n} a_{n} \\
	&=& 2 \sum_{n=1}^{\infty} n \widetilde{P}_{2,n-1}(q,p)a_{n} \\
	&=& \sum_{n=0}^{\infty}  \widetilde{P}_{2,n}(q,p) b_n,
\end{eqnarray*}
where $b_n:= 2(n+1) a_{n+1}$. We prove the convergence of the series in \eqref{ser3}.
    If we assume that $q \in \mathbb{R}$ the result is trivial, so we suppose that $q \notin \mathbb{R}$. Since $(q-p)^{*(n+1)}$ is a slice hyperholomorphic function, by the representation formula for axially polyanalytic functions of order $2$, see Theorem \ref{polyrap} and \eqref{polyf}, we have  for $n \geq 1$:
    \begin{eqnarray*}
    	 |2(n+1)\widetilde{P}_{2,n}(q,p) a_{n+1}|&=& |\bar{D}(q-p)^{*(n+1)}a_{n+1}|\\
        &=&\biggl|(n+1) \left[(q_I-p)^{n}+ (q_{-I}-p)^{n}\right]a_{n+1}\\
        && + (n+1) I_{\underline{q}}I\left[(q_{-I}-p)^{(n)}- (q_{I}-p)^{n}\right]a_{n+1} \\
        && +(\underline{q})^{-1}I_{\underline{q}}I\left[(q_{-I}-p)^{(n+1)}- (q_{I}-p)^{(n+1)}\right]a_{n+1} \biggl|,
    \end{eqnarray*}
    where if $p\in\mathbb{C}_I$, we take $q_{\pm I}=x \pm yI$. So we get
    \begin{eqnarray}
        | 2(n+1)\widetilde{P}_{2,n}(q,p) a_{n+1}| & \leq& 2(n+1) \left[\left|(q_I-p)^{n}a_{n+1} \right|+ \left|(q_{-I}-p)^{n}a_{n+1} \right|\right]\\
        \nonumber
        &&+|\underline{q}|^{-1}\left[\left|(q_I-p)^{n+1}a_{n+1} \right|+ \left|(q_{-I}-p)^{n+1}a_{n+1} \right|\right].
    \end{eqnarray}
    Since $q \in \Omega(p,R)$, by the hypothesis on the coefficients $ \{a_n\}_{n \in \mathbb{N}}$ we get that the series
    $$ \sum_{n=1}^{\infty} \left|(q_{\pm I}-p)^{n}a_{n+1} \right| \quad \hbox{and} \quad \sum_{n=1}^{\infty} \left|(q_{\pm I}-p)^{n+1}a_{n+1} \right|,$$
    are convergent, therefore the series in \eqref{ser3} is convergent.

\end{proof}

\begin{remark}
	\label{rempoly}
	If we consider $p=0$ in \eqref{ser3} we get
	\begin{equation}
		\label{zeropoly3}
		\bar{D}f(q)= 2 \sum_{n=0}^{\infty} (n+1)P_{2,n}(q) a_{n+1},
	\end{equation}
	since $ \widetilde{P}_{2,n}(q,0)=P_{2,n}(q)$. Formula \eqref{zeropoly3} is the Taylor expansion in terms of $P_{2,n}$  in a neighbourhood of the origin of $\bar{D}f$.
\end{remark}

Another useful way of writing the polyanalytic regular series is by means of the Clifford-Appell polynomials.

\begin{proposition}
Let $f$ be a slice hyperholomorphic function in a neighbourhood of $p \in \mathbb{H}$ that admits a $*$-Taylor expansion  as in \eqref{seriesNN} with coefficients $\{a_n\}_{n \in \mathbb{N}_0} \subseteq \mathbb{H}$,  convergent (absolutely and uniformly on the compact subsets) in $\Omega(p,R)$, where $\frac1R=\limsup_n|a_n|^{1/n}$. Then the polyanalytic regular series of order $2$, for $q\in\Omega(p,R)$, can be written as
    \begin{eqnarray}
    \label{polyTaylor}
        \sum_{n=0}^{\infty}  \widetilde{P}_{2, n}(q,p) b_n&=&2 \sum_{n=0}^{\infty}(n+1)\left( \sum_{k=0}^{n} \binom{n}{k} (k +2) \mathcal{Q}_{k}(q) p^{n- k} (-1)^{n- k} \right.\\
        \nonumber
        && \left. - q_0 \sum_{k=1}^{n}k \binom{n}{k} \mathcal{Q}_{k-1}(q) p^{n-k} (-1)^{n-k}\right) a_{n+1},
    \end{eqnarray}
Moreover if $q \notin \mathbb{R}$ we have
    $$ \sum_{n=0}^{\infty} \widetilde{P}_{2, n}(q,p) b_n= 2 \partial_{q_0} \left( \sum_{n=0}^{\infty} (q-p)^{*n} a_n \right)- \underline{q}^{-1} \left( \sum_{n=0}^{\infty} (\bar{q}-p)^{*n}a_n- \sum_{n=0}^{\infty} (q-p)^{*n} a_n \right).$$
\end{proposition}
\begin{proof}
By formula \eqref{starLeR} and formula \eqref{poly4} we have
    \begin{eqnarray*}
        \bar{D}f(q) &=&  \sum_{n=0}^{\infty} \sum_{k=1}^{n} \binom{n}{k}2k  \mathcal{P}_{2,k-1}(q) p^{n-k} (-1)^{n-k} a_n \\
        &=& 2\sum_{n=1}^{\infty} \sum_{k=1}^{n} \binom{n-1}{k-1} n \mathcal{P}_{2,k-1}(q) p^{n-k} (-1)^{n-k} a_n \\
        &=& 2\sum_{n=1}^{\infty} \sum_{k=0}^{n-1} \binom{n-1}{k}  n \mathcal{P}_{2,k}(q) p^{n-k-1} (-1)^{n-k-1} a_n\\
        &=& 2\sum_{n=0}^{\infty} \sum_{k=0}^{n}  \binom{n}{k}  (n+1) \mathcal{P}_{2,k}(q) p^{n-k} (-1)^{n-k} a_{n+1}.
    \end{eqnarray*}
From the polyanalytic decomposition of the polynomials $\mathcal{P}_{2,k}(q)$, see \eqref{poly1} we deduce
    \begin{eqnarray*}
        \bar{D}f(q) &=& 2\sum_{n=0}^{\infty}  (n+1) \sum_{k=0}^{n} \binom{n}{k}(k+2) \mathcal{Q}_k(q)   p^{n-k} (-1)^{n-k} a_{n+1}\\
        &&- q_0 \sum_{n=1}^{\infty} (n+1) \sum_{k=1}^{n} \binom{n}{k} k  \mathcal{Q}_{k-1}(q) p^{n-k} (-1)^{n-k} a_{n+1}.
    \end{eqnarray*}
Hence the result follows by Theorem \ref{conve4}.
    Now,  we consider $ q \notin \mathbb{R}$. By formula \eqref{NNpoly} we get
    \begin{eqnarray*}
        \sum_{n=0}^{\infty} \widetilde{P}_{2, n}(q,p) b_n&=& 2 \sum_{n=0}^{\infty} (n+1) \widetilde{P}_{2, n}(q,p) a_{n+1}\\
        &=& 2 \sum_{n=0}^{\infty} (n+1) (q-p)^{*n}a_{n+1}-(\underline{q})^{-1} \sum_{n=0}^{\infty} \left[(\bar{q}-p)^{*(n+1)}-(q-p)^{*(n+1)} \right]a_{n+1} \\
        &=& 2 \sum_{n=1}^{\infty} n (q-p)^{*(n-1)}a_{n}-(\underline{q})^{-1} \sum_{n=1}^{\infty} \left[(\bar{q}-p)^{*n}-(q-p)^{*n} \right]a_{n} \\
        &=& 2 \partial_{q_0} \left( \sum_{n=0}^{\infty} (q-p)^{*n} a_n \right)- (\underline{q})^{-1} \left( \sum_{n=0}^{\infty} (\bar{q}-p)^{*n}a_n- \sum_{n=0}^{\infty} (q-p)^{*n} a_n \right),
    \end{eqnarray*}
    which ends the proof.
\end{proof}
In the example below we expand the polyanalytic kernel $P^L_2(p,q)$ as polyanalytic regular series of order 2.
\begin{example}
The polyanalytic kernel appearing in the integral representation of the axially polyanalytic functions of order $2$, see \eqref{polykernel}, has the following expansion
\begin{equation}
\label{z2}
P_2^L(p,q) = \sum_{n=0}^{\infty} \widetilde{P}_{2,n}(q,p+1)a_n, \qquad q \in \Omega(p+1,1),
\end{equation}
where $\{ a_{n}\}_{n \in \mathbb{N}}:= \{-2 (-1)^n (n+1)\}_{n \in \mathbb{N}}$. Indeed, we apply the conjugate Cauchy-Fueter operator to \eqref{ex1}, with respect to the variable $q$. By \eqref{polyf} we have
\begin{eqnarray*}
	\bar{D} S^{-1}_L(p,q)&=& \sum_{n=0}^{\infty} (-1)^{n+1}\bar{D} (q-p-1)^{*n}\\
	&=& 2 \sum_{n=1}^{\infty} (-1)^n n \widetilde{P}_{2, n-1}(q,p+1)\\
	&=& -2 \sum_{n=0}^{\infty} (-1)^n (n+1) \widetilde{P}_{2, n}(q,p+1).
\end{eqnarray*}
Since $\bar{D} S^{-1}_L(p,q)=P_2^L(p,q)$, see \eqref{polykernel}, we get \eqref{z2}.
\end{example}

\subsection{Polyanalytic spherical series of order 2}

Another possible way to get a series expansion for an axially polyanalytic function of order $2$ is to consider spherical series which allow us to have  convergence in suitable Euclidean neighbourhoods. Thus we introduce the following notion:

\begin{definition}
    Let $U$ be an axially symmetric open set in $\mathbb H$. Let $f$ be a function slice hyperholomorphic in $U$ admitting a spherical series expansion at an arbitrarily fixed point $p\in U$ convergent in a convenient Euclidean open set contained in  $U$. Then we say that $\bar{D}f$ has a polyanalytic spherical series of order $2$.
\end{definition}

Using the spherical expansion of the slice hyperholomorphic function $f$, we can write a compact formula for the polyanalytic spherical series of order $2$ of $h=\bar Df$.

\begin{theorem}
	\label{polyseries1}
	Let $f$ be a slice hyperholomorphic function in a neighborhood of $p \in \mathbb{H}$ containing $U(p,R)$ for some $R>0$, where it admits a spherical expansion \eqref{spherical}, with coefficients
	$\{a_n\}_{n\in\mathbb{N}_0}\subseteq\mathbb{H}$,  convergent (absolutely and uniformly on the compact subsets) in $U(p,R)$. Let $h$ be the axially polyanalytic function of order 2 in a neighborhood of $p$ given by $h=\overline{D}f$. Then, $h$ has the following polyanalytic spherical expansion of order 2
	\begin{eqnarray}
\nonumber
			h(q)=\bar D f(q)&=&4\sum_{n=0}^{\infty} (n+1)(q-p_0)\,Q^{n}_p(q)\left[a_{2n+2}+(q-p)a_{2n+3}\right]\\
					\label{polyseries}
			&&+4\sum_{n=0}^{\infty}(q_0-p_0)\,\mathcal{L}_{n+1}(q,p)\left[a_{2n+2}+(\bar q-p)a_{2n+3}\right]
			+4\sum_{n=0}^{\infty}Q^{n}_p(q)\,a_{2n+1},
	\end{eqnarray}
where the function $\mathcal{L}_n(q,p)$ is defined in \eqref{key4}.
\end{theorem}
\begin{proof}
We apply $\bar D$ term by term to \eqref{spherical}. The exchange between $\bar D$ and the summation
is justified as in Theorem \ref{harm11}. We start proving the result for $q \notin \mathbb{R}$. By \eqref{compactPoly}, applied first to $Q_p^n(q)$, we have
\begin{equation}
	\label{Dbar}
\bar D\left(Q_p^n(q)\right)=4n(q-p_0)Q_p^{n-1}(q)+4(q_0-p_0)\mathcal{L}_n(q,p),
\end{equation}
where we have used \eqref{laurentI} and \eqref{derv}. Similarly, applying \eqref{compactPoly} to $Q_p^n(q)(q-p)$ and using again \eqref{laurentI}, we obtain
\begin{equation}
\label{Dbar1}
\bar D\left(Q_p^n(q)(q-p)\right)=4n(q-p_0)Q_p^{n-1}(q)(q-p)+4(q_0-p_0)\mathcal{L}_n(q,p)(\bar q-p)+4Q_p^n(q).
\end{equation}
Therefore, applying $\bar D$ term by term to \eqref{spherical}
and substituting \eqref{Dbar} and \eqref{Dbar1}, and changing the indexes we get \eqref{polyseries}. It remains to consider the real points. Let
$q=q_0\in U(p,R)\cap\mathbb R$. Since $Q_p(q_0)\in\mathbb R$, we have
$$
\mathcal L_n(q_0,p)=nQ_p^{n-1}(q_0).
$$
By \eqref{realbar} we deduce that $\overline{D} (\cdot))=4 \partial_{q_0}(\cdot)$. Since $Q_p(q)$ and $Q_p(q)(q-p)$ are slice hyperholomorphic,by \eqref{derv} we have
\begin{equation}
\label{realDbar}
\bar D[Q_p^n(q_0)]=8n(q_0-p_0)Q_p^{n-1}(q_0),
\end{equation}
and
\begin{equation}
\label{realDbar1}
\bar D[Q_p^n(q_0)(q_0-p)]=8n(q_0-p_0)Q_p^{n-1}(q_0)(q_0-p)+4Q_p^n(q_0).
\end{equation}
Therefore, the right-hand side of \eqref{polyseries} coincides with the series obtained by applying $\bar D$ term by term to \eqref{spherical} at the real points.
\end{proof}

Now we show that the polyanalytic spherical series of order $2$ of $h=\bar D f$ has the same  convergence  set as that of its "primitive" function $f$ and of the spherical series, and so the same radius of convergence of the harmonic and of the regular series.

\begin{proposition}
	\label{convepoly}
	With the notations in Theorem \ref{polyseries1}, let $ \{a_n \}_{n \in \mathbb{N}_0} \subseteq \mathbb{H}$ be such that
	$$ \limsup_{n \to \infty} |a_n|^{\frac{1}{n}}= \frac{1}{R}.$$
	The polyanalytic spherical series of order $2$ of $h=\bar{D}f$ converges (absolutely and uniformly on the compact subsets) in $U( p,R)$
	for $p \in \mathbb{H}$.
\end{proposition}
\begin{proof}
	Let $K$ be a compact subset of $U( p,R)$. If $q \in K$ then $| Q_p^n(q)| \leq r^{2n}$ for some $r$ such that $0<r<R$. Fix $R'$ with $r<R'<R$ so that $|a_n| \leq (1/R')^n$ for $n$ sufficiently large. By Theorem \ref{polyseries1} we know that a spherical polyanalytic series of order $2$ can be decomposed in two parts. We start by estimating the terms of the second series in  \eqref{polyseries}. By Lemma \ref{newres} we have
	\begin{eqnarray*}
		4\left|\mathcal{L}_{n+1}(q,p)\right|\left|q_0-p_0\right|
\left(\left|a_{2n+2}\right|+\left|\bar{q}-p\right|\left|a_{2n+3}\right|\right)  &\leq & \frac{4(n+1)}{R'^3} \left(\frac{r}{R'} \right)^{2n} \sqrt{r^2+p_1^2}\\
		&& \times \left[R'+ \left(\sqrt{r^2+p_1^2}+p_1\right)\right] \\
		&=:& \alpha_n.
	\end{eqnarray*}
	Hence the second series of \eqref{polyseries} is dominated by the series $ \sum_{n=0}^{\infty} \alpha_n$. We note that here we used $|\mathcal{L}_{n+1}(q,p)|\leq (n+1)r^{2n}$, see \eqref{key4},
and $|\bar q-p|\leq \sqrt{r^2+p_1^2}+p_1$ by \eqref{newin}.
	\newline
	We now focus on the first summation in \eqref{polyseries}. By Lemma \ref{r1} and Lemma \ref{newres} we get
	\begingroup\allowdisplaybreaks
	\begin{eqnarray*}
		4 \left|(n+1)(q-p_0)Q_p^{n}(q)[a_{2n+2}+(q-p)a_{2n+3} \right|&\leq& \frac{4(n+1)}{R'^3} \left(\frac{r}{R'} \right)^{2n} \sqrt{r^2+p_1^2}\\
		&& \times \left[R'+ \left(\sqrt{r^2+p_1^2}+p_1\right)\right]\\
		&&:=\beta_{n}.
	\end{eqnarray*}
	\endgroup
	Hence the first part of the expansion is dominated  by the convergent series $\sum_{n=0}^{\infty} \beta_n$.
	Finally the third series in \eqref{polyseries} can be estimated by $\sum_{n=0}^{\infty}\gamma_n$, where $\gamma_n:= \frac{4}{R'}\left(\frac{r}{R'}\right)^{2n}$.
	
	Thus we deduce that the polyanalytic spherical series is dominated by
	$$| \bar{D}f(q)| \leq \sum_{n=0}^{\infty} (\alpha_n+  \beta_n+\gamma_n).$$
	Since $\sum_{n=0}^{\infty} (\alpha_n+  \beta_n + \gamma_n)$ is convergent, by the ratio test, we get the result.
\end{proof}

We can write the polyanalytic spherical series of order $2$ in terms of the polyanalytic functions $ \mathcal{P}_{2,n}(q)$ defined in \eqref{poly4b}.

\begin{theorem}
	\label{Taylorpoly}
	Let $f$ be a slice hyperholomorphic function in an axially symmetric open set $U$. Let us assume  that $f$ admits a spherical series at an arbitrarily fixed point $p=p_0+Ip_1\in U$ as in \eqref{spherical} with coefficients $\{a_n\}_{n \in \mathbb{N}_0} \subseteq \mathbb{H}$,  convergent (absolutely and uniformly on the compact subsets) in $U(p,R)$, $R>0$. Then the polyanalytic spherical series of order $2$ of $h=\bar{D}f$, for $q\in U(p,R)$, can be written as
	\begingroup\allowdisplaybreaks
	\begin{eqnarray}
		\label{Taylorpoly2}
		\bar{D}f(q) &=&2 \left[\sum_{n=1}^{\infty}\sum_{k=0}^{n-1} \binom{n-1}{k} (2n)\mathcal{P}_{2,2k+1}(q-p_0)p_1^{2(n-k-1)}a_{2n} \right.\\
		\nonumber
		&&\left. +\sum_{n=0}^{\infty}\sum_{k=0}^{n} \binom{n}{k} (2k+1)\mathcal{P}_{2,2k}(q-p_0) p_1^{2(n-k)}a_{2n+1} \right.\\
		\nonumber
		&& \left. -\sum_{n=1}^{\infty} \sum_{k=0}^{n-1} \binom{n-1}{k} (2n) \mathcal{P}_{2, 2k+1}(q-p_0) p_1^{2(n-k)-1}I a_{2n+1}\right].
	\end{eqnarray}
	\endgroup
\end{theorem}
\begin{proof}
	By the binomial theorem and formula \eqref{poly4} we have
	\begingroup\allowdisplaybreaks
	\begin{eqnarray}
		\nonumber
		\bar{D}\left(Q_p^n(q)\right)&=&  \bar{D} \left(\sum_{k=0}^{n} \binom{n}{k} (q-p_0)^{2k} p_1^{2(n-k)} \right) \\
		\nonumber
		&=& 2\sum_{k=1}^{n} \binom{n}{k} 2k \mathcal{P}_{2,2k-1}(q-p_0) p_1^{2(n-k)}\\
		\label{f21}
		&=&2 \sum_{k=0}^{n-1} \binom{n-1}{k} (2n)\mathcal{P}_{2,2k+1}(q-p_0)p_1^{2(n-k-1)}.
	\end{eqnarray}
	\endgroup
	
	Now, we focus on the second part of the spherical series. By \eqref{poly4} we can write
	\begingroup\allowdisplaybreaks
	\begin{eqnarray}
		\nonumber
		&&\bar{D}\left(\sum_{k=0}^{n} \binom{n}{k} (q-p_0)^{2k+1}p_1^{2(n-k)}- \sum_{k=0}^{n} \binom{n}{k} (q-p_0)^{2k}p_1^{2(n-k)+1}I \right)\\
		\nonumber
		&&=2\sum_{k=0}^{n} \binom{n}{k} (2k+1)\mathcal{P}_{2,2k}(q-p_0) p_1^{2(n-k)}\\
		\nonumber
		&& \, \, \, \, \,-2 \sum_{k=1}^{n} \binom{n}{k} (2k) \mathcal{P}_{2, 2k-1}(q-p_0) p_1^{2(n-k)+1}I\\
		\nonumber
		&&=2\sum_{k=0}^{n} \binom{n}{k} (2k+1)\mathcal{P}_{2,2k}(q-p_0) p_1^{2(n-k)}\\
		\label{f22}
		&& \, \, \, \, \,-2 \sum_{k=0}^{n-1} \binom{n-1}{k} (2n) \mathcal{P}_{2, 2k+1}(q-p_0) p_1^{2(n-k)-1}I.
	\end{eqnarray}
	\endgroup
	
	By putting together \eqref{f21} and \eqref{f22} we get the result.

\end{proof}

\begin{remark}
	\label{zeropoly}
	Setting $p=0$ in \eqref{Taylorpoly2}, we observe that the nonzero terms correspond to $k=n-1$, and thus we obtain
	\begingroup\allowdisplaybreaks
	\begin{eqnarray}
		\label{zeropoly1}
		h(q)=\bar{D}f(q)&=& 2 \sum_{n=1}^{\infty} (2n)\mathcal{P}_{2, 2n-1}(q) a_{2n}+2 \sum_{n=0}^{\infty} (2n+1) \mathcal{P}_{2, 2n}(q) a_{2n+1}\\
		\nonumber
		&=& 2 \sum_{n=0}^{\infty} (n+1) \mathcal{P}_{2,n}(q)a_{n+1}.
	\end{eqnarray}
	\endgroup
	Formula \eqref{zeropoly1} is the Taylor expansion of a function  $ h\in\mathcal{AP}_2^L(U)$, where $U$ is an axially symmetric domain containing the origin. We observe that the same result was obtained in Remark \ref{rempoly} by taking $p=0$ in the polyanalytic regular series of order $2$.
\end{remark}	

In the next result we show an alternative way of writing a spherical polyanalytic series of
order 2.

\begin{proposition}
	\label{TaylorPoly1}
Let $f$ be a slice hyperholomorphic function in an axially symmetric open set $U$ that admits a spherical series at an arbitrarily fixed point $p=p_0+Ip_1\in U$ as in \eqref{spherical} with coefficients $\{a_n\}_{n \in \mathbb{N}_0} \subseteq \mathbb{H}$, convergent (absolutely and uniformly on the compact subsets) in $U(p,R)$, $R>0$. Then  for $q\in U(p,R)\setminus \mathbb{R}$ we can write the polyanalytic spherical series of order $2$ of $h=\bar{D}f$ as
     \begingroup\allowdisplaybreaks
    \begin{eqnarray*}
       h(q)= \bar{D}f(q) &=& 2\partial_{q_0}\left[ \sum_{n=0}^{\infty} Q_p^n(q) a_{2n}+Q_p^{n}(q)(q-p) a_{2n+1} \right] \\
        &&+( \underline{q})^{-1} \left( \sum_{n=0}^{\infty} Q_p^n(q) a_{2n}+Q_p^n(q)(q-p) a_{2n+1} \right.\\
        && \left.-\sum_{n=0}^{\infty} (Q_p^n(\bar{q}) a_{2n}+Q_p^n(\bar{q})(\bar{q}-p) a_{2n+1}) \right).
    \end{eqnarray*}
    \endgroup
\end{proposition}
\begin{proof}
    By using \eqref{starLeR} and formula \eqref{poly5} we get
    \begingroup\allowdisplaybreaks
    \begin{eqnarray}
        \nonumber
        \bar{D} \left(Q_p^n(q) \right) &=&\bar{D} \left(  \sum_{ k=0}^{n} \binom{n}{k} (q-p_0)^{2k}p_1^{2(n-k)} \right) \\
        \nonumber
        &=& 2 \sum_{k=1}^{n} \binom{n}{k}\partial_{q_0} (q-p_0)^{2k}p_1^{2(n-k)}-(\underline{q})^{-1} \sum_{k=1}^{n} \binom{n}{k} (\bar{q}-p_0)^{2k} p_1^{2(n-k)} \\
        \nonumber
        && +(\underline{q})^{-1} \sum_{k=1}^{n} \binom{n}{k} (q-p_0)^{2k} p_1^{2(n-k)}\\
        \nonumber
        &=& 2 \partial_{q_0} \left(\sum_{k=0}^{n} \binom{n}{k} (q-p_0)^{2k}p_1^{2(n-k)}\right)- (\underline{q})^{-1}\left[ (\bar{q}-p_0)^2+p_1^2 \right]^n\\
        \nonumber
        && +(\underline{q})^{-1}\left[ (q-p_0)^2+p_1^2 \right]^n\\
        \label{ex0}
        &=& 2 \partial_{q_0} \left[ Q_p^n(q)\right]- \underline{q}^{-1} \left[ Q_p^n(\bar{q})-Q_p^n(q) \right].
    \end{eqnarray}
    \endgroup
    By using similar arguments for the second part of spherical series we have
    \begin{eqnarray}
        \nonumber
        \bar{D}\left( Q_p^n(q) (q-p) \right)&=& 2\left( \sum_{k=1}^{n} \binom{n}{k} \partial_{q_0} (q-p_0)^{2k+1}p_1^{2(n-k)}-\sum_{k=1}^{n} \binom{n}{k} \partial_{q_0} (q-p_0)^{2k} p_1^{2(n-k)+1} I\right)\\
        \nonumber
        && - (\underline{q})^{-1} \left( \sum_{k=0}^{n} \binom{n}{k} (\bar{q}-p_0)^{2k+1}p_1^{2(n-k)}- \sum_{k=0}^{n} \binom{n}{k}(\bar{q}-p_0)^{2k} p_1^{2(n-k)+1} I \right)+\\
        \nonumber
        && +(\underline{q})^{-1} \left(\sum_{k=0}^{n} \binom{n}{k} (q-p_0)^{2k+1}p_1^{2(n-k)}-\sum_{k=0}^{n} \binom{n}{k}(q-p_0)^{2k} p_1^{2(n-k)+1} I \right)\\
        \label{ex11}
        &=& 2 \partial_{q_0} \left(Q_p^n(q)(q-p) \right)- \underline{q}^{-1} \left[ Q_p^n(\bar{q})(\bar{q}-p)- Q_p^n(q)(q-p) \right].
    \end{eqnarray}
    The result follows by putting together \eqref{ex0} and \eqref{ex11}.
\end{proof}

We note that, similarly to the cases of axially harmonic and regular functions, the connection between the truncated polyanalytic regular and spherical series of order 2 can also be expressed using the right global operator.

\begin{theorem}
	
	Let $p \in \mathbb{H}\setminus\mathbb{R}$, $ \{a_n \}_{n \in \mathbb{N}_0} \subseteq \mathbb{H}$ be a given sequence. Define the sequence $ \{b_n\}_{n \in \mathbb{N}_0} \subseteq \mathbb{H}$ by
    $$
b_{n-1}=-2n \left(c_n a_{2n}+c_{n-1}a_{2n-1} \right), \qquad n \geq 1,
$$
where $c_{n}:= 2^n n! (-1)^n$. Then we have the following connection between the truncated polyanalytic regular series and the truncated polyanalytic spherical series of order $2$:
    \begin{eqnarray*}
	\sum_{n=0}^{N}  \widetilde{P}_{2, n}(q,p) b_n&=& - 4 \sum_{n=0}^{N} V_p^{n+1}
	\left[\mathcal L_{n+1}(q,p)(q_0-p_0)+(n+1)Q_p^n(q)(q-p_0)\right]a_{2n+2}\\
	&&-4\sum_{n=0}^{N-1}V_p^{n+1}\left[\mathcal L_{n+1}(q,p)(q_0-p_0)(\bar q-p)+(n+1)Q_p^n(q)(q-p_0)(q-p)\right]a_{2n+3}\\
&&-4 \sum_{n=0}^{N} V_p^n(Q_p^n(q))a_{2n+1}
\end{eqnarray*}
where $N \in \mathbb{N}_0$ is fixed.
\end{theorem}
\begin{proof}
    The result follows by using arguments similar to those used to prove Theorem \ref{conna}.
\end{proof}

\section{Laurent series: axially polyanalytic functions of order $2$ }\label{Sect:10}

In this section we complete the study by investigating the Laurent series of $h=\bar{D}f$, where $f$ is a slice hyperholomorphic function which is expanded in series at a point $p \in \mathbb{H}$.

\subsection{Laurent polyanalytic regular series of order 2}

We now introduce and study a series  around a generic quaternion $p$, which  we call Laurent polyanalytic regular series of order $2$. It is based on the application of the operator $\bar{D}$ to a slice hyperholomorphic function expanded as a $*$-Laurent series.

\begin{definition}
Let $\Omega \subseteq\mathbb{H}$ be an axially symmetric open set and let $f$ be a function slice hyperholomorphic in $\Omega$ that admits a  $*$-Laurent expansion at $p\in\mathbb H$ convergent in a set contained in $\Omega$. We say that $h=\bar{D}f$ has a Laurent polyanalytic regular series of order 2 at $p$.
\end{definition}
To study these series we compute $\bar{D}(q-p)^{-n*_{q,L}}$. As customary in this work, we write $*$ to denote the  $*_{q,L}$-products.

\begin{theorem}
	\label{LauNpoly}
Let $q,p \in \mathbb{H}$ satisfy $q \notin [p]$, and let $\bar{D}$ denote the conjugate Fueter operator with respect to $q$. Then for $n \geq 1$ we have
	$$ \bar{D}(q-p)^{-*n}=-2 \left(n(\bar{q}-p)^{*(n+1)}+ \sum_{k=0}^{n-1} (\bar{q}-p)^{*(n-k)}*(q-p)^{*(k+1)}\right) \mathcal{Q}_{c,p}^{-n-1}(q).$$
\end{theorem}
\begin{proof}
 We start by observing that $ D+ \bar{D}=2 \partial_{q_0}$, so by Theorem \ref{actionN1} we have
	\begin{eqnarray}
		\nonumber
		\bar{D}(q-p)^{-*n}&=& 2 \partial_{q_0} (q-p)^{-*n}-D (q-p)^{-*n}\\
		\nonumber
		&=&-2n (q-p)^{-*(n+1)}-2 \sum_{k=0}^{n-1} (\bar{q}-p)^{*(n-k-1)}*(q-p)^{*k} \mathcal{Q}_{c,p}^{-n}(q)\\
		\label{polynegative}
		&=&-2n (q-p)^{-*(n+1)}-2 \sum_{k=0}^{n-1} (\bar{q}-p)^{*(n-k)}*(q-p)^{*(k+1)} \mathcal{Q}_{c,p}^{-n-1}(q).
	\end{eqnarray}
	By Proposition \ref{powerN} we get that
	$$ \left[ (q-p)^{*(n+1)}*(\bar{q}-p)^{*(n+1)} \right] \mathcal{Q}_{c,p}^{-n-1}(q)=1,$$
	which implies
	\begin{equation}
		\label{negative}
		(q-p)^{-*(n+1)}= (\bar{q}-p)^{*(n+1)} \mathcal{Q}_{c,p}^{-n-1}(q).
	\end{equation}
	Thus, by plugging \eqref{negative} in \eqref{polynegative} we have
	\begin{eqnarray*}
		\bar{D}(q-p)^{-*n}&=& -2n (\bar{q}-p)^{*(n+1)} \mathcal{Q}_{c,p}^{-n-1}(q)-2 \left( \sum_{k=0}^{n-1} (\bar{q}-p)^{*(n-k)}*(q-p)^{*(k+1)} \right) \mathcal{Q}_{c,p}^{-n-1}(q)\\
		&=&-2 \left(n(\bar{q}-p)^{*(n+1)}+ \sum_{k=0}^{n-1} (\bar{q}-p)^{*(n-k)}* (q-p)^{*(k+1)}\right) \mathcal{Q}_{c,p}^{-n-1}(q),
	\end{eqnarray*}
	which proves the result.
\end{proof}
The above formula can be rewritten as
\begin{equation}
	\label{appPolyN}
	\bar{D}(q-p)^{-*n}=-2 \mathcal{R}_{2,n}(q,p) \mathcal{Q}_{c,p}^{-n-1}(q),
\end{equation}
where
\begin{equation}
\mathcal{R}_{2,n}(q,p):=n(\bar{q}-p)^{*(n+1)}+ \sum_{k=0}^{n-1} (\bar{q}-p)^{*(n-k)}*(q-p)^{*(k+1)}.
\end{equation}
The functions $\mathcal{R}_{2,n}(q,p)$ defined above possess the properties described in the next result.
\begin{proposition} Let $n \in \mathbb{N}$ and $p$, $q \in \mathbb{H}$ such that $q \notin [p]$. Then we have
	\begin{itemize}
		\item[1)] $\mathcal{R}_{2,n}(q,p)$ satisfy
		\begin{equation}
			\label{polyharmoN}
			\mathcal{R}_{2,n}(q,p)= n (\bar{q}-p)^{*(n+1)}+ \mathcal{H}_{n}(q,p) \mathcal{Q}_{c,p}(q),
		\end{equation}
where $\mathcal{H}_{n}(q,p)$ are defined in \eqref{harmN}.
		\item[2)] $\mathcal{R}_{2,n}(q,p)$  satisfy
		$$ \mathcal{R}_{2,n}(q,p)= -2 \mathcal{M}_n(q,p)p-2 \mathcal{M}_{n-1}(q,p) \mathcal{Q}_{c,p}(q)+2 q_0 \mathcal{M}_n(q,p),$$
where $\mathcal{M}_n(q,p)$ are defined in \eqref{calMn}.
		\item [3)] The function $\mathcal{R}_{2,n}(q,p) \mathcal{Q}_{c,p}^{-n-1}(q)$ is axially polyanalytic of order $2$.
		\item[4)] The functions $\mathcal{R}_{2,n}(q,p)$ are left slice polyanalytic of order $n+2$ in $q$, and are right slice hyperholomorphic in the variable $p$.
		\item[5)] For $q \notin \mathbb{R}$ we have
		\begin{equation}
			\label{polyLauN}
			\mathcal{R}_{2,n}(q,p)= n ( \bar{q}-p)^{*(n+1)}+\frac{(\underline{q})^{-1}}{2} \left[ (q-p)^{*n}- (\bar{q}-p)^{*n}\right] \mathcal{Q}_{c,p}(q)
		\end{equation}
	\end{itemize}
\end{proposition}
\begin{proof}

	\begin{itemize}
		\item[1)] It follows by \eqref{harmN} and by the fact that $\mathcal{Q}_{c,p}(q)=(q-p) *(\bar{q}-p)$.
		\item[2)] By \eqref{appPolyN}, \eqref{starL} and \eqref{polykernel} we get
		\begin{eqnarray*}
			\mathcal{R}_{2,n}(q,p)&=&- \frac{1}{2} \left(\bar{D} (q-p)^{-*n} \right) \mathcal{Q}_{c,p}^{n+1}(q)\\
			&=& \frac{(-1)^{n+1}}{2(n-1)!} \left( \partial_{q_0}^{n-1} \bar{D} S^{-1}_L(p,q) \right) \mathcal{Q}_{c,p}^{n+1}(q)\\
			&=& \frac{(-1)^{n+1}}{2(n-1)!} \left(\partial_{q_0}^{n-1}   P_2^L(p,q)\right)\mathcal{Q}_{c,p}^{n+1}(q).
		\end{eqnarray*}
		Now by using the expression of the polyanalytic kernel $P_2^L(p,q)$ and the Leibniz rule we get
		\begin{eqnarray*}
			\mathcal{R}_{2,n}(q,p)&=& \frac{(-1)^{n+1}}{2(n-1)!} \partial_{q_0}^{n-1} [-F_L(p,q)p] \mathcal{Q}_{c,p}^{n+1}(q)+ \frac{(-1)^n}{2 (n-1)!} \partial_{q_0}^{n-1} \left[q_0 F_L(p,q) \right] \mathcal{Q}_{c,p}^{n+1}(q)\\
			&=&\frac{(-1)^{n+1}}{2(n-1)!} \partial_{q_0}^{n-1} [-F_L(p,q)]p\mathcal{Q}_{c,p}^{n+1}(q)+\frac{(-1)^n}{2(n-1)!} \left[q_0 \partial_{q_0}^{n-1}F_L(p,q)\right. \\
			&&\left. +(n-1) \partial_{q_0}^{n-2} F_L(p,q) \right] \mathcal{Q}_{c,p}^{n+1}(q).
		\end{eqnarray*}
		Since the $F$-kernel $F_L(p,q)$ is axially Fueter regular in $q$ we can replace $ \partial_{q_0}^{n-1}$ with $ \frac{\bar{D}^{n-1}}{2^{n-1}}$. Thus by formula \eqref{derivativeF} we have
		\begin{eqnarray*}
			\mathcal{R}_{2,n}(q,p)&=& \frac{(-1)^{n}}{ (n-1)! 2^{n}} \left(\bar{D}^{n-1}F_L(p,q)\right)\mathcal{Q}_{c,p}^{n+1}(q)p	+ \frac{(-1)^n}{(n-1)! 2^{n}}q_0 \left(\bar{D}^{n-1}F_L(p,q)\right) \mathcal{Q}_{c,p}^{n+1}(q)\\
			&&+ \frac{(-1)^{n+1}}{2^{n-1} (n-2)!} \left(\bar{D}^{n-2}F_L(p,q) \right) \mathcal{Q}_{c,p}^{n+1}(q)\\
			&=& -2 \mathcal{M}_n(q,p) p+2 q_0 \mathcal{M}_n(q,p)-2 \mathcal{M}_{n-1}(q,p) \mathcal{Q}_{c,p}(q),
		\end{eqnarray*}
		as stated.
		\item[3)] Since for $ q \notin [p]$ the function $ (q-p)^{-*n}$ is slice hyperholomorphic by \eqref{appPolyN} and the Fueter mapping theorem we get that the function $ \mathcal{R}_{2,n}(q,p)\mathcal{Q}_{c,p}^{-n-1}(q)$ is axially polyanalytic of order 2.
		\item[4)]  By \eqref{polyharmoN} and reasoning as in the proof of point 2) of Proposition \ref{harmpoly} we get that the functions $\mathcal{R}_{2,n}(q,p) $ are left slice polyanalytic of order $n+2$. By formula \eqref{polyharmoN} and point 2) of Proposition \ref{harmLauN}, we deduce that the functions $\mathcal{R}_{2,n}(q,p)$ are right slice hyperholomorphic in $p$.
		\item[5)] Formula \eqref{polyLauN} follows by combining formula \eqref{closedN1} and \eqref{polyharmoN}.
	\end{itemize}
\end{proof}

\begin{remark}
If we set $ p=0 $ in \eqref{appPolyN}, we obtain
\begin{equation}
	\label{appbar}
	\bar{D}(q^{-n})=-2R_n(q)|q|^{-2(n+1)}, \qquad n \in \mathbb{N},
\end{equation}
where the polynomials $ R_n(q) $ are given by
\begin{equation}
	\label{polyR}
	R_n(q)=\mathcal{R}_{2,n}(q,0)= n \bar{q}^{n+1}+ \sum_{i=0}^{n-1} \bar{q}^{n-i}q^{i+1}.
\end{equation}
Furthermore, if $q \in \mathbb{H}\setminus \mathbb{R} $, by \eqref{polyLauN} with $p=0$, we have
\begin{equation}
	\label{closedN3}
	R_n(q)=n \bar{q}^{n+1} + \frac{(\underline{q})^{-1}|q|^2 }{2}(q^{n}- \bar{q}^{n}).
\end{equation}

\end{remark}

The Laurent polyanalytic regular series of order $2$ can be written in a more compact way.
\begin{theorem}
	\label{PolyL}
Let $\Omega \subseteq\mathbb{H}$ be an axially symmetric open set and let $f$ be a slice hyperholomorphic function in $\Omega$ that admits $*$-Laurent expansion at an arbitrarily fixed point $p\in\mathbb H$
\begin{equation}
	\label{NNL123}
	f(q)= \sum_{n=0}^{\infty} (q-p)^{*n}a_n+ \sum_{n=1}^{\infty} (q-p)^{-*n}a_{-n}, \qquad \{a_n\}_{n \in \mathbb{Z}} \subset \mathbb{H},
\end{equation}
convergent (absolutely and uniformly on the compact subsets) in $\Omega(p,R_1, R_2) \subset \Omega$
where $ \frac{1}{R_2}= \limsup_{n \to \infty} |a_n|^{\frac{1}{n}}$ and $R_1= \limsup_{n \to \infty} |a_{-n}|^{\frac{1}{n}}$. Then we can write the Laurent polyanalytic regular series of order $2$ of $h=\bar{D}f$ as
	\begin{eqnarray}
		\label{auxPoly0}
		h(q)=\bar{D}f(q)&=&\sum_{n=0}^{\infty}  \widetilde{P}_{2,n}(q,p) c_n+ \sum_{n=1}^{\infty} \mathcal{R}_{2,n}(q,p) \mathcal{Q}_{c,p}^{-n-1}(q)c_{-n},
	\end{eqnarray}
where $q \in \Omega(p,R_1, R_2)$,  $\{c_n\}_{n \geq 0}:=	\{2 (n+1)a_{n+1}\}_{n \geq 0}$ and $\{c_{-n}\}_{n \geq 1}:=	\{-2a_{-n}\}_{n \geq 1}$.

\end{theorem}
\begin{proof}
By formulas \eqref{poly4} and \eqref{appPolyN} we have
	\begin{eqnarray*}
	\bar{D}f(q)&=& \sum_{n=0}^{\infty} \bar{D} (q-p)^{*n} a_n + \sum_{n=1}^{\infty} \bar{D}(q-p)^{-*n}a_{-n}\\
	&=& 2 \sum_{n=0}^{\infty} (n+1) \widetilde{P}_{2,n}(q,p) a_{n+1}-2 \sum_{n=1}^{\infty} \mathcal{R}_{2,n}(q,p) \mathcal{Q}_{c,p}^{-n-1}(q)a_{-n}\\
	&=& \sum_{n=0}^{\infty}  \widetilde{P}_{2,n}(q,p) c_n+ \sum_{n=1}^{\infty} \mathcal{R}_{2,n}(q,p) \mathcal{Q}_{c,p}^{-n-1}(q)c_{-n}.
\end{eqnarray*}
The convergence of the first series follows by Theorem \ref{conve4}. We focus on the second series. If $q \in \mathbb{R}$ the result is trivial. If we consider $q \notin \mathbb{R}$, by formula \eqref{appPolyN} and Theorem \ref{polyrap}, and taking $q_{\pm I}\in [q]$ if $p\in\mathbb{C}_I$, we get
	\begin{eqnarray*}
		-2\mathcal{R}_{2,n}(q,p) \mathcal{Q}_{c,p}^{-n-1}(q)a_{-n}&=&\bar{D}(q-p)^{-*n}a_{-n}\\
		&=&-n \left[ (q_I-p)^{-(n+1)}+ (q_{-I}-p)^{-(n+1)}\right]a_{-n}\\
		&&-nI_{\underline{q}}I \left[ (q_{-I}-p)^{-(n+1)}- (q_{I}-p)^{-(n+1)}\right]a_{-n}\\
		&& +(\underline{q})^{-1}I_{\underline{q}}I \left[ (q_{-I}-p)^{-n}- (q_{I}-p)^{-n}\right]a_{-n}.
	\end{eqnarray*}
	So we have
	\begin{eqnarray*}
		|2\mathcal{R}_{2,n}(q,p) \mathcal{Q}_{c,p}^{-n-1}(q)a_{-n}|& \leq & 2 n \left[ \left|(q_I-p)^{-(n+1)}a_{-n}\right|+ \left|(q_{-I}-p)^{-(n+1)}a_{-n}\right| \right]\\
		&& +|\underline{q}|^{-1} \left[ \left|(q_I-p)^{-n}a_{-n}\right|+ \left|(q_{-I}-p)^{-n}a_{-n}\right|\right].
	\end{eqnarray*}
	Since $q \in \Omega(p,R_1, R_2)$  we get that the series
	$$ \sum_{n=1}^{\infty} \left|(q_{\pm I}-p)^{-(n+1)}a_{-n}\right|, \qquad \hbox{and} \qquad \sum_{n=1}^{\infty} \left|(q_{\pm I}-p)^{-n}a_{-n}\right|,$$
	are convergent. So we have that the series in \eqref{auxPoly0} is convergent where stated.

\end{proof}

\begin{remark}
	If we consider $p=0$ in \eqref{auxPoly0}, and using the fact that $\mathcal{Q}_{c,0}(q)=|q|^2$ we get the following expression of a polyanalytic Laurent series of order 2 in a neighbourhood of the origin:
	
	\begin{equation}
		\label{polyN1}
		\bar{D}f(q)= \sum_{n=0}^{\infty} \mathcal{P}_{2,n}(q)c_n+ \sum_{n=1}^{\infty} |q|^{-2(n+1)} R_n(q) c_{-n},
	\end{equation}
	where $\{c_n\}_{n \geq 0}:= \{2(n+1)a_{n+1}\}_{n \geq 0}$ and $\{c_{-n}\}_{n \geq 1}:= \{-2a_{-n}\}_{n \geq 1}$.
	
\end{remark}

Similarly to what happens in the case of  Laurent series for Fueter regular or axially harmonic functions, also the Laurent polyanalytic regular series can be written in terms of the kernel of the integral representation of the axially polyanalytic function of order $2$.

\begin{proposition}
	\label{polyL5}
Let $f$ be a slice hyperholomorphic function that admits a $*$-Laurent expansion at an arbitrarily fixed point $p$ as in \eqref{NNL123} with coefficients $ \{a_n\}_{n \in \mathbb{Z}} \subseteq \mathbb{H}$,  convergent (absolutely and uniformly on the compact subsets) in $\Omega(p,R_1,R_2)$, with $R_1$, $R_2$ as in Theorem \ref{PolyL}. Then for $q \in \Omega(p,R_1,R_2) \setminus\mathbb{R}$ the Laurent polyanalytic series of order $2$,  can be written as
\begin{equation}
\label{polyLLL}
h(q)=\bar{D}f(q)= \sum_{n=0}^{\infty}  \widetilde{P}_{2,n}(q,p) b_n+ \sum_{n=1}^{\infty} \frac{(-1)^n}{(n-1)!} \partial_{q_0}^{n-1} P_2^L(p,q)b_{-n}.
\end{equation}
where $\{ b_n\}_{n \geq 0}:=\{2(n+1)a_{n+1}\}_{n \geq 0}$ and $\{b_{-n}\}_{n \geq 1}:=\{a_{-n}\}_{n \geq 1}$.
\end{proposition}
\begin{proof}
	By Proposition \ref{Lau5}, formula \eqref{polykernel}, expression \eqref{polyf} and the fact that the conjugate Fueter operator commutes with $ \partial_{q_0}$ we get
	\begin{eqnarray*}
		h(q)=\bar{D}f(q)&=& \sum_{n=0}^{\infty} \bar{D} (q-p)^{*n} a_n + \sum_{n=1}^{\infty} \frac{(-1)^n}{(n-1)!} \partial_{q_0}^{n-1}  \bar{D} S^{-1}_L(p,q)a_{-n}\\
		&=& 2 \sum_{n=1}^{\infty} n \widetilde{P}_{2,n-1}(q,p) a_n + \sum_{n=1}^{\infty} \frac{(-1)^n}{(n-1)!} \partial_{q_0}^{n-1} P_2^L(p,q)a_{-n}\\
		&=& 2 \sum_{n=0}^{\infty} (n+1) \widetilde{P}_{2,n}(q,p) a_{n+1} + \sum_{n=1}^{\infty} \frac{(-1)^n}{(n-1)!} \partial_{q_0}^{n-1} P_2^L(p,q)a_{-n}.
	\end{eqnarray*}
Finally, by setting the coefficients $ \{b_n\}_{n \in \mathbb{Z}}$ as in the statement we get the result.
\end{proof}
\begin{remark}
	The Laurent polyanalytic regular series of order 2 can be written in terms of the slice hyperholomorphic Cauchy kernel by using Proposition \ref{corr}.
	\\Moreover if we assume that $p=0$ in \eqref{polyLLL} we obtain a Laurent polyanalytic series of order 2 in a neighbourhood of the origin in terms of the polyanalytic Cauchy kernel (see \eqref{polycauchy}):
	
		$$ \bar{D} f(q)= \sum_{n=0}^{\infty} \mathcal{P}_{2, n}(q) b_n+ \sum_{n=1}^{\infty} \frac{4(-1)^n}{(n-1)!} \partial_{q_0}^{n-1}\left(q_0E(q) \right)b_{-n},$$
	where $\{b_n\}_{n \geq 0}:= \{2 (n+1)a_{n+1}\}_{n \geq 0}$ and $\{b_{-n}\}_{n \geq 1}:=\{a_{-n}\}_{n \geq 1}$.	
\end{remark}

In \cite{Polyf2} the authors proved that the polyanalytic kernel can be written in terms of the polyanalytic polynomials $ \mathcal{P}_{2,n}(q)$, see \eqref{poly4b}. Precisely, we have
\begin{equation}
	\label{polyser}
	P_{2}^L(p,q)= 2 \sum_{n=0}^{\infty} (n+1) \mathcal{P}_{2,n}(q)p^{-2-n}, \qquad |q|<|p|.
\end{equation}
The above expansion and \eqref{poly1} pave the way to get a Laurent polyanalytic regular series in terms of the Clifford-Appell polynomials $ \mathcal{Q}_n(q)$.

\begin{proposition}
Let $f$ be a slice hyperholomorphic function that admits a $*$-Laurent expansion centered at an arbitrarily fixed point $p$ as in \eqref{Lau1} whose coefficients are $ \{a_n\}_{n \in \mathbb{Z}} \subseteq \mathbb{H}$,  convergent (absolutely and uniformly on the compact subsets) in $\Omega(p,R_1,R_2)$. Assume moreover that the region $|q|<|p|$ has nonempty intersection with $\Omega(p,R_1,R_2)$. Then, on that intersection, we can write the Laurent polyanalytic regular series of order $2$ of $h=\bar{D}f$ as
	\begingroup\allowdisplaybreaks
	\begin{eqnarray}
		\nonumber
		h(q)=\bar{D}f(q)&=&\sum_{n=0}^{\infty} \widetilde{P}_{2,n}(q,p)b_{n}+ 2 \sum_{n=1}^{\infty} \sum_{k=n-1}^{\infty} \binom{k}{n-1} (k+1)_2 \mathcal{Q}_{k-n+1}(q) p^{-2-k}b_{-n}\\
		\nonumber
		&&-2q_0 \sum_{n=1}^{\infty} \sum_{k=n}^{\infty} \binom{k-1}{n-1} \mathcal{Q}_{k-n}(q)p^{-2-k}b_{-n}\\
		\label{polyL3}
		&&-2 \sum_{n=2}^{\infty} \sum_{k=n}^{\infty} \binom{k-1}{n-2} (k)_2 \mathcal{Q}_{k-n+1}(q)p^{-2-k}b_{-n}.
	\end{eqnarray}
	\endgroup
	Moreover, if $q \notin \mathbb{R}$ we can write the Laurent polyanalytic regular series of order $2$ as
	\begin{eqnarray*}
		h(q)=\bar{D}f(q)&=& 2 \partial_{q_0}\left(\sum_{n \in \mathbb{Z}}  (q-p)^{*n}a_n \right)- (\underline{q})^{-1} \sum_{n \in \mathbb{Z}} \left[ (\bar{q}-p)^{*n}-(q-p)^{*n} \right]a_n.
	\end{eqnarray*}
\end{proposition}
\begin{proof}
We first assume that the region $|q|<|p|$ has a nonempty intersection with the domain of convergence of the $*$-Laurent series. So, using the expansion of the Laurent polyanalytic regular series of order $2$ proved in \eqref{polyLLL} and series expansion of the polyanalytic kernel, see \eqref{polyser}, we have
	\begin{eqnarray}
		\nonumber
		\bar{D}f(q)&=&\sum_{n=0}^{\infty} \widetilde{P}_{2,n}(q,p)b_{n}+ \sum_{n=1}^{\infty} \frac{(-1)^n}{(n-1)!} \partial_{q_0}^{n-1} P_2^L(p,q)b_{-n}\\
		\label{polyL0}
		&=& \sum_{n=0}^{\infty} \widetilde{P}_{2,n}(q,p)b_{n}+ 2\sum_{n=1}^{\infty} \sum_{k=0}^{\infty} \frac{k+1}{(n-1)!}\partial_{q_0}^{n-1} \mathcal{P}_{2,k}(q) p^{-2-k} b_{-n}.
	\end{eqnarray}
In order to show formula \eqref{polyL3} we plug into \eqref{polyL0} the polyanalytic decomposition of the polynomials $ \mathcal{P}_{2,k}(q)$, see \eqref{poly1}, and we obtain
	\begin{eqnarray}
		\nonumber
		\bar{D}f(q)&=& \sum_{n=0}^{\infty} \widetilde{P}_{2,n}(q,p)b_{n}+ 2 \sum_{n=1}^{\infty} \sum_{k=n-1}^{\infty} \frac{(k+1)(k+2)}{(n-1)!} \partial_{q_0}^{n-1} \mathcal{Q}_k(q) p^{-2-k}b_{-n}\\
		\label{cliffL}
		&&-2 \sum_{n=1}^{\infty} \sum_{k=n-1}^{\infty}\frac{k(k+1)}{(n-1)!} \partial_{q_0}^{n-1} \left(q_0 \mathcal{Q}_{k-1}(q) \right)p^{-2-k}b_{-n}.
	\end{eqnarray}
	Since the Clifford-Appell polynomials are axially Fueter regular by formula  \eqref{appProp} we get
	\begin{equation}
		\label{polyL4}
		\partial_{q_0} \mathcal{Q}_n(q)=n \mathcal{Q}_{n-1}(q), \qquad n \geq 1.
	\end{equation}
The Leibniz rule yields
	\begin{eqnarray}
		\nonumber
		\partial_{q_0}^{n-1} \left(q_0 \mathcal{Q}_{k-1}(q) \right)&=& \sum_{\ell=0}^{n-1} \binom{n-1}{\ell} \left( \partial_{q_0}^{n-1-\ell}q_0 \right) \left(\partial_{q_0}^\ell \mathcal{Q}_{k-1}(q) \right)\\
		\nonumber
		&=& q_0 \left( \partial_{q_0}^{n-1} \mathcal{Q}_{k-1}(q) \right)+ \binom{n-1}{n-2} \left( \partial_{q_0} q_0 \right)  \left( \partial_{q_0}^{n-2} \mathcal{Q}_{k-1}(q)\right)\\
		\label{polyL2}
		&=&q_0 \frac{(k-1)!}{(k-n)!} \mathcal{Q}_{k-n}(q)+(n-1) \frac{(k-1)!}{(k-n+1)!} \mathcal{Q}_{k-n+1}(q).
	\end{eqnarray}
	Finally by plugging \eqref{polyL4} and \eqref{polyL2} into \eqref{cliffL} we arrive at
	\begin{eqnarray*}
		\nonumber
		\bar{D}f(q)&=& \sum_{n=0}^{\infty} \widetilde{P}_{2,n}(q,p)b_{n}+ 2\sum_{n=1}^{\infty} \sum_{k=n-1}^{\infty} \frac{(k+1)(k+2)k!}{(n-1)!(k-n+1)!} \mathcal{Q}_{k-n+1}(q) p^{-2-k}b_{-n}\\
		&&-2q_0 \sum_{n=1}^{\infty} \sum_{k=n}^{\infty}\frac{(k-1)!}{(n-1)!(k-n)!} \mathcal{Q}_{k-n}(q)p^{-2-k}b_{-n}\\
		&&-2 \sum_{n=1}^{\infty} \sum_{k=n}^{\infty}\frac{k(k+1)(k-1)!}{(n-2)!(k-n+1)!} \mathcal{Q}_{k-n+1}(q)p^{-2-k}b_{-n}\\
		&=&	 \sum_{n=0}^{\infty} \widetilde{P}_{2,n}(q,p)b_{n}+ 2 \sum_{n=1}^{\infty} \sum_{k=n-1}^{\infty} \binom{k}{n-1} (k+1)_2 \mathcal{Q}_{k-n+1}(q) p^{-2-k}b_{-n}\\
		&&-2q_0 \sum_{n=1}^{\infty} \sum_{k=n}^{\infty} \binom{k-1}{n-1} \mathcal{Q}_{k-n}(q)p^{-2-k}b_{-n}\\
		&&-2 \sum_{n=1}^{\infty} \sum_{k=n}^{\infty} \binom{k-1}{n-2} (k)_2 \mathcal{Q}_{k-n+1}(q)p^{-2-k}b_{-n}.
	\end{eqnarray*}
	This proves formula \eqref{polyL3}. Now, we suppose that $q \notin \mathbb{R}$. By Theorem \ref{PolyL} and formula \eqref{NNpoly} we have that
	\begin{equation}
		\label{polyNL}
		2 \sum_{n=0}^{\infty} (n+1) \widetilde{P}_{2,n}(q,p)a_{n+1}= 2 \sum_{n=0}^{\infty} \partial_{q_0} (q-p)^{* n}a_n- (\underline{q})^{-1} \sum_{n=0}^{\infty} \left[(\bar{q}-p)^{*n}-(q-p)^{*n} \right]a_{n}.
	\end{equation}
	By \eqref{polyLauN} and Proposition \ref{powerN} we have
	\begin{eqnarray}
		\nonumber
		-2\sum_{n=1}^{\infty} \mathcal{R}_{2,n}(q,p) \mathcal{Q}_{c,p}^{-n-1}(q)a_{-n}&=&-2 \sum_{n=1}^{\infty} n (\bar{q}-p)^{*(n+1)} \mathcal{Q}_{c,p}^{-n-1}(q)a_{-n}\\
		\nonumber
		&&- (\underline{q})^{-1} \sum_{n=1}^{\infty} \left[ (q-p)^{*n}-(\bar{q}-p)^{*n}\right] \mathcal{Q}_{c,p}^{-n}(q)a_{-n}\\
		\label{polyNL1}
		&=& 2 \sum_{n=1}^{\infty} \partial_{q_0} (q-p)^{-*n}a_{-n}- (\underline{q})^{-1} \sum_{n=1}^{\infty} \left[(\bar{q}-p)^{-*n}-(q-p)^{-*n}\right]a_{-n}.
	\end{eqnarray}
Thus, by \eqref{auxPoly0} and by combining \eqref{polyNL} and \eqref{polyNL1}, we obtain

	\begin{eqnarray*}
		\bar{D}f(q)&=&2 \sum_{n=0}^{\infty} \partial_{q_0} (q-p)^{* n}a_n- (\underline{q})^{-1} \sum_{n=0}^{\infty} \left[(\bar{q}-p)^{*n}-(q-p)^{*n} \right]a_{n}\\
		&& + 2 \sum_{n=1}^{\infty} \partial_{q_0} (q-p)^{-*n}a_{-n}- (\underline{q})^{-1} \sum_{n=1}^{\infty} \left[(\bar{q}-p)^{-*n}-(q-p)^{-*n}\right]a_{-n}\\
		&=& 2 \partial_{q_0}\left(\sum_{n \in \mathbb{Z}}  (q-p)^{*n}a_n \right)- (\underline{q})^{-1}\sum_{n \in \mathbb{Z}} \left[  (\bar{q}-p)^{*n}-(q-p)^{*n} \right]a_n,
	\end{eqnarray*}
which proves the last assertion.
\end{proof}

\subsection{Laurent polyanalytic spherical series}

We now discuss a Laurent polyanalytic series of order $2$ centered at a generic quaternion $p$ that converges in an open Euclidean neighbourhood.

\begin{definition}
	Let $ \Omega$ be an axially symmetric open set. Let $f$ be a function slice hyperholomorphic in $\Omega$ with a spherical Laurent expansion at $p$ convergent in a subset of $\Omega$. Then we say that $h=\bar{D}f$ has a Laurent polyanalytic spherical series of order $2$ at an arbitrarily fixed point $p$.
\end{definition}

By using the exact form of the Laurent spherical series we can provide a compact expression of the Laurent polyanalytic spherical series of order $2$ of $h=\bar{D}f$.

\begin{theorem}
	\label{Polysphreal}
Let $f$ be a slice hyperholomorphic function in a neighbourhood of $p\in\mathbb H$ containing the Cassini shell $U(p,R_1,R_2)$, for some $0<R_1<R_2$, where it admits the spherical Laurent expansion \eqref{Lau2} with coefficients $\{a_n\}_{n\in\mathbb Z}\subseteq\mathbb H$, converging (absolutely and uniformly on compact subsets) in $U(p,R_1,R_2)$. Let $h=\bar Df$ be the axially polyanalytic function of order $2$ in a neighbourhood of $p$. Then $h$ admits the following Laurent polyanalytic
spherical expansion of order $2$:
	\begin{equation}
		\label{Laupolyreal}
		\begin{split}
			h(q)=\bar D f(q)=\ &4\sum_{n\in\mathbb{Z}} n\,(q-p_0)\,Q^{n-1}_p(q)\left[a_{2n}+(q-p)a_{2n+1}\right]\\
			&+4\sum_{n\in\mathbb{Z}}(q_0-p_0)\,\mathcal{L}_n(q,p)\left[a_{2n}+(\bar q-p)a_{2n+1}\right]
			+4\sum_{n\in\mathbb{Z}}Q^{n}_p(q)\,a_{2n+1},
		\end{split}
	\end{equation}
	with $\mathcal{L}_n(q,p)$ as in \eqref{key4}.
\end{theorem}
\begin{proof}
	Apply $\bar D$ term by term to \eqref{Lau2}. The exchange between the summation and the operator follows as Theorem \ref{harmospherL}. The expression \eqref{Laupolyreal} follows by \eqref{Dbar}, \eqref{Dbar1}. It remains to consider the real points. Let
	$q=q_0\in U(p,R_1,R_2)\cap\mathbb R$. Since $Q_p(q_0)\in\mathbb R$, we have
	$$
	\mathcal L_n(q_0,p)=nQ_p^{n-1}(q_0),\qquad n\in\mathbb Z.
	$$
By \eqref{realDbar} and \eqref{realDbar1}, which remain valid for every $n\in\mathbb{Z}$, the right-hand side of \eqref{Laupolyreal} coincides precisely with the series obtained by applying $\bar D$ term by term to the Laurent spherical expansion \eqref{Lau2}.
\end{proof}

We now show that the Laurent polyanalytic spherical series of order $2$ has the same convergence set  as the slice hyperholomorphic Laurent spherical series.

\begin{proposition}
	With the notations in Theorem \ref{Polysphreal} let $ \{a_n\}_{n \in \mathbb{Z}} \subseteq \mathbb{H}$, set
	$$ R_1:= \limsup_{n \to \infty}  |a_{-n}|^{\frac{1}{n}}, \qquad \hbox{and} \qquad \frac{1}{R_2}:= \limsup_{n \to \infty}  |a_n|^{\frac{1}{n}},$$
	with $R_1<R_2$.	The Laurent polyanalytic spherical series of order $2$ of $h=\bar{D}f$ converges (absolutely and uniformly on compact subsets) in the Cassini shell:
	$$ U(p,R_1,R_2):= \{q \in \mathbb{H}\, : \, R_1^2 < | (q-p_0)^2+p_1^2|<R_2^2\},$$
	where $p=p_0+Ip_1 \in \mathbb{H}$, $I \in \mathbb{S}$.	
\end{proposition}
\begin{proof}
Let $K$ be a compact subset of the Cassini shell $U(p,R_1,R_2)$. Then there exist $r_1,r_2$ such that $R_1<r_1<r_2<R_2$ and $r_1^{2n}<|Q_p^n(q)|<r_2^{2n}$ for $q\in K$. Choose $R'$ and $R''$ such that $R_1<R'<r_1<r_2<R''<R_2$. By the assumptions on the coefficients, we have $|a_n|\leq (1/R'')^n$ and $|a_{-n}|\leq (R')^n$ for all sufficiently large $n$. By Theorem \ref{Polysphreal} we can write the Laurent polyanalytic spherical series of order $2$ as
	\begin{eqnarray}
		\nonumber
		h(q)=\bar{D}f(q)&=& 4 \sum_{n=0}^{\infty} (n+1)\mathcal{L}_{n+1}(q,p)(q_0-p_0) \left[a_{2n+2}+(\bar{q}-p)a_{2n+3} \right]\\
		\nonumber
		&&+ 4 \sum_{n=0}^{\infty} (n+1)Q_p^{n}(q)(q-p_0) \left[a_{2n+2}+(q-p)a_{2n+3} \right] +4 \sum_{n=0}^{\infty}Q_p^n(q)a_{2n+1}\\
		\nonumber
		&&+4 \sum_{n=1}^{\infty} \mathcal{L}_{-n}(q,p)(q_0-p_0)\left[a_{-2n}+(\bar{q}-p)a_{-2n+1}\right]\\
		\label{Tpoly}
		&& -4 \sum_{n=1}^{\infty} n Q_p^{-n-1}(q)(q-p_0)\left[a_{-2n}+(q-p)a_{-2n+1}\right]+ 4 \sum_{n=1}^{\infty}Q_p^{-n}(q)a_{-2n+1}.
	\end{eqnarray}
	The convergence of the Taylor part of \eqref{Tpoly} follows by Proposition \ref{convepoly}, so we focus on the remaining part. By Lemma \ref{r1} and Lemma \ref{newres} we have
	\begin{eqnarray*}
		|\mathcal{L}_{-n}(q,p)||q_0-p_0|\left[|a_{-2n}|+|\bar{q}-p||a_{-2n+1}|\right] & \leq &
		\frac{n}{r_1^2 R'}  \left(\frac{R'}{r_1}\right)^{2n} \sqrt{r_2^2 +p_1^2}\\
		&& \times \left( R'+\left(p_1+\sqrt{r_2^2+p_1^2}\right)\right):= A_n,
	\end{eqnarray*}
where we used $|\mathcal{L}_{-n}(q,p)|\leq n|Q_p(q)|^{-n-1}$,  see \eqref{key4}, and $|\bar q-p|\leq\sqrt{r_2^2+p_1^2}+p_1$, see \eqref{newin}.
	Similarly, we have
	$$ n |Q_p^{-n-1}(q)||q-p_0|\left[|a_{-2n}|+|q-p||a_{-2n+1}|\right] \leq A_n.$$
	and
	$$ |Q_p^{-n}(q)a_{-2n+1}| \leq  \frac{1}{R'} \left(\frac{R'}{r_1}\right)^{2n}:= B_n,$$
	so, we have
	$$ | \bar{D}f(q)| \leq  4\sum_{n=1}^{\infty} (2A_n+B_n).$$
	By the ratio test the series $\sum_{n=1}^{\infty} (A_n+B_n)$ is convergent. This concludes the proof.
\end{proof}

In suitable subsets of $\mathbb H$ we can write an expansion of the spherical Laurent polyanalytic expansion in terms of the functions $R_n(q)$ defined in \eqref{polyR} and the polyanalytic functions $\mathcal{P}_{2,n}(q)$ defined in \eqref{poly4b}.
\begin{proposition}
	\label{polyN4}
	Let $q, p=p_0+Ip_1 \in \mathbb{H}$. We assume that $f$ is a slice hyperholomorphic function that admits a spherical Laurent expansion at an arbitrarily fixed point $p$ as in \eqref{Lau2} with coefficients $ \{a_n\}_{n \in \mathbb{Z}} \subseteq \mathbb{H}$,  convergent (absolutely and uniformly on the compact subsets) in  $U(p,R_1,R_2)$, $0<R_1<R_2$. Assume that the region $|p_1|<|q-p_0|$ meets $U(p,R_1,R_2)$; then, on that intersection, we can write the Laurent polyanalytic spherical series of order $2$ of $h=\bar Df$ as
	\begin{eqnarray}
		\nonumber
		h(q)=\bar{D}f(q)&=&C_T(q,p)-2 \sum_{n=1}^{\infty} \sum_{k=0}^{\infty} (-1)^k \binom{n+k-1}{k} R_{2(n+k)}(q-p_0)|q-p_0|^{-4(n+k)-2}p_1^{2k}a_{-2n}\\
		\nonumber
		&&-2 \sum_{n=1}^{\infty} \sum_{k=0}^{\infty} (-1)^k \binom{n+k-1}{k} |q-p_0|^{-4(n+k)}R_{2(n+k)-1}(q-p_0) p_1^{2k}a_{-2n+1}\\
		\label{Laurentpoly}
		&& +2 \sum_{n=1}^{\infty} \sum_{k=0}^{\infty} (-1)^k \binom{n+k-1}{k} R_{2(n+k)}(q-p_0)|q-p_0|^{-4(n+k)-2}p_1^{2k+1}Ia_{-2n+1}.
	\end{eqnarray}
	where $C_T(q,p)$ is the Taylor part of the Laurent polyanalytic spherical series of order $2$, given in  formula \eqref{Taylorpoly2}. If instead the region $|q-p_0|<|p_1|$ meets $U(p,R_1,R_2)$, then on that intersection another representation of the Laurent spherical polyanalytic series of order $2$ is given by
	\begin{eqnarray}
		\nonumber
		h(q)=\bar{D}f(q)&=& C_T(q,p)-4 \sum_{n=1}^{\infty}\sum_{k=0}^{\infty} (-1)^k \binom{n+k}{k}n \mathcal{P}_{2,2k+1}(q-p_0) p_1^{-2(n+k+1)}a_{-2n}\\
		\nonumber
		&& +2 \sum_{n=1}^{\infty}\sum_{k=0}^{\infty}(-1)^k \binom{n+k-1}{k} (2k+1) \mathcal{P}_{2,2k}(q-p_0) p_1^{-2(n+k)}a_{-2n+1}\\
		\label{LL}
		&&+4\sum_{n=1}^{\infty}\sum_{k=0}^{\infty} (-1)^k \binom{n+k}{k} n\mathcal{P}_{2, 2k+1}(q-p_0) p_1^{-2(n+k+1)+1}Ia_{-2n+1},
	\end{eqnarray}
	
\end{proposition}
\begin{proof}
	We apply the operator $\bar{D}$ to a function $f$ expanded as in \eqref{Las}.
	The expression of the Taylor part $C_T(q,p)$ follows by Theorem \ref{Taylorpoly}. Let $p=p_0+Ip_1$. We first assume that the region $|p_1|<|q-p_0|$ has a nonempty intersection with the domain of convergence of the spherical Laurent series. Hence, we may expand $Q_p^{-n}(q)$ as in \eqref{negLa}. Then, using \eqref{appbar}, we obtain
	\begin{eqnarray}
		\nonumber
		\bar{D}\left(Q_p^{-n}(q) \right)&=& \sum_{k=0}^{\infty} (-1)^k \binom{n+k-1}{k} \bar{D}(q-p_0)^{-2(n+k)} p_1^{2k}\\
		\label{polyN2}
		&=& -2 \sum_{k=0}^{\infty} (-1)^k \binom{n+k-1}{k} R_{2(n+k)}(q-p_0)|q-p_0|^{-4(n+k)-2}p_1^{2k}.
	\end{eqnarray}
	Similarly, by \eqref{negLa} and by using \eqref{appbar} we have
	\begin{eqnarray}
		\nonumber
		\bar{D}\left(Q_p^{-n}(q)(q-p)\right)&=& \sum_{k=0}^{\infty}(-1)^k \binom{n+k-1}{k} \bar{D} \left[(q-p_0)^{-2(n+k)+1} \right] p_1^{2k}\\
		\nonumber
		&& -\sum_{k=0}^{\infty}(-1)^k \binom{n+k-1}{k} \bar{D} \left[(q-p_0)^{-2(n+k)} \right] p_1^{2k+1}I\\
		\nonumber
		&=& -2 \sum_{k=0}^{\infty} (-1)^k \binom{n+k-1}{k}  R_{2(n+k)-1}(q-p_0)|q-p_0|^{-4(n+k)} p_1^{2k}\\
		\label{polyN3}
		&& +2 \sum_{k=0}^{\infty} (-1)^k \binom{n+k-1}{k} R_{2(n+k)}(q-p_0)|q-p_0|^{-4(n+k)-2}p_1^{2k+1}I.
	\end{eqnarray}
	Formula \eqref{Laurentpoly} follows by putting together \eqref{polyN2} and \eqref{polyN3}. \\Now, we assume that the intersection of the region $|q-p_0| <|p_1|$ with the domain of convergence of the spherical Laurent is nonempty. The expression of the Taylor part follows by Theorem \ref{Taylorpoly}, so we focus on applying the operator $\bar{D}$ to the other part of \eqref{Las}. We can use the expansions of $Q_p^{-n}(q)$ as in \eqref{negLa11} and by \eqref{NNpoly} we have
	\begin{eqnarray}
		\nonumber
		\bar{D}(Q_p^{-n}(q))&=& \sum_{k=0}^{\infty} (-1)^k \binom{n+k-1}{k} \bar{D} (q-p_0)^{2k} p_1^{-2(n+k)}\\
		\nonumber
		&=&4 \sum_{k=1}^{\infty} (-1)^k \binom{n+k-1}{k} k \mathcal{P}_{2,2k-1}(q-p_0) p_1^{-2(n+k)}\\
		\nonumber
		&=&4n \sum_{k=0}^{\infty} (-1)^k \binom{n+k-1}{k-1} \mathcal{P}_{2,2k-1}(q-p_0) p_1^{-2(n+k)}\\
		\label{Laupoly}
		&=&-4n \sum_{k=0}^{\infty} (-1)^k \binom{n+k}{k} \mathcal{P}_{2,2k+1}(q-p_0) p_1^{-2(n+k+1)},
	\end{eqnarray}
	and
	\begin{eqnarray}
		\nonumber
		\bar{D}\left(Q_p^{-n}(q)(q-p)\right)&=& 2 \sum_{k=0}^{\infty}(-1)^k \binom{n+k-1}{k} (2k+1) \mathcal{P}_{2,2k}(q-p_0) p_1^{-2(n+k)}\\
		\nonumber
		&&-4 \sum_{k=1}^{\infty} (-1)^k \binom{n+k-1}{k} k \mathcal{P}_{2,2k-1}(q-p_0) p_1^{-2(n+k)+1}I\\
		\nonumber
		&=&2 \sum_{k=0}^{\infty}(-1)^k \binom{n+k-1}{k} (2k+1) \mathcal{P}_{2,2k}(q-p_0) p_1^{-2(n+k)}\\
		\label{Laupoly3}
		&&+4n \sum_{k=0}^{\infty} (-1)^k \binom{n+k}{k} \mathcal{P}_{2,2k+1}(q-p_0) p_1^{-2(n+k)-1}I
	\end{eqnarray}
	Formula \eqref{LL} follows by putting together \eqref{Laupoly} and \eqref{Laupoly3}.
\end{proof}

\begin{remark}
	If we take $p=0$ in \eqref{Laurentpoly}, by Remark \ref{zeropoly} and \eqref{poly1} we have
	$$
	C_T(q,0)=2 \sum_{n=0}^{\infty} (n+1) \mathcal{P}_{2,n}(q)a_{n+1}.\\
	$$
	The only nonzero terms in the other series appearing in \eqref{Laurentpoly} correspond to $k=0$. Hence, they reduce to
	$$
	-2 \sum_{n=1}^{\infty} R_{2n}(q)|q|^{-4n-2}a_{-2n}-2 \sum_{n=1}^{\infty} R_{2n-1}(q) |q|^{-4n}a_{-2n+1}=-2 \sum_{n=1}^{\infty} R_n(q)|q|^{-2(n+1)}a_{-n}.$$
	Thus we get back to the polyanalytic Laurent series of order $2$ in a neighbourhood of the origin obtained in \eqref{polyN1}.
\end{remark}

\begin{proposition}
	Let  $p=p_0+Ip_1\in \mathbb{H}$ be a fixed point. We assume that $f$ is a slice hyperholomorphic function that admits a spherical Laurent expansion at $p$ as in \eqref{Lau2}, with coefficients  $ \{a_n\}_{n \in \mathbb{Z}} \subseteq \mathbb{H}$,  convergent (absolutely and uniformly on the compact subsets) in a Cassini shell $U(p,R_1,R_2)$, $0<R_1<R_2$. Then, for every $q\in U(p,R_1,R_2)\setminus\mathbb{R}$ with $|q-p_0|\not=|p_1|$, we can write the Laurent polyanalytic spherical series of order $2$ as
		\begingroup\allowdisplaybreaks
	\begin{eqnarray}
		\nonumber
		h(q)=\bar{D}f(q)&=& 2\partial_{q_0} \left[ \sum_{n \in \mathbb{Z}} Q_p^n(q) \left( a_{2n}+(q-p)a_{2n+1} \right)\right]\\
		\label{Laupoly4}
		&& +(\underline{q})^{-1} \left( \sum_{n \in \mathbb{Z}}Q_p^n(q) \left( a_{2n}+(q-p)a_{2n+1} \right)-\sum_{n \in \mathbb{Z}}Q_p^n(\bar{q}) \left( a_{2n}+(\bar{q}-p)a_{2n+1} \right) \right).
	\end{eqnarray}
\endgroup
\end{proposition}
\begin{proof}
We first focus on the Taylor part of $\bar{D}f(q)$, which we denote by $C_T(q,p)$. By Proposition~\ref{TaylorPoly1}, we have
	\begingroup\allowdisplaybreaks
\begin{eqnarray}
	\nonumber
C_T(q,p)&=&2 \partial_{q_0}\left[ \sum_{n=0}^{\infty} Q_p^n(q) a_{2n}+Q_p^{n}(q)(q-p) a_{2n+1} \right] \\
\nonumber
&&+( \underline{q})^{-1} \left( \sum_{n=0}^{\infty} Q_p^n(q) a_{2n}+Q_p^n(q)(q-p) a_{2n+1} \right.\\
\label{LLP}
&& \left.-\sum_{n=0}^{\infty} Q_p^n(\bar{q}) a_{2n}-Q_p^n(\bar{q})(\bar{q}-p) a_{2n+1} \right).
\end{eqnarray}
\endgroup
We now focus on the two series in \eqref{Laurentpoly} and \eqref{LL}. Since these series converge in different domains of convergence, we divide the proof into two cases. Let $p=p_0+Ip_1$. In the sequel of the proof, we assume that the regions $|p_1|<|q-p_0|$ and $|q-p_0|<|p_1|$ have a nonempty intersection with the domain of convergence of the spherical Laurent series.
\newline
\newline
\emph{Case I: $|p_1|<|q-p_0|$}.
\newline
\newline
We focus on the first series of \eqref{Laurentpoly}. By \eqref{closedN3} and \eqref{negLa} we have
	\begingroup\allowdisplaybreaks
	\begin{eqnarray}
		\nonumber
		&&-2 \sum_{n=1}^{\infty}\sum_{k=0}^{\infty} (-1)^k \binom{n+k-1}{k} R_{2(n+k)}(q-p_0)|q-p_0|^{-4(n+k)-2}p_1^{2k}a_{-2n}\\
		\nonumber
		&=&-2 \sum_{n=1}^{\infty} \sum_{k=0}^{\infty}(-1)^k \binom{n+k-1}{k} |q-p_0|^{-4(n+k)-2} \left[2(n+k)(\bar{q}-p_0)^{2(n+k)+1}+   \right.\\
		\nonumber
		&&\left. + \frac{(\underline{q})^{-1}|q-p_0|^2}{2}\left( (q-p_0)^{2(n+k)}-(\bar{q}-p_0)^{2(n+k)} \right) \right] p_1^{2k}a_{-2n}\\
		\nonumber
		&=& 2  \sum_{n=1}^{\infty}\sum_{k=0}^{\infty} (-1)^k \binom{n+k-1}{k} (-2(n+k)) (q-p_0)^{-2(n+k)-1}p_1^{2k}a_{-2n}\\
		\nonumber
		&& -(\underline{q})^{-1} \sum_{n=1}^{\infty} \sum_{k=0}^{\infty} (-1)^k \binom{n+k-1}{k} \left[(\bar{q}-p_0)^{-2(n+k)}-(q-p_0)^{-2(n+k)} \right]p_1^{2k}a_{-2n}\\
		\label{Laupoly5}
		&=&2 \partial_{q_0}  \sum_{n=1}^{\infty}\left[Q_p^{-n}(q)\right]a_{-2n}-(\underline{q})^{-1}  \sum_{n=1}^{\infty}\left[Q_p^{-n}(\bar{q})-Q_p^{-n}(q) \right]a_{-2n}.
	\end{eqnarray}
\endgroup
Now, we consider the second summation of \eqref{Laurentpoly}. By \eqref{closedN3} and \eqref{negLa} we have
\begingroup\allowdisplaybreaks
\begin{eqnarray}
	\nonumber
	&&-2 \sum_{n=1}^{\infty}\sum_{k=0}^{\infty} (-1)^k \binom{n+k-1}{k} |q-p_0|^{-4(n+k)}R_{2(n+k)-1}(q-p_0) p_1^{2k}a_{-2n+1}\\
	\nonumber
	&&=-2 \sum_{n=1}^{\infty}\sum_{k=0}^{\infty} (-1)^k \binom{n+k-1}{k} |q-p_0|^{-4(n+k)}\left[\left(2(n+k)-1 \right)(\bar{q}-p_0)^{2(n+k)} \right. \\
	\nonumber
	&&  \left. + \frac{(\underline{q})^{-1} |q-p_0|^2}{2}\left( (q-p_0)^{2(n+k)-1}-(\bar{q}-p_0)^{2(n+k)-1} \right) \right]a_{-2n+1}\\
	\nonumber
	&&= 2 \sum_{n=1}^{\infty}\sum_{k=0}^{\infty} (-1)^k \binom{n+k-1}{k}  (q-p_0)^{-2(n+k)} (1- 2(n+k))p_1^{2k} a_{-2n+1}\\
	\nonumber
	&&- (\underline{q})^{-1} \sum_{n=1}^{\infty}\sum_{k=0}^{\infty} (-1)^k \binom{n+k-1}{k}\left((\bar{q}-p_0)^{-2(n+k)+1}-(q-p_0)^{-2(n+k)+1} \right)a_{-2n+1}\\
	\nonumber
	&&= 2 \partial_{q_0}\left(\sum_{n=1}^{\infty} \sum_{k=0}^{\infty} (-1)^k \binom{n+k-1}{k}  (q-p_0)^{-2(n+k)+1} p_1^{2k} a_{-2n+1} \right)\\
	\nonumber
	&&- (\underline{q})^{-1} \sum_{n=1}^{\infty} \sum_{k=0}^{\infty} (-1)^k \binom{n+k-1}{k}\left((\bar{q}-p_0)^{-2(n+k)+1}-(q-p_0)^{-2(n+k)+1} \right)a_{-2n+1}\\
\nonumber
	&&  =   2 \sum_{n=1}^{\infty} \partial_{q_0} \left[Q_p^{-n}(q)(q-p_0)\right]a_{-2n+1}\\
		\label{Laupoly6}
	&&\qquad -(\underline{q})^{-1}\sum_{n=1}^{\infty} \left[ Q_p^{-n}(\bar{q})(\bar{q}-p_0)-Q_p^{-n}(q)(q-p_0) \right]a_{-2n+1}.
\end{eqnarray}
\endgroup
Finally, we compute the third summation of \eqref{Laurentpoly}. By \eqref{Laupoly5} we have
	\begingroup\allowdisplaybreaks
	\begin{eqnarray}
		\nonumber
		&&2 \sum_{n=1}^{\infty}\sum_{k=0}^{\infty} (-1)^k \binom{n+k-1}{k} R_{2(n+k)}(q-p_0)|q-p_0|^{-4(n+k)-2}p_1^{2k+1}Ia_{-2n+1}\\
		\label{Laupoly7}
		&=&-2 \sum_{n=1}^{\infty}\partial_{q_0} Q_p^{-n}(q)p_1Ia_{-2n+1}+(\underline{q})^{-1}\sum_{n=1}^{\infty} \left[Q_p^{-n}(\bar{q})-Q_p^{-n}(q) \right]p_1Ia_{-2n+1}.
	\end{eqnarray}
\endgroup

Hence formula \eqref{Laupoly4} follows by putting together \eqref{LLP}, \eqref{Laupoly5}, \eqref{Laupoly6} and \eqref{Laupoly7}.
\newline
\newline
\emph{Case II: $|q-p_0|<|p_1|$}.
\newline
\newline
We consider the first summation of \eqref{LL}. By \eqref{poly5} and \eqref{negLa11} we have
	\begingroup\allowdisplaybreaks
\begin{eqnarray}
	\nonumber
&&-4 \sum_{n=1}^{\infty}\sum_{k=0}^{\infty} (-1)^k \binom{n+k}{k}n \mathcal{P}_{2,2k+1}(q-p_0) p_1^{-2(n+k+1)}a_{-2n}\\
\nonumber
&&= -2\sum_{n=1}^{\infty}\sum_{k=0}^{\infty} (-1)^k \binom{n+k}{k} \frac{n}{k+1} (q-p_0)^{2k+1} p_1^{-2(n+k+1)}a_{-2n}\\
\nonumber
&&+(\underline{q})^{-1}\sum_{n=1}^{\infty}\sum_{k=0}^{\infty} (-1)^k \binom{n+k}{k} \frac{n}{k+1} \left[(\bar{q}-p_0)^{2(k+1)}-(q-p_0)^{2(k+1)} \right] p_1^{-2(n+k+1)}a_{-2n}\\
\nonumber
&&= -2\sum_{n=1}^{\infty}\sum_{k=0}^{\infty} (-1)^k \binom{n+k}{k+1} \partial_{q_0}(q-p_0)^{2(k+1)} p_1^{-2(n+k+1)}a_{-2n}\\
\nonumber
&&+(\underline{q})^{-1}\sum_{n=1}^{\infty}\sum_{k=0}^{\infty} (-1)^k \binom{n+k}{k+1} \left[(\bar{q}-p_0)^{2(k+1)}-(q-p_0)^{2(k+1)} \right] p_1^{-2(n+k+1)}a_{-2n}\\
\nonumber
&&= 2\sum_{n=1}^{\infty}\sum_{k=0}^{\infty} (-1)^k \binom{n+k-1}{k} \partial_{q_0}(q-p_0)^{2k} p_1^{-2(n+k)}a_{-2n}\\
\nonumber
&&-(\underline{q})^{-1}\sum_{n=1}^{\infty}\sum_{k=0}^{\infty} (-1)^k \binom{n+k-1}{k} \left[(\bar{q}-p_0)^{2k}-(q-p_0)^{2k} \right] p_1^{-2(n+k)}a_{-2n}\\
\nonumber
&=& 2 \sum_{n=1}^{\infty}  \partial_{q_0} \left[Q_p^{-n}(q)\right]a_{-2n}\\
\label{PP}
&&- (\underline{q})^{-1} \sum_{n=1}^{\infty}\left[Q_p^{-n}(\bar{q})-Q_p^{-n}(q)\right]a_{-2n}.
\end{eqnarray}
\endgroup

The second summation of \eqref{LL} can be dealt with using \eqref{poly5} and \eqref{negLa11}, so we have
	\begingroup\allowdisplaybreaks
\begin{eqnarray}
\nonumber
&&2 \sum_{n=1}^{\infty}\sum_{k=0}^{\infty}(-1)^k \binom{n+k-1}{k} (2k+1) \mathcal{P}_{2,2k}(q-p_0) p_1^{-2(n+k)}a_{-2n+1}\\
\nonumber
&=&  \sum_{n=1}^{\infty}\sum_{k=0}^{\infty}(-1)^k \binom{n+k-1}{k} \partial_{q_0} (q-p_0)^{2k+1}p_1^{-2(n+k)}a_{-2n+1}\\
\nonumber
&& -(\underline{q})^{-1}\sum_{n=1}^{\infty}\sum_{k=0}^{\infty}(-1)^k \binom{n+k-1}{k} \left[(\bar{q}-p_0)^{2k+1}-(q-p_0)^{2k+1}\right]p_1^{-2(n+k)}a_{-2n+1}\\
\nonumber
&=& 2\sum_{n=1}^{\infty} \partial_{q_0} \left[Q_p^{-n}(q)(q-p_0)\right]a_{-2n+1}\\
\label{PP1}
&&-(\underline{q})^{-1}\sum_{n=1}^{\infty} \left[Q_p^{-n}(\bar{q})(\bar{q}-p_0)-Q_p^{-n}(q)(q-p_0)\right]a_{-2n+1}
\end{eqnarray}
\endgroup
Finally, we focus on the third summation of \eqref{LL}. By \eqref{PP} we have
\begin{eqnarray}
	\nonumber
&&4\sum_{n=1}^{\infty}\sum_{k=0}^{\infty} (-1)^k \binom{n+k}{k} n\mathcal{P}_{2, 2k+1}(q-p_0) p_1^{-2(n+k+1)+1}Ia_{-2n+1}\\
\label{PP2}
&=& -2 \sum_{n=1}^{\infty}  \partial_{q_0} \left[Q_p^{-n}(q)\right]p_1Ia_{-2n+1}+(\underline{q})^{-1} \sum_{n=1}^{\infty}\left[Q_p^{-n}(\bar{q})-Q_p^{-n}(q)\right]p_1Ia_{-2n+1}.
\end{eqnarray}

Formula \eqref{Laupoly4} follows by putting together \eqref{LLP}, \eqref{PP}, \eqref{PP1} and \eqref{PP2}.
\end{proof}

\section{Applications to operator theory}\label{CONCREMK}

Our work inserts in the field of the generalizations of the concept of holomorphicity in dimensions greater than one. Indeed, the set of holomorphic functions $f:\Omega \subseteq \mathbb{C} \to \mathbb{C}$ of one complex variable, denoted by $\mathcal{O}(\Omega)$, and the associated functional calculus admit extensions that are discussed below.

(I) The system of Cauchy-Riemann equations for functions $f:\Pi\subseteq \mathbb{C}^n \to \mathbb{C}$ gives the theory of holomorphic functions in several complex variables.

The exploration of spectral theory based on the theory of several complex variables and their Cauchy formula gives rise to the so-called holomorphic functional calculus for $n$-tuples of operators and it is based on  Taylor's joint spectrum and its further developments.
This calculus was initiated in \cite{TAYLOR1,TAYLOR2,TAYLOR3}, but see also \cite{BOUR1,BOUR2}.

(II) The holomorphicity of vector fields is linked with quaternionic-valued functions and, more broadly, with Clifford algebra-valued functions. The Fueter-Sce-Qian theorem yields two distinct extensions namely two classes of hyperholomorphic functions as discussed in preceding sections.

Both these classes of functions possess a Cauchy formula applicable to defining functions of quaternionic operators or of $n$-tuples of (non-commuting) operators. More precisely:

(II-A) The Cauchy formula of slice hyperholomorphic functions leads to the $S$-functional calculus for quaternionic linear operators or for $n$-tuples of non-commuting operators, and more generally for Clifford operators. This calculus is grounded in the concept of the $S$-spectrum, on which the spectral theorem for quaternionic operators and for Clifford operators are based.

(II-B) The Cauchy formula of monogenic functions gives rise to the monogenic functional calculus, based on the monogenic spectrum. This calculus has also links with the Weyl functional calculus.

\medskip
The functions systematically explored in this paper, along with their corresponding functional calculi, can collectively be referred to as quaternionic fine structures within the framework of the spectral theory based on the $S$-spectrum and belong to the developments in (II-A).
The integral representations of the functions of the fine structures are employed to introduce new functional calculi designed for quaternionic operators, encompassing both bounded and unbounded operators. It is worthwhile to note that, recently, these calculi have been extended to sectorial type operators.
\\ As we have discussed in this work, the four distinct classes of functions in the quaternionic fine structures and their functional calculi are: slice hyperholomorphic functions (resulting in the $S$-functional calculus), axially harmonic functions (leading to the $Q$-functional calculus), axially polyanalytic functions of order $2$ (resulting in the $P_2$-functional calculus), and axially Fueter regular functions (leading to the $F$-functional calculus). These calculi are based on the Cauchy formula for slice hyperholomorphic functions given by
\begin{equation*}
f(q)=\frac{1}{2\pi}\int_{\partial(U\cap\mathbb{C}_I)} S_L^{-1}(p,q)
dp_If(p),
\end{equation*}
where the (left) Cauchy kernel is considered in the second form, namely
$$
S_L^{-1}(p,q):=(p-\overline{q})(p^2-2pq_0+|q|^2)^{-1}.
$$
Applying the operators of the quaternionic fine structure of Dirac type,
that is $D$, $\bar{D}$ and $\Delta$ to the Cauchy kernel $S_L^{-1}(p,q)$ we obtain
\begin{equation}\label{KENELSFINE}
\mathcal{Q}_{c,p}^{-1}(q)=-\frac{1}{2}DS_L^{-1}(p,q),\qquad P_2^L(p,q)=\bar{D}S_L^{-1}(p,q),\qquad F_L(p,q)=\Delta S_L^{-1}(p,q),
\end{equation}
and these new kernels give an integral representation of the functions of the quaternionic fine structure (see all the details in the
Subsection \ref{INTRAPP} on the integral representations of the functions of the fine structure of Dirac type).
We now recall the definition of $F$-spectrum associated with the spectral theory on the $S$-spectrum in its commutative version, that is associated with the Cauchy kernels in the second form.

\begin{definition}[$F$-spectrum]\label{defi_F_spectrum}
Let $T=T_0+T_1e_1+T_2e_2+T_3e_3$  be a bounded quaternionic operator with commuting components $T_\ell$, for $\ell=0,...,3$ and let $\overline{T}:=T_0-T_1e_1-T_2e_2-T_3e_3$ be its conjugate and set $|T|^2=T\overline{T}=\sum_{\ell=0}^3T^2_\ell$. According to the invertibility of the operator
$$
\mathcal{Q}_{c,p}(T):=p^2-2pT_0+|T|^2
$$
 we define the \textit{$F$-resolvent set}
\begin{equation}\label{Eq_F_resolvent_set}
\rho_F(T):=\Set{p\in\mathbb{H} | \mathcal{Q}_{c,p}(T)\text{ is bijective}},
\end{equation}
and the \textit{$F$-spectrum} as the complement
\begin{equation}\label{Eq_F_spectrum}
\sigma_F(T):=\mathbb{H}\setminus\rho_F(T).
\end{equation}
\end{definition}

We observe that when dealing with bounded operators $T$, the $S$-spectrum
\begin{equation*}
\sigma_S(T):=\Set{p\in\mathbb{H} | Q_p(T)\text{ is not bijective}},\quad\text{with }Q_p(T):=T^2-2p_0T+|p|^2,
\end{equation*}
coincides with the $F$-spectrum in \eqref{Eq_F_spectrum}. It is important to note that the $F$-spectrum can be viewed as a commutative counterpart of the $S$-spectrum, which is more general and it does not require that the components of the operator $T$ commute among themselves.

\medskip
The above definitions are the starting point to define:
\\ - the
$S$-functional calculus
$$
f(T)=\frac{1}{2\pi}\int_{\partial(U\cap\mathbb{C}_I)}S_L^{-1}(p,T)dp_If(p),
$$
- the $Q$-functional calculus, often called the harmonic functional calculus in the quaternionic setting
$$
 Df(T)=-\frac{1}{\pi}\int_{\partial(U\cap\mathbb{C}_I)}\mathcal{Q}_{c,p}^{-1}(T)dp_If(p),
$$
- the $P_2$-functional calculus
$$
\bar{D}f(T)=\frac{1}{2\pi}\int_{\partial(U\cap\mathbb{C}_I)}P_2^L(p,T)dp _If(p),
$$
- the $F$-functional calculus
$$
\Delta f(T)=\frac{1}{2\pi}\int_{\partial(U\cap\mathbb{C}_I)}F_L(p,T)dp_If(p),
$$
where $T$ is a bounded quaternionic operator with commuting components.
The open set $U$ contains the $S$-spectrum of $T$ and all the above functional calculi
depend neither on $U$ nor on the imaginary unit $I\in \mathbb{S}$, and in the case of the
$Q$-functional calculus, $P_2$-functional calculus and the $F$-functional calculus, they do not depend on the kernels of the operators $D$, $\bar{D}$ and $\Delta$, respectively.
The resolvent operators in the fine structure, i.e.,
$S_L^{-1}(p,T)$, $\mathcal{Q}_{c,p}^{-1}(T)$, $P_2^L(p,T)$ and $F_L(p,T)$ are obtained from the Cauchy kernel
$S_L^{-1}(p,q)$ and the kernels $\mathcal{Q}_{c,p}^{-1}(q)$,  $P_2^L(p,q)$ and $F_L(p,q)$ of the functions of the fine structure given in (\ref{KENELSFINE}) by replacing $q$ with the operator $T$, respectively.

\medskip
It is worthwhile to note that the $F$-functional calculus is a monogenic functional calculus in the spirit of the Cauchy formula of monogenic functions, but it is based on slice hyperholomorphic functions,  via an integral transform, and on the $S$-spectrum (used instead of the monogenic spectrum  to define a calculus via the monogenic Cauchy formula).

\medskip
We remark that the fine structure on the $S$-spectrum is more suitable for operators with commuting components, whereas the $S$-functional calculus is naturally defined also for operators with noncommuting components.
The spectral theory on the $S$-spectrum
 and all its variations arising from the fine structures on the $S$-spectrum
constitute natural spectral theories for vector operators.
An extension of the fine structures to Clifford algebras has been developed in \cite{Fivedim} and \cite{CDP25}.

\medskip
We finally highlight that the extension of the holomorphic functional calculus to sectorial operators leads to the $H^\infty$-functional calculus, initially introduced in the complex context in the paper \cite{McI1} and further discussed in the books \cite{Haase, HYTONBOOK1, HYTONBOOK2}. This calculus has proven valuable in addressing boundary value problems, as evidenced by its applications highlighted in \cite{MC10,MC97,MC06}. The boundedness of the $H^\infty$-functional calculus relies on appropriate quadratic estimates while unbounded operators for a specific class of functions are discussed  in \cite{CDP23}. The $H^\infty$-functional calculus for the quaternionic fine structures is treated in \cite{CPS,DPS}.
It is worth noting that the $H^\infty$-functional calculus is defined in the monogenic case as well, as detailed in \cite{JM}. This approach, pioneered by A. McIntosh and collaborators, is further explored in the books \cite{JBOOK,TAOBOOK}.

\end{document}